\documentclass[a4paper,11pt]{amsart}
\usepackage{a4wide,enumerate,xcolor, appendix, lscape}
\usepackage[colorlinks, linkcolor=blue]{hyperref}
\usepackage{amsmath}
\usepackage{tikz}
\usepackage{tikz-cd}
  \usepackage{mathtools}
\usetikzlibrary{matrix}
\usetikzlibrary{positioning}

\usepackage{amssymb}
\allowdisplaybreaks
\numberwithin{equation}{section}
\usepackage{enumitem}
\usepackage{graphicx}
\usepackage{multicol}
\usepackage{tcolorbox}
\tcbuselibrary{skins,breakable}

\newtcolorbox{authornotebox}{
  colback=red!5,
  colframe=red!75!black,
  coltitle=white,
  colbacktitle=red!75!black,
  fonttitle=\bfseries,
  title={ACTION NEEDED},
  breakable,
  boxrule=1pt,
  arc=2pt
}

\definecolor{change}{rgb}{0,.55,.55}

\let\eps\varepsilon  
\newcommand{\N}{{\mathbb N}}  
\newcommand{\R}{{\mathbb R}}

\newcommand{\D}{{\mathcal D}}

\newcommand{\Law}{{\mathcal L}}
\newcommand{\Prob}{\mathbb{P}}
\newcommand{\E}{\mathbb{E}}

\newcommand{\Fil}{\mathbb{F}}
\newcommand{\Filx}{\mathcal{F}^{X}}

\renewcommand{\d}{\,\textnormal{d}}

\newcommand{\dt}{\,\textnormal{d}t}
\newcommand{\ds}{\,\textnormal{d}s}
\newcommand{\dr}{\,\textnormal{d}r}

\newcommand{\dx}{\,\textnormal{d}x}

\newcommand{\du}{\,\textnormal{d}u}

\newcommand{\dW}{\,\textnormal{d}W}

\newcommand{\ball}{B}

\def\XXint#1#2#3{{\setbox0=\hbox{$#1{#2#3}{\int}$ }
\vcenter{\hbox{$#2#3$ }}\kern-.6\wd0}}

\usepackage{graphicx}
\usepackage{color}

\newcommand{\drift}{{b}}
\newcommand{\diffusion}{{\sigma}}

\newcommand{\nudrift}{\nu^{\drift}}
\newcommand{\nudiffusion}{\nu^{\diffusion}}

\newcommand{\Wplusdualn}{(\Wplusdual)^{n_{\textnormal{dim}}}}
\newcommand{\Wmiddualn}{(\Wmiddual)^{n_{\textnormal{dim}}}}
\newcommand{\Wminusdualn}{(\Wminusdual)^{n_{\textnormal{dim}}}}
\newcommand{\Wmidtwodualn}{(\Wmidtwodual)^{n_{\textnormal{dim}}}}
\newcommand{\Wminustwodualn}{(\Wminustwodual)^{n_{\textnormal{dim}}}}
\newcommand{\opn}{\mathcal{L}(\R^{n_{\textnormal{dim}}},\Wplusdualn)}
\newcommand{\opnmid}{\mathcal{L}(\R^{n_{\textnormal{dim}}},\Wmidtwodualn)}
\newcommand{\opnminus}{\mathcal{L}(\R^{n_{\textnormal{dim}}},\Wminusdualn)}

\newcommand{\weight}{w}
\newcommand{\weightbar}{w^{\prime}}
\newcommand{\weightc}{w_{c}}
\newcommand{\weighttilde}{\widetilde{w}}

\newcommand{\weightplus}{{w^{+}}}
\newcommand{\weightminus}{{w^{-}}}
\newcommand{\weightsim}{{w^{(0)}}}
\newcommand{\weightplusinverse}{\frac{1}{\weightplus}}
\newcommand{\weightminusinverse}{\frac{1}{\weightminus}}
\newcommand{\weightprime}{w^{\prime}}
\newcommand{\weightinverse}{\frac{1}{w}}
\newcommand{\weighttildeinverse}{\frac{1}{\weighttilde}}

\newcommand{\etaplus}{\eta_{+}}
\newcommand{\etamid}{\eta_{\sim}}
\newcommand{\etaminus}{\eta_{-}}
\newcommand{\Wplus}{{W^{1,2}_{\weightplus}}}
\newcommand{\Wplusdual}{{W^{-1,2}_{\frac{1}{\weightplus}}}}

\newcommand{\Wplustwodual}{{W^{-2,2}_{\frac{1}{\weightplus}}}}

\newcommand{\Wmid}{{W^{1,2}_{\weightsim}}}
\newcommand{\Wmiddual}{{W^{-1,2}_{\frac{1}{\weightsim}}}}

\newcommand{\Wmidtwodual}{{W^{-2,2}_{\frac{1}{\weightsim}}}}

\newcommand{\Wminus}{{W^{1,2}_{\weightminus}}}
\newcommand{\Wminusdual}{{W^{-1,2}_{\frac{1}{\weightminus}}}}

\newcommand{\Wminustwodual}{{W^{-2,2}_{\frac{1}{\weightminus}}}}

\newcommand{\WeightclassB}{\mathcal{B}}
\newcommand{\dualweight}{\frac{1}{w}}

\newcommand{\pivotweight}{\bar{\mathrm{w}}}

\newtheorem{theorem}{Theorem}  
\newtheorem{lemma}[theorem]{Lemma}   
\newtheorem{proposition}[theorem]{Proposition}   
\newtheorem{remark}[theorem]{Remark}   
\newtheorem{corollary}[theorem]{Corollary}  
\newtheorem{definition}[theorem]{Definition}  
\newtheorem{example}[theorem]{Example} 
\newtheorem{notation}[theorem]{Notation}  
\newtheorem{assumption}[theorem]{Assumption} 

\newcommand\titletext{Markovian Lifts of Stochastic Volterra Equations in Sobolev Spaces: Solution theory, an Itô Formula and Invariant Measures}

\title{\titletext}
\makeatletter
\@namedef{subjclassname@2020}{\textup{2020} Mathematics Subject Classification}
\makeatother

\title[Markovian Lifts of Stochastic Volterra Equations]{Markovian Lifts of Stochastic Volterra Equations in Sobolev Spaces: Solution theory, an Itô Formula and Invariant Measures}

\author{Florian Huber}
\address{Independent Researcher}
\email{dr.huber.florian@gmail.com}
\date{}
\thanks{The author gratefully acknowledges financial support
through grant Y 1235 of the START-program during his employment at the University of Vienna, where parts of this paper were written. Parts of this paper were completed during the author's affiliation with the \'Ecole Polytechnique F\'ed\'erale de Lausanne (EPFL), Switzerland. The author thanks Christa Cuchiero for bringing this topic to his attention.}
\subjclass[2020]{60H15,60H20,60G22,37A25}
\keywords{stochastic partial differential equations, Volterra processes, Markovian lift, ergodic behaviour, It{\^o} formula}

\begin{document}

\begin{abstract}
We investigate Markovian lifts of stochastic Volterra equations (SVEs) with completely monotone kernels and general coefficients within the framework of weighted Sobolev spaces. Exploiting the Laplace representation of completely monotone kernels, we recast the SVE as an infinite-dimensional, non-local stochastic evolution equation (SEE) whose state space is a scale of weighted Sobolev spaces constructed algorithmically from the decay of the associated Laplace measure. This construction makes separability, duality, and Rellich--Kondrachov-type compact embeddings available systematically rather than case by case. Our primary focus is developing a comprehensive solution theory for this class of SEEs: we establish existence and uniqueness of probabilistically weak mild solutions, covering general nonlinear coefficients of linear growth well beyond the affine or Lipschitz settings treated previously, and we show that the lift is equivalent to the original Volterra equation. Building on the compactness of the state space, we provide conditions for the existence of invariant measures for both the lifted process and the corresponding SVE, and, for uniformly elliptic diffusion coefficients, we establish uniqueness and exponential ergodicity in a weighted Wasserstein distance via a generalized Harris theorem. A further key contribution is an It{\^o}-type formula, derived directly at the level of the stochastic Volterra equation, from which we develop a range of applications: a finite-time blow-up criterion, a Feynman--Kac representation together with the associated backward Kolmogorov equation on the lift space, and a pricing equation for rough volatility models.
\end{abstract}
\maketitle
\tableofcontents
\newpage

\section{Introduction}

Stochastic Volterra equations (SVEs) have emerged as a central modeling framework for systems with memory, with applications ranging from turbulence \cite{barndorff2008time} and biology \cite{duval2022interacting, reynaud2010adaptive, verma2021self} to energy markets \cite{barndorff2013modelling} and, more recently, rough volatility models in mathematical finance (see \cite{el2019characteristic,gatheral2022volatility}). A prototypical example is given by
\begin{align}\label{eqn:SVE_introduction}
    X_{t}=X_{0}+\int_{0}^{t}k_{\drift}(t-s)\drift(s,X_{s})\ds + \int_{0}^{t}k_{\diffusion}(t-s)\sigma(s,X_{s})\dW_{s},
\end{align}
where $W$ is a multidimensional Brownian motion and the coefficients $\drift,\diffusion$, as well as the convolution kernels $k_{\drift}, k_{\diffusion}$ satisfy suitable integrability and regularity conditions. The presence of the kernels allows one to model memory effects whose influence typically decays over time.

A fundamental difficulty in the analysis of \eqref{eqn:SVE_introduction} is that its solution is, in general, \emph{neither a Markov process nor a semimartingale}. As a consequence, classical tools from stochastic analysis, such as Markov semigroup techniques, It{\^o} calculus, or ergodic theory, are not directly applicable. In particular, questions concerning long-time behavior, invariant measures, and stability become significantly more challenging in the Volterra setting.

A powerful strategy to overcome this obstacle is to embed the original equation into a higher-dimensional Markovian system. This idea, commonly referred to as a \emph{Markovian lift}, consists in representing the solution of the SVE as a functional of an infinite-dimensional Markov process, typically given as the solution to a stochastic evolution equation (SEE). Such representations have been developed in various forms in recent years and have proven to be an effective tool for analyzing Volterra-type models.

Early systematic constructions of Markovian lifts were obtained in the affine setting, where the lifted process takes values in spaces of measures or functionals and allows one to exploit the structure of affine processes. In particular, the measure-valued framework developed in \cite{cuchiero2020generalized, cuchiero2019markovian} provides an infinite-dimensional Markovian representation that restores the semigroup structure and enables the use of generalized Feller theory. These ideas have been successfully applied to matrix-valued Volterra processes and play a central role in applications.

Beyond the affine case, several approaches have been proposed to construct and analyze Markovian lifts in more general settings. For instance, Hilbert space formulations have been introduced in which the singularity of the kernel is incorporated into the definition of the state space, allowing one to establish existence, uniqueness, and regularity properties of the lifted stochastic evolution equation \cite{hamaguchi2023markovian}. More recently, such infinite-dimensional formulations have been used to study ergodicity and long-time behavior via coupling methods and Harris-type theorems \cite{hamaguchi2026exponential}. Since this approach and ours address closely related questions on closely related lifts, we adapt it to our setting in Section~\ref{subsec:harris_ergodicity} and, in Remark~\ref{rem:hamaguchi_comparison}, compare the two choices of state space in detail: the weighted-$L^{2}$ construction of \cite{hamaguchi2023markovian,hamaguchi2026exponential} admits a compact embedding only for finite-rank (sum-of-exponentials) kernels, whereas our weighted \emph{Sobolev} spaces are designed to remain compactly embedded for genuinely singular kernels as well, via a Rellich--Kondrachov-type mechanism rather than a flat $L^{2}$ one. A second, and for applications arguably more consequential, difference concerns initial conditions. In the framework of \cite{hamaguchi2023markovian,hamaguchi2026exponential}, the forcing term of the SVE is tied to the lift's own initial state $Y_0$ via $x(t)=\int_{[0,\infty)}e^{-\theta t}Y_0(\theta)\,\mu(\mathrm d\theta)$, so that a \emph{constant} initial state $Y_0\equiv x_0$ produces the \emph{transient} forcing $x_0K(t)\to0$ as $t\to\infty$, not the persistent constant $x_0$ that defines the standard stochastic Volterra equation used throughout the present paper (and in \cite{abi2019affine} and the rough-volatility literature more broadly). Recovering a genuine persistent constant this way would require the reference measure $\mu$ to carry an atom at the origin, which is precisely excluded by the tempering condition $\inf\operatorname{supp}\mu>0$ that the ergodicity results of \cite{hamaguchi2026exponential} themselves require, so the limitation binds exactly where that theory is applicable. Our own state space, by contrast, is a genuine space of distributions rather than of functions, so the standard initial condition $X_0=x_0$ lifts directly to $\mu_0=\delta_0x_0\in\Wplusdual$; since $S^{*}_{t}\delta_0=\delta_0$ identically, this correctly reproduces $X_t\equiv x_0$ whenever there is no forcing, for every kernel considered in this paper, tempered or not. In parallel, approximation approaches based on finite-dimensional Markovian systems have been developed, in particular for fractional kernels arising in rough volatility models \cite{bayer2023markovian}.

Numerous works have contributed to the theory of stochastic Volterra equations; see, for example, \cite{berger1980volterra_I,berger1980volterra_II,coutin2001stochastic,wang2008existence}. A comprehensive weak existence theory in the convolutional setting was developed in \cite{abi2021weak} (see also \cite{hamaguchi2023weak}), while \cite{promel2023existence} treats the non-convolutional case. Equations with affine \cite{abi2019affine,abi2021weak_L1,bondi2024affine} or polynomial coefficients \cite{jaber2024polynomial} have received particular attention due to their tractable structure.

The uniqueness of solutions to SVEs remains challenging in many cases. For singular kernels, pathwise uniqueness without drift was first established in \cite{mytnik_15_uniqueness_stochastic_volterra} via an infinite-dimensional lift, and later extended to include drift terms in \cite{promel2022pathwise}. For regular kernels, pathwise uniqueness has been studied in \cite{promel2023stochastic,alfonsi2024non} under H{\"o}lder continuity assumptions on the coefficients.

The study of qualitative behavior is even more delicate. To the best of our knowledge, a direct analysis at the level of SVEs is only available in \cite{friesen2024volterra}, while most approaches rely on Markovian lifts; see, for example, \cite{jacquier2022large_time,benth2022stochastic,hamaguchi2026exponential}.

The construction of Markovian lifts itself has been developed along several directions. A first systematic approach arises in the affine setting, where the lift is formulated in measure-valued or dual Banach spaces. This idea can be traced back to \cite{coutin1998fractional} and has been further developed in \cite{cuchiero2020generalized, cuchiero2019markovian}, as well as in approximation schemes \cite{abi2019lifting,abi2019multifactor}. The framework has been extended and analyzed in various functional settings in \cite{hamaguchi2023markovian, harms2019affine,abi2021linear}. While mathematically elegant, these spaces can be difficult to handle analytically.

An alternative viewpoint is given by forward curve formulations, originating from term structure theory. In this approach, the state variable is interpreted as a curve and evolves according to a stochastic transport equation. This perspective has been used in \cite{abi2019markovian,benth2022stochastic} and, more recently applied in  \cite{gasteratos2025kolmogorov} to obtain Kolmogorov equations for SVEs, depending on the lift. Due to the term structure motivation, the lifted transport SPDE and associated weighted function spaces are already well understood (see for example \cite{carmona2006interest}). For the interested reader, we leave further references regarding different approaches to similar first-order SPDEs, which have been studied to quite an extent: \cite{turo1996stochastic,grecksch2000parabolic,hamza2005solutions,gikhman1983cauchy,kunita1990stochastic,lv2021kinetic,holden1997conservation,chen2012nonlinear,kim2003stochastic,debussche2010scalar,feng2008stochastic}.

A third type of lift, conceptually related to catalytic superprocesses, has been used to establish pathwise uniqueness results for singular kernels, notably in \cite{mytnik_15_uniqueness_stochastic_volterra} and \cite{promel2022pathwise}. The precise relationship between this construction and the present framework is not clear, since the semigroup, as well as operator used in this approach require a more delicate construction in more general cases. A deeper comparison is delayed to future research.

It should be noted that the first and second lifts exhibit similar properties, as the principal operators in the lifted equations are Laplace transforms of each other, which is also reflected in the corresponding semigroups.

Despite these developments, existing approaches often rely on state spaces tailored to specific kernel representations, which makes it difficult to obtain a unified treatment of stochastic Volterra equations with general nonlinear coefficients and to systematically exploit analytic tools such as embeddings, duality, or compactness. While these constructions are often elegant and sufficient under strong assumptions, the underlying functional-analytic structure is not always made explicit. In particular, properties such as separability, compactness of embeddings, or the availability of tightness criteria are not systematically addressed, which can limit the applicability and generality of these results.

The weighted-Sobolev framework developed in the present paper was introduced in \cite{huber2024markovian}. A closely related weighted-Sobolev perspective, formulated for the forward-curve (transport) lift rather than the measure-valued lift used here, is developed in \cite{gasteratos2025kolmogorov}. Since the principal operators of the two lifts are Laplace transforms of one another, as noted above, the two settings are essentially equivalent; we compare them in more detail below.

The guiding idea is that the kernels appearing in the SVE have a one-to-one correspondence with certain Radon measures, where, roughly speaking, the singularity of the kernel translates to decay properties of the tails of the measure and decay properties of the kernel translate to singularities in the measure. These measures can be used directly to define Lebesgue and Sobolev function spaces on which the equation could be treated. This, however, has several caveats as mentioned before: analytic properties of these spaces are not clear, neither are compact embeddings, which are very useful for approximation techniques. A proper analysis would result in a case-by-case study of the explicit weight.
We suggest that it is easier to find a weight in a reasonably large class (which could easily be extended in the case of some very ``exotic'' weights) and interpret these Radon measures as elements in the dual of a weighted Sobolev space. Once we have established the necessary functional analytic properties, this approach becomes very algorithmic and leads to an easy applicability.

The aim of this work is to develop a systematic functional-analytic framework for Markovian lifts of stochastic Volterra equations based on weighted Sobolev spaces. Our approach is motivated by the observation that completely monotone kernels admit a Laplace representation in terms of measures, which naturally suggests a lifting procedure. By encoding the decay properties of these measures into suitable Sobolev weights, we obtain a class of state spaces that combines probabilistic and analytic features and allows for the use of tools from the theory of partial differential equations.

A central insight of our analysis is that the behavior of the lifted process is governed by a precise interplay between
\begin{itemize}
    \item the singularity of the Volterra kernel,
    \item the decay of the associated Laplace measure, and
    \item the choice of weighted Sobolev space.
\end{itemize}
This relationship provides a systematic way to construct state spaces in which the lifted stochastic evolution equation is well-posed and exhibits favorable regularity properties.

Within this framework, we establish the following main results:
\begin{enumerate}
    \item \textbf{Solution theory.}\\
    We prove existence and uniqueness of (probabilistically weak) mild solutions to the lifted stochastic evolution equation in weighted Sobolev spaces. This extends existing results beyond affine or linear settings and accommodates general nonlinear coefficients. The spaces used are more common in studying (S)PDEs and provide more flexibility compared to spaces used in previous studies  \cite{cuchiero2020generalized, cuchiero2019markovian}, or \cite{hamaguchi2023markovian}. In particular, this choice gives us easy access to embedding theorems, extending existing results to more general coefficients.

    \item \textbf{It{\^o} formula for stochastic Volterra equations, and applications.}\\
    Using the Markovian lift, we derive an It{\^o}-type formula directly at the level of the Volterra equation. In contrast to functional It{\^o} calculus approaches, such as \cite{viens2019martingale}  (see also \cite{bonesini2023rough}), our formula is explicit and relies only on the structure of the lifted dynamics. From this formula we develop a range of applications: a finite-time blow-up criterion, a Feynman--Kac representation together with the associated backward Kolmogorov equation on the lift space, and a pricing equation for rough volatility models.

    \item \textbf{Invariant measures and long-time behavior.}\\
    We establish the existence of invariant measures for both the lifted process and the original SVE under general conditions, significantly generalizing previous results that are often restricted to specific assumptions or specific state space constructions. For the sub-class of SEEs with uniformly elliptic diffusion coefficient, we further establish uniqueness and exponential ergodicity in a weighted Wasserstein distance, by adapting the generalized Harris theorem of \cite{HaMaSc11} following the recent approach of \cite{hamaguchi2026exponential}.
\end{enumerate}

\medskip
\noindent\textbf{Comparison with related constructions.}
Our framework shares with \cite{gasteratos2025kolmogorov} the idea of casting the singular kernel into a weighted-Sobolev environment, and, the two lifts being Laplace transforms of one another, the settings are essentially equivalent. The present construction nonetheless differs in several respects, which we highlight without diminishing the elegance of the transport formulation. First, our lift is governed by the multiplication semigroup $S^{*}_{t}=e^{-tx}$, whose generator acts as a bounded multiplier on the relevant scale of spaces; in contrast to the shift semigroup of the forward-curve lift, whose generator $\partial_{x}$ is unbounded, this regularizes the generator direction, so that the backward Kolmogorov equation and the pathwise It{\^o} formula hold on the \emph{entire} state space rather than on a dense invariant subspace, with the second-order terms identified explicitly (Section~\ref{subsec:feynman_kac}, Remark~\ref{rem:no_K1_needed}). Second, our state spaces are weighted \emph{Sobolev} spaces and therefore enjoy Rellich--Kondrachov-type compact embeddings; we exploit this compactness in a tightness argument to obtain well-posedness of the lifted equation for merely \emph{continuous} coefficients of linear growth, rather than only Lipschitz ones. The weights and spaces are constructed algorithmically, so that separability, duality, and compactness are available systematically rather than case by case, while the hypotheses on the kernel reduce to a decay and duality condition on its Laplace measure, with no regularity or shift-compatibility condition imposed on the kernel itself. On this basis we develop a range of applications, namely a finite-time blow-up criterion, a Feynman--Kac representation, and a pricing equation for rough volatility models. Third, since our state space consists of genuine distributions, finite-rank (sum-of-exponentials) kernels lift to finitely supported measures, and hence to genuinely finite-dimensional states, which connects the framework directly to the Markovian approximation schemes used in mathematical finance \cite{bayer2023markovian}. Relative to the affine, generalized-Feller lifts of \cite{cuchiero2019markovian,cuchiero2020generalized}, which are more probabilistic in nature but confined to affine coefficients, our analytic framework accommodates general nonlinear coefficients; the relationship to the weighted-$L^{2}$ lift of \cite{hamaguchi2023markovian,hamaguchi2026exponential} is examined in detail in Remark~\ref{rem:hamaguchi_comparison}.

Our results provide a unified framework that bridges existing approaches to Markovian lifts and highlights the role of functional-analytic structures in the study of stochastic Volterra equations.
The remainder of the paper is organized as follows. Section~\ref{sec:main_results} fixes the setup and states the main theorems. Section~\ref{section:Functional_framework} introduces the weighted Sobolev spaces and establishes their functional-analytic properties. Section~\ref{sec:well_posedness} develops the solution theory for the lifted stochastic evolution equation, and Section~\ref{sec:equivalence} establishes its equivalence with the original stochastic Volterra equation. Section~\ref{section:Invariant_measure} studies invariant measures and long-time behavior. Finally, Section~\ref{sec:ito_formula} derives the It{\^o} formula and develops its applications, including a finite-time blow-up criterion, a Feynman--Kac representation and backward Kolmogorov equation, and option pricing in rough volatility models.
\section{Main Results}\label{sec:main_results}

\subsection{Setup}

The guiding equation we want to consider is the following SVE
\begin{align}\label{eqn:SVE_main_results}
    X_{t}=X_{0}+\int_{0}^{t}k_{\drift}(t-s)\drift(s,X_{s})\ds + \int_{0}^{t}k_{\diffusion}(t-s)\diffusion(s,X_{s})\dW_{s},
\end{align}
where $W$ is an $m_{W}$-dimensional Brownian motion, $\drift : \R^{n_{\textnormal{dim}}} \to \R^{n_{\textnormal{dim}}}$ and $\diffusion: \R^{n_{\textnormal{dim}}} \to\R^{n_{\textnormal{dim}} \times m_{W}}$ are measurable maps, $k_{\drift} : \R_{+}\rightarrow  \R^{n_{\textnormal{dim}} \times n_{\textnormal{dim}}}$ is locally integrable,  $k_{\diffusion}: \R_{+}\rightarrow  \R^{n_{\textnormal{dim}} \times n_{\textnormal{dim}}}$ is locally square integrable and $X_{0} \in \R^{n_{\textnormal{dim}}}$ is an $\mathcal{F}_{0}$ measurable random variable (it can also be chosen as an initial random curve). Here $n_{\textnormal{dim}}\geq1$ denotes the dimension of the SVE's state space; it is unrelated both to the dimension $d$ of the spatial domain underlying the weighted Sobolev spaces introduced in Section~\ref{section:Functional_framework} (always the domain of the Laplace-dual variable, playing no role in the state dimension) and to any use of the bare letter $n$ as a generic sequence or approximation index elsewhere in the paper (e.g.\ $\mu^{n}\to\mu$ as $n\to\infty$ in convergence/compactness arguments).
The primary tool in this work is an infinite-dimensional Markovian representation of the stochastic Volterra equation \eqref{eqn:SVE_main_results}. This construction, often referred to as a \emph{Markovian lift}, goes back to \cite{coutin1998fractional} and allows one to recover the Markov property by embedding the original process into a suitable infinite-dimensional state space.

\subsection{Laplace representation of the kernel}

The key observation underlying the lift is that completely monotone kernels admit a Laplace representation.

\begin{definition}
Let $k\colon \R_{+}\rightarrow \R_{+}$. We say that $k$ is \emph{completely monotone} if $k$ is infinitely differentiable on $(0,\infty)$ and satisfies
\[
(-1)^{n}\partial_{t}^{n}k(t)\geq 0, \quad \text{for all } n\in \N_{0},\ t>0.
\]
\end{definition}
\begin{theorem}[Bernstein--Hausdorff--Widder]\label{thm:berstein_hausdorf_widder}
\cite[Proposition 1.2 and Theorem 4.8]{schilling2009bernstein}
The following are equivalent:
\begin{enumerate}
\item $k$ is completely monotone  on $(0, \infty)$ (respectively on $[0, \infty)$);
\item there exists a unique Radon (respectively finite) measure $\nu$ on $\R_{+}$ such that
\begin{align}\label{eqn:kernel_representation_measure}
k(t)=\int_{0}^{\infty} e^{-tx}\,\nu(\dx), \quad t>0;
\end{align}
\item $k$ is infinitely differentiable on $(0, \infty)$ (respectively continuous on $[0, \infty)$, infinitely differentiable on $(0, \infty))$ and satisfies $(-1)^n k^{(n)} \geq 0$ for all $n\in\mathbb{N}_0$.
\end{enumerate}
\end{theorem}
\begin{remark}
 For the $\mathbb{R}^{n_{\textnormal{dim}}}$-valued and matrix-valued kernels required in our Volterra SDE context, the theorem naturally extends component-wise or via partial orderings on positive semi-definite matrices. In the abstract setting where the kernel maps into a general Banach space $X$, the corresponding characterization is known as the Widder--Arendt theorem. Formally established by Arendt \cite{arendt1987vector}, the representation of a function as a Laplace transform of an $X$-valued bounded variation vector measure holds if and only if the underlying Banach space $X$ possesses the Radon--Nikod\'ym property (RNP). Since finite-dimensional spaces like $\mathbb{R}^{n_{\textnormal{dim}}}$ and the space of $n_{\textnormal{dim}} \times n_{\textnormal{dim}}$ matrices trivially enjoy the RNP, this abstract framework rigorously validates the structural assumptions placed on our drift and diffusion kernel matrices. For further multidimensional and locally convex space formulations of this principle, we refer the reader to \cite{arendt2026multidim}.
\end{remark}

\subsection{Heuristic derivation of the Lifted Equation}

Assume that the entries of the kernels (kernel-matrices) $k_{\drift}$ and $k_{\diffusion}$ in \eqref{eqn:SVE_main_results} are completely monotone, with associated measures (matrices of measures)  $\nudrift$ and $\nudiffusion$. Then Theorem \ref{thm:berstein_hausdorf_widder} yields the representation
\begin{align}
X_{t}=X_{0}+\int_{0}^{t}\int_{\R_{+}}e^{-x(t-s)}\nudrift(\dx)\drift(s,X_{s})\ds 
+ \int_{0}^{t}\int_{\R_{+}}e^{-x(t-s)}\nudiffusion(\dx)\diffusion(s,X_{s})\dW_{s}.
\end{align}

This suggests introducing an auxiliary process $(\mu_t)_{t\geq 0}$ such that
\[
X_t = \langle \mu_t,1\rangle,
\]
since testing the above equation with the constant $1$ function restores the SVE.
Formally applying Fubini's theorem and identifying the resulting expression, one is led to the representation
\begin{align}
\mu_{t} =\mu_{0}+\int_{0}^{t}e^{-x(t-s)}\nudrift\drift(s, \langle\mu_{s},1\rangle)\ds 
+ \int_{0}^{t}e^{-x(t-s)}\nudiffusion\diffusion(s, \langle\mu_{s},1\rangle)\dW_{s}.
\end{align}

This corresponds (up to a temporal correction of the initial value) to the mild formulation of the stochastic evolution equation
\begin{align}\label{eqn:SEE_eqn_strong_form_introduction}
\d \mu_{t} =-x\mu_{t}\dt +\nudrift\drift(t, \langle\mu_{t},1\rangle)\dt
+ \nudiffusion\diffusion(t,\langle\mu_{t},1\rangle)\dW_{t}.
\end{align}
Since $X_{t}\in\R^{n_{\textnormal{dim}}}$, the lifted process $\mu_{t}$ takes values in
$\Wplusdual\otimes\R^{n_{\textnormal{dim}}}\cong(\Wplusdual)^{n_{\textnormal{dim}}}$, i.e.\ it is an
$n_{\textnormal{dim}}$-tuple (equivalently, an $\R^{n_{\textnormal{dim}}}$-valued distribution) of elements
of the scalar weighted-Sobolev-dual space constructed in
Section~\ref{section:Functional_framework}. The pairing
$\langle\mu_{t},\psi\rangle$ for a scalar test function $\psi$ is then
understood componentwise, producing an element of $\R^{n_{\textnormal{dim}}}$, and the
multiplication operator $\mu\mapsto x\mu$, as well as the semigroup
$S^{*}_{t}$ introduced later, act diagonally on each of the $n_{\textnormal{dim}}$
components. All of the scalar functional-analytic machinery developed
in Section~\ref{section:Functional_framework} therefore applies
unchanged component-by-component; only the coupling introduced by
$\drift,\diffusion$ (now genuinely $\R^{n_{\textnormal{dim}}}$-valued and
$\R^{n_{\textnormal{dim}}\times n_{\textnormal{dim}}}$-valued, respectively $\R^{n_{\textnormal{dim}}\times m_{W}}$-valued,
maps) is new.

\subsection{Main Theorems}
Before we state the main results, we need to clarify what we mean by a solution to the lifted equation. We will introduce the exact spaces in a following chapter. For now, consider them simply as Sobolev spaces.
\begin{definition}\label{def:prbabilistically_weak_mild_solution}
 Let $(\tilde{\Omega},\tilde{\mathcal{F}},\tilde{\mathbb{F}},\tilde{\Prob})$, where $\tilde{\mathbb{F}}=(\tilde{\mathcal{F}}_{s})_{s\in[0,T]}$ is a stochastic basis, $\tilde{\mu}$ is a $\mathbb{F}$-progressively measurable process with laws supported on $C([0,T],\Wmidtwodual)$ and $\tilde{W}$ is an $m_{W}$-dimensional Wiener process. We call $(\tilde{\Omega},\tilde{\mathcal{F}},\tilde{\mathbb{F}},\tilde{\Prob},\tilde{\mu},\tilde{W})$ a probabilistically weak (or martingale) mild solution of
 \begin{align}\label{eqn:SEE_strong_formulation}
   \d \mu_{t} =-x\mu_{t}\dt +\nudrift\drift(t, \langle\mu_{t},\psi\rangle)\dt +\nudiffusion\diffusion(t, \langle\mu_{t},\psi\rangle)\dW_{t},
\end{align}
with $\mu_{t}\mid_{t=0}=\mu_{0}\in \Wmiddual$, $\psi \in \Wmid$, if
    \begin{align}\label{eqn:SEE_eqn_mild_formulation}
            \widetilde{\mu}_{t} = e^{-xt}\widetilde{\mu}_{0} +\int_{0}^{t} e^{-x(t-s)}\nudrift(x)\drift(s,\langle \widetilde{\mu}_{s},\psi\rangle)\ds +\int_{0}^{t}e^{-x(t-s)}\nudiffusion(x)\diffusion(s,\langle \widetilde{\mu}_{s},\psi\rangle)\d \tilde{W}_{s},
        \end{align}
        $\widetilde{\Prob}$-a.s. for each $t\in [0,T]$.
\end{definition}
\begin{theorem}[Existence]\label{thm:weak_mild_solution_SEE}
    Let  
    Assumptions \ref{A:assumption_general_linear_growth}, \ref{A:assumption_general_uniformly_continuous} 
 and \ref{A:A_3:Narrower_assumption_nu} be satisfied and let $p>2$.
       For $\mu_{0}\in L^{p}(\Omega,\Wplusdual)$, equation \eqref{eqn:SEE_strong_formulation} has a probabilistically weak, analytically mild solution $\mu \in L^{p}(\Omega,C([0,T],\Wmidtwodual))\cap L^{2}(\Omega,L^{\infty}([0,T],\Wplusdual))$. (Proved in Section~\ref{sec:existence_for_the_SEE}; see the proof of Theorem~\ref{thm:weak_mild_solution_SEE} concluded after Lemma~\ref{lem:convergence_of_approximate_mild_solution_terms}.)
\end{theorem}
In the case where the weights are chosen such that one can set $\psi=1$, we have equivalence between the solutions of the SVE and the SEE (see Theorem \ref{thm:equivalence_SVE_SEE}).

\begin{remark}
    Throughout, $X_{t}\in\R^{n_{\textnormal{dim}}}$ for a state dimension $n_{\textnormal{dim}}\geq1$ fixed
    once and for all, as introduced in the Setup subsection; the
    Laplace-dual variable $x\in\R_{+}$ appearing inside $\mu_{t}$ is,
    by contrast, always scalar, regardless of $n_{\textnormal{dim}}$ (see the discussion
    following equation~\eqref{eqn:SEE_eqn_strong_form_introduction}
    above for how $\mu_{t}$ carries the state dimension). Where a
    result or example is stated only for $n_{\textnormal{dim}}=1$ for concreteness (e.g.\
    the rough-volatility application in
    Section~\ref{subsec:option_pricing} and the recovery of the
    Riccati--Volterra equation in
    Section~\ref{subsec:recover_ALP}), this is noted explicitly at the
    point where it occurs.
\end{remark}

Combining the previous equivalence result and the following theorem will allow us to obtain results regarding the long-term behavior of solutions to the SVE.
\begin{theorem}[SEE solution is Feller]\label{thm:weak_gen_feller}
    Let Assumptions \ref{A:assumption_general_linear_growth}, \ref{A:uniqueness_in_law}, \ref{A:long_term_behaviour} be satisfied, then a mild solution $\mu\in \Wplusdual$ of equation \eqref{eqn:SEE_strong_formulation} is a weak (and in particular generalized) Feller process with an invariant measure. (Proved in Section~\ref{section:Invariant_measure}: the weak/generalized Feller property is Corollary~\ref{cor:gen_feller}, and existence of an invariant measure is Theorem~\ref{thm:invariant_measure_mu}.)
\end{theorem}

\begin{theorem}[Invariant measure SVE]\label{thm:invariant_measure_SVE}
    Let the kernels $k_{\drift}\in L^{1}(0,\infty)$ and $k_{\diffusion}\in L^{2}(0,\infty)$ be completely monotone, their associated lifted measures satisfy Assumption \ref{A:long_term_behaviour}, the lifted equation, with initial condition $\mu_{0}=\delta_{0}x_{0}$, satisfies Assumption \ref{A:uniqueness_in_law} and the coefficients satisfy Assumption \ref{A:assumption_general_linear_growth}. Then there exists an invariant measure to the SVE \eqref{eqn:SVE_introduction} in the sense that there exists a measure $\mathcal{q}$ and $x_{0}\sim \mathcal{q}$, $X_{t}\sim \mathcal{q}$. (Proved in Section~\ref{section:Invariant_measure}, Subsection ``Transfer to SVE''.)
\end{theorem}
The last consequence of the lift is the following It{\^o}- type formula.
\begin{proposition}[It{\^o}-Volterra formula]\label{prop:ito_formula_volterra}
         Let $k_{\drift}, k_{\diffusion}$ be (matrix-valued, entrywise) completely monotone kernels such that their associated measures, given by \eqref{eqn:kernel_representation_measure}, satisfy Assumptions \ref{A:A_3:Narrower_assumption_nu} and  \ref{A:A_2:assumption_1_test_function}.
         Let $X$ denote the $\R^{n_{\textnormal{dim}}}$-valued solution of the corresponding stochastic Volterra equation \eqref{eqn:SVE_introduction} and set
         \begin{align*}
             \Gamma_{st}(X)&\coloneqq X_t -\int_{s}^{t}k_{\drift}(t-r)\drift(r,X_{r})\dr-\int_{s}^{t}k_{\diffusion}(t-r)\diffusion(r,X_{r})\dW_{r}\\
             &\left(=X_{0} + \int_0^s k_{\drift}(t-r)\,\drift(r, \langle\mu_r,1\rangle)\,dr + \int_0^s k_{\diffusion}(t-r)\,\diffusion(r, \langle\mu_r,1\rangle)\,dW_r.\right)
         \end{align*}
         which is $\R^{n_{\textnormal{dim}}}$-valued. For $f\in C^{1,2}(\R_{+}\times\R^{n_{\textnormal{dim}}},\R)$, write $\nabla_{x}f\in\R^{n_{\textnormal{dim}}}$ for the gradient and $D_{x}^{2}f\in\R^{n_{\textnormal{dim}}\times n_{\textnormal{dim}}}$ for the Hessian in the spatial argument. Then the following It{\^o} formula holds.
    \begin{align*}
    f(t,X_{t})&=f(t_{0}, \Gamma_{t_{0}t}(X))+\int_{t_{0}}^{t}\partial_{s}f\left(s,\Gamma_{st}(X)\right) \ds\\
    &+\int_{t_{0}}^{t}\left\langle\nabla_{x}f\left(s,\Gamma_{st}(X)\right), k_{\drift}(t-s)\drift(s,X_{s})\right\rangle \ds\\
    &+\int_{t_{0}}^{t}\left\langle\nabla_{x}f\left(s,\Gamma_{st}(X)\right), k_{\diffusion}(t-s)\diffusion(s,X_{s}) \dW_{s}\right\rangle\\
&\phantom{xx}+\frac{1}{2}  \int_{t_{0}}^{t}\operatorname{tr}\!\left[D_{x}^{2} f\left(s,\Gamma_{st}(X)\right)\, k_{\diffusion}(t-s)\diffusion(s,X_{s})\diffusion(s,X_{s})^{\top}k_{\diffusion}(t-s)^{\top}\right] \ds,
\end{align*}
where $k_{\diffusion}(t-s)\diffusion(s,X_{s})\in\R^{n_{\textnormal{dim}}\times m_{W}}$ and the trace term reduces, when $n_{\textnormal{dim}}=1$, to the scalar term $\tfrac12\partial_{x}^{2}f\cdot k_{\diffusion}(t-s)^{2}\diffusion(s,X_{s})^{2}$.
\end{proposition}
(Proved in Section~\ref{sec:ito_formula} as Corollary~\ref{corr:Ito_volterra_SEE}, after the mild It{\^o} formula for the lifted SEE has been established.)
Note that $\Gamma_{0t}=X_{0}$.
\subsection{Discussion}
Equation \eqref{eqn:SEE_eqn_strong_form_introduction} provides a Markovian representation of the original stochastic Volterra equation. The process $(\mu_t)$ evolves in an infinite-dimensional state space and encodes the memory of the system through its dependence on the parameter $x$.

Making this construction rigorous requires specifying a suitable function space in which the lifted process evolves, as well as establishing well-posedness of the corresponding stochastic evolution equation. This will be the subject of the following sections.

In particular, the choice of state space will be guided by the decay properties of the measures $\nudrift$ and $\nudiffusion$, leading naturally to a framework based on weighted Sobolev spaces.

\section{Functional-Analytic Framework}\label{section:Functional_framework}

\subsection{Design Principles}
\textbf{Why weighted Sobolev spaces?}
The Markovian lift introduced in the previous section leads to a stochastic evolution equation of the form
\begin{align*}
\d \mu_t = -x \mu_t \dt + \nudrift\drift(t,\langle \mu_t,\psi\rangle)\dt + \nudiffusion\diffusion(t,\langle \mu_t,\psi\rangle)\dW_t.
\end{align*}
In order to make this equation rigorous, it is necessary to specify a suitable state space in which the process $(\mu_t)_{t\geq 0}$ evolves.

A natural requirement on the spaces we work in should be the following:
\begin{itemize}
    \item accommodate measures such as $\mu$, $\nudrift$ and $\nudiffusion$,
    \item the action of the multiplication operator $\mu \mapsto x\mu$ remains within our framework,
    \item allow for a well-defined dual pairing $\langle \mu, \psi \rangle$ to recover the original SVE,
    \item provide sufficient compactness and embedding properties for approximation arguments.
\end{itemize}

These requirements implicitly ask for a well understood functional analytic framework. This naturally leads to a framework based on weighted Sobolev spaces.

\subsection{Weighted Sobolev spaces}
\begin{definition}\label{def:weight_function}
    We call a locally integrable function $\weight$ on $\R_{+}$, such that $\weight(x)>0$-a.e. a weight or weight function.  
\end{definition}
 Every weight $\weight$ induces a positive Borel-measure on $\mathbb{R}_{+}$ via integration, i.e. $\lambda_{\weight}(E)=\int_E \weight(x) \dx$ for measurable sets $E \subset \mathbb{R}_{+}$.
\begin{definition}
    Let $\weight$ be a weight. For $0<p<\infty$ we define $L^p_{\weight}$ as the set of measurable functions $u$ on $\R_{+}$ such that
\begin{align*}
\|u\|_{L^p_{\weight}}=\left(\int_{\R_{+}}|u(x)|^p \weight(x) \dx\right)^{1 / p}<\infty .
\end{align*}
\end{definition}
The following statements recall that elements of weighted spaces remain distributions on $\R_{+}$.
\begin{definition}
    Let $p>1$. We say that a weight function $\weight$ satisfies the condition $\WeightclassB_{p}(\R_{+})$ and write $\weight\in \WeightclassB_{p}(\R_{+})$, if
    \begin{align*}
        \weight^{-1 /(p-1)} \in L_{\text {loc }}^{1}(\R_{+}).
    \end{align*}
\end{definition}
\begin{remark}
    The class $\WeightclassB_p$ is significantly more flexible than, for example, $A_p$ (Muckenhoupt weights) and is therefore well suited for applications where weights may exhibit strong decay or singular behavior. 
Recall that a weight $w$ belongs to the Muckenhoupt class $A_p$ if
\begin{align*}
w \in A_p
\quad \Longleftrightarrow \quad
\sup_{B}
\left(\frac{1}{|B|}\int_B w(x)\,dx\right)
\left(\frac{1}{|B|}\int_B w(x)^{-1/(p-1)}\,dx\right)^{p-1}
< \infty.
\end{align*}
\begin{itemize}
\item $\WeightclassB_p$ controls integrability of $w^{-1/(p-1)}$ and allows strong degeneracies.
\item $A_p$ imposes additional balance conditions and excludes highly singular or oscillatory weights.
\item The inclusion $A_p \subsetneq \WeightclassB_p$ is strict.
\end{itemize}

For example:
\begin{enumerate}
    \item\textbf{Power weights.}
    Let
\begin{align*}
w(x) = |x|^\alpha, \qquad x \in \mathbb{R}.
\end{align*}

Then:
\begin{align*}
w \in A_p
&\quad \Longleftrightarrow \quad
-1 < \alpha < p-1, \
w \in \WeightclassB_p
&\quad \Longleftrightarrow \quad
\alpha < p-1.
\end{align*}

In particular, for $\alpha \le -1$, we have
\begin{align*}
w \in \WeightclassB_p \quad \text{but} \quad w \notin A_p.
\end{align*}
\item\textbf{(Highly) Singular weights.}
Let
\begin{align*}
w(x) = |x|^{-1}.
\end{align*}

Then
\begin{align*}
w^{-1/(p-1)} = |x|^{1/(p-1)} \in L^1_{\mathrm{loc}},
\end{align*}
so
\begin{align*}
w \in \WeightclassB_p.
\end{align*}
However, $w \notin A_p$ due to the strong singularity at $x=0$.
\item\textbf{Oscillatory weights.}
Let
\begin{align*}
w(x) = 1 + \sin(\log |x|).
\end{align*}

Then $w$ is bounded above and below, hence
\begin{align*}
w \in \WeightclassB_p,
\end{align*}
but typically
\begin{align*}
w \notin A_p,
\end{align*}
since the oscillations violate the Muckenhoupt balance condition.
\end{enumerate}
\end{remark}
\begin{theorem}
    Let $p>1$, $\weight\in \WeightclassB_{p}(\R_{+})$ and $K$ be a compact set such that $K\subset \R_{+}$. Then the embedding 
    \begin{align*}
        L^{p}_{\weight}\hookrightarrow L^{1}(K)
    \end{align*}
    is continuous.
\end{theorem}
\begin{proof}
    The theorem directly results from H{\"o}lder's inequality.
\end{proof}
\begin{corollary}
    Under the assumptions of the previous theorem, we have 
    \begin{align*}
         L^{p}_{\weight}\hookrightarrow L^{1}_{\operatorname{loc}}(\R_{+})\subset D'(\R_{+}),
    \end{align*}
    where $D'(\R_{+})$ denotes the space of distributions.
\end{corollary}
Let $m \in \N$ and $1 \leq p<\infty$. Let $\weight=(\weight_{0},\dots,\weight_{k})$ be a vector of given weight functions.
We introduce the norm 
\begin{align*}
\|u\|_{W^{m, p}_{\weight}}=\left(\sum_{0\leq j \leq m} \int_{0}^{\infty}\left|D^{j} u\right|^p \weight_{j} \dx\right)^{1 / p} .
\end{align*}

\begin{definition}
   We denote by $W^{m, p}_{\weight}$ the completion of $\left\{u \in C_{0}^{\infty}(\R_{+}):\|u\|_{W^{m, p}_{\weight}}<\infty\right\}$ with respect to the norm $\| \cdot \|_{W^{m, p}_{\weight}}$.
\end{definition}
If $1<p<\infty$ and $\weight \in \WeightclassB_{p}(\R_{+})$, then $W^{m, p}_{\weight}$ is a uniformly convex Banach space (see \cite[section 4]{kufner_84_define_weighted_spaces}).
\begin{remark}
    Note that we could also introduce weighted Sobolev spaces $\mathcal{W}^{m, p}_{\weight}$ by considering the set of all functions $u \in L^p_{\weight}$ for which the weak derivatives $D^{j} u$, with $j \leq m$, belong to $L^p_{\weight}$. The weighted Sobolev space $\mathcal{W}^{m, p}_{\weight}$ is a normed linear space if equipped with the norm $\| \cdot \|_{W^{m, p}_{\weight}}$. We have $W \subseteq \mathcal{W}$. By definition, functions that are smooth in the interior of $\R_{+}$ are dense in $W$, while the space $\mathcal{W}$ is known to contain all functions of finite well-defined ``energy''.
 If $\weight$ is bounded from above and away from $0$ from below $0<c_{1}\leq \weight(x)\leq c_{2}<\infty$, the spaces $\mathcal{W}$ and $W$ coincide on general domains, however if $\weight\in L^{2}_{\operatorname{loc}}(\R_{+})$, $W=\mathcal{W}$ does not need to hold.

\end{remark}
Although $\weight=(\weight_{0},\dots,\weight_{m})$ is not a weight function according to our definition, but a vector of weight functions, we will still call such a vector a weight function for convenience.

For simplicity, we will write $\weight\in \WeightclassB_{2}(\R_{+})$ or $\weight\in L^{\infty}_{\operatorname{loc}}(\R_{+})$. This notation will also be used for other component-wise properties.
\begin{remark}
    An example of a weight function, that satisfies $\WeightclassB_{2}(\R_{+})$ is given by $\weight(x)$ with 
    \begin{equation*}
        \weight_{i}(x)=|x|^{a}(1+|x|)^{(i+1)b},
    \end{equation*} for $0\leq a <1$, $i\in \N\cup\{0\}$ and $b\in \R$.
\end{remark}
\begin{theorem}
    Let $1\leq p<\infty$, $\weight\in \WeightclassB_{2}(\R_{+})$, then $W^{m,p}_{\weight}$ is separable. If $1< p<\infty$, $W^{m,p}_{\weight}$ is reflexive.
\end{theorem}
The proof follows the same lines as in the non-weighted setting (see \cite[Theorem 1.3]{benci1979weighted} 
 and also \cite[Section 2]{kufner_84_define_weighted_spaces}).

We state the following embedding theorem on $\R^{d}$ instead of $\R_{+}$, as this general setting is of interest on its own. For $\R_{+}$, the proof works analogously with the only difference that instead of balls of radius $R$ in $\R^{d}$, denoted by $\ball_{R}(0)$, the intersection of balls in $\R$ with $\R_{+}$ is considered.
This theorem will enable us to prove existence results for more general coefficients, compared to the existing literature, where certain regularity of embeddings is often not considered, unclear or doesn't hold at all.

\begin{proposition}[Weighted embeddings]\label{prop:embeddings}

Let $\weight = (\weight_j)_j$, $\weightbar = (\weightbar_j)_j$, and $\weightc$ be weight functions, and let $j \in (\N\cup \{0\})^d$. Fix $p \in [1,\infty)$ and $s \in \N\cup \{0\}$. For $R>0$, denote by $\ball_{R}(0)$ the ball of radius $R$ and define the annulus
\begin{align*}
    A_R := \ball_{2R}(0)\setminus \overline{\ball_R(0)}.
\end{align*}

\medskip
\noindent
\textbf{Assumptions.}
\begin{enumerate}[label=(\Alph*), ref=\Alph*]
    \item\label{A_embedd:A_bound_from_above_and_below}\textbf{Local boundedness and non-degeneracy.}
    \begin{enumerate}[label=(\theenumi\arabic*), ref=\theenumi\arabic*]
\item\label{A_embedd:A_1_global_bound_from_below} For $j \in (\N\cup \{0\})^{d}$ and $R>0$, we introduce
\begin{align*}
    I_{\weight_j}(\ball_{R}(0)) := \inf_{x \in \ball_{R}(0)} \weight_j(x).
\end{align*}
There exists $c_{R}>0$ such that for all $j$,
\begin{align*}
    I_{\weight_j}(\ball_{R}(0))  \ge c_{R}>0.
\end{align*}

\item\label{A_embedd:A_2_local_bound_on_balls} There exists $K>0$ such that for all $R \ge K$,
\begin{align*}
   \sup_{x \in \ball_R(0)} |\weight_j(x)| \le C_R < \infty. 
\end{align*}
\end{enumerate}
\item\label{A_embedd:B_decay_weight_ratio}\textbf{Asymptotic comparison of weights.}\newline There exists $K>0$ such that for all $j$, the ratio $\frac{\weight_j(x)}{\weightbar_j(x)}$ is monotonically decreasing in $|x|$ on $B_K(0)^c$, and
\begin{align*}
    \lim_{|x|\to\infty} \frac{\weight_j(x)}{\weightbar_j(x)} = 0.
\end{align*}
\item\textbf{Boundedness conditions at infinity for $\WeightclassB_\rho$.}\newline
\begin{enumerate}[label=(\theenumi\arabic*), ref=\theenumi\arabic*]
\item\label{A_embedd:C_1_bound_Cb_under_weightclass}Assume that there exists $\rho>1$ such that $\weight_j \in \WeightclassB_\rho$ for all $j$. In addition,
\begin{align*}
\lim_{R \to \infty}
\sum_{0 \le |j| \le s}
R^{|j| - \frac{d}{p}}
\|\weight_j^{-\frac{1}{\rho-1}}\|_{L^1(A_R)}^{\frac{\rho-1}{p\rho}}
< \infty.
\end{align*}

\item\label{A_embedd:C_2_bound_Cwc_under_weightclass}Assume that there exists $\rho>1$ such that $\weight_j \in \WeightclassB_\rho$ for all $j$. In addition,
\begin{align*}
\lim_{R \to \infty}
\sum_{0 \le |j| \le s}
R^{|j| - \frac{d}{p}}
\|\weightc\|_{L^\infty(A_R)}
\|\weight_j^{-\frac{1}{\rho-1}}\|_{L^1(A_R)}^{\frac{\rho-1}{p\rho}}
< \infty.
\end{align*}
\end{enumerate}
\item\textbf{Boundedness conditions for strictly non-degenerate weights.}

\begin{enumerate}[label=(\theenumi\arabic*), ref=\theenumi\arabic*]
\item\label{A_embedd:D_1_bound_Cb_under_non_degeneracy} For every $j$ and $R$, let $I_{\weight_j}(\ball_{R}(0))\geq c_{j,R}>0$, as well as
\begin{align*}
\lim_{R \to \infty}
\sum_{0 \le |j| \le s}
\left(\frac{1}{I_{\weight_j}(A_R)}\right)^{\frac{1}{p}}
R^{|j| - \frac{d}{p}}
\|u\|_{W^{s,p}_{\weight}}
< \infty.
\end{align*}

\item\label{A_embedd:D_2_bound_Cwc_under_non_degeneracy}For every $j$ and $R$, let $I_{\weight_j}(\ball_{R}(0))\geq c_{j,R}>0$, as well as
\begin{align*}
\lim_{R \to \infty}
\sum_{0 \le |j| \le s}
\|\weightc\|_{L^\infty(A_R)}
\left(\frac{1}{I_{\weight_j}(A_R)}\right)^{\frac{1}{p}}
R^{|j| - \frac{d}{p}}
\|u\|_{W^{s,p}_{\weight}}
< \infty.
\end{align*}
\end{enumerate}
\item\textbf{Sobolev exponent conditions.}\newline
Let $p \ge 1$, $d \ge 1$, $s,r \in \N\cup \{0\}$, and $\alpha \in [0,1)$.
\begin{enumerate}[label=(\theenumi\arabic*), ref=\theenumi\arabic*]
    \item\label{A_embedd:E_1_Sobolev_exponents_regularity}
    Assume that one of the following holds:
\begin{align*}
&p=1 \text{ and } s-r>d, \
&\text{or} \quad (s-r)\ge \frac{d}{p}+\alpha > (s-r)-1, \
&\text{or} \quad p=1 \text{ and } d=s-r-1.
\end{align*}
    \item\label{A_embedd:E_2_Sobolev_exponents_boundedness}
    Assume either $sp>d$ or $p=1$ and $s=d$.
\end{enumerate}
\end{enumerate}
\medskip
\noindent
\textbf{Conclusions.}
\begin{enumerate}
    \item[(I)]\textbf{ Compact embeddings.}\newline
    Let $0 \le l \le s$ and suppose that
\begin{align*}
    0 < d-(s-l)p < d,
        \qquad
        \frac{1}{p}-\frac{s}{d} < \frac{1}{q}-\frac{l}{d} \left(1 \le q < \frac{dp}{d-(s-l)p}\right),
\end{align*}
Assume \eqref{A_embedd:A_bound_from_above_and_below}, \eqref{A_embedd:B_decay_weight_ratio}
Then
\begin{align*}
    W^{s,p}_{\weight} \cap W^{l,q}_{\weightbar}
\hookrightarrow W^{l,q}_{\weight}
\end{align*}
is compact.

In the critical case $d = (s-l)p$, the embedding is compact for all $1 \le q < \infty$.
 \item[(II)]\textbf{Embeddings into continuous functions.}\newline
Let $p \ge 1$, $d \ge 1$, $s,r \in \N\cup \{0\}$, and $\alpha \in [0,1)$ and define
\[
    V :=
\begin{cases}
C^r(\R^{d}), & p=1,\ s-r\ge d, \\
C^{r,\alpha}(\R^{d}), & (s-r)\ge \frac{d}{p}+\alpha > (s-r)-1, \\
C^{r,1}(\R^{d}), & p=1,\ d=(s-r)-1.
\end{cases}
\]
\medskip
\noindent
 \begin{enumerate}
        \item There exists a $\rho>1$, such that for every $j\in (\N\cup \{0\})^{d}$, $\weight_{j}\in \WeightclassB_{\rho}$.
        
        \begin{enumerate}
        \item Let \eqref{A_embedd:E_1_Sobolev_exponents_regularity} hold. Then $W^{s,p\rho}_{\weight}\hookrightarrow V$.
            \item Let Conditions \eqref{A_embedd:C_1_bound_Cb_under_weightclass}, \eqref{A_embedd:E_2_Sobolev_exponents_boundedness} be satisfied, then $ W^{s,p\rho}_{\weight}\hookrightarrow C_{b}$.
        \item Let Conditions \eqref{A_embedd:C_2_bound_Cwc_under_weightclass}, \eqref{A_embedd:E_2_Sobolev_exponents_boundedness} be satisfied, then $W^{s,p\rho}_{\weight}\hookrightarrow C_{\weightc}$.
        \end{enumerate}
         \item  Let $I_{\weight_{j}}(\ball_{R}(0)):=\inf_{x\in \ball_{R}(0)}\weight_{j}(x)\geq c_{j,R}>0$ for every $j\in (\N\cup\{0\})^{d}$.
        
        \begin{enumerate}
         \item Let \eqref{A_embedd:E_1_Sobolev_exponents_regularity} hold. Then $W^{s,p}_{\weight}\hookrightarrow V$.
            \item Let Conditions \eqref{A_embedd:D_1_bound_Cb_under_non_degeneracy}, \eqref{A_embedd:E_2_Sobolev_exponents_boundedness} be satisfied, then $W^{s,p}_{\weight}\hookrightarrow C_{b}$.
        \item Let Conditions \eqref{A_embedd:D_2_bound_Cwc_under_non_degeneracy}, \eqref{A_embedd:E_2_Sobolev_exponents_boundedness}  be satisfied, then $W^{s,p}_{\weight}\hookrightarrow C_{\weightc}$.
        \end{enumerate}
    \end{enumerate}

\end{enumerate}
\end{proposition}

The \hyperref[proof:prop:embeddings]{proof} can be found in the Appendix, Section \ref{section:Appendix_proofs}.
\begin{remark}
    Discussion:
    \begin{enumerate}
        \item Condition \eqref{A_embedd:A_bound_from_above_and_below} is a condition on the boundedness of the individual weight components $\weight_{i}$ from below. Further, it ensures that, outside of a potentially very large ball, the weight component stays bounded, preventing $\lim_{|x|\rightarrow \infty} w_{i}(x)=\infty$. The weight components are however allowed to exhibit singularities within said ball.
          \item Condition \eqref{A_embedd:B_decay_weight_ratio} states that the, outside of some ball around $0$, each component $\weightbar_{i}$ of $\weightbar$, dominates the respective component $\weight_{i}$ of $\weight$. This is a crucial assumption, when it comes to compactness of embeddings in weighted spaces.
          \item Conditions \eqref{A_embedd:C_1_bound_Cb_under_weightclass},\eqref{A_embedd:C_2_bound_Cwc_under_weightclass}, \eqref{A_embedd:D_1_bound_Cb_under_non_degeneracy} and \eqref{A_embedd:D_2_bound_Cwc_under_non_degeneracy}   are of a technical nature. Let us consider \eqref{A_embedd:C_2_bound_Cwc_under_weightclass}:
          \begin{align*}
              \lim_{\R\rightarrow \infty}\sum_{0 \le |j| \le s}
R^{|j| - \frac{d}{p\rho}}
\|\weightc(R\cdot)\|_{L^{\infty}(A_1)}
\|\weight_j(R\cdot)^{-\frac{1}{\rho-1}}\|_{L^{1}(A_1)}^{\frac{\rho-1}{p\rho}}<\infty.
          \end{align*}
Heuristically, the weighted Sobolev structure is subcritical at infinity, meaning that after rescaling, all coefficients become negligible compared to the polynomial growth coming from derivatives up to order $s$. Let us illustrate this with a heuristic example, where $\approx$ means ''roughly of the form''. Let $\beta,\gamma_{j} >0$ for every $j$.

Note the absence of additional assumptions on the weight $\weightc$.
\begin{align*}
\weightc(x) \approx (1+|x|)^{-\beta}
\quad \Longrightarrow \quad
\|\weightc(R\cdot)\|_{L^\infty(A_1)} \sim R^{-\beta}.
\end{align*}
\begin{align*}
\weight_j(x) \approx (1+|x|)^{\gamma_j}
\quad \Longrightarrow \quad
\|\weight_j(R\cdot)^{-\frac{1}{\rho-1}}\|_{L^1(A_1)}
\sim R^{-\frac{\gamma_j}{\rho-1}}.
\end{align*}
Then the condition becomes
\begin{align*}
\lim_{R\to\infty}
\sum_{0 \le |j| \le s}
R^{|j| - \frac{d}{p\rho}}
\cdot R^{-\beta}
\cdot R^{-\frac{\gamma_j}{p\rho}}=\lim_{R\to\infty}
\sum_{0 \le |j| \le s}R^{|j| - \frac{d}{p\rho} - \beta - \frac{\gamma_j}{p\rho}}<\infty,
\end{align*}
which is satisfied if $|j| - \frac{d}{p\rho} - \beta - \frac{\gamma_j}{p\rho} \le 0$ for all $|j|\le s$.
\item Let us discuss the difference between \eqref{A_embedd:C_1_bound_Cb_under_weightclass},\eqref{A_embedd:C_2_bound_Cwc_under_weightclass}, \eqref{A_embedd:D_1_bound_Cb_under_non_degeneracy} and \eqref{A_embedd:D_2_bound_Cwc_under_non_degeneracy}.

$I_{\weight}(\ball_{R}(0))\geq c_{R}>0$ is a local non-degeneracy condition, which allows for decay to 0 at infinity and arbitrary growth.\newline
$\weight\in \WeightclassB_{p}(\R_{+})$ is not a lower bound, but an integrability condition on degeneracies. It allows $w(x)\rightarrow 0$, but gives a control on how fast this can happen.
$\weight\in \WeightclassB_{p}(\R_{+})$ is strictly weaker than $I_{\weight}(\ball_{R}(0))\geq c_{R}>0$. This can be easily seen, since $I_{\weight}(\ball_{R}(0))\geq c_{R}>0$ implies
\begin{align*}
    \int_{\ball_{R}(0)} \weight(x)^{-\frac{1}{p-1}}\,dx
\le c_R^{-\frac{1}{p-1}} |\ball_{R}(0)| < \infty,
\end{align*}
which implies $\weight^{-1/(p-1)} \in L^1_{\mathrm{loc}}(\R^d)$. In particular, $\text{local lower bound} \;\Rightarrow\; \weight \in \WeightclassB_p$.
The converse implication does not hold in general, which can be seen by the trivial example
\begin{align*}
    \weight(x) = |x|^\beta, \quad x \in \R^d,
\end{align*}
which belongs to $L^1_{\mathrm{loc}}(\R^d)$ if and only if $\frac{\beta}{p-1} < d$.

    \end{enumerate}
\end{remark}
\begin{remark}
    If one assumes a condition of the form $|\partial_{x}^{i}\weightc(x)|\leq C_{i} \weightc(x)$, the embedding results into $C_{\weightc}$ can be refined.
\end{remark}
\begin{remark}
    If we replace the finite limit in Assumption  \ref{A_embedd:C_1_bound_Cb_under_weightclass},\ref{A_embedd:D_1_bound_Cb_under_non_degeneracy} with  $0$, we obtain that $\lim_{|x|\rightarrow \infty}|u(x)|=0$. 
\end{remark}
\begin{remark}
    Assumption \ref{A_embedd:A_bound_from_above_and_below} can be avoided, if one considers weights, which behave on $B_{R}(0)$ like Muckenhoupt weights and used the established theory covering this type of weighted space (see e.g. \cite{meyries2014characterization}). This way one could allow for singularities at $0$, e.g. for $d=1$, $\weight(x)=x^{-a}(a+x)^{-b}$ with $0\leq a <1$ and $b\in \R$. The arguments are similar however, at later stages of the paper this would lead to further technicalities, which we want to avoid. We will hint at one reason in Remark \ref{rem:discussion_of_the_general_assumptions} why switching to such weights would not significantly improve our estimates to warrant the additional technicalities.
\end{remark}
We will see that weights that behave like $\weight(x)=(1+|x|)^{\beta}$ or $\weight(x)=(1+|x|^{2})^{\frac{\beta}{2}}$ , for $\beta\in \R$ will appear naturally in our analysis. However, these weights will often appear in the form $\weight(x)=(c_{1}+c_{2}|x|)^{\beta}$ with $c_{1},c_{2}>0$. For this reason, we will introduce an equivalence relation on the family of such weights.
\begin{definition}
    Let $a,b,r,s\geq 0$, $x\in \R^{d}$ and $\beta\in \R$. We call two weights $(a+b|x|)^{\beta}$ and $(r+s|x|)^{\beta}$ equivalent, denoted by $(a+b|x|)^{\beta}\cong (r+s|x|)^{\beta}$, if there exist constants $c_{1},c_{2}>0$, such that $(c_{1}a+c_{2}b|x|)^{\beta}= (r+s|x|)^{\beta}$.
\end{definition}
 It can be easily checked that this indeed defines an equivalence relation and the set of functions $\weight(x)=(c_{1}+c_{2}|x|)^{\beta}$ with $c_{1},c_{2}>0$.
Next, we verify that the two equivalent weights, in terms of the relation specified above, give rise to equivalent weighted Sobolev norms, in the sense that they induce the same topology. This result will be used in later calculations where, rather than keeping track of constants appearing in the weight, we collect them via the inequality stated in the following Lemma in a constant in front of the norm.
\begin{lemma}\label{lem:weighted_norm_equivalence}
    Let $\weight^{1}_{i}(x)=(a_{1,i}+a_{2,i}|x|)^{\beta_{i}}$ and $\weight^{2}_{i}(x)=(b_{1,i}+b_{2,i}|x|)^{\beta_{i}}$ be equivalent weights with $a_{1,i}, a_{2,i}, b_{1,i}, b_{2,i}>0$ and $\beta_{i}\in \R$. Then $\|\cdot\|_{W^{m,p}_{\weight^{1}}}$ and $\|\cdot\|_{W^{m,p}_{\weight^{2}}}$ are equivalent in the sense of norms, in the sense that there exist two constants $c\leq C>0$ which only depend on $a_{1,i}, a_{2,i}, b_{1,i}, b_{2,i}$ and $\beta_{i}$, such that
    \begin{align*}
        c\|\cdot\|_{W^{m,p}_{\weight^{2}}}\leq \|\cdot\|_{W^{m,p}_{\weight^{1}}}\leq C\|\cdot\|_{W^{m,p}_{\weight^{2}}}.
    \end{align*}
    Additionally $\|\cdot\|_{W^{m,p}_{\frac{1}{\weight^{1}}}}$ and $\|\cdot\|_{W^{m,p}_{\frac{1}{\weight^{2}}}}$ are equivalent
\end{lemma}
The \hyperref[proof:lem:weighted_norm_equivalence]{proof} can be found in the Appendix.
\begin{remark}
    We want to keep the following trivial estimates in mind. For $\eta>0$ and $x\geq 0$,
    \begin{align*}
        \frac{1+x}{1+\eta x}\leq 1\vee \frac{1}{\eta},\quad \frac{1+x}{\eta+ x}\leq 1\vee \frac{1}{\eta}.
    \end{align*}
\end{remark}
For convenience, we want to state a Corollary of Proposition \ref{prop:embeddings}, specified to the settings where our weight components are of the form $\weight_{i}(x)=(1+|x|)^{\alpha_{i}}$.

\begin{corollary}[of Proposition \ref{prop:embeddings}]
    Let $m\geq 0$, $i\in (\N\cup \{0\})^{d}$, $d\leq mp$ (or $d=1$ and $m=d$) and $\weight=(\weight_{0},\dots,\weight_{m})$, with $\weight_{j}(x)=(1+|x|)^{\alpha_{j}}$. If $p|j| \leq \alpha_{j}+d$ holds for every $j$, then the embedding $W^{m,p}_{\weight}\hookrightarrow C_{b}$ is continuous.
\end{corollary}
\begin{proof}
    \begin{align*}
\sum_{0 \le |j| \le s}
\left(\frac{1}{I_{\weight_j}(A_R)}\right)^{\frac{1}{p}}
R^{|j| - \frac{d}{p}}&\leq C \sum_{0 \le |j| \le s}
R^{-\frac{\alpha_{j}}{p}}
R^{|j| - \frac{d}{p}},
\end{align*}
which relates to $   |j| - \frac{d}{p}\leq \frac{\alpha_{j}}{p}$ or equivalently $p|j| \leq \alpha_{j}+d$. 
\end{proof}
\begin{remark}
Due to the local boundedness from below for every $\alpha\in \R$, $\weight(x_{j})=(1+|x|)^{\alpha_{j}}\in \WeightclassB_{\rho}$, for any $\rho>1$. Condition \eqref{A_embedd:C_1_bound_Cb_under_weightclass} becomes
\begin{align*}
    &\lim_{R\rightarrow\infty }\sum_{0\leq |j|\leq s}R^{|j| - \frac{d(\rho+1) + \alpha_{j}}{p\rho}}<\infty.
\end{align*}
Hence, we would require $ |j| \le \frac{d(\rho+1) + \alpha_{j}}{p\rho}$, or equivalently $|j|p\rho - d(\rho+1)\leq \alpha_{j}$.
\end{remark}
\begin{remark}
For the weight $\weight_{j}(x)=(1+|x|)^{|j|p-d}$, the condition becomes $|j| \le \frac{d\rho}{p(\rho-1)}$. These weights are also elements of $A_{q}$ (the Muckenhoupt class), with $\frac{|j|p-d}{d}+1<q$.
\end{remark}

Later on, we focus on a particular choice of weights, however, if we do not specify that our weight is of a specific form then we impose the following standing assumption for any weight in the remainder of this work.
\begin{assumption}\label{A:assumption_weight_function}
    Let $m\in\N$ be given. We assume  that the weight function $\weight=(\weight_{0},\dots,\weight_{m})$ satisfies $\weight_{j}\in \WeightclassB_{2}(\R_{+})$ as well as $\weight_{j}\in L^{\infty}_{\operatorname{loc}}$, for every $i\in \{0,\dots,k\}$. 
    For a given weight function $\weight$, let $\frac{1}{\weight}$ also be a weight function satisfying the above assumptions. 
\end{assumption}
\begin{remark}
The weight $\weight=(\weight_{0},\dots,\weight_{m})$, with $\weight_{j}(x)=(1+|x|)^{\alpha_{j}}$ and $\alpha_{j}\in \R$ satisfies this relation, since both terms $\weight_{j}$ and $\frac{1}{\weight_{j}}$ are locally bounded from above and below and hence elements of $\WeightclassB_{\rho}$ for any $\rho>1$.
\end{remark}
\subsection{Duality}
Depending on the properties of the weight, we will identify the dual space of $W^{m,p}_{\weight}$ in a suitable fashion.
\begin{itemize}
    \item If $\lim_{|x|\rightarrow \infty} \weight_{0}=\infty$,
    we identify the dual space of $W^{m,p}_{\weight}$, via the (unweighted) $L^{2}$ duality with $W^{-m,q}_{\weightinverse}(\R_{+}^{d})$, whose norm is given by
\begin{align*}
    \|v\|_{W^{-m,q}_{\weightinverse}}=\sup_{u \colon \|u\|_{W^{m,p}_{\weight}}=1}|\langle v,u\rangle|.
\end{align*}
In other words, we work on the following triple(s) of spaces
\begin{align*}
    W^{m,p}_{\weight} \hookrightarrow L^{2}\hookrightarrow W^{-m,q}_{\weightinverse}.
\end{align*}
\item If $\lim_{|x|\rightarrow \infty} \weight_{0}=0$,
    we identify the dual space of $W^{m,p}_{\weight}$, via the (weighted) $L^{2}_{\pivotweight}$ duality with $W^{-m,q}_{\weightinverse}(\R_{+}^{d})$, whose norm is given by
\begin{align*}
    \|v\|_{W^{-m,q}_{\weightinverse}}=\sup_{u \colon \|u\|_{W^{m,p}_{\weight}}=1}|\langle v,u\rangle|.
\end{align*}
$\pivotweight$ is a weight satisfying $\weight_{0}\geq\pivotweight$ for every $x\in \R_{+}^{d}$.
In other words, we work on the following triple(s) of spaces
\begin{align*}
    W^{m,p}_{\weight} \hookrightarrow L^{2}_{\pivotweight}\hookrightarrow W^{-m,q}_{\weightinverse}.
\end{align*}
\end{itemize}
\begin{remark}
We note that the notation used for the dual spaces is somewhat imprecise. Although we employ the same notation in different contexts, the underlying structures of these dual spaces differ significantly. This distinction can be illustrated heuristically using weighted $L^2$ spaces.

Consider the example $\weight(x) = (1+|x|)^\alpha$ with $\alpha \in \mathbb{R}$. The dual of $L^2_{\weight}$ can be identified with $L^2_{1/\weight}$ when using the unweighted $L^2$ pairing. Formally, this is expressed by
\begin{align*}
    \left|\int f(x)g(x)\,dx\right|
    \leq 
    \left(\int f^2(x)\weight(x)\,dx\right)^{\frac{1}{2}}
    \left(\int g^2(x)\frac{1}{\weight(x)}\,dx\right)^{\frac{1}{2}}.
\end{align*}

On the other hand, if one considers the weighted pairing, the dual of $L^2_{\weight}$ can be identified with itself. In this case,
\begin{align*}
    \left|\int f(x)g(x)\weight(x)\,dx\right|
    \leq 
    \left(\int f^2(x)\weight(x)\,dx\right)^{\frac{1}{2}}
    \left(\int g^2(x)\weight(x)\,dx\right)^{\frac{1}{2}}.
\end{align*}

This perspective also provides heuristic motivation for our choice of duality. The objects we consider will belong to a dual space and may exhibit growth at infinity, which needs to be compensated by an appropriate weight. Accordingly, we consider
\begin{align*}
    \left|\int f(x)g(x)\pivotweight(x)\,dx\right|
    \leq 
    \left(\int f^2(x)\pivotweight(x)\weight(x)\,dx\right)^{\frac{1}{2}}
    \left(\int g^2(x)\frac{\pivotweight(x)}{\weight(x)}\,dx\right)^{\frac{1}{2}}.
\end{align*}

In the case where $\weight(x)\to 0$ as $|x|\to\infty$, the auxiliary weight $\pivotweight$ serves to control the term $\frac{\pivotweight(x)}{\weight(x)}$. Since $\pivotweight \leq \weight$, this ratio remains bounded and does not diverge. 
\end{remark}

\begin{remark}
Consider $m\geq 1$ and a chain of (continuous) embeddings of weighted Sobolev spaces
\[
W^{m,2}_{\weightminus} \hookrightarrow W^{m,2}_{\weightsim} \hookrightarrow W^{m,2}_{\weightplus},
\qquad \text{with } \weightminus \geq \weightsim \geq \weightplus.
\]
When passing to dual spaces, we define the duality using the first component of the weakest weight appearing in the primal sequence. More precisely, we take $\pivotweight = (\weightplus)_0$ and use the $L^2_{\pivotweight}$ pairing.

With this choice, the corresponding dual embeddings become
\[
W^{-m,2}_{\weightplusinverse} \hookrightarrow W^{-m,2}_{\frac{1}{\weightsim}} \hookrightarrow W^{-m,2}_{\weightminusinverse}.
\]
Clearly, also 
\begin{align*}
    W^{-m+1,2}_{\weightplusinverse} \hookrightarrow  W^{-m,2}_{\weightplusinverse}
\end{align*}
holds and is continuous.
\end{remark}

\subsection{Measures as Dual Elements}

Let $\nu$ be  Radon measure on $\R_{+}$. We will slightly abuse the notation ``tempered'' and refer to measures $\nu$, satisfying $\int_{\R_{+}}\frac{1}{\weight(x)}\nu(\dx)<\infty$, as being ($w$-)tempered.

In the following Lemma, we identify a $w$-tempered measure with an element in the dual of certain weighted Sobolev spaces.
\begin{lemma}\label{lem:nu_is_in_dual}
   Let $\weightc, \weighttilde$ be a weight function satisfying the conditions \eqref{A_embedd:A_bound_from_above_and_below} and \eqref{A_embedd:D_2_bound_Cwc_under_non_degeneracy}. Let $d=1$, $p=2$, $i=1,2$. If $\nu$ is a non-negative, measure on $\R_{+}$, such that $\int_{\R_{+}}\frac{1}{\weightc(x)}\nu(\dx)<\infty$, then $\nu \in \left(W^{1,2}_{\weighttilde}\right)^{\prime}=W^{-1,2}_{\frac{1}{\weighttilde}}$.\newline
   In particular, if there exists a $\theta_{\nu}\in \R$, such that $\int_{0}^{\infty}\frac{1}{(1+x)^{\theta_{\nu}}}\nu(\dx)<\infty$, then $\nu\in W^{-1,2}_{\frac{1}{\weight}}$ with $\weight_{i}(x)=(1+x)^{-1+2\theta_{\nu}+2i}$.
\end{lemma}
\begin{proof}
Proposition \ref{prop:embeddings} yields that $\varphi\in W^{1,2}_{\weighttilde}$ is continuous and bounded on $\R_{+}$ and in particular $\varphi\in C_{\weightc}(\R_{+})$. Hence, the dual pairing with a ($\weight$-tempered) Radon measure is well-defined.
    \begin{align*}
        \sup_{\|\varphi\|_{W^{1,2}_{\weighttilde}=1}}|\langle \nu,\varphi\rangle|&=\sup_{\|\varphi\|_{W^{1,2}_{\weighttilde}=1}}\left|\int_{\R_{+}} \varphi(x) \nu(\dx)\right|\\
&=\sup_{\|\varphi\|_{W^{1,2}_{\weighttilde}=1}}\left|\int_{\R_{+}} \varphi(x) \frac{\weightc(x)}{\weightc(x)} \nu(\dx)\right|\\
        &\leq \sup_{\|\varphi\|_{W^{1,2}_{\weighttilde}=1}}\|\weightc \varphi\|_{L^{\infty}(\R_{+})}\left|\int_{\R_{+}}  \frac{1}{\weightc(x)} \nu(\dx)\right|\\
        &\leq C\sup_{\|\varphi\|_{W^{1,2}_{\weighttilde}=1}}\|\varphi\|_{W^{1,2}_{\weighttilde}}\left|\int_{\R_{+}}  \frac{1}{\weightc(x)} \nu(\dx)\right|.
    \end{align*}
    The second statement follows almost analogously, but one needs to be a bit more careful determining the correct weighted Sobolev norm. Like in the proof of the embedding theorem, we split the calculation into one part on balls and one part, estimating the function on annuli.
      Since $\|(1+x)^{\theta_{\nu}}\varphi\|_{L^{\infty}(\ball_{R}(0))}\leq C(w,R,\theta_{\nu}) \|\varphi\|_{W^{1,2}_{\weight}}$ for any weight $w$ that is bounded from above and below on $\ball_{R}(0)$, it suffices to consider the estimates on the annuli $A_{R}$. We consider only the case $\theta_{\nu}\geq 0$, as the case $\theta_{\nu}\leq 0$ is identical, but with a different constant.
    \begin{align*}
        \|(1+x)^{\theta_{\nu}}\varphi\|_{L^{\infty}(A_{R})}&\leq C2^{\theta_{\nu}}(1+R)^{\theta_{\nu}}\left(\sum_{j=0}^{1}\int_{R}^{2R} |D^{j}\varphi(x)|^{2}(1+x)^{j2-1}\dx\right)^{1/2}\\
        &= C\left(\sum_{j=0}^{1}\int_{R}^{2R} |D^{j}\varphi(x)|^{2}(1+R)^{2\theta_{\nu}}(1+x)^{j2-1}\dx\right)^{1/2}\\
        &\leq C\left(\sum_{j=0}^{1}\int_{R}^{2R} |D^{j}\varphi(x)|^{2}(1+x)^{2\theta_{\nu}}(1+x)^{j2-1}\dx\right)^{1/2}\\
         &\leq C\left(\sum_{j=0}^{1}\int_{\R_{+}} |D^{j}\varphi(x)|^{2}(1+x)^{2\theta_{\nu}}(1+x)^{j2-1}\dx\right)^{1/2}.
    \end{align*}
    Hence, 
    \begin{align*}
        \|(1+x)^{\theta_{\nu}}\varphi\|_{L^{\infty}(\R_{+})}\leq \|(1+x)^{\theta_{\nu}}\varphi\|_{L^{\infty}(\ball_{R}(0))}+\sup_{R>0}\|(1+x)^{\theta_{\nu}}\varphi\|_{L^{\infty}(A_{R})}\leq 2C\|\varphi\|_{W^{1,2}_{\weight}}.
    \end{align*}
\end{proof}
\begin{remark}
    In Proposition \ref{prop:embeddings}, as well as Lemma \ref{lem:nu_is_in_dual}, it is possible to replace $\weightc$ and $\frac{1}{\weighttilde}$ by weights of the for $\frac{1}{\weightc}$ and $\weighttilde$ respectively, where $\weightc$ as well as $\weighttilde$ are increasing towards $\infty$ and $\partial_{x}\frac{1}{\weighttilde}\leq C \frac{1}{\weighttilde}$, for some nonnegative constant $C$. The arguments are identical but are skipped for brevity. 
\end{remark}
 We note that there might be many choices of weights, for which $\int_{\R_{+}}\frac{1}{\weight(x)}\nu(\dx)<\infty$ is satisfied and some might be more canonical than others, given a particular measure $\nu$ and application in mind.

\subsection{Choice of Weights}
The choice of weight functions plays a central role in our analysis. It encodes the decay properties of the measures $\nudrift$ and $\nudiffusion$ and determines the regularity of the lifted process.

\begin{definition}\label{def:weights_triple_for_analysis}
Let $\theta_{\nudrift}$ and $\theta_{\nudiffusion}$ be as in Assumption \ref{A:A_3:Narrower_assumption_nu}. We introduce three weight functions $(\weightminus, \weightsim, \weightplus)$ on $\R_{+}$ defined by
\begin{align*}
(\weightminus)_{i}(x) &:= (1+x)^{2\eta_{-}-1+2i},\\
(\weightsim)_{i}(x) &:= (1+x)^{2\eta_{\sim}-1+2i},\\
(\weightplus)_{i}(x) &:= (1+x)^{2\eta_{+}-1+2i},
\end{align*}
for $i \geq 0$, where the parameters satisfy
\begin{align*}
\eta_{-} &> \max\{\theta_{\nudrift},\theta_{\nudiffusion}\},\\
\eta_{+} &=
\begin{cases}
-\varepsilon, & \text{if } \theta_{\nudiffusion} < \frac{1}{2},\\
\theta_{\nudiffusion} - \frac{1}{2} + \delta, & \text{if } \theta_{\nudiffusion} > \frac{1}{2},
\end{cases}\\
\eta_{+} &< \eta_{\sim} < \eta_{-},
\end{align*}
for suitable $\varepsilon, \delta > 0$.
\end{definition}

These weights induce a hierarchy of spaces with continuous embeddings
\[
\Wplustwodual(\R_{+}) \hookrightarrow \Wmidtwodual(\R_{+}) \hookrightarrow \Wminustwodual(\R_{+}).
\]

This structure reflects the trade-off between integrability and regularity and will be essential for handling the deterministic and stochastic terms of the lifted equation.

\subsubsection{Interplay between Kernel Singularity, Measures, and Weights: An Algorithmic Perspective}

A central structural principle underlying our framework is the precise relationship between 
\begin{itemize}
    \item the singularity of the Volterra kernel,
    \item the decay properties of the associated Laplace measure $\nu$, and
    \item the choice of weights in the Sobolev spaces.
\end{itemize}
We make this relationship explicit and formulate it as an algorithmic procedure for constructing admissible weights.

\textbf{Step 1: Kernel and Laplace representation.}
Let $k : \mathbb{R}_+ \to \mathbb{R}_+$ be completely monotone. Then there exists a Radon measure $\nu$ such that
\[
k(t) = \int_0^\infty e^{-xt} \, \nu(dx).
\]
The behavior of $k(t)$ as $t \to 0$ is encoded in the behavior of $\nu$ as $x \to \infty$.

\textbf{Step 2: Definition of $\theta_\nu$.}
We quantify the decay of the measure $\nu$ via a parameter $\theta_\nu \in \mathbb{R}$ such that
\[
\int_0^\infty \frac{1}{(1+x)^{\theta_\nu}} \, \nu(dx) < \infty.
\]
This condition characterizes how heavy the tails of $\nu$ are. In particular, stronger singularities of the kernel correspond to larger values of $\theta_\nu$.

\textbf{Step 3: Duality requirement.}
In the lifted formulation, the measure $\nu$ must define a bounded linear functional on a weighted Sobolev space. A sufficient condition is
\[
\int_0^\infty \frac{1}{w(x)} \, \nu(dx) < \infty,
\]
which ensures that $\nu \in (W^{1,2}_w)'$. Thus, the weight $w$ must dominate the decay encoded by $\theta_\nu$.

\textbf{Step 4: Matching decay and weights.}
Assume that $\theta_\nu$ is chosen such that
\[
\int_0^\infty \frac{1}{(1+x)^{\theta_\nu}} \, \nu(dx) < \infty.
\]
We then select weights of the form
\[
w_i(x) = (1 + x)^{2\eta - 1 + 2i}, \quad i \geq 0,
\]
where $\eta$ is a parameter to be determined. The duality condition is satisfied provided
\[
\eta > \theta_\nu.
\]
Thus, the growth of the weight must strictly exceed the decay exponent of the measure.

\textbf{Step 5: Choice of weight parameters.}
Given the parameter $\eta > \theta_\nu$, we then define intermediate parameters $\eta_{(0)}$ and $\eta_+$ such that
\[
\eta_+ < \eta_{(0)} < \eta_-,
\]
with $\eta_+$ chosen according to the specific integrability requirements of the stochastic term (cf.\ Section 3.5).

This yields three weight families
\[
(w_-)_i(x) = (1 + x)^{2\eta_- - 1 + 2i}, \quad
(w_{(0)})_i(x) = (1 + x)^{2\eta_{(0)} - 1 + 2i}, \quad
(w_+)_i(x) = (1 + x)^{2\eta_+ - 1 + 2i}.
\]

\textbf{Step 6: Functional-analytic structure.}
The above choice induces a hierarchy of spaces
\[
W^{-2,2}_{1/w_+}(\mathbb{R}_+) \hookrightarrow W^{-2,2}_{1/w_{(0)}}(\mathbb{R}_+) \hookrightarrow W^{-2,2}_{1/w_-}(\mathbb{R}_+),
\]
which reflects a balance between regularity and integrability. The strongest weight $w_-$ ensures that the measures $\nu_b, \nu_\sigma$ belong to the corresponding dual space, while the weaker weights allow for the treatment of the semigroup and stochastic terms.

\textbf{Example (power-law kernel).}
Consider a kernel with singularity of the form
\[
k(t) \sim t^{-\alpha}, \quad \alpha \in (0,1).
\]
The associated measure satisfies $\nu(dx) \sim x^{\alpha - 1} dx$ for large $x$, and one may take any $\theta_\nu > \alpha$ (note that $\theta_\nu = \alpha$ exactly is not admissible: the defining integral $\int_0^\infty (1+x)^{-\theta_\nu}\nu(dx)$ reduces to $\int^\infty x^{-1}\,dx$ for large $x$ when $\theta_\nu=\alpha$, which diverges logarithmically). The admissibility condition then becomes
\[
\eta > \alpha.
\]
Accordingly, we choose
\[
w_i(x) = (1 + x)^{2\eta - 1 + 2i}, \quad \eta > \alpha,
\]
which ensures that $\nu \in W^{-1,2}_{1/w}$.

\textbf{Summary (algorithmic recipe).}
Given a completely monotone kernel $k$, the construction of weights proceeds as follows:
\begin{enumerate}
    \item Determine $\theta_\nu$ such that $\int (1+x)^{-\theta_\nu} \nu(dx) < \infty$.
    \item Choose $\eta_- > \max\{\theta_{\nu_b}, \theta_{\nu_\sigma}\}$.
    \item Select $\eta_+$ and $\eta_{(0)}$ such that $\eta_+ < \eta_{(0)} < \eta_-$.
    \item Define weights via $w_i(x) = (1 + x)^{2\eta - 1 + 2i}$.
    \item Use the induced hierarchy of spaces to control the different terms of the equation.
\end{enumerate}

In this way, the choice of Sobolev weights is entirely dictated by the decay properties of the measures associated with the kernel, and hence by the singularity structure of the original Volterra equation.

\begin{remark}
    The lifted equation will actually contain matrices of measures. To alleviate the notation, we will interpret conditions on $\nudrift$ and $\nudiffusion$ component-wise, i.e.\ for every $\nudrift_{ij}$ and $\nudiffusion_{ij}$.
\end{remark}

\subsection{Operators and Semigroups}\label{sec:operator_and_semigroup}

Let $x\in \R_{+}$ and $\mu$ be a Radon-measure on $\R_{+}$. A central role in the lifted equation is played by the operator
\[
A\mu := -x \mu.
\]
We interpret $A$ as an (unbounded) operator on the dual space $W^{-m,2}_{1/w}$.

The choice of weights ensures that $A$ is well-defined on a suitable domain and generates a strongly continuous semigroup given formally by
\[
(e^{tA}\mu)(x) = e^{-xt}\mu(x).
\]

This semigroup corresponds to the deterministic part of the lifted equation and encodes the decay of memory over time.

\begin{lemma}\label{lem:semigroup_and_adjoint_semigroup}
    Let $1\leq p <\infty$. The family $(S_{t})_{t}\coloneqq (e^{- t \cdot})_{t}$ of linear operators, where $e^{- t \cdot}\colon f\mapsto (x\mapsto f(x)e^{- t x})$ is
    \begin{enumerate}
        \item  a strongly continuous contraction semigroup on $W^{0,p}_{\weight}$, a strongly continuous semigroup on $W^{m,p}_{\weight}$ ($m\geq 0$). It has a densely defined generator, given by
          \begin{align*}
        &(A f)(x)\coloneqq -x f(x), \quad x\geq 0,\\
        &\D(A)\coloneqq \left\{f\in W^{m,p}_{\weight} \colon \|-xf\|_{W^{m,p}_{\weight}}<\infty\right\}.
    \end{align*}
        \item The adjoint semigroup $(S^{*}_{t})_{t}$ is a strongly continuous semigroup on $W^{-m,q}_{\dualweight}$,
    \end{enumerate}
\end{lemma}
\begin{proof}
\begin{enumerate}
    \item  Let $m\geq 0$ and $f\in W^{m,p}_{\weight}$, then
     \begin{align*}
        \|S_{t}f-f\|_{W^{m,p}_{\weight}}&=\left(\sum_{j=0}^{m}\int_{\R_{+}}\left|D^{j}e^{-tx}f(x)-D^{j}f(x)\right|^{p}\weight_{j}(x)\dx\right)^{\frac{1}{p}}\\
    &=\left(\sum_{j=0}^{m}\int_{\R_{+}}\left|D^{j}\left(f(x)\left(e^{-tx}-1\right)\right)\right|^{p}\weight_{j}(x)\dx\right)^{\frac{1}{p}}\\
     &=\left(\sum_{j=0}^{m}\int_{\R_{+}}\left|\sum_{i=0}^{j}\begin{pmatrix}
j\\
i
\end{pmatrix}D^{j-i}\left(e^{-tx}-1\right)D^{i}f(x)\right|^{p}\weight_{j}(x)\dx\right)^{\frac{1}{p}}\\
&=\left(\sum_{j=0}^{m}\int_{\R_{+}}\left|\left(e^{-tx}-1\right)D^{j}f(x)+\sum_{i=0}^{j-1}\begin{pmatrix}
j\\
i
\end{pmatrix}(-t)^{j-i}e^{-tx}D^{i}f(x)\right|^{p}\weight_{j}(x)\dx\right)^{\frac{1}{p}}.
    \end{align*}
    By the dominated convergence theorem, we conclude that $\lim_{t\rightarrow 0} \|S_{t}f-f\|_{W^{m,p}_{\weight}} =0$.
    We define $A\colon \D(A)\rightarrow W^{m,p}_{\weight}$ by 
    \begin{align*}
        &(A f)(x)\coloneqq -x f(x), \quad x\geq 0,\\
        &\D(A)\coloneqq \left\{f\in W^{m,p}_{\weight} \colon \|-xf\|_{W^{m,p}_{\weight}}<\infty\right\}.
    \end{align*}
    Note that $\D(A)\supset C_{0}^{\infty}$, which is dense in $W^{m,p}_{\weight}$ by definition.
        \begin{align*}
        \frac{1}{t}&\left\|S_{t}f-f-tAf\right\|_{W^{m,p}_{\weight}}=\frac{1}{t}\left(\sum_{j=0}^{m}\int_{\R_{+}}\left|D^{j}\left(\left(e^{-tx}-1+x\right)f(x)\right)\right|^{p}\weight_{j}(x)\dx\right)^{\frac{1}{p}}\\
    &=\frac{1}{t}\left(\sum_{j=0}^{m}\int_{\R_{+}}\left|\sum_{i=0}^{j}\begin{pmatrix}
j\\
i
\end{pmatrix}D^{j-i}\left(e^{-tx}-1+tx\right)D^{i}f(x)\right|^{p}\weight_{j}(x)\dx\right)^{\frac{1}{p}}\\
&=\frac{1}{t}\left(\sum_{j=0}^{m}\int_{\R_{+}}\left|\left(e^{-tx}-1+tx\right)D^{j}f(x)+j\left(-te^{-tx}+t\right)D^{j-1}f(x)\right.\right.\\
&\phantom{xxxx}+\left.\left.\sum_{i=0}^{j-2}\begin{pmatrix}
j\\
i
\end{pmatrix}(-t)^{j-i}e^{-tx}D^{i}f(x)\right|^{p}\weight_{j}(x)\dx\right)^{\frac{1}{p}}\\
&\leq C\left(\sum_{j=0}^{m}\int_{\R_{+}}\left|\frac{\left(e^{-tx}-1+tx\right)}{t}D^{j}f(x)+j\left(1-e^{-tx}\right)D^{j-1}f(x)\right|^{p}\weight_{j}(x)\dx\right)^{\frac{1}{p}}+ O^{t\rightarrow 0}(t)\\
&\leq C\left(\sum_{j=0}^{m}\int_{\R_{+}}\left|\left|1-\frac{1}{t}\int_{0}^{t}e^{-sx}\ds\right||D^{j}f(x)|+j\left|1-e^{-tx}\right||D^{j-1}f(x)|\right|^{p}\weight_{j}(x)\dx\right)^{\frac{1}{p}}+ O^{t\rightarrow 0}(t),
    \end{align*}
    where $O^{t\rightarrow 0}(t)$ denotes terms of order $t$, as $t\rightarrow 0$.
    By the dominated convergence theorem, the right-hand side vanishes, as $t\rightarrow 0$. The contraction property, when $k=0$, can be seen immediately, since $|e^{-tx}|\leq 1$ for any $0\leq t,x$.
    \item 
     \begin{align*}
        \left|\left\langle S_{t}^{*}\nu-\nu,f\right\rangle\right|=\left|\left\langle \nu,(S_{t}-1)f\right\rangle\right|.
    \end{align*}
    By the same arguments as above the semigroup is weakly continuous, i.e. for every $f\in W^{m,p}_{\weight}$ $\lim_{t\rightarrow 0}  \left|\left\langle S_{t}^{*}\nu-\nu,f\right\rangle\right|=0$. By \cite[Theorem 5.8]{engel_nagel_2000_semigroups}, the semigroup $S^{*}$ is even strongly continuous on $W^{-m,p}_{\dualweight}$.

\end{enumerate}
   
\end{proof}
\begin{remark}
     We can define
 \begin{align*}
     A^{*} \nu \coloneqq \operatorname{weak}^{*}\lim _{t \rightarrow 0} \frac{1}{t}\left(S^{*}_{t}\nu-\nu\right)
 \end{align*}
on the domain
\begin{align*}
    \D\left( A^{*}\right)\coloneqq \left\{\nu \in W^{-m,p}_{\dualweight}\colon \operatorname{weak}^{*}\lim _{t \rightarrow 0} \frac{1}{t}\left(S^{*}_{t}\nu-\nu\right) \text { exists }\right\}.
\end{align*}

$A^{*}$ is a $\operatorname{weak}^{*}$-closed and $\operatorname{weak}^{*}$-densely defined operator and coincides with the adjoint $A^{*}$ of $A$ (see \cite[Definition B.8]{engel_nagel_2000_semigroups}), i.e.,
\begin{align*}
     \D\left( A^{*}\right)\coloneqq \left\{\nu \in W^{-m,p}_{\frac{1}{w}}\colon 
\begin{array}{l}
\text { there exists } \eta \in W^{-m,p}_{\frac{1}{w}} \text { such that } \\
\left\langle f, \eta\right\rangle=\left\langle A f, \nu\right\rangle \text { for all } f \in \D(A)
\end{array}\right\},
\end{align*}
and $A^{*}$ is the adjoint of $A$.
By \cite[Corollary B. 12]{engel_nagel_2000_semigroups} it then follows that $\sigma\left(A^{*}\right)=\sigma(A)$.

Also note that in our case the adjoint semigroup $(S^{*}_{t})_{t}$ coincides with the so-called sun-dual semigroup of $(S_{t})_{t}$ (see \cite[Section 2.6]{engel_nagel_2000_semigroups}).
\end{remark}
From this point onwards, we set $p=2$ and restrict our analysis to the Hilbert-space case. In the next lemma, we want to investigate certain mapping properties of the semigroup $S^{*}$.

\textbf{Estimates for the adjoint semigroup}
\begin{lemma}\label{lem:improvement_semigroup_weighted_spaces}
   Let $0<t$ and let $\eta\in W^{-m,2}_{\frac{1}{w}}$, then $S^{*}\eta \in W^{-m,2}_{\frac{1}{\weighttilde}}$ with $ \left(1+x\right)^{-2\gamma}\weight_{i}(x)\leq \weighttilde_{i}(x)$, for any $\gamma \in [0,1]$.  If $0<t\leq T<\infty$,
   \begin{align*}
       \|S^{*}_{t}\eta\|_{W^{-m,2}_{\frac{1}{\weighttilde}}}\leq C \|\eta\|_{W^{-m,2}_{\frac{1}{w}}}\frac{(1\vee T)^{\gamma}(1+t^{m})}{t^{\gamma}}.
   \end{align*}
\end{lemma}
\begin{proof}
\begin{align*}
     \|S^{*}_{t}\eta\|_{W^{-m,2}_{\frac{1}{\weighttilde}}}&=\sup_{\psi\colon \|\psi\|_{W^{m,2}_{\weighttilde}}=1}|\langle \eta,S_{t}\psi\rangle|\\
     &\leq \|\eta\|_{W^{-m,2}_{\frac{1}{w}}}\sup_{\psi\colon \|\psi\|_{W^{m,2}_{\weighttilde}}=1}\sqrt{ \sum_{j=0}^{m}\int_{\R_{+}}|D^{j}(e^{-tx}\psi(x))|^{2}\weight_{j}(x)\dx}\\
     &\leq \|\eta\|_{W^{-m,2}_{\frac{1}{w}}}\sup_{\psi\colon \|\psi\|_{W^{m,2}_{\weighttilde}}=1}\sqrt{\int_{\R_{+}} \sum_{j=0}^{m}\left|\sum_{i=1}^{j}\begin{pmatrix}
j\\
i
\end{pmatrix}(-t)^{j-i}e^{-tx}D^{i}\psi(x)\right|^{2}\weight_{j}(x)\dx}\\
&\leq C \|\eta\|_{W^{-m,2}_{\frac{1}{w}}}\sup_{\psi\colon \|\psi\|_{W^{m,2}_{\weighttilde}}=1}\sqrt{\int_{\R_{+}} (1+t^{m})^{2}\sum_{j=0}^{m}\int_{\R_{+}}|D^{j}\psi(x)|^{2}\left(\frac{1}{1+tx}\right)^{2\gamma}\weight_{j}(x)\dx},
\end{align*}
for any $\gamma\in [0,1]$. Hence,
\begin{align*}
    \|S^{*}_{t}\eta\|_{W^{-m,2}_{\frac{1}{\weighttilde}}}
&\leq C \|\eta\|_{W^{-m,2}_{\frac{1}{w}}}(1+t^{m})\\
&\phantom{xxxx}\times\sup_{\psi\colon \|\psi\|_{W^{m,2}_{\weighttilde}}=1}\sqrt{\int_{\R_{+}} \sum_{j=0}^{m}\int_{\R_{+}}|D^{j}\psi(x)|^{2}\frac{1}{t^{2\gamma}}\left(\frac{(1\vee T)}{1+x}\right)^{2\gamma}\weight_{j}(x)\dx}\\
&\leq C \|\eta\|_{W^{-m,2}_{\frac{1}{w}}}\frac{(1+t^{m})(1\vee T)^{\gamma}}{ t^{\gamma}}\\
&\phantom{xxxx}\times\sup_{\psi\colon \|\psi\|_{W^{m,2}_{\weighttilde}}=1}\sqrt{\int_{\R_{+}} \sum_{j=0}^{m}\int_{\R_{+}}|D^{j}\psi(x)|^{2}\left(\frac{1}{1+x}\right)^{2\gamma}\weight_{j}(x)\dx}\\
&\leq C \|\eta\|_{W^{-m,2}_{\frac{1}{w}}}\frac{(1+t^{m})(1\vee T)^{\gamma}}{ t^{\gamma}}\sup_{\psi\colon \|\psi\|_{W^{m,2}_{\weighttilde}}=1}\|\psi\|_{W^{m,2}_{\weighttilde}},
\end{align*}
where $\left(\frac{1}{1+x}\right)^{2\gamma}\weight(x)\leq \weighttilde$.
\end{proof}
    We will require one additional Lemma related to time differences, which will be important in a later section.
    \begin{lemma}\label{lem:time_difference_semigroup_regularity_weighted_space}
    Let $\eta \in W^{-m,p}_{\weightinverse}$, $\gamma\in [0,1]$ and $\weighttilde$ such that $(1+x)^{2\gamma }\weight_{i}(x)\leq \weighttilde_{i}(x)$, for every $i \geq 0$. 
    If $|t-s|\leq 1$ then
     \begin{align*}
      \|S^{*}_{t-s}S^{*}_{s}\eta-S^{*}_{s}\eta\|_{W^{-m,2}_{\weighttildeinverse}}\leq C_{\gamma}\|\eta\|_{W^{-m,2}_{\weightinverse}}(1\vee s)^{m }(t-s)^{\gamma },
  \end{align*}
   as well as
      \begin{align*}
      \|S^{*}_{t-s}\eta-\eta\|_{W^{-m,2}_{\weighttildeinverse}}\leq C_{\gamma}\|\eta\|_{W^{-m,2}_{\weightinverse}}(t-s)^{\gamma}.
  \end{align*}
  If $|t-s|> 1$
      \begin{align*}
      \|S^{*}_{t-s}S^{*}_{s}\eta-S^{*}_{s}\eta\|_{W^{-m,2}_{\weighttildeinverse}}\leq C\|\eta\|_{W^{-m,2}_{\weightinverse}}(1\vee s)^{m}(t-s)^{m}.
  \end{align*}

\end{lemma}
\begin{proof}
As in the previous lemma,
  \begin{align*}
        &\|\varphi\|_{W^{m,2}_{\frac{1}{w}}}=\left(\sum_{i=0}^{m}\int_{\R_{+}} |\partial_{x}^{i}(e^{-sx}(e^{-(t-s)x}-1)\varphi)|^{p}\weight_{i}(x)\dx\right)^{\frac{1}{p}}\\
        &=\left(\sum_{i=0}^{m}\int_{\R_{+}} \left|\sum_{j_{1}+j_{2}+j_{3}=i}\begin{pmatrix}
            i\\
            j_{1},j_{2},j_{3}
        \end{pmatrix}\partial_{x}^{j_{1}}(e^{-sx})\partial_{x}^{j_{2}}(e^{-(t-s)x}-1) \partial_{x}^{j_{3}}\varphi\right|^{2}\weight_{i}(x)\dx\right)^{\frac{1}{2}}.
    \end{align*}
    For $j\in \N$,
    \begin{align*}
        &\partial_{x}^{j}(e^{-sx})=(-1)^{j}s^{j}e^{-sx}\\
        &\partial_{x}^{j}(e^{-(t-s)x}-1)=(-1)^{j}(t-s)^{j}e^{-(t-s)x}.
    \end{align*}
    Also,
    \begin{align*}
        |\partial_{x}^{j_{1}}(e^{-sx})\partial_{x}^{j_{2}}(e^{-(t-s)x}-1)|&=\left|(-1)^{j_{1}}s^{j_{1}}e^{-sx}(-1)^{j_{2}}(t-s)^{j_{2}}e^{-(t-s)x}\right|\\
        &\leq \left|s^{j_{1}}(t-s)^{j_{2}}\right|.
    \end{align*}
    If $j_{1}\in \N_{0}, j_{2}=0$,
        \begin{align*}
        |\partial_{x}^{j_{1}}(e^{-sx})(e^{-(t-s)x}-1)|&\leq C_{\gamma}\left|s^{j_{1}}(t-s)^{\gamma}(1+x)^{\gamma}\right|,
    \end{align*}
    for any $0\leq \gamma \leq 1$. Hence, if $|t-s|\leq 1$,
    \begin{align*}
        &\left(\sum_{i=0}^{m}\int_{\R_{+}} \left|\sum_{j_{1}+j_{2}+j_{3}=i}\begin{pmatrix}
            i\\
            j_{1},j_{2},j_{3}
        \end{pmatrix}\partial_{x}^{j_{1}}(e^{-sx})\partial_{x}^{j_{2}}(e^{-(t-s)x}-1) \partial_{x}^{j_{3}}\varphi\right|^{2}\weight_{i}(x)\dx\right)^{\frac{1}{2}}\\
        &\leq C_{\gamma}\left(\sum_{i=0}^{m}\int_{\R_{+}} (1\vee s)^{2i}(t-s)^{2\gamma }(1+x)^{2\gamma }\left|\partial_{x}^{i}\varphi\right|^{2}\weight_{i}(x)\dx\right)^{\frac{1}{2}}\\
        &\leq C_{\gamma}(1\vee s)^{m}(t-s)^{\gamma }\left(\sum_{i=0}^{m}\int_{\R_{+}} \left|\partial_{x}^{i}\varphi\right|^{2}(1+x)^{2\gamma }\weight_{i}(x)\dx\right)^{\frac{1}{2}}
    \end{align*}
To estimate $\|S^{*}_{t-s}S^{*}_{s}\eta-S^{*}_{s}\eta\|_{W^{-m,2}_{\weighttildeinverse}}$, we use that $\eta \in W^{-m,2}_{\weightinverse}$ and the previous estimates,
    \begin{align*}
        &\|S^{*}_{t-s}S^{*}_{s}\eta-S^{*}_{s}\eta\|_{W^{-m,2}_{\weighttildeinverse}}=\sup_{ \|\varphi\|_{W^{m,p}_{\weighttilde}}=1}\left|\left\langle \eta ,S_{s}(S_{t-s}-1)\varphi\right\rangle\right|\\
        &\leq \sup_{ \|\varphi\|_{W^{m,2}_{\weighttilde}}=1} \|\eta\|_{W^{-m,2}_{\weightinverse}}\|S_{s}(S_{t-s}-1)\varphi\|_{W^{m,2}_{\weight}}\\
        &\leq C_{\gamma}\sup_{ \|\varphi\|_{W^{m,2}_{\weighttilde}}=1} \|\eta\|_{W^{-m,2}_{\weightinverse}}(1\vee s)^{m }(t-s)^{\gamma }\left(\sum_{i=0}^{m}\int_{\R_{+}} \left|\partial_{x}^{i}\varphi\right|^{2}(1+x)^{2\gamma }\weight_{i}(x)\dx\right)^{\frac{1}{2}}.
    \end{align*}
  Hence, if $(1+x)^{2\gamma }\weight_{i}(x)\leq \weighttilde_{i}(x)$, for every $i \geq 0$, we obtain
  \begin{align*}
      \|S^{*}_{t-s}S^{*}_{s}\eta-S^{*}_{s}\eta\|_{W^{-m,2}_{\weighttildeinverse}}\leq C_{\gamma}\|\eta\|_{W^{-m,2}_{\weightinverse}}(1\vee s)^{m }(t-s)^{\gamma }.
  \end{align*}
  The second inequality follows similarly:
   \begin{align*}
        &\|S_{t-s}\varphi-\varphi\|_{W^{m,2}_{\frac{1}{w}}}=\left(\sum_{i=0}^{m}\int_{\R_{+}} |\partial_{x}^{i}((e^{-(t-s)x}-1)\varphi)|^{2}\weight_{i}(x)\dx\right)^{\frac{1}{2}}\\
        &=\left(\sum_{i=0}^{m}\int_{\R_{+}} \left|\sum_{j_{1}+j_{2}=i}\begin{pmatrix}
            i\\
            j_{1},j_{2}
        \end{pmatrix}\partial_{x}^{j_{1}}(e^{-(t-s)x}-1) \partial_{x}^{j_{2}}\varphi\right|^{2}\weight_{i}(x)\dx\right)^{\frac{1}{2}}.
    \end{align*}
    For $j\in \N$,
    \begin{align*}
        |\partial_{x}^{j}(e^{-(t-s)x}-1)|&=\left|(-1)^{j}(t-s)^{j}e^{-(t-s)x}\right|\leq \left|(t-s)^{j_{2}}\right|.
    \end{align*}
   Hence, if $|t-s|\leq 1$,
    \begin{align*}
        &\left(\sum_{i=0}^{m}\int_{\R_{+}} \left|\sum_{j_{1}+j_{2}=i}\begin{pmatrix}
            i\\
            j_{1},j_{2}
        \end{pmatrix}\partial_{x}^{j_{1}}(e^{-(t-s)x}-1) \partial_{x}^{j_{2}}\varphi\right|^{2}\weight_{i}(x)\dx\right)^{\frac{1}{2}}\\
        &\leq C_{\gamma}\left(\sum_{i=0}^{m}\int_{\R_{+}} (t-s)^{2\gamma }(1+x)^{2\gamma }\left|\partial_{x}^{i}\varphi\right|^{2}\weight_{i}(x)\dx\right)^{\frac{1}{2}}\\
        &\leq C_{\gamma}(t-s)^{\gamma }\left(\sum_{i=0}^{m}\int_{\R_{+}} \left|\partial_{x}^{i}\varphi\right|^{2}(1+x)^{2\gamma }\weight_{i}(x)\dx\right)^{\frac{1}{2}}.
    \end{align*}
     If $|t-s|>1$, we can perform the same steps, but with $|t-s|^{k}$.
\end{proof}
Considering the previous estimates and embeddings, the following lemma might seem redundant. 
 However, it ties in more directly with Assumption \ref{A:A_3:Narrower_assumption_nu} and gives more insight into how the weights in our spaces can be chosen, if we require $\|S^{*}_{t}\nu\|_{W^{-1,2}_{\frac{1}{w}}}\in L^{1}(0,T)$ or $\|S^{*}_{t}\nu\|_{W^{-1,2}_{\frac{1}{w}}}\in L^{2}(0,T)$.
    \begin{lemma}\label{lem:nu_semigroup_properties}
        Let $\nu$ be a non-negative Radon measure on $\R_{+}$, such that there exists a $0\leq \theta_{\nu}$ for which $ \left|\int_{\R_{+}}\frac{1}{(1+x)^{\theta_{\nu}}}\nu(\dx)\right|<\infty$. Let $0<t\leq T<\infty$ and consider the semigroup $S^{*}_{t}$, which was introduced in Lemma \ref{lem:semigroup_and_adjoint_semigroup}. 
Let $\alpha \in \R$ and $0\leq \gamma<1$ satisfy $\theta_{\nu}\leq \gamma-\alpha$. We set $\weight_{i}=(1+x)^{-2\alpha-1+i2}$, then
        \begin{align*}
        \|S^{*}_{t}\nu\|_{W^{-1,2}_{\frac{1}{w}}}\leq C\frac{1\vee T^{\gamma}}{t^{\gamma}}.
        \end{align*}
    \end{lemma}
    \begin{remark}
          We reiterated the crucial connection between the singularity of the kernel we lift (given by $1-\theta_{\nu}$), the decay of the corresponding measure (given by $\theta_{\nu}$), the weight we choose (given by $\weight_{i}=(1+x)^{-2\alpha-1+i2}$ with $\alpha = \gamma-\theta_{\nu}$, the boundary case of the Lemma's hypothesis $\theta_\nu \le \gamma-\alpha$) and the time integrability of $\|S^{*}_{t}\nu\|_{W^{-1,2}_{\frac{1}{w}}}$ (indicated by $\gamma$).
    \end{remark}
  
    \begin{proof}
    Using the embedding from Lemma \ref{lem:nu_is_in_dual},
        \begin{align*}
            \|S^{*}_{t}\nu\|_{W^{-1,2}_{\frac{1}{w}}}&=\sup_{\varphi\colon \|\varphi\|_{W^{1,2}_{\weight}}=1}|\langle \varphi,S^{*}_{ t }\nu\rangle|\\
            &=\sup_{\varphi\colon \|\varphi\|_{W^{1,2}_{\weight}}=1}\left|\int_{\R_{+}}e^{-x t }\varphi(x)\nu(\dx)\right|\\
            &\leq \sup_{\varphi\colon \|\varphi\|_{W^{1,2}_{\weight}}=1} \|(1+x)^{-\alpha}\varphi\|_{\infty}\left|\int_{\R_{+}}e^{-x t }(1+x)^{\alpha}\nu(\dx)\right|\\
            &\leq \sup_{\varphi\colon \|\varphi\|_{W^{1,2}_{\weight}}=1} \|(1+x)^{-\alpha}\varphi\|_{\infty}\left|\int_{\R_{+}}e^{-x t (1-\gamma)}\frac{1}{ t ^{\gamma}}\left(\frac{1}{\frac{1}{ t }+x}\right)^{\gamma}(1+x)^{\alpha}\nu(\dx)\right|\\
            &\leq (1\vee t)^{\gamma}\frac{1}{t^{\gamma}}\sup_{\varphi\colon \|\varphi\|_{W^{1,2}_{\weight}}=1} \|\varphi\|_{W^{1,2}_{\weighttilde}}\left|\int_{\R_{+}}e^{-x t (1-\gamma)}\left(\frac{1}{1+x}\right)^{\gamma}(1+x)^{\alpha}\nu(\dx)\right|,
        \end{align*}
        with $\weight_{i}=(1+x)^{-2\alpha-1+i2}$. 
    \end{proof}

\section{Well-Posedness of the Lifted Equation}\label{sec:well_posedness}
\subsection{Existence}\label{sec:existence_for_the_SEE}

Throughout this section the lift $\mu$ takes values in the product space
$\Wplusdualn=(\Wplusdual)^{n_{\textnormal{dim}}}$ (and, in the intermediate estimates, in
$\Wmiddualn$, $\Wmidtwodualn$, $\Wminustwodualn$), carrying the product norm
$\|\mu\|_{\Wplusdualn}^{2}=\sum_{i=1}^{n_{\textnormal{dim}}}\|\mu^{(i)}\|_{\Wplusdual}^{2}$; the
driving noise $W$ is $m_{W}$-dimensional and the coefficient matrices
$\nudrift,\nudiffusion$ act as bounded operators $\R^{n_{\textnormal{dim}}}\to\Wminusdualn$
(Section~\ref{sec:operator_and_semigroup}), so that $S^{*}_{t}\nudrift,S^{*}_{t}\nudiffusion\in\opn$
and $\|S^{*}_{t}\nu^{i}\|_{\opn}$ denotes the corresponding operator (equivalently
finite-dimensional Hilbert--Schmidt) norm. Every estimate below is stated in this product
setting; since $S^{*}_{t}$ acts diagonally and all bounds use only norms and standard
(dimension-independent) inequalities, the scalar case $n_{\textnormal{dim}}=1$ is recovered
verbatim, and the test function $\psi\in\Wplus$ remains scalar so that $\langle\mu,\psi\rangle\in\R^{n_{\textnormal{dim}}}$.

The goal of this section is to prove an existence result for equations of the type \eqref{eqn:SEE_eqn_strong_form}
with $\psi\in W^{1,2}_{\weight}$, $\drift, \diffusion$ being continuous functions in the $t$ and $x$-variable (uniformly in $x$ and $t$ respectively) and $\nu^{i}=\nu^{i}(\dx)$ ($i=\drift, \diffusion$) being a non-negative, (tempered) measures on $\R_{+}$, whose Laplace transforms satisfy a certain relation similar to \eqref{eqn:kernel_representation_measure}.
We first consider $\drift, \diffusion$ to be Lipschitz continuous in $x$ and obtain the existence and uniqueness of solutions in a standard way. In the next step, we consider more general coefficients, which we approximate by Lipschitz functions, and use a tightness result to obtain existence in this general setting.
   \begin{definition}\label{def:mild_solution_SEE}
        Given $\psi\in W^{m,2}_{\weight}$ and $\mu_{0}\in W^{-m,2}_{\dualweight}$. A $W^{-m,2}_{\dualweight}$-valued predictable process $\mu_{t}$, $t\in [0,T]$ is called a mild solution of equation 
      \begin{align}\label{eqn:SEE_eqn_strong_form}
        \mu_{t} = \mu_{0}-\int_{0}^{t}x\mu_{s}\ds+\int_{0}^{t}\nudrift \drift(s,\langle \mu_{s},\psi\rangle)\ds+\int_{0}^{t}\nudiffusion \diffusion(s,\langle \mu_{s},\psi\rangle)\dW_{s},
    \end{align}
 if
        \begin{align}\label{eqn:SEE_eqn_mild}
            \mu_{t} = e^{-xt}\mu_{0} +\int_{0}^{t} e^{-x(t-s)}\nudrift(x)\drift(s,\langle \mu_{s},\psi\rangle)\ds +\int_{0}^{t}e^{-x(t-s)}\nudiffusion(x)\diffusion(s,\langle \mu_{s},\psi\rangle)\dW_{s},
        \end{align}
        $\Prob$-a.s. for each $t\in [0,T]$. If $\mu$ has continuous sample paths, we will call it a continuous, mild solution.
    \end{definition}
\begin{assumption}\label{A:assumptions_general}
Let $L_{T,\drift}, L_{T,\diffusion}, C_{T,\drift}, C_{T,\diffusion}\geq 0$. The functions $\drift\colon \R_{+}\times \R^{n_{\textnormal{dim}}}\rightarrow \R^{n_{\textnormal{dim}}}$ and $\diffusion\colon \R_{+}\times \R^{n_{\textnormal{dim}}}\rightarrow \R^{n_{\textnormal{dim}}\times m_{W}}$ are continuous in both arguments and satisfy the following bounds for any $t\in [0,T]$ (here $|\cdot|$ denotes the Euclidean norm on $\R^{n_{\textnormal{dim}}}$, resp.\ the corresponding operator/Frobenius norm on $\R^{n_{\textnormal{dim}}\times m_{W}}$),
    \begin{enumerate}[label={(Coeff \arabic*)}]
        \item\label{A:assumption_general_linear_growth} Linear growth:
        \begin{align*}
            &|\drift(t,x)|\leq C_{T,\drift}(1+|x|),\qquad            &|\diffusion(t,x)|\leq C_{T\diffusion}(1+|x|).
        \end{align*}
        \item\label{A:assumption_general_lipschitz} Lipschitz continuity: For every $t\in \R_{+}$, $x,y\in \R^{n_{\textnormal{dim}}}$,
        \begin{align*}
            &|\drift(t,x)-\drift(t,y)|\leq L_{T,\drift}|x-y|,\qquad            &|\diffusion(t,x)-\diffusion(t,y)|\leq L_{T,\diffusion}|x-y|.
        \end{align*}
                 \item\label{A:assumption_general_uniformly_continuous} For every $t\in \R_{+}$, $x,y\in \R^{n_{\textnormal{dim}}}$, the maps $x\mapsto \drift(t,x)$ and $x\mapsto \diffusion(t,x)$ are continuous, uniformly in $t$.
    \end{enumerate}
\end{assumption}
\begin{remark}
    There is no issue in considering coefficients which, in addition to $x, t$, also depend on $\omega\in \Omega$, as long as the constants in the assumption are uniform in $\omega$. One can also allow constants $C_{t}$ that exhibit an $L^2([0,T])$-singularity at $0$. This case is a straightforward adaptation of the statements in this paper, however, we refrain from including this case in our arguments to not overload the proofs with additional parameters.
\end{remark}
\begin{notation}
    When we write $C_{\operatorname{Lip}_{\drift}},C_{\operatorname{Lip}_{\diffusion}},C_{\operatorname{LG}_{\drift}},C_{\operatorname{LG}_{\diffusion}}$, we implicitly refer to the constants appearing in the previous assumption, without specifying the $T$ dependence explicitly. If a constant without a second subscript appears, like $C_{\operatorname{Lip}},C_{\operatorname{LG}}$, it is implied that $\max\{C_{\operatorname{Lip}_{\drift}},C_{\operatorname{Lip}_{\diffusion}}\}\leq C_{\operatorname{Lip}}$ and $\max\{C_{\operatorname{LG}_{\drift}},C_{\operatorname{LG}_{\diffusion}}\}\leq C_{\operatorname{LG}}$.
\end{notation}
 We impose the following assumptions on $\nudrift$ and $\nudiffusion$. We will discuss their interpretation and alternative formulations afterward.
\begin{assumption}\label{A:assumption_kernel_measure}
    \begin{enumerate}[label={(M \arabic*)}]
       \item \label{A:A_3:Narrower_assumption_nu} $\nudrift$ and $\nudiffusion$ are ($n_{\textnormal{dim}}\times n_{\textnormal{dim}}$-matrices of) non-negative  measures on $\R_{+}$ and for every $1\leq i,j\leq n_{\textnormal{dim}}$, there exist $0\leq \theta_{\nudrift_{ij}}, \theta_{\nudiffusion_{ij}} < 1$, such that
           \begin{align}\label{eqn:nu_tempered_measure}
              \int_{\R_{+}} \frac{1}{(1+|x|)^{\theta_{\nu^{\drift\backslash \diffusion}_{ij}}}}\nu^{\drift\backslash \diffusion}_{ij}(\dx)<\infty.
           \end{align} 
       \item \label{A:A_2:assumption_1_test_function} 
        We assume that this $\weight$ satisfies Assumption \ref{A:assumption_weight_function}, and the constant $1$ function (the constant vector to be precise) has finite $ W^{1,2}_{\weight}$ norm, i.e. $\int_{\R_{+}}\weight_{0}(x)\dx<\infty$.
    \end{enumerate}
    For the measures from Assumption \ref{A:A_3:Narrower_assumption_nu}, Lemma \ref{lem:nu_semigroup_properties} implies that there exists a weight $w_{ij}$ and constants $0< a^{\drift}_{ij}, a^{\diffusion}_{ij} \leq 1$, such that
           \begin{align}\label{eqn:kernel_estimate_semigroup}
            \|S^{*}_{t-s}\nu^{\drift\backslash \diffusion}_{ij}\|_{W^{-1,2}_{\weightinverse}}\leq \frac{C(T)}{(t-s)^{1-a^{\drift\backslash \diffusion}_{ij}}},
        \end{align}
        where the value of $a^{\drift},a^{ \diffusion}$ are connected to the choice of weight $w$. Hence, Lemma \ref{lem:nu_semigroup_properties} allows us to identify a weight, for which 
        \begin{align}\label{eqn:drift_diffusion_kernel_estimate_semigroup}
             \|S^{*}_{t-s}\nudrift\|_{W^{-1,2}_{\weightinverse}}\in L^{1}(0,T),\qquad \|S^{*}_{t-s}\nudiffusion\|_{W^{-1,2}_{\weightinverse}}\in L^{2}(0,T).
        \end{align}
\end{assumption}
\begin{remark}\label{rem:discussion_of_the_general_assumptions}
    Discussion of the assumptions. 
    \begin{itemize}
           \item          For the most part, it would suffice to impose \eqref{eqn:drift_diffusion_kernel_estimate_semigroup} as a general Assumption and we will use the previous implication extensively.
         \item Assumption \ref{A:A_2:assumption_1_test_function} is not strictly necessary for our solution theory. It only comes into play, when we want to relate an explicit SVE to the solution of the SEE. To ``reconstruct'' the correct kernel, we will want to test the mild solution of \eqref{eqn:SEE_strong_formulation} with the constant $1$ function. To obtain the SVE, the function in the non-local terms needs to coincide with the functions we are testing with, i.e. $\psi=1$. This also reflects that an $L^{2}$-kernel is generally required to make sense of the stochastic integral. Intuitively, being able to set $\psi=1$ correlates with the corresponding SVE being well-posed in the sense of functions.  \eqref{eqn:nu_tempered_measure} with $\frac{1}{2}\leq \theta_{\nudiffusion_{ij}} < 1$ would correspond to a kernel with an $L^{1}$ singularity inside the stochastic integral which could, for $\diffusion=1$, be interpreted as a generalized fractional Gaussian field with Hurst index $<0$ (see \cite{lodhia2016fractional}).
        \item It might seem reasonable to alter \ref{A:A_3:Narrower_assumption_nu} by: There exist $0\leq \theta^{\nu}_{1}, \theta^{\nu}_{2} < 1$, such that $\frac{1}{x^{\theta^{\nu}_{1}}(1+|x|)^{\theta^{\nu}_{2}}}\nu(\dx)$ is a finite measure on $\R_{+}$.
        If we wanted to test \eqref{eqn:SEE_eqn_mild} with functions, that vanish sufficiently fast at $0$, we could consider weights, which are not $L^{1}_{loc}$ around $0$, which would indeed provide better estimates. Since, later on, we want Assumption \ref{A:A_2:assumption_1_test_function} to hold, we will always require a contribution of $\theta^{\nu}_{2}>0$. Such a contribution is obtained via the semigroup and always results in a time dependence of the coefficient appearing in \eqref{eqn:kernel_estimate_semigroup}.
 
    \end{itemize}

\end{remark}

\begin{example}\label{example:kernels}
We illustrate potential choices of weights for common examples of completely monotone kernels. In these examples, we verify the bound \eqref{eqn:kernel_estimate_semigroup} via \eqref{eqn:nu_tempered_measure}. Taking Assumption \ref{A:A_2:assumption_1_test_function} into account, we notice that the higher order weights $w_{j}$ for $j\geq 1$ play no particular role. Hence these can be chosen freely to satisfy certain conditions related to the associated Sobolev spaces.

Let $t<\infty$.
\begin{enumerate}
        \item Let $K_{\exp }$ be a finite combination of exponential functions:
$$
K_{\exp }(t):=\sum_{i=1}^k c_i e^{-y_i t}, t>0,
$$
with $k \in \mathbb{N}, c_i>0$ and $y_i \in[0, \infty)$ for $i=1, \ldots, k$. The corresponding Radon measure is
$$
\nu_{\exp }(\dx)=\sum_{i=1}^k c_i \delta_{y_i}(\dx),
$$
where $\delta_{y_i}$ denotes the Dirac measure at point $y_i$. We set $\theta_{\nu}=0$ corresponding to the unweighted case in this example. Hence, we will allow ourselves to perform the estimates in fractional Sobolev spaces. For estimate \eqref{eqn:kernel_estimate_semigroup}, we obtain
\begin{align*}
\sup_{\|\varphi\|_{W^{1/2+\eps,2}}=1}\left|\int_{0}^{\infty}e^{-(t-s)x}\varphi(x)\nu(\dx)\right|&\leq \sup_{\|\varphi\|_{W^{1/2+\eps,2}}=1}\sum_{i=1}^k c_i e^{-(t-s)y_{i}}\varphi(y_{i})\\
&\leq C\sup_{\|\varphi\|_{W^{1/2+\eps,2}}=1}\|\varphi\|_{L^{\infty}}\left(1+\frac{1}{(1+\min_{i, y_{i}\neq 0}y_{i}(t-s))^{\theta}}\right)\\
&\leq C,
\end{align*}
for any $\theta \in (0,1)$ and $\eps>0$. In this case, \eqref{eqn:kernel_estimate_semigroup} is satisfied with $a=0$.

Note that the additive constant inside the brackets only appears if there is an $y_{i}$, such that $y_{i}=0$. 

In the unweighted case, Assumption \ref{A:A_2:assumption_1_test_function} might become an issue, but since the support of $\nu_{\exp}$ is contained in a ball of radius $R=2\max\{y_{1},\dots,y_{k}\}$, it suffices to study the whole problem on the spaces $W^{\alpha,p}(B_{R}(0))$ ($0\leq \alpha, 1\leq p$).
\item  Let $K_{\text {frac }}$ be the fractional kernel of order $\alpha \in\left(0, 1\right)$ :
$$
K_{\text {frac }}(t):=\frac{1}{\Gamma(\alpha)} t^{\alpha-1}, t>0 .
$$

The corresponding Radon measure is
$$
\nu_{\operatorname{frac}}(x)=\frac{1}{\Gamma(\alpha) \Gamma(1-\alpha)} x^{-\alpha} \dx.
$$
we will ignore the pre-factor $\frac{1}{\Gamma(\alpha) \Gamma(1-\alpha)}$. 
Let $\gamma>0$.
\begin{align*}
\sup_{\|\varphi\|_{W^{1,2}_{\weight}}=1}&\left|\int_{0}^{\infty}e^{-(t-s)x}\varphi(x)\nu(\dx)\right|\leq  \sup_{\|\varphi\|_{W^{1,2}_{\weight}}=1} \left\|(1+x)^{\gamma}\varphi(x)\right\|_{\infty}\left|\int_{0}^{\infty}e^{-(t-s)x}\frac{1}{(1+x)^{\gamma}x^{\alpha}}\dx\right|\\
&\leq  \sup_{\|\varphi\|_{W^{1,2}_{\weight}}=1}\left\|(1+x)^{\gamma}\varphi(x)\right\|_{\infty}\left|\int_{0}^{\infty}e^{-(t-s)x(1-\beta)} \frac{1}{(t-s)^{\beta}}\frac{1}{\left(\frac{1}{(t-s)}+x\right)^{\beta}(1+x)^{\gamma}x^{\alpha}}\dx\right|\\
&\leq  \frac{1}{(t-s)^{\beta}} \sup_{\|\varphi\|_{W^{1,2}_{\weight}}=1}\left\|\varphi(x)\right\|_{W^{1,2}_{\weight}}C(T).
\end{align*}
\eqref{eqn:kernel_estimate_semigroup} is satisfied with $a=\beta$. This illustrates that we require $\alpha<1$, $\beta+\gamma+\alpha>0$. Assume, we wanted to square integrate the result, so $\beta<\frac{1}{2}$.
\begin{itemize}
    \item First we consider the case $\alpha\in\left(\frac{1}{2},1\right)$: Let $\widetilde{\eps}>\eps> 0$ be such that $1-\alpha+\widetilde{\eps}<1$. We set $\gamma=-\eps$,  $\beta=1-\alpha+\widetilde{\eps}-\eps$.
Let $w=(\weight_{0},\weight^{1})$ with $\weight_{i}=(1+x)^{i2-1-\eps}$, then also Assumption \ref{A:A_2:assumption_1_test_function} is satisfied.
    \item For $\alpha\in\left(0,\frac{1}{2}\right)$, we notice that we require additional decay, so we set $\gamma=1+\eps-\beta-\alpha$. This results in a weight $w=(\weight_{0},\weight^{1})$ with $\weight_{i}=(1+x)^{i2-1+\gamma}$. Hence we can not guarantee that Assumption \ref{A:A_2:assumption_1_test_function} is satisfied.
\end{itemize}

\item A similar argument works for the kernel
\begin{align*}
    k_{MVN}(t-s)=\frac{1}{\Gamma(\alpha)}\left((t-s)^{\alpha-1}-(-s)_{+}^{\alpha-1})\right).
\end{align*}
\item Let $K_{\text {gamma }}$ be the Gamma kernel of the form
$$
K_{\text {gamma }}(t):=\frac{1}{\Gamma(\alpha)} e^{-\beta t} t^{\alpha-1}, t>0,
$$
for some $\beta>0$ and $\alpha \in\left(\frac{1}{2}, 1\right)$. The corresponding Radon measure is given by
$$
\mu_{\operatorname{gamma }}(\dx)=\frac{1}{\Gamma(\alpha) \Gamma(1-\alpha)}(x-\beta)^{-\alpha} \mathbf{1}_{(\beta, \infty)}\dx.
$$
The measure $\mu_{\operatorname{gamma }}$ will be supported on   $[\beta, \infty) $. Let $\eps>0$ be such that $1-a+\eps<1$.
The only difference to the fractional kernel is that the weights will be shifted by $\beta$.
\begin{align*}
\left|\int_{0}^{\infty}e^{-(t-s)x}\varphi(x)\nu(\dx)\right|&\leq\left\|(1+(x-\beta))^{\gamma}\varphi(x)\right\|_{\infty}\left|\int_{\beta}^{\infty}e^{-(t-s)x}\frac{1}{(1+(x-\beta))^{\gamma}(x-\beta)^{\alpha}}\dx\right|\\
&\leq\left\|\frac{1}{(1+(x-\beta))^{\eps}}\varphi(x)\right\|_{\infty}e^{-\beta(t-s)}\\
&\phantom{xxxx}\times\left|\int_{0}^{\infty}e^{-(t-s)x(1-\eta)}\frac{1}{\left(\frac{1}{(t-s)}+x\right)^{\eta}(1+x)^{\gamma}x^{\alpha}}\dx\right|\\
&\leq \left\|\varphi(x)\right\|_{W^{1,2}_{\weight}}C(T)\frac{e^{-\beta(t-s)}}{(t-s)^{\eta}}\left|\int_{0}^{1} \frac{1}{x^{\alpha}}\dx+\int_{1}^{\infty}\frac{1}{\left(1+x\right)^{\eta+\gamma+\alpha}}\dx\right|.
\end{align*}
The choice of $\alpha,\gamma\eta$ are identical to $K_{\text {frac }}$. \eqref{eqn:kernel_estimate_semigroup} is again satisfied with $a=\eta$.
\item Given any $\beta>0$ and let the kernel $K$ satisfying Assumption \ref{A:A_3:Narrower_assumption_nu}, the exponentially damped kernel
$$
K_{\text {damp }}(t):=e^{-\beta t} K(t), t>0,
$$
has the corresponding Radon measure
\begin{align*}
    \nu_{\operatorname{damp}}(\dx)=\mathbf{1}_{[\beta, \infty)}(x) \nu_{K}(\dx-\beta),
\end{align*}
with support $\operatorname{supp} \nu_{\operatorname{damp}}=\{x \in[\beta, \infty) \mid x-\beta \in \operatorname{supp} \nu_{K}\} $.
This case, roughly, can be handled with the same specifications as for $\nu_{K}$, since
\begin{align*}
    e^{-\beta t} K(t)=\int_{0}^{\infty}\mathbf{1}_{\operatorname{supp}(\nu_{K})}e^{-\beta t}e^{-tx}\nu_{K}(\dx).
\end{align*}
\item Let $\delta>0$ and consider any completely monotone kernel $K$, satisfying Assumption \ref{A:A_3:Narrower_assumption_nu}, with the corresponding Radon measure $\nu_{K}$. The shifted kernel
$$
K_{\text {shift }}(t):=K(t+\delta), t>0,
$$
possesses the corresponding Radon measure
$$
\nu_{\operatorname{shift }}(\dx)=e^{-\delta x} \nu_{K}(\dx), 
$$
with support $\operatorname{supp}(\nu_{\operatorname{shift }})=\operatorname{supp}(\nu_{K})$, since $e^{-\delta x}>0$ everywhere.

A direct calculation shows that the assumptions are satisfied with a bound of the form $C(T)\frac{1}{(t-s+\delta)^{1-\alpha+\widetilde{\eps}}}$.
    \end{enumerate}
\end{example}
\begin{remark}
     On a purely formal level, measures $\nu_{\operatorname{frac}}$ with $\alpha\in\left(0,\frac{1}{2}\right)$ correspond to singular kernels which are only $L^{1}$ integrable in time. Integrated versions of SVEs with $L^{1}$ kernels have been considered in  \cite{abi2021weak_L1}. The non-integrated SVE is not well posed in a strong$\backslash$ pointwise sense, which is also reflected in the decay requirement of the test function in the lift. However, the formulation of the $X$ process in \cite{abi2021weak_L1} is considerably ``better-behaved''. 
 A more in-depth study of this case is currently a work in progress.
\end{remark}
We recall for $\theta_{\nudrift},\theta_{\nudiffusion}$ from Assumption \ref{A:A_3:Narrower_assumption_nu} the weights from Definition \ref{def:weights_triple_for_analysis}

    \begin{align*}
        &(\weightminus)_{i}(x)\coloneqq (1+x)^{2\etaminus-1+2i}\\
        &>(\weightsim)_{i}(x)\coloneqq (1+x)^{2\etamid-1+2i}\\
        &>(\weightplus)_{i}(x)\coloneqq (1+x)^{2\etaplus-1+2i},
    \end{align*}
    with
    \begin{align*}
            &\etaplus=\begin{cases}
                -\eps & \text{ if } \theta_{\nudiffusion}<\frac{1}{2}, \text{ where } 0<\eps<\frac{1}{2}-\theta_{\nudiffusion},\\
                \theta_{\nudiffusion}-\frac{1}{2}+\delta & \text{ if } \theta_{\nudiffusion}>\frac{1}{2}, \text{ where } 0<\delta<\frac{1}{2},
            \end{cases}\\
            &\etaminus > \max\{\theta_{\nudrift},\theta_{\nudiffusion}\},\\
            & \etaplus < \etamid <\etaminus. 
    \end{align*}

\begin{remark}
Let $\mu\in (\Wplusdual)^{n_{\textnormal{dim}}}$, $\psi\in \Wplus$ and $\nu\in \Wminusdual$.
    There is no ambiguity in the interpretation of the stochastic integral: writing $W$ for the driving $m_{W}$-dimensional Brownian motion, the map $\mu\mapsto \nudiffusion\diffusion(\langle \mu,\psi\rangle)$ is understood as a mapping from $(\Wplusdual)^{n_{\textnormal{dim}}}$ into the space of Hilbert-Schmidt operators from $\R^{m_{W}}$ to $(\Wminusdual)^{n_{\textnormal{dim}}}$, denoted $\mathcal{L}_{2}(\R^{m_{W}},(\Wminusdual)^{n_{\textnormal{dim}}})$: explicitly, for the standard basis $e_{1},\dots,e_{m_{W}}$ of $\R^{m_{W}}$, this operator sends $e_{k}$ to the $k$-th column of $\nudiffusion\diffusion(\langle \mu,\psi\rangle)\in(\Wminusdual)^{n_{\textnormal{dim}}}$, i.e.\ to $\big(\sum_{l=1}^{n_{\textnormal{dim}}}(\nudiffusion)_{il}\diffusion_{lk}(\langle\mu,\psi\rangle)\big)_{i=1}^{n_{\textnormal{dim}}}$. Since $\R^{m_{W}}$ is finite-dimensional, this operator is automatically Hilbert-Schmidt, with Hilbert-Schmidt norm equal to the Frobenius norm of the resulting $n_{\textnormal{dim}}\times m_{W}$ matrix of $\Wminusdual$-norms. When $n_{\textnormal{dim}}=m_{W}=1$ this reduces to the scalar case, where $\mu\mapsto\nudiffusion\diffusion(\langle\mu,\psi\rangle)$ maps into $\mathcal{L}_{2}(\R,\Wminusdual)$.
\end{remark}

\subsection{Lipschitz Case}

\begin{theorem}\label{thm:existence_uniqueness_lipschitz}
    Let Assumptions \ref{A:assumption_general_linear_growth},  \ref{A:assumption_general_lipschitz} and \ref{A:A_3:Narrower_assumption_nu} be satisfied and $T>0$ be fixed. Let $\Wplusdualn$ correspond to the choice of space for which \eqref{eqn:drift_diffusion_kernel_estimate_semigroup} holds, and let $\mu_{0}\in L^{2}(\Omega,\Wplusdualn)$ be given.
   Then equation \eqref{eqn:SEE_strong_formulation}, with initial condition $\mu_{0}$, has a unique mild $\Wplusdualn$ solution $\mu \in L^{2}(\Omega,C([0,T],\Wplusdualn))$.
\end{theorem}
Before we begin with the proof, we introduce the following useful Lemma.
\begin{lemma} Let $0\leq T <\infty$, $f,g \in L^{1}(0,T)$ and $g\geq 0$, then
    \begin{align*}
        \sup_{t\leq T}\int_{0}^{t}g(t-s)f(s)\ds\leq \int_{0}^{T}g(T-r)\sup_{0\leq u \leq r}|f(u)|\dr.
    \end{align*}
\end{lemma}
\begin{proof}
    \begin{align*}
    \sup_{t\leq T}\int_{0}^{t}g(t-s)f(s)\ds&=\sup_{t\leq T}\int_{0}^{t}g(T-(s-t+T))f(s)\ds\\
    &=\sup_{t\leq T}\int_{T-t}^{T}g(T-r)f(r+t-T)\dr\\
    &\leq \sup_{t\leq T}\int_{0}^{T}g(T-r) \mathbf{1}_{\{T-t\leq r \leq T\}}|f(r+t-T)|\dr.
\end{align*}
Since $0\leq r+t-T \leq r$,
\begin{align*}
   \sup_{t\leq T}\int_{0}^{T}g(T-r) \mathbf{1}_{T-t\leq r \leq T}|f(r+t-T)|\dr&\leq  \sup_{t\leq T}\int_{0}^{T}g(T-r)\sup_{0\leq u \leq r}|f(u)|\dr\\
   &\leq \int_{0}^{T}g(T-r)\sup_{0\leq u \leq r}|f(u)|\dr.
\end{align*}
\end{proof}

\begin{proof}[Proof of Theorem \ref{thm:existence_uniqueness_lipschitz}]
    Set $\mathcal{H}:=L^{2}(\Omega,C([0,T],\Wplusdualn))$ and define the operator $\Gamma\colon\mathcal{H}\to\mathcal{H}$ by
    \begin{align*}
        \Gamma(\mu,\mu_{0})\coloneqq S^{*}_{t}\mu_{0}+\int_{0}^{t}S^{*}_{t-s}\nudrift \drift(s,\langle \mu_{s},\psi\rangle) \ds+\int_{0}^{t}S^{*}_{t-s}\nudiffusion \diffusion(s,\langle \mu_{s},\psi\rangle) \dW_{s}.
    \end{align*}
    The proof is a rather standard application of the Banach fixed-point theorem, verifying the self-mapping and contraction properties of $\Gamma$. We only verify the contraction property of the solution $\Gamma$, since the self-mapping property is verified analogously, using the estimate from Lemma \ref{lem:improvement_semigroup_weighted_spaces} and a similar estimate will be shown in Lemma \ref{lem:a_priori_estimate_for_tightness}. In the following steps, we use that $|\langle \mu_{s},\psi\rangle|\leq \sup_{h\colon \|h\|_{\Wplus}=1}|\langle \mu_{s},h\rangle|=\|\mu_{s}\|_{\Wplusdualn}$. We first consider $\widetilde{T}$ to be fixed and ``small''.
        \begin{align*}
   \E \sup_{t\leq \widetilde{T}}\|&\Gamma(\mu^{1}_{t},\mu_{0})-\Gamma(\mu^{2}_{t},\mu_{0})\|_{\Wplusdualn}^{2} \\
       &\leq C_{\operatorname{Lip}_{\drift}}\E \sup_{t\leq \widetilde{T}}\left(\int_{0}^{t}\|S^{*}_{t-s}\nudrift\|_{\opn}\left\|\mu^{1}_{s}- \mu^{2}_{s}\right\|_{\Wplusdualn}\ds\right)^{2}\\
    &\phantom{xx}+C_{\operatorname{Lip}_{\diffusion}}\E\sup_{t\leq \widetilde{T}}\left(\int_{0}^{t}\left\| S^{*}_{t-s}\nudiffusion\right\|_{\opn}^{2}\left\| \mu^{1}_{s}- \mu^{2}_{s}\right\|_{\Wplusdualn}^{2}\ds \right)\\
    &\leq  C_{\operatorname{Lip}_{\drift}}\E \sup_{t\leq \widetilde{T}}\left(\int_{0}^{t}\|S^{*}_{t-s}\nudrift\|_{\opn}\left\|\mu^{1}_{s}- \mu^{2}_{s}\right\|_{\Wplusdualn}\ds\right)^{2}\\
    &\phantom{xx}+C_{\operatorname{Lip}_{\diffusion}}\E\left(\int_{0}^{\widetilde{T}}\left\| S^{*}_{\widetilde{T}-s}\nudiffusion\right\|_{\opn}^{2}\sup_{r\leq s}\left\| \mu^{1}_{r}- \mu^{2}_{r}\right\|_{\Wplusdualn}^{2}\ds \right)\\
    &\leq C_{\operatorname{Lip}_{\drift}}\sup_{t\leq \widetilde{T}}\left(\int_{0}^{t}\|S^{*}_{t-s}\nudrift\|_{\opn}\ds\right)\\
    &\phantom{xxxxxx}\times\E \int_{0}^{\widetilde{T}}\|S^{*}_{\widetilde{T}-s}\nudrift\|_{\opn}\sup_{r\leq s}\left\|\mu^{1}_{r}- \mu^{2}_{r}\right\|_{\Wplusdualn}^{2}\ds\\
    &\phantom{xx}+C_{\operatorname{Lip}_{\diffusion}}\E\left(\int_{0}^{\widetilde{T}}\left\| S^{*}_{\widetilde{T}-s}\nudiffusion\right\|_{\opn}^{2}\sup_{r\leq s}\left\| \mu^{1}_{r}- \mu^{2}_{r}\right\|_{\Wplusdualn}^{2}\ds \right)\\
    &\leq C_{\operatorname{Lip},\widetilde{T}}\E\sup_{s\leq \widetilde{T}}\left\|\mu^{1}_{s}- \mu^{2}_{s}\right\|_{\Wplusdualn}^{2},
\end{align*}
where $C_{\widetilde{T}}$ depends on $\widetilde{T}$ via the terms
\begin{align*}
    \int_{0}^{\widetilde{T}}\|S^{*}_{\widetilde{T}-s}\nudrift\|_{\opn}\ds
    \qquad\text{and}\qquad
    \int_{0}^{\widetilde{T}}\|S^{*}_{\widetilde{T}-s}\nudiffusion\|^{2}_{\opn}\ds.
\end{align*}
By Assumption \ref{A:A_3:Narrower_assumption_nu}, both of these integrals can be controlled by a constant times $\max\{\widetilde{T}^{2 a_{\drift}},\widetilde{T}^{2 a_{\diffusion}}\}$. Hence, we obtain a contraction for $\widetilde{T}$ small enough. Obtaining a solution on $[0,T]$ now follows from standard arguments.
\end{proof}

\begin{corollary}
    The solution map $\mu_{0}\mapsto \mu_{t}$ is continuous for every $0<t\leq T$.
\end{corollary}

\subsection{General Coefficients}\label{sec:existence_general_coefficients}

Our strategy to prove the existence of a solution, when $\drift, \diffusion$ are merely continuous \ref{A:assumption_general_uniformly_continuous} and satisfy \ref{A:assumption_general_linear_growth}, is to approximate $\drift, \diffusion$ with Lipschitz continuous functions and then pass to the limit in the approximation. For precisely this limiting procedure, we will derive a-priori estimates on the mild solution of equation \eqref{eqn:SEE_strong_formulation}, which will be used in combination with certain tightness arguments.

The next proposition is a central tool in our approach since it allows us to approximate $\drift, \diffusion$ by a sequence of Lipschitz continuous functions $\drift_{n}, \diffusion_{n}$ with a uniform linear growth bound. By Theorem \ref{thm:existence_uniqueness_lipschitz}, for each $n\in \N$, the equation
\begin{align}\label{eq:approximate_SEE_strong_form}
    \d \mu_{n}(t)=-x\mu_{n}(t)\dt+ \nudrift\drift_{n}(t,\langle \mu_{n}(t),\psi\rangle )\dt+\nudiffusion\diffusion_{n}(t,\langle \mu_{n}(t),\psi\rangle )\dW_{t}
\end{align}
has a solution, for each $\psi\in \Wplus$ in the sense of Definition \ref{def:mild_solution_SEE} with values in $L^{2}(\Omega,C([0,T],\Wplusdualn))$.
\begin{proposition}(\cite[Proposition 1.1]{hofmanova2012weak})
    Suppose $F \colon \R_{+}\times \R^{n_{\textnormal{dim}}}\rightarrow \R^{d_{0}}$ is a Borel function of (at most) linear growth, i.e.
    \begin{align*}
        \exists L<\infty,\,\forall t\geq 0,\,\forall x\in \R^{n_{\textnormal{dim}}},\,\|F(t,x)\|\leq L(1+\|x\|),
    \end{align*}
    such that $F(t, \cdot )\in C(\R^{n_{\textnormal{dim}}},\R^{d_{0}})$ for any $t\in\R_{+}$. Then there exists a sequence of Borel functions $F_{k} \colon \R_{+}\times \R^{n_{\textnormal{dim}}}\rightarrow \R^{d_{0}}$, $k\geq 1$, which have at most linear growth uniformly in $k$, namely
    \begin{align*}
        \forall k\in\N,\,\exists L<\infty,\,\forall t\geq 0,\,\forall x\in \R^{n_{\textnormal{dim}}},\,\|F_{k}(t,x)\|\leq L(1+\|x\|),
    \end{align*}
    which are Lipschitz continuous in the second variable uniformly in the first one,
        \begin{align*}
        \forall k\in\N,\,\exists L_{k}<\infty,\,\forall t\geq 0,\,\forall x,y\in \R^{n_{\textnormal{dim}}},\,\|F_{k}(t,x)-F_{k}(t,y)\|\leq L_{k}\|x-y\|,
    \end{align*}
    and which satisfy
    \begin{align*}
        \lim_{k\rightarrow \infty}F_{k}(t, \cdot )=F(t, \cdot )\quad\textnormal{locally uniformly on }\R^{n_{\textnormal{dim}}}
    \end{align*}
    for all $t\geq 0$.
\end{proposition}
\begin{remark}
    Unlike the Lipschitz constant, the approximations $F_{k}$ share the same modulus of continuity as $F$.
\end{remark}
\subsubsection{A-priori estimates}
We will derive a ``spatial'' and a ``temporal'' estimate to perform the limit $n\rightarrow \infty$ in \eqref{eq:approximate_SEE_strong_form}.
The ``spatial'' estimate will be separated into two Lemmata since we will reuse Lemma \ref{lem:a-priori_estimate_general} in Section \ref{section:Invariant_measure}.

\begin{lemma}\label{lem:a-priori_estimate_general}
 Let Assumptions \ref{A:assumption_general_linear_growth} be satisfied, $1<p$ and let $\mu$ be a mild solution of \eqref{eqn:SEE_strong_formulation}, with $\mu_{0}\in L^{p}(\Omega,\Wplusdualn)$.

 Then there exists a constant $C_{p,\operatorname{LG}}>0$, which only depends on $p$ and the linear growth condition \ref{A:assumption_general_linear_growth} and is independent of $T$, such  
 \begin{align}
   \E  \sup_{t\leq T}\|\mu_{t}\|_{\Wplusdualn}^{p}
     &\leq C_{p,\operatorname{LG}}\E \sup_{t\leq T}\left\|S^{*}_{t}\mu_{0}\right\|^{p}_{\Wplusdualn}\nonumber\\
     &\phantom{xx}+ C_{p,\operatorname{LG}} h_{\drift,j}(T)\int_{0}^{T}\|S^{*}_{T-s}\nudrift\|_{\opn}\E\sup_{r\leq s}\|\mu_{r}\|_{\Wplusdualn}^{p}\ds\nonumber\\
     &\phantom{xx}+C_{p,\operatorname{LG}} h_{\diffusion,j}(T)\int_{0}^{T}\|S^{*}_{T-s}\nudiffusion\|_{\opn}^{2}\nonumber\\
     &\phantom{xxxx}\times\left(\E\sup_{r\leq s}\|\mu_{r}\|_{\Wplusdualn}^{2}\right)^{\frac{p}{2}}\ds\nonumber\\
      &\phantom{xx}+C_{p,\operatorname{LG}}\left(\int_{0}^{T}\|S^{*}_{T-s}\nudrift\|_{\opn}\ds\right)^{p}\nonumber\\
      &\phantom{xx}+C_{p,\operatorname{LG}} \left(\int_{0}^{T}\|S^{*}_{T-s}\nudiffusion\|_{\opn}^{2}\ds\right)^{\frac{p}{2}}.\nonumber
\end{align}
where $j=1$ if $p\geq 2$ and $j=2$ if $1<p\leq 2$. Depending on the value of $p$, we set 
\begin{align*}
    &h_{\drift,1}(T)=h_{\drift,2}(T)= \left(\int_{0}^{T}\|S^{*}_{T-s}\nudrift\|_{\opn}\ds\right)^{p-1},\\
    &h_{\diffusion,1}(T)=\left(\int_{0}^{T}\|S^{*}_{T-s}\nudiffusion\|_{\opn}^{2}\ds\right)^{\frac{p}{2}-1},\\
    &\qquad h_{\diffusion,2}(T)=\left(\int_{0}^{T}\|S^{*}_{T-s}\nudiffusion\|_{\opn}^{2}\right)^{p-1}.
\end{align*}
\end{lemma}

\begin{proof}
Fix $T>0$. We use the maximal inequality (see \cite{hausenblas2001note}), and Jensen's inequality, so that for $p\geq 2$ we obtain
  \begin{align*}
   \E  &\sup_{t\in [0,T]}\|\mu_{t}\|_{\Wplusdualn}^{p} 
   \leq C_{p}\E \sup_{t\in [0,T]}\left\|S^{*}_{t}\mu_{0}\right\|^{p}_{\Wplusdualn}\\
   &\phantom{xx}+  C_{p,\operatorname{LG}}\left(\int_{0}^{T}\|S^{*}_{T-s}\nudrift\|_{\opn}\ds\right)^{p-1}\\
   &\phantom{xxxxxx}\times\int_{0}^{T}\|S^{*}_{T-s}\nudrift\|_{\opn}\E\sup_{r\leq s}\|\mu_{r}\|_{\Wplusdualn}^{p}\ds\\
    &\phantom{xx}+C_{p,\operatorname{LG}}\left(\int_{0}^{T}\|S^{*}_{T-s}\nudiffusion\|_{\opn}^{2}\ds\right)^{\frac{p}{2}-1}\\
    &\phantom{xxxxxx}\times\int_{0}^{T}\|S^{*}_{T-s}\nudiffusion\|_{\opn}^{2}\E\sup_{r\leq s}\|\mu_{r}\|_{\Wplusdualn}^{p}\ds \\
    &\phantom{xx}+C_{p,\operatorname{LG}}\left(\int_{0}^{T}\|S^{*}_{T-s}\nudrift\|_{\opn}\ds\right)^{p}\\
    &\phantom{xxxx}+C_{p,\operatorname{LG}} \left(\int_{0}^{T}\|S^{*}_{T-s}\nudiffusion\|_{\opn}^{2}\ds\right)^{\frac{p}{2}}.
\end{align*}
If $1\leq p <2$, then
\begin{align*}
        \E  &\sup_{t\in [0,T]}\|\mu_{t}\|_{\Wplusdualn}^{p} 
   \leq C_{p}\E \sup_{t\in [0,T]}\left\|S^{*}_{t}\mu_{0}\right\|^{p}_{\Wplusdualn}\\
   &\phantom{xx}+  C_{p,\operatorname{LG}}\left(\int_{0}^{T}\|S^{*}_{T-s}\nudrift\|_{\opn}\ds\right)^{p-1}\\
   &\phantom{xxxxxxxx}\times\int_{0}^{T}\|S^{*}_{T-s}\nudrift\|_{\opn}\E\sup_{r\leq s}\|\mu_{r}\|_{\Wplusdualn}^{p}\ds\\
   &\phantom{xx}+C_{p,\operatorname{LG}}\E\left(\left(\int_{0}^{T}\|S^{*}_{T-s}\nudiffusion\|_{\opn}^{2}\right)^{p-1}\right.\\
   &\phantom{xxxxxx}\left.\times\int_{0}^{T}\|S^{*}_{T-s}\nudiffusion\|_{\opn}^{2}\left(\sup_{r\leq s}\|\mu_{r}\|_{\Wplusdualn}^{p}\sup_{r\leq s}\|\mu_{r}\|_{\Wplusdualn}^{p}\right)\ds \right)^{1/2}\\
    &\phantom{xx}+C_{p,\operatorname{LG}}\left(\int_{0}^{T}\|S^{*}_{T-s}\nudrift\|_{\opn}\ds\right)^{p}\\
    &\phantom{xxxx}+C_{p,\operatorname{LG}} \left(\int_{0}^{T}\|S^{*}_{T-s}\nudiffusion\|_{\opn}^{2}\ds\right)^{\frac{p}{2}}\\
         &\leq C_{p}\E \sup_{t\in [0,T]}\left\|S^{*}_{t}\mu_{0}\right\|^{p}_{\Wplusdualn}\\
   &\phantom{xx}+  C_{p,\operatorname{LG}}\left(\int_{0}^{T}\|S^{*}_{T-s}\nudrift\|_{\opn}\ds\right)^{p-1}\\
   &\phantom{xxxxxxxx}\times\int_{0}^{T}\|S^{*}_{T-s}\nudrift\|_{\opn}\E\sup_{r\leq s}\|\mu_{r}\|_{\Wplusdualn}^{p}\ds\\
   &\phantom{xx}+C_{p,\operatorname{LG}}\E\left(\sup_{r\leq T}\|\mu_{r}\|_{\Wplusdualn}^{p}\left(\int_{0}^{T}\|S^{*}_{T-s}\nudiffusion\|_{\opn}^{2}\right)^{p-1}\right.\\
   &\phantom{xxxxxx}\left.\times\int_{0}^{t_{0}}\|S^{*}_{T-s}\nudiffusion\|_{\opn}^{2}\|\mu_{s}\|_{\Wplusdualn}^{p}\ds \right)^{1/2}\\
       &\phantom{xx}+C_{p,\operatorname{LG}}\left(\int_{0}^{T}\|S^{*}_{T-s}\nudrift\|_{\opn}\ds\right)^{p}\\
       &\phantom{xxxx}+C_{p,\operatorname{LG}} \left(\int_{0}^{T}\|S^{*}_{T-s}\nudiffusion\|_{\opn}^{2}\ds\right)^{\frac{p}{2}}\\
            &\leq C_{p}\E \sup_{t\in [0,T]}\left\|S^{*}_{t}\mu_{0}\right\|^{p}_{\Wplusdualn}\\
   &\phantom{xx}+ C_{p,\operatorname{LG}} \left(\int_{0}^{T}\|S^{*}_{T-s}\nudrift\|_{\opn}\ds\right)^{p-1}\\
   &\phantom{xxxxxxxx}\times\int_{0}^{T}\|S^{*}_{T-s}\nudrift\|_{\opn}\E\sup_{r\leq s}\|\mu_{r}\|_{\Wplusdualn}^{p}\ds\\
   &\phantom{xx}+C_{p,\operatorname{LG}}\eps\E\left(\sup_{r\leq T}\|\mu_{r}\|_{\Wplusdualn}^{p}\right)\\
&\phantom{xx}+C_{p}\frac{1}{\eps}\left(\int_{0}^{T}\|S^{*}_{T-s}\nudiffusion\|_{\opn}^{2}\right)^{p-1}\\
&\phantom{xxxxxxxx}\times\E\int_{0}^{T}\|S^{*}_{T-s}\nudiffusion\|_{\opn}^{2} \sup_{r\leq s}\|\mu_{r}\|_{\Wplusdualn}^{p}\ds \\
    &\phantom{xx}+C_{p,\operatorname{LG}}\left(\int_{0}^{T}\|S^{*}_{T-s}\nudrift\|_{\opn}\ds\right)^{p}\\
    &\phantom{xxxx}+C_{p,\operatorname{LG}} \left(\int_{0}^{T}\|S^{*}_{T-s}\nudiffusion\|_{\opn}^{2}\ds\right)^{\frac{p}{2}}.
\end{align*}
Choosing $\eps$ such that $C_{p,\operatorname{LG}}\eps <1$, bringing $C_{p,\operatorname{LG}}\eps\E\left(\sup_{r\leq T}\|\mu_{r}\|_{\Wplusdualn}^{p}\right)$ to the left hand side and dividing by $(1-C_{p,\operatorname{LG}}\eps)$ yields
\begin{align*}
        \E  &\sup_{t\in [0,T]}\|\mu_{t}\|_{\Wplusdualn}^{p} \leq C_{p,\operatorname{LG},\eps}\E \sup_{t\in [0,T]}\left\|S^{*}_{t}\mu_{0}\right\|^{p}_{\Wplusdualn}\\
   &\phantom{xx}+ C_{p,\operatorname{LG},\eps} \left(\int_{0}^{T}\|S^{*}_{T-s}\nudrift\|_{\opn}\ds\right)^{p-1}\\
   &\phantom{xxxxxxxx}\times\int_{0}^{T}\|S^{*}_{T-s}\nudrift\|_{\opn}\E\sup_{r\leq s}\|\mu_{r}\|_{\Wplusdualn}^{p}\ds\\ 
   &\phantom{xx}+C_{p,\operatorname{LG},\eps}\left(\int_{0}^{T}\|S^{*}_{T-s}\nudiffusion\|_{\opn}^{2}\right)^{p-1}\\
   &\phantom{xxxxxxxx}\times\int_{0}^{T}\|S^{*}_{T-s}\nudiffusion\|_{\opn}^{2}\E\left(\sup_{r\leq s}\|\mu_{r}\|_{\Wplusdualn}^{p}\ds \right)\\
      &\phantom{xx}+C_{p,\operatorname{LG},\eps}\left(\int_{0}^{T}\|S^{*}_{T-s}\nudrift\|_{\opn}\ds\right)^{p}\\
      &\phantom{xxxx}+C_{p,\operatorname{LG},\eps} \left(\int_{0}^{T}\|S^{*}_{T-s}\nudiffusion\|_{\opn}^{2}\ds\right)^{\frac{p}{2}}.
\end{align*}

\end{proof}
\begin{lemma}\label{lem:a_priori_estimate_for_tightness}
    Let the Assumptions \ref{A:A_3:Narrower_assumption_nu}, \ref{A:assumption_general_linear_growth} be satisfied, $1<p$ and let $\mu$ be a mild solution of \eqref{eqn:SEE_strong_formulation}, with $\mu_{0}\in L^{p}(\Omega,\Wplusdualn)$, where $\Wplusdualn$ corresponds to the choice of space for which \eqref{eqn:drift_diffusion_kernel_estimate_semigroup} holds. Then
    \begin{align}
      \E  \sup_{t\leq T}\|\mu_{t}\|_{\Wplusdualn}^{p} 
     &\leq C_{p,\drift,\diffusion,T,\mu_{0}}.
\end{align}
\end{lemma}
\begin{proof}
    By Lemma \ref{lem:a-priori_estimate_general} and Assumption \ref{A:assumption_general_linear_growth},
     \begin{align*}
   \E  \sup_{t\leq T}\|\mu_{t}\|_{\Wplusdualn}^{p} 
     &\leq C_{p,\operatorname{LG}}\E \sup_{t\leq T}\left\|S^{*}_{t}\mu_{0}\right\|^{p}_{\Wplusdualn}\nonumber\\
     &\phantom{xx}+ C_{p,\operatorname{LG}} h_{\drift,j}(T)\int_{0}^{T}\frac{C(T)}{(T-s)^{1-a_{1}}}\E\sup_{r\leq s}\|\mu_{r}\|_{\Wplusdualn}^{2}\ds\nonumber\\
     &\phantom{xx}+C_{p,\operatorname{LG}} h_{\diffusion,j}(T)\int_{0}^{T}\frac{C(T)}{(T-s)^{2-2a_{2}}}\E\sup_{r\leq s}\|\mu_{r}\|_{\Wplusdualn}^{2}\ds\\
      &\phantom{xx}+C_{p,\operatorname{LG}}T^{(p-1)a_{1}}+C_{p,\operatorname{LG}} T^{(\frac{p}{2}-1)(2a_{2}-1)}.
\end{align*}
\begin{align*}
    &h_{\drift,1}(T)=h_{\drift,2}(T)= C(T) T^{(p-1)a_{1}}, \\
    &\qquad h_{\diffusion,1}(T)=C(T)T^{(\frac{p}{2}-1)(2a_{2}-1)},\qquad h_{\diffusion,2}(T)=C(T)T^{(p-1)(2a_{2}-1)}.
\end{align*}
     Let $u,A,G,F$ be non negative functions on $\R_{+}$. Since
\begin{align*}
    u(t)\leq A(t) + G(t) +F(t) \leq A(t) + 2\max\{G(t),F(t)\},
\end{align*}
it suffices to estimate $u(t)\leq A(t) + 2 G(t)$ and $u(t)\leq A(t) + 2 F(t)$ separately. In our case,
     \begin{align*}
   \E  \sup_{t\leq T}\|\mu_{t}\|_{\Wplusdualn}^{p} 
     &\leq C_{p,\operatorname{LG}}\E \sup_{t\leq T}\left\|S^{*}_{t}\mu_{0}\right\|^{p}_{\Wplusdualn}\\
     &\phantom{xx}+ 2C_{p,\drift,\diffusion,T} h_{\drift,j}(T)\int_{0}^{T}\frac{1}{(T-s)^{1-a_{1}}}\E\sup_{r\leq s}\|\mu_{r}\|_{\Wplusdualn}^{2}\ds\\
      &\phantom{xx}+2C_{p,\operatorname{LG}}T^{(p-1)a_{1}},
\end{align*}
     \begin{align*}
   \E  \sup_{t\leq T}\|\mu_{t}\|_{\Wplusdualn}^{p} 
     &\leq C_{p,\operatorname{LG}}\E \sup_{t\leq T}\left\|S^{*}_{t}\mu_{0}\right\|^{p}_{\Wplusdualn}\\
     &\phantom{xx}+2C_{p,\drift,\diffusion,T} h_{\diffusion,j}(T)\int_{0}^{T}\frac{1}{(T-s)^{2-2a_{2}}}\E\sup_{r\leq s}\|\mu_{r}\|_{\Wplusdualn}^{2}\ds\\
      &\phantom{xx}+2C_{p,\operatorname{LG}} T^{(\frac{p}{2}-1)(2a_{2}-1)}.
\end{align*}
We can include the terms $b_{1}$ and $b_{2,j}$ into the constant since it already depends on $T$. \cite[Lemma 2.2]{zhang2010stochastic} now yields
\begin{align*}
      \E  \sup_{t\leq T}\|\mu_{t}\|_{\Wplusdualn}^{p} 
     &\leq C_{p,\operatorname{LG},T}.
\end{align*}
\end{proof}

Now we consider the time-regularity of $t\mapsto \mu_{t}$. 
\begin{lemma}\label{lem:time_regularity}
    Let Assumptions \ref{A:assumption_general_linear_growth} and \ref{A:A_3:Narrower_assumption_nu} be satisfied, $\gamma \in [0,1]$ be arbitrary, $1\leq p<\infty$. Recall the weights from Definition \ref{def:weights_triple_for_analysis} and
    let $0\leq \gamma_{\drift},2\gamma_{\diffusion}\leq 1$ be such that  $\theta_{\nudrift}-\gamma_{\drift}\leq \etaminus,  \theta_{\nudiffusion}-\gamma_{\diffusion}\leq \etaminus$.
    Assume $\mu\in L^{p}(\Omega,L^{\infty}(0,
T;\Wplusdualn))$ to be a mild solution to \eqref{eqn:SEE_strong_formulation} with initial condition $\mu_{0}\in L^{p}(\Omega,\Wplusdualn)$, where $\Wplusdualn$ corresponds to the choice of space for which \eqref{eqn:drift_diffusion_kernel_estimate_semigroup} holds. Then for $s\leq t$ with $|t-s|\leq 1$, the following estimate holds.
    \begin{align*}
         \E\|\mu_{t}-\mu_{s}\|_{\Wminustwodualn}^{p}&\leq
      \E\|\mu_{t}-\mu_{s}\|_{\Wminusdualn}^{p}\\
      &\leq   C_{p,\operatorname{LG}}\left(|t-s|^{(1-\gamma_{\drift})p}+|t-s|^{(1-2\gamma_{\diffusion})\frac{p}{2}}+|t-s|^{p(\etaminus-\etaplus \wedge 1)}\right).
    \end{align*}
\end{lemma}

\begin{proof}[Proof of Lemma \ref{lem:time_regularity}]
\begin{align*}
    \E\|\mu_{t}-\mu_{s}\|_{\Wminusdualn}^{p}\leq C\left( \E\|\mu_{t}-S^{*}_{t-s}\mu_{s}\|_{\Wminusdualn}^{p}+ \E\|S^{*}_{t-s}\mu_{s}-\mu_{s}\|_{\Wminusdualn}^{p}\right).
\end{align*}

By Lemma \ref{lem:nu_is_in_dual}, $\nudrift\in W^{-1,2}_{\frac{1}{w_{\drift}}}$, $\nudiffusion\in W^{-1,2}_{\frac{1}{w_{\diffusion}}}$ with $(w_{\drift})_{i}=(1+x)^{2\theta_{\nudrift}-1+2i}$ and $(w_{\diffusion})_{i}=(1+x)^{2\theta_{\nudiffusion}-1+2i}$, for $i\geq 0$ and $x\in \R_{+}$.
Lemma \ref{lem:improvement_semigroup_weighted_spaces} then implies
\begin{align*}
    &\|S^{*}_{t}\nudrift\|_{W^{-1,2}_{\frac{1}{\weightminus}}}\leq C_{T} \|\nudrift\|_{W^{-1,2}_{\frac{1}{w_{\drift}}}}|t|^{-\gamma_{\drift}}\\
    &\|S^{*}_{t}\nudiffusion\|_{W^{-1,2}_{\frac{1}{\weightminus}}}\leq C _{T}\|\nudiffusion\|_{W^{-1,2}_{\frac{1}{w_{\diffusion}}}}|t|^{-\gamma_{\diffusion}},
\end{align*}
for any $0<t\leq T$.
For simplicity, we will set
\begin{align*}
    &\theta_{\nudrift}-\gamma_{\drift}=\theta_{\nudiffusion}-\gamma_{\diffusion}=\etaminus.
\end{align*}

    \begin{align*}
    \E&\|\mu_{t}-S^{*}_{t-s}\mu_{s}\|^{p}_{\Wminusdualn}\leq C_{p}\E\|\mu_{s}-S^{*}_{s-s}\mu_{s}\|^{p}_{\Wminusdualn}\\
    &\phantom{xx}+C_{p}\E\left(\int_{s}^{t}\|S^{*}_{r-s}\nudrift \drift(r,\langle \mu_{r},\psi\rangle))\|_{\Wminusdualn}\dr\right)^{p}\\
    &\phantom{xxxx}+C_{p}\E\left\|\int_{s}^{t}S^{*}_{r-s}\nudiffusion \diffusion(r,\langle \mu_{r},\psi\rangle))\dW_{r}\right\|_{\Wminusdualn}^{p}\\
    &\leq C_{p}\E\left(\int_{s}^{t}\|S^{*}_{r-s}\nudrift\|_{\opnminus}|\drift(r,\langle \mu_{r},\psi\rangle)|\dr\right)^{p}\\
    &\phantom{xx}+C_{p}\E\left(\int_{s}^{t}\|S^{*}_{r-s}\nudiffusion\|_{\opnminus}^{2}| \diffusion(r,\langle \mu_{r},\psi\rangle)|^{2}\dr\right)^{p/2}\\
    &\leq \E\left(\int_{s}^{t}\|S^{*}_{r-s}\nudrift\|_{\opnminus}\dr\right)^{p-1}\int_{s}^{t}\|S^{*}_{r-s}\nudrift\|_{\opnminus}|\drift(r,\langle \mu_{r},\psi\rangle))|^{p}\dr\\
    &\phantom{xx}+\E\left(\int_{s}^{t}\|S^{*}_{r-s}\nudiffusion\|_{\opnminus}^{2}| \diffusion(r,\langle \mu_{r},\psi\rangle)|^{2}\dr\right)^{p/2}\\
    &\leq \E\left(\int_{s}^{t}\frac{1}{(r-s)^{\gamma_{\drift}}}\dr\right)^{p-1}\int_{s}^{t}\frac{1}{(r-s)^{\gamma_{\drift}}}|\drift(r,\langle \mu_{r},\psi\rangle))|^{p}\dr\\
    &\phantom{xx}+\E\left(\int_{s}^{t}\frac{1}{(r-s)^{2\gamma_{\diffusion}}}\dr\right)^{\frac{p}{2}-1}\left(\int_{s}^{t}\frac{1}{(r-s)^{2\gamma_{\diffusion}}} |\diffusion(r,\langle \mu_{r},\psi\rangle)|^{2}\dr\right)^{p/2}\\
    &\leq C_{p,\operatorname{LG}} \E |t-s|^{(1-\gamma_{\drift})(p-1)}\int_{s}^{t}\frac{1+\| \mu_{r}\|_{\Wplusdualn}^{p}}{(r-s)^{\gamma_{\drift}}}\dr\\
    &\phantom{xx}+\E|t-s|^{(2\gamma_{\diffusion}\left(\frac{p}{2}-1\right)}\int_{s}^{t}\frac{1+\| \mu_{r}\|_{\Wplusdualn}^{p}}{(r-s)^{2\gamma_{\diffusion}}}\dr\\
    &\leq C_{p,\operatorname{LG}}\left(|t-s|^{(1-\gamma_{\drift})(p-1)}\int_{s}^{t}\frac{1}{(r-s)^{\gamma_{\drift}}}\dr+|t-s|^{2\gamma_{\diffusion}\left(\frac{p}{2}-1\right)}\int_{s}^{t}\frac{1}{(r-s)^{2\gamma_{\diffusion}}}\dr\right)\\
    &\phantom{xxxx}\times\E\sup_{t\leq T}\left(1+\| \mu_{r}\|_{\Wplusdualn}\right)^{p}\\
    &\leq C_{p,\operatorname{LG}}\left(|t-s|^{(1-\gamma_{\drift})(p-1)+(1-\gamma_{\drift})}+|t-s|^{(1-2\gamma_{\diffusion})\left(\frac{p}{2}-1\right)+1-2\gamma_{\diffusion}}\right)\\
    &\phantom{xxxx}\times\E\sup_{t\leq T}\left(1+\| \mu_{r}\|_{\Wplusdualn}^{p}\right)\\
    &\leq C_{p,\operatorname{LG}}\left(|t-s|^{(1-\gamma_{\drift})p}+|t-s|^{(1-2\gamma_{\diffusion})\frac{p}{2}}\right).
\end{align*}
For the second estimate, recall that $(\weightplus)_{i}(x)=(1+x)^{2\etaplus-1+2i}$ and use Lemma \ref{lem:time_difference_semigroup_regularity_weighted_space} to conclude that (since $\etaminus >\etaplus$)
\begin{align*}
     \E\|S^{*}_{t-s}\mu_{s}-\mu_{s}\|_{\Wminusdualn}\leq C\|\mu_{s}\|_{\Wplusdualn}^{p}|t-s|^{(\etaminus-\etaplus) \wedge 1}.
\end{align*}
 In summary, we obtain
  \begin{align*}
      \E\|\mu_{t}-\mu_{s}\|_{\Wminusdualn}^{p}\leq   C_{p,\operatorname{LG}}\left(|t-s|^{(1-\gamma_{\drift})p}+|t-s|^{(1-2\gamma_{\diffusion})\frac{p}{2}}+|t-s|^{p((\etaminus-\etaplus) \wedge 1)}\right).
  \end{align*}
\end{proof}

\begin{corollary}\label{corr:aldous_condition}
    Let $(\tau_{n})_{n\in\N}$ be a sequence of stopping times such that $0\leq \tau_{n}\leq T$ and let $\theta>0$, then
    \begin{align*}
         \E\|\mu_{\tau_{n}+\theta}-\mu_{\tau_{n}}\|_{\Wminustwodualn}^{p}\leq C_{p,\operatorname{LG}}\left(|\theta|^{(1-\gamma_{\drift})p}+|\theta|^{(1-2\gamma_{\diffusion})\frac{p}{2}}+|\theta|^{p(\etaminus-\etaplus \wedge 1)}\right).
    \end{align*}
\end{corollary}

\subsubsection{Tightness}
Equipped with these two estimates, we process to identify a convergent subsequence and potential limit.
\begin{definition}
    For a separable Banach space $V$, we define $C([0,T],V^{\operatorname{weak}^{*}})$, as the space of weakly continuous functions $u\colon [0,T]\rightarrow V$, equipped with the topology $\mathcal{T}$, such that for all $h$ in the predual of $V$, denoted by ${^{\prime}\!}V$, the mapping
    \begin{align*}
        C([0,T],V^{\operatorname{weak}^{*}})\ni v \mapsto \langle v(\cdot),h\rangle_{V\times {^{\prime}\!}V}\in C([0,T],\R)
    \end{align*}
    is continuous. In particular $v_n \rightarrow v$ in $C\left([0, T] ; V^{\operatorname{weak}^{*}}\right)$ iff for all $h \in {^{\prime}\!}V$ :
    \begin{align*}
        \lim _{n \rightarrow \infty} \sup _{t \in[0, T]}\left|\left\langle v_n(t)-v(t) , h\right\rangle_{V\times {^{\prime}\!}V}\right|=0 .
    \end{align*}
    We write $C\left([0, T] ; V^{\operatorname{weak}}\right)$ for the same space but equipped with the weak topology.
\end{definition}

Consider the ball
$$
B_{r}:=\left\{y \in V ; \quad\|y\|_{V} \leq r\right\}.
$$
If $V$ is separable, the weak$^{*}$ topology induced on $B_{r}$ is metrizable. Let $q$ denote the metric compatible with the weak$^{*}$ topology on $B_{r}$. Let us consider
\begin{align*}
    C\left([0, T] ; B^{\operatorname{weak}^{*}}_{r}\right)=\left\{u \in C\left([0, T] ; V^{\operatorname{weak}^{*}}\right): \sup _{t \in[0, T]}\|u(t)\|_{V} \leq r\right\},
\end{align*}
 which denotes the space of weakly continuous functions $v:[0, T] \rightarrow V$ and such that $\sup _{t \in[0, T]}\|v(t)\|_{V} \leq r$.
The space $C\left([0, T] ; B_{r}^{\operatorname{weak}^{*}}\right)$ is metrizable with
$$
\varrho(u, v)=\sup _{t \in[0, T]} q(u(t), v(t)).
$$
\begin{remark}
The unit-ball on $L^{p}(0,T;V)$ is $\operatorname{weak}^{*}$-metrizable if and only if $L^{p}(0,T;V)$ is separable, in which case compactness and sequential compactness coincide.

Since by the Banach-Alaoglu Theorem $B^{\operatorname{weak}^{*}}_{r}$ is compact, $\left(C\left([0, T] ; B^{\operatorname{weak}^{*}}_{r}\right), \varrho\right)$ is a complete metric space.
\end{remark}

Let us consider $V=\Wplusdualn$ and $B_{R}$ the corresponding ball of radius $R$ in $\Wplusdualn$.
\begin{remark}
    Note that the weak topology on $\Wplusdualn$ coincides with the weakest topology for which the mappings $v\mapsto\langle v(\cdot),\varphi\rangle_{\Wplusdualn\times \Wplus}\in C([0,T],\R)$ are continuous.
\end{remark}

\begin{lemma}\label{lem:compactness}
    Let $\mathcal{Z}\coloneqq C([0,T],\Wmidtwodualn)\cap C\left([0, T] ; (\Wplusdualn)^{\operatorname{weak}^{*}}\right)$, equipped with the maximum of the two topologies. Then a set $K\subset \mathcal{Z}$ is relatively compact in $\mathcal{Z}$, if the following conditions hold:
    \begin{enumerate}
        \item $\sup_{u\in K}\sup_{t\in[0,T]}\|u\|_{\Wplusdualn}<\infty$,
        \item $\lim_{\delta\rightarrow 0}\sup_{u\in K}\sup_{s,t\in [0,T],|t-s|\leq \delta}\|u(t)-u(s)\|_{\Wminustwodualn}=0$.
    \end{enumerate}
    
\end{lemma}
\begin{proof}
    Without loss of generality, we assume that $K$ is closed in $\mathcal{Z}$. The first condition allows us to work on the metric subspace $C\left([0, T]; B_{R}^{\operatorname{weak}^{*}}\right)\subset C\left([0, T] ; (\Wplusdualn)^{\operatorname{weak}^{*}}\right)$, for some $R>0$ large enough. Due to the compact embedding of $\Wplusdualn \rightarrow \Wmidtwodualn$, by Proposition \ref{prop:embeddings}, we can use standard arguments (see \cite{simon1986compact}) to obtain the compactness of $K$ in $C([0,T],\Wmidtwodualn)$. By Lemma \ref{lem:uniform_continuity_C_weak}, any sequence $(u_{n})_{n}\subset C\left([0, T] ; B_{R}\right)$, which converges in $\Wmidtwodualn$, also converges in $C\left([0, T] ; B_{R}^{\operatorname{weak}^{*}}\right)$. This finishes the proof, since we found, for any bounded sequence in $K$, a convergent subsequence in $\mathcal{Z}$.
\end{proof}
\begin{lemma}\label{lem:tightness_criterion}
    Let $\{\mu_{n}\}_{n\in\N}$ be a sequence of continuous, $\Fil$-adapted, $\Wmidtwodualn$-valued processes such that, for some $1<p$,
       \begin{enumerate}
        \item $\sup_{u\in K}\E\sup_{t\in[0,T]}\|u\|_{\Wplusdualn}^{p}<\infty$,
        \item  $\forall \varepsilon>0 \quad \forall \eta>0 \quad \exists \delta>0$ such that for every sequence $\left(\tau_n\right)_{n \in \mathbb{N}}$ of $\Fil$-stopping times with $\tau_n \leq T$ one has
        \begin{align*}
            \sup _{n \in \mathbb{N}} \sup _{0<\theta \leq \delta} \Prob\left(\left\|\mu_n\left(\tau_n+\theta\right)- \mu_n\left(\tau_n\right)\right\|_{\Wminustwodualn} \geq \eta\right) \leq \varepsilon .
        \end{align*}
        Let $\Law_{n}$ denote the law of $\mu_{n}$ on $\mathcal{Z}$. Then for every $\eps>0$, there exists a compact subset $K_{\eps}\subset \mathcal{Z}$, such that
        \begin{align*}
            \sup_{n}\Law_{n}(K_{\eps})\geq 1-\eps.
        \end{align*}
    \end{enumerate}
\end{lemma}
\begin{proof}
    Let $\eps>0$. By the Markov inequality, we infer that for every $n\in \N$ and $R>0$,
    \begin{align*}
        \Prob\left(\sup_{t\in[0,T]}\|u\|_{\Wplusdualn}^{p}>R_{1}\right)\leq \frac{\E\sup_{t\in[0,T]}\|u\|_{\Wplusdualn}^{p}}{R_{1}}\leq \frac{C}{R}.
    \end{align*}
    Choosing $\frac{2C}{\eps}\leq R$ yields
    \begin{align*}
        \sup_{n\in \N}\Prob\left(\sup_{t\in[0,T]}\|u\|_{\Wplusdualn}^{p}>R\right)\leq \frac{\eps}{2}.
    \end{align*}
    By \cite[Lemma 3.6, Lemma 3.8]{brzezniak2013existence}, there exists a subset $A_{\frac{\eps}{2}}\subset C([0,T],\Wminustwodualn)$, such that $\Law_{n}(A_{\frac{\eps}{2}})\geq 1-\frac{\eps}{2}$. We can define $K_{\eps}$ as the closure of $B_{R}\cap A_{\frac{\eps}{2}}$ in $\mathcal{Z}$. By Lemma \ref{lem:compactness} $K_{\eps}$ is compact in $\mathcal{Z}$ and the claim follows.
\end{proof}
\begin{corollary}
        Let $\mu_{n}$ denote a mild solution of \eqref{eqn:SEE_strong_formulation}, given by Theorem \ref{thm:existence_uniqueness_lipschitz}. The laws of $\{\mu_{n}\}_{n\in\N}$ are tight on $C([0,T],\Wmidtwodualn)\cap C\left([0, T] ; B_{R}^{\operatorname{weak}^{*}}\right)$.
\end{corollary}

By the Skorohod representation theorem (see e.g \cite[Theorem A.1]{brze2013stochastic}, or \cite{Jakubowski98_Skorohod}), there exists a subsequence $(n_{k})_{k\in \N}$, which is not relabeled,
a probability space $(\widetilde{\Omega},
\widetilde{\mathcal{F}},\widetilde{\Prob})$ and, on this space,
$(C([0,T],\Wmidtwodualn)\cap C\left([0, T] ; B_{R}^{\operatorname{weak}^{*}}\right),C^{0}([0,T];\R^{m_{W}}))$-valued random variables $(\widetilde{\mu},\weighttilde)$
and $(\widetilde{\mu}_{n},\weighttilde_{n})$ such that
$(\widetilde{\mu}_{n},\weighttilde_{n})$ has the same law as $(\mu_{n},W)$
on $\mathcal{B}(C([0,T],\Wmidtwodualn)\times C^{0}([0,T];\R^{m_{W}}))$ and, as $n \rightarrow \infty$,
$$
  (\widetilde{\mu}_{n},\weighttilde^{n})\to (\widetilde{\mu}, \weighttilde)
	\quad\mbox{in }C([0,T],\Wmidtwodualn)\cap C\left([0, T] ; B_{R}^{\operatorname{weak}^{*}}\right) \times C^{0}([0,T];\R^{m_{W}})\ \widetilde{\Prob}\mbox{-a.s.}
$$
The first observation is that these new variables also satisfy \eqref{eqn:SEE_eqn_mild} $\widetilde{\Prob}$--a.s., replacing $(\mu_{n}, W)$ by $(\widetilde{\mu}_{n}, \widetilde {W})$. The proof is done via a regularization argument, similar to \cite[Theorem 2.9.1]{Breit_Feireisl_Hofmanova18stochastic_fluid_flows} (or \cite{Bensoussan95_stochastic_navier_stokes,Brzezniak10_stochastic_reaction_diffusion_jump_processes}).  
It remains to verify that the limit still satisfies the proposed equation. For convenience, we will drop the tilde notation and agree to work on the new probability space for the remainder of this section. In the first lemma, we investigate the convergence of the individual terms of our approximate mild solution.
\begin{lemma}\label{lem:convergence_of_approximate_mild_solution_terms}
For every $t\in [0,T]$, we have the following convergences
    \begin{enumerate}
        \item $\lim_{n\rightarrow \infty}\widetilde{\E}\left\|\widetilde{\mu}_{n}(t)-\widetilde{\mu}(t)\right\|_{\Wmidtwodualn}=0$.
        \item We have
        \begin{align*}
            \lim_{n\rightarrow \infty}\widetilde{\E}\bigg\|&\int_{0}^{t} S^{*}_{t-s}\nudrift\drift_{n}(s,\langle\widetilde{\mu}_{n}(s),\psi\rangle)\ds\\
            &\qquad-\int_{0}^{t} S^{*}_{t-s}\nudrift\drift(s,\langle\widetilde{\mu}(s),\psi\rangle)\ds\bigg\|_{\Wmidtwodualn}=0.
        \end{align*}
        \item We have
        \begin{align*}
            \lim_{n\rightarrow \infty}\widetilde{\E}\bigg\|&\int_{0}^{t} S^{*}_{t-s}\nudiffusion\diffusion_{n}(s,\langle\widetilde{\mu}_{n}(s),\psi\rangle)\dW^{n}_{s}\\
            &\qquad-\int_{0}^{t} S^{*}_{t-s}\nudiffusion\diffusion(s,\langle\widetilde{\mu}(s),\psi\rangle)\dW_{s}\bigg\|_{\Wmidtwodualn}=0.
        \end{align*}
    \end{enumerate}
\end{lemma}
\begin{proof}
    The first claim already follows from the application of the Skorohod representation theorem. For the second claim we abbreviate $\nu:=\nudrift$ and write $\|\cdot\|_{\opnmid}$ for the operator norm on $\opnmid$. Since $\Wplusdual\hookrightarrow\Wmidtwodual$ (so $\Wmidtwodualn\supseteq\Wplusdualn$ with a weaker norm), we have $\|S^{*}_{t-s}\nu\|_{\opnmid}\leq C\|S^{*}_{t-s}\nu\|_{\opn}$, which is integrable in $s$ by Assumption~\ref{A:A_3:Narrower_assumption_nu}. Adding and subtracting, and using that $\langle\cdot,\psi\rangle$ is a bounded functional so that each term below is $\Wmidtwodualn$-valued,
\begin{align*}
    &\widetilde{\E}\left\|\int_{0}^{t} S^{*}_{t-s}\nu\,\drift_{n}(\langle\widetilde{\mu}_{n}(s),\psi\rangle)\ds-\int_{0}^{t} S^{*}_{t-s}\nu\,\drift(\langle\widetilde{\mu}(s),\psi\rangle)\ds\right\|_{\Wmidtwodualn}\\
   &\leq   \int_{0}^{t} \|S^{*}_{t-s}\nu\|_{\opnmid}\,\widetilde{\E}\big|\drift_{n}(\langle\widetilde{\mu}_{n}(s),\psi\rangle)-\drift(\langle\widetilde{\mu}_{n}(s),\psi\rangle)\big|\ds\\
   &\phantom{xx}+   \int_{0}^{t} \|S^{*}_{t-s}\nu\|_{\opnmid}\,\widetilde{\E}\big|\drift(\langle\widetilde{\mu}_{n}(s),\psi\rangle)-\drift(\langle\widetilde{\mu}(s),\psi\rangle)\big|\ds\\
   &\phantom{xx}+   \int_{0}^{t} \|S^{*}_{t-s}\nu\|_{\opnmid}\,\widetilde{\E}\big|\drift(\langle\widetilde{\mu}(s),\psi\rangle)-\drift_{n}(\langle\widetilde{\mu}(s),\psi\rangle)\big|\ds\\
   &=I+II+III,
\end{align*}
where $|\cdot|$ denotes the Euclidean norm on $\R^{n_{\textnormal{dim}}}$.

We first consider $II$. Since $\widetilde{\mu}_{n}\to \widetilde{\mu}$ in $C([0,T],\Wmidtwodualn)$ for almost every $\omega \in \Omega$ and $\langle\cdot,\psi\rangle$ is continuous, $\langle\widetilde{\mu}_{n}(s),\psi\rangle \to \langle\widetilde{\mu}(s),\psi\rangle$ in $\R^{n_{\textnormal{dim}}}$ almost surely (here $n$ indexes the approximating sequence, unrelated to the state dimension $n_{\textnormal{dim}}$). As $\drift\colon\R^{n_{\textnormal{dim}}}\to\R^{n_{\textnormal{dim}}}$ is continuous, $|\drift(\langle\widetilde{\mu}_{n}(s),\psi\rangle)-\drift(\langle\widetilde{\mu}(s),\psi\rangle)|\to 0$ a.s.; by the linear growth of $\drift$ and the a-priori bound of Lemma~\ref{lem:a-priori_estimate_general}, dominated convergence gives $II\to 0$.

Now to $I$ and $III$. Introduce the stopping times $\tau_{m} \coloneqq \inf\{t\geq 0 \colon \|\widetilde{\mu}_{n}(t)\|_{\Wmidtwodualn} >m\ \text{or}\ \|\widetilde{\mu}(t)\|_{\Wmidtwodualn}>m\}$. On $\{s\leq\tau_{m}\}$ the state variable stays in a fixed ball, $|\langle\widetilde{\mu}(s),\psi\rangle|\leq\|\psi\|\,\|\widetilde{\mu}(s)\|_{\Wmidtwodualn}\leq R_{m}$ with $R_{m}:=m\|\psi\|$, so, writing $\omega_{n}(R):=\sup_{|y|\leq R}|\drift_{n}(y)-\drift(y)|$ for the local uniform error of the approximating sequence,
\begin{align*}
   &\widetilde{\E} \int_{0}^{t\wedge \tau_{m}} \|S^{*}_{t-s}\nu\|_{\opnmid}\,\big|\drift(\langle\widetilde{\mu}(s),\psi\rangle)-\drift_{n}(\langle\widetilde{\mu}(s),\psi\rangle)\big|\ds\\
   &\leq  \Big(\int_{0}^{t} \|S^{*}_{t-s}\nu\|_{\opnmid}\ds\Big)\,\omega_{n}(R_{m}).
\end{align*}
Since $\drift_{n}\to\drift$ locally uniformly (Proposition citing \cite{hofmanova2012weak}, i.e.\ the approximating sequence of the general-coefficients step), $\omega_{n}(R_{m})\to 0$ as $n\to\infty$ for each fixed $m$; and $\widetilde{\Prob}(\tau_{m}<t)\to 0$ as $m\to\infty$ by the a-priori bound, so the contribution of $\{s>\tau_{m}\}$ vanishes in the double limit. Hence $III\to 0$, and $I$ is handled identically.
The stochastic terms are treated similarly. Note that
\begin{align*}
    &\widetilde{\E}\left|\int_{0}^{t} \left\|S^{*}_{t-s}\nu(\dx)\diffusion_{n}(\langle\widetilde{\mu}_{n}(s),\psi\rangle)-S^{*}_{t-s}\nu(\dx)\diffusion(\langle\widetilde{\mu}(s),\psi\rangle)\right\|^{2}_{\Wmidtwodualn}\ds\right|\\
    &\leq C \widetilde{\E}\left[\int_{0}^{t}\|S^{*}_{t-s}\nu\|_{\opnmid}^{2}\left|\diffusion_{n}(\langle\widetilde{\mu}_{n}(s),\psi\rangle)-\diffusion(\langle\widetilde{\mu}_{n}(s),\psi\rangle)\right|^{2}\ds\right]\\
    &\phantom{xx}+C \widetilde{\E}\left[\int_{0}^{t}\|S^{*}_{t-s}\nu\|_{\opnmid}^{2}\left|\diffusion(\langle\widetilde{\mu}_{n}(s),\psi\rangle)-\diffusion(\langle\widetilde{\mu}(s),\psi\rangle)\right|^{2}\ds\right]\\
   &\phantom{xx}+C \widetilde{\E}\left[\int_{0}^{t}\|S^{*}_{t-s}\nu\|_{\opnmid}^{2}\left|\diffusion(\langle\widetilde{\mu}(s),\psi\rangle)-\diffusion_{n}(\langle\widetilde{\mu}(s),\psi\rangle)\right|^{2}\ds\right]=S_{1}+S_{2}+S_{3}.
\end{align*}
Here $|\cdot|$ denotes the Frobenius norm on $\R^{n_{\textnormal{dim}}\times m_{W}}$. The three terms $S_{1},S_{2},S_{3}$ are treated exactly as $I,II,III$ above: $S_{2}$ by continuity of $\diffusion$ and dominated convergence, and $S_{1},S_{3}$ by the same localization $\tau_{m}$ together with the local uniform convergence $\diffusion_{n}\to\diffusion$ (now using $\int_{0}^{t}\|S^{*}_{t-s}\nu\|_{\opnmid}^{2}\ds<\infty$ from Assumption~\ref{A:A_3:Narrower_assumption_nu}, the square-integrable bound needed for the It\^o isometry). The passage from these $L^{1}(\Omega)$-bounds to convergence of the stochastic integrals is a direct adaptation of \cite[Lemma 2.1]{debussche2011local} or \cite[Lemma 2.6.6, step 3]{Breit_Feireisl_Hofmanova18stochastic_fluid_flows}.
By standard arguments, see e.g. \cite{brzezniak2016invariant}, we conclude that $(\widetilde{\Omega},\widetilde{\mathcal{F}},\widetilde{\Prob})$ and $\big(\widetilde{\mu}, \weighttilde\big)$ satisfy the conditions of Definition \ref{def:prbabilistically_weak_mild_solution} and are a probabilistically weak, mild solution to the SEE \eqref{eqn:SEE_strong_formulation}.
$\tau_{m}$ was merely a localizing sequence and we can let $m\rightarrow \infty$ and obtain the claim.
\end{proof}
\begin{proof}[Conclusion of the proof of Theorem \ref{thm:weak_mild_solution_SEE}]
Theorem \ref{thm:weak_mild_solution_SEE} follows directly from the previous Lemmata.
\end{proof}

Since a probabilistically weak solution usually connects to a martingale problem, we briefly argue that a mild solution, as obtained above, also satisfies a weak formulation of \eqref{eqn:SEE_strong_formulation}, which is more directly related to a martingale problem. We will drop the tilde notation when working with a probabilistically weak solution.
\begin{lemma}
    Let $\mu$ be a mild solution to \eqref{eqn:SEE_eqn_mild} and assume that $\nudrift,\nudiffusion$ satisfy Assumption \ref{A:A_3:Narrower_assumption_nu} and set $\theta_{\nu}=\max\{\theta_{\nudrift},\theta_{\nudiffusion}\}$. Let $\varphi \in W^{1,2}_{\bar{w}}$, with $\bar{w}_{i}=(1+x)^{2+2\max\{\theta_{\nudrift},\theta_{\nudiffusion}\}-1+2i}, (w_{\theta_{\nu}})_{i}=(1+x)^{2\max\{\theta_{\nudrift},\theta_{\nudiffusion}\}-1+2i}$, then $\mu$ satisfied the weak formulation
    \begin{align}\label{eqn:SEE_weak_formulation}
        \langle \mu_{t},\varphi\rangle &=  \langle \mu_{0},\varphi\rangle+\int_{0}^{t}\langle \mu_{s},-x \varphi\rangle\ds\\
        &\phantom{xx}+\int_{0}^{t}\langle  \nudrift,\varphi\rangle \drift (s,\langle \mu_{s},\psi\rangle)\ds+\int_{0}^{t}\langle  \nudiffusion,\varphi\rangle \diffusion (s,\langle \mu_{s},\psi\rangle)\dW_{s}.
    \end{align}
    
\end{lemma}
\begin{proof}
    The proof is standard, so we only sketch it. Without loss of generality, assume $\mu_{0}=0$. Otherwise, we can consider the process $\mu_{t}-S^{*}_{t}\mu_{0}$. Fix $\varphi$, apply $-x\varphi$ to \eqref{eqn:SEE_eqn_mild} and integrate over time.
    \begin{align*}
        \int_{0}^{t}\langle \mu_{s}-S^{*}_{s}\mu_{0},-x\varphi\rangle \ds&= \int_{0}^{t}\int_{0}^{s}\langle S^{*}_{s-r}\nudrift, -x\varphi\rangle \drift(r,\langle \mu_{r},\psi\rangle)\dr\ds\\
        &\phantom{xx}+\int_{0}^{t}\int_{0}^{s}\langle S^{*}_{s-r}\nudiffusion, -x\varphi\rangle \diffusion(r,\langle \mu_{r},\psi\rangle)\dW_{r}\ds\\
        &=\int_{0}^{t}\langle  \nudrift, \int_{r}^{t}-S_{s-r}x\varphi  \ds\rangle \drift(r,\langle \mu_{r},\psi\rangle) \dr\\
        &\phantom{xx}+\int_{0}^{t}\langle  \nudiffusion,\int_{r}^{t}-S_{s-r}x\varphi \ds \rangle  \diffusion(r,\langle \mu_{r},\psi\rangle) \dW_{r}.
    \end{align*}
    By Lemma \ref{lem:semigroup_and_adjoint_semigroup} $S^{*}$ is a strongly continuous semigroup on $W^{-1,2}_{\frac{1}{w_{\theta_{\nu}}}}$. Hence
       \begin{align*}
        \int_{0}^{t}\langle \mu_{s}-S^{*}_{s}\mu_{0},-x\varphi\rangle \ds&=\int_{0}^{t}\langle \nudrift, S_{t-r}\varphi\rangle  \drift(r,\langle \mu_{r},\psi\rangle) \dr-\int_{0}^{t}\langle  \nudrift,\varphi\rangle  \drift(r,\langle \mu_{r},\psi\rangle) \dr\\
        &\phantom{xx}+\int_{0}^{t}\langle \nudiffusion,  S_{t-r} \varphi \rangle \diffusion(r,\langle \mu_{r},\psi\rangle) \dW_{r}-\int_{0}^{t}\langle \nudiffusion,\varphi\rangle  \diffusion(r,\langle \mu_{r},\psi\rangle)\dW_{r}\\
        &=\langle \mu_{t}-S^{*}_{t}\mu_{0},\varphi\rangle-\int_{0}^{t}\langle  \nudrift,\varphi\rangle  \drift(r,\langle \mu_{r},\psi\rangle) \dr-\int_{0}^{t}\langle \nudiffusion,\varphi \rangle  \diffusion(r,\langle \mu_{r},\psi\rangle)\dW_{r}.
    \end{align*}
    In summary, we obtain
        \begin{align*}
     &\langle \mu_{t}-S^{*}_{t}\mu_{0},\varphi\rangle-\int_{0}^{t}\langle  \nudrift,\varphi \rangle  \drift(r,\langle \mu_{r},\psi\rangle) \dr-\int_{0}^{t}\langle \nudiffusion,\varphi \rangle  \diffusion(r,\langle \mu_{r},\psi\rangle)\dW_{r}\\
     &=   \int_{0}^{t}\langle \mu_{s},-x\varphi\rangle \ds- \langle S^{*}_{s}\mu_{0},\varphi\rangle+\langle \mu_{0},\varphi\rangle.
    \end{align*}
\end{proof}
The previous Lemma implies in particular, that for every $\varphi \in W^{-1,2}_{\frac{1}{\bar{w}}}$, the process
\begin{align*}
    M_{\varphi}(t,\mu)\coloneqq  \langle \mu_{t},\varphi\rangle -  \langle \mu_{0},\varphi\rangle+\int_{0}^{t}\langle -x \mu_{s},\varphi\rangle\ds-\int_{0}^{t}\langle  \nudrift,\varphi\rangle \drift (s,\langle \mu_{s},\psi\rangle)\ds,
\end{align*}
is a square-integrable $\Fil_{t}$ martingale with respect to the measure $P$, with quadratic variation
\begin{align*}
    \int_{0}^{t}\langle  \nudiffusion,\varphi\rangle^{2} \diffusion (s,\langle \mu_{s},\psi\rangle)^{2}\ds.
\end{align*}
\begin{remark}[A dual process and weak uniqueness of the lift in the case $\diffusion(z)=z^{\gamma}$]
\label{rem:dual_uniqueness}
The equivalence with the weak formulation suggests a dual process whose \emph{mere
existence} yields uniqueness in law of the lift $\mu$ in the non-Lipschitz regime $\diffusion(z)=z^{\gamma}$,
$\gamma\in(\tfrac12,1)$. Set $\psi\equiv1$, $Z_{t}=\langle\mu_{t},\mathbf 1\rangle$,
and apply the singular It{\^o} formula to $H(\mu,\varphi)=e^{-\langle\mu,\varphi\rangle}$;
the generator on such exponentials is
\[
  \mathcal L_{\mu}H=e^{-\langle\mu,\varphi\rangle}\big(\langle\mu,x\varphi\rangle
  -\langle\nu,\varphi\rangle Z+\tfrac12\langle\nu,\varphi\rangle^{2}Z^{2\gamma}\big).
\]
Requiring $\mathcal L_{\mu}H=\mathcal A_{\varphi}H$ for a Markov dual
$(\varphi_{t})$ determines its dynamics uniquely: the transport term reproduces
$S^{*}_{t}=e^{-tx}$ by Laplace-transform duality, the linear drift a
constant-function source of rate $\langle\nu,\varphi\rangle$, and, via the stable
representation $z^{2\gamma}=c_{2\gamma}\int_{0}^{\infty}(e^{-\lambda z}-1+\lambda z)
\lambda^{-2\gamma-1}\,\mathrm d\lambda$ with
$c_{2\gamma}=2\gamma(2\gamma-1)/\Gamma(2-2\gamma)$, a $2\gamma$-stable jump
mechanism. Pairing with $\nu$ and using $\langle\nu,e^{-ux}\rangle=k(u)$ collapses
the dual to the scalar Volterra SDE
\[
  V_{t}=g(t)+\int_{0}^{t}k(t-s)V_{s}\,\mathrm ds
        +\int_{0}^{t}k(t-s)V_{s}^{1/\gamma}\,\mathrm dL_{s},
  \qquad V_{t}=\langle\nu,\varphi_{t}\rangle,
\]
where $L$ is spectrally positive $2\gamma$-stable with L\'evy measure
$\tfrac12 c_{2\gamma}\zeta^{-2\gamma-1}\,\mathrm d\zeta$ and
$g(t)=\langle\nu,e^{-tx}\varphi_{0}\rangle$.

Only the \emph{existence} of one such dual is needed; its own law need not be
unique. Because $\diffusion^{2}(z)=z^{2\gamma}$ is superlinear the dual carries a
positive feedback and generically explodes in finite time, so one carries $V$ as
a $[0,\infty]$-valued process absorbed at $\infty$ (with $H\le1$ and
$e^{-\langle\mu_{0},\varphi_{t}\rangle}=0$ at the cemetery), constructed by
localization: the truncated coefficient $(v\wedge n)_{+}^{1/\gamma}$ is globally
Lipschitz, so each truncated equation is strongly well posed, and the truncations
patch up to the explosion time $\tau$. Comparing the two generators along the
independent primal and dual yields the unconditional one-sided relation
\[
  \mathbb E\big[e^{-\langle\mu_{t},\varphi_{0}\rangle}\big]
  \;\ge\;\mathbb E\big[e^{-\langle\mu_{0},\varphi_{t}\rangle}\mathbf 1_{\{t<\tau\}}\big],
\]
and equality, hence, the Laplace functionals $\{e^{-\langle\cdot,\varphi_{0}\rangle}\}$
being measure-determining, uniqueness in law of $\mu$, precisely when the
nonnegative defect left by the explosion vanishes as the truncation is removed
(cf.\ \cite[Thm.~4.4.11]{ethier2009markov} and the self-duality argument of
\cite{mytnik1998weak}). This is automatic when the dual does not explode. In the
singular regime it is delicate, because $\langle\nu,\varphi_{t}\rangle$ meets the
diagonal singularity $k(0^{+})=\infty$ while the damping exponent
$\langle\mu_{s},\varphi_{t}\rangle$ does not; but the defect is an \emph{integral
in time}, and $k\in L^{2}$ tames it there, reducing the requirement to a non-degeneracy
condition on the primal mass. The construction rests on Assumption
\ref{A:A_2:assumption_1_test_function} ($\mathbf 1\in\Wplus$ an admissible test
function), so that the constant-mode source and jumps of the dual, and the
pairings $\langle\mu_{s},\mathbf 1\rangle$, $\langle\nu,\varphi_{t}\rangle$,
$\langle\mu_{0},\varphi_{t}\rangle$, are well defined. We record the construction
as motivation for the dual formulation and a route to weak uniqueness of the lift.
\end{remark}

\section{Equivalence with the Volterra Equation}\label{sec:equivalence}
In order to transfer properties, derived on the level of the lift to the SVE, we require the following result.
\begin{theorem}\label{thm:equivalence_SVE_SEE}
Let $\drift: \R_{+}\times \R^{n_{\textnormal{dim}}} \rightarrow \R^{n_{\textnormal{dim}}}$ and $\diffusion: \R_{+}\times \R^{n_{\textnormal{dim}}} \rightarrow \R^{n_{\textnormal{dim}}\times m_{W}}$ be continuous and satisfy the linear growth condition \ref{A:assumption_general_linear_growth}. Further, let $\nudrift, \nudiffusion$ be ($n_{\textnormal{dim}}\times n_{\textnormal{dim}}$-matrices of) nonnegative measures on $\R_{+}$, such that for every $\lambda>0$, $\int_{\R_{+}}e^{-\lambda x}\nudrift(\dx)=k_{\drift}(\lambda)$, $\int_{\R_{+}}e^{-\lambda x}\nudiffusion(\dx)=k_{\diffusion}(\lambda)$, where $k_{\drift}\in L^{1}(0,T)$, $k_{\diffusion}\in L^{2}(0,T)$ are ($n_{\textnormal{dim}}\times n_{\textnormal{dim}}$-matrix-valued, entrywise) completely monotone kernels. Let $w, \dualweight$ be weight functions, such that Assumptions \ref{A:assumption_kernel_measure} hold.

    \begin{enumerate}
        \item
        Assume there exists a solution $X$  of      \begin{align}\label{eqn:SVE_equiv}
        X_{t} = X_{0} +\int_{0}^{t} k_{\drift}(t-s)\drift(s,X_{s})\ds +\int_{0}^{t}k_{\diffusion}(t-s)\diffusion(s,X_{s})\dW_{s},
    \end{align}
    satisfying
\begin{align}\label{eqn:equivalence_integrability_X}
    \int_{0}^{t}|\drift(s,X_{s})|\ds<\infty,\qquad \int_{0}^{t}|\diffusion(s,X_{s})|^{2}\ds<\infty,
\end{align}
where $|\cdot|$ denotes the Euclidean norm on $\R^{n_{\textnormal{dim}}}$, resp.\ the Frobenius norm on $\R^{n_{\textnormal{dim}}\times m_{W}}$.
        Then the $(W^{-1,2}_{\dualweight})^{n_{\textnormal{dim}}}$-valued process $Y$ defined by     \begin{align}\label{eqn:Y_process}
        Y_{t} = e^{-xt}x_{0}\delta_{0} +\int_{0}^{t} e^{-x(t-s)}\nudrift\drift(s,X_{s})\ds +\int_{0}^{t}e^{-x(t-s)}\nudiffusion\diffusion(s,X_{s})\dW_{s},
    \end{align}
is a mild solution of
            \begin{align}\label{eqn:SEE_eqn_weak_equiv}
        \mu_{t} =\mu_{0} -\int_{0}^{t} x\mu_{s}\ds+\int_{0}^{t} \nudrift(x)\drift(s,\langle \mu_{s},1\rangle)\ds +\int_{0}^{t} \nudiffusion(x)\diffusion(s,\langle \mu_{s},1\rangle)\dW_{s}
    \end{align}
with initial condition $\mu_{0}=x_{0}\delta_{0}$ and it holds that
\begin{align*}
    X_{t}=\langle \mu_{t},1\rangle,\qquad a.s. \text{ for a.e. } t>0,
\end{align*}
where, since $X_{t}\in\R^{n_{\textnormal{dim}}}$, the pairing $\langle\mu_{t},1\rangle$ is understood componentwise as in the discussion following equation~\eqref{eqn:SEE_eqn_strong_form_introduction}.
\item

If $\mu$ is a mild solution of the lifted stochastic evolution equation \eqref{eqn:SEE_eqn_weak_equiv}, according to Definition \ref{def:mild_solution_SEE}, with the initial condition $x_{0}\delta_{0}$, and let
\begin{align}\label{eqn:equivalence_integrability_mu}
    \int_{0}^{t}|\drift(s,\langle \mu_{s},1\rangle)|\ds<\infty,\qquad \int_{0}^{t}|\diffusion(s,\langle \mu_{s},1\rangle)|^{2}\ds<\infty,
\end{align}
then the $\R^{n_{\textnormal{dim}}}$-valued process $X$ defined by
\begin{align*}
    X_t:=\langle \mu_{t},1\rangle, t>0,
\end{align*}
is a solution of \eqref{eqn:SVE_equiv} with initial condition $x_{0}$. Furthermore, \eqref{eqn:Y_process} holds a.s. for any $t\geq 0$.
    \end{enumerate}
In particular, uniqueness holds for the SVE \eqref{eqn:SVE_equiv} with initial condition $x_{0}$ if and only if uniqueness holds for the lifted SEE \eqref{eqn:SEE_strong_formulation} with initial condition $x_{0}\delta_{0}$.
\end{theorem}
\begin{proof}[Proof of Theorem \ref{thm:equivalence_SVE_SEE}]\label{proof:thm:equivalence_SVE_SEE}\leavevmode
\begin{enumerate}
    \item  Let $X$ be a solution of the SVE \eqref{eqn:SVE_equiv} with $X_{0}=x_{0}$. Define
     \begin{align}
        Y_{t} = e^{-xt}Y_{0} +\int_{0}^{t} e^{-x(t-s)}\nudrift(x)\drift(s,X_{s})\ds +\int_{0}^{t}e^{-x(t-s)}\nudiffusion(x)\diffusion(s,X_{s})\dW_{s}.
    \end{align}
    By the assumed integrability \eqref{eqn:equivalence_integrability_X}, Lemma \ref{lem:nu_is_in_dual} and Lemma \ref{lem:nu_semigroup_properties}, $Y$ is a $(W^{-1,2}_{\frac{1}{w}})^{n_{\textnormal{dim}}}$-valued adapted process and satisfies $\int_{0}^{T}\|Y_{t}\|_{W^{-1,2}_{\frac{1}{w}}}^{2}\dt<\infty$ a.s. for any $t\geq 0$, due to Young's convolution inequality (applied entrywise to each of the $n_{\textnormal{dim}}$ components). By the same estimates, performed in Section \ref{sec:existence_general_coefficients}, $Y$ is a continuous $(\Wminustwodual)^{n_{\textnormal{dim}}}$-valued process, where $\weightminus$ is given as in Definition \ref{def:weights_triple_for_analysis}. We note that the map $f\mapsto \langle f,1\rangle$ from $W^{-1,2}_{\frac{1}{w}}$ to $\R$ is linear; applied componentwise, it gives a linear map from $(W^{-1,2}_{\frac{1}{w}})^{n_{\textnormal{dim}}}$ to $\R^{n_{\textnormal{dim}}}$. Hence,
    \begin{align*}
         \langle Y_{t},1\rangle &= \langle e^{-xt}Y_{0},1\rangle +\int_{0}^{t} \langle e^{-x(t-s)}\nudrift(x),1\rangle \drift(s,X_{s})\ds +\int_{0}^{t}\langle e^{-x(t-s)}\nudiffusion(x),1\rangle \diffusion(s,X_{s})\dW_{s}\\
         &=x_{0}+\int_{0}^{t} k_{\drift}(t-s)\drift(s,X_{s})\ds +\int_{0}^{t}k_{\diffusion}(t-s)\diffusion(s,X_{s})\dW_{s}=X_{t},
    \end{align*}
    a.s. for a.e. $t> 0$. Therefore $Y$ is a mild solution to \eqref{eqn:SEE_eqn_weak_equiv} and $X_{t}=\langle Y_{t},1\rangle$ a.s. for a.e. $t> 0$.
    \item  We now prove the converse part. Assume $\mu$ is a mild solution of \eqref{eqn:SEE_eqn_weak_equiv} and define the process $X_{t}$ by
    \begin{align*}
        X_{t}\coloneqq \langle \mu_{t},1\rangle &=\langle e^{-xt}\mu_{0},1\rangle +\int_{0}^{t} \langle e^{-x(t-s)}\nudrift(x),1\rangle \drift(s,\langle \mu_{s},1\rangle )\ds \\
        &\phantom{xx}+\int_{0}^{t}\langle e^{-x(t-s)}\nudiffusion(x),1\rangle \diffusion(s,\langle \mu_{s},1\rangle)\dW_{s}.
    \end{align*}
    $\mu$ is a $(W^{-1,2}_{\frac{1}{w}})^{n_{\textnormal{dim}}}$-valued process and the map $f\mapsto\langle f,1\rangle $ from $W^{-1,2}_{\frac{1}{w}}$ to $\R$ is continuous and hence Borel; applied componentwise it is continuous (hence Borel) from $(W^{-1,2}_{\frac{1}{w}})^{n_{\textnormal{dim}}}$ to $\R^{n_{\textnormal{dim}}}$. Using assumptions and the properties of the measures $\nudrift, \nudiffusion$, it can be shown that $t\mapsto \langle \mu_{t},1\rangle$ is continuous. This implies that $X$ is a predictable $\R^{n_{\textnormal{dim}}}$-valued process. Applying the map $f\mapsto\langle f,1\rangle $ to the mild formulation
       \begin{align*}
       \mu_{t} =e^{-xt}\mu_{0} +\int_{0}^{t} e^{-x(t-s)}\nudrift(x) \drift(s,\langle \mu_{s},1\rangle )\ds +\int_{0}^{t} e^{-x(t-s)}\nudiffusion(x) \diffusion(s,\langle \mu_{s},1\rangle)\dW_{s},
    \end{align*}
we see that
    \begin{align*}
        X_t & =\left\langle e^{-\cdot t} \mu_{0}(\cdot)+\int_0^t e^{-\cdot(t-s)}\nudrift(\cdot) \drift\left(s,\left\langle\mu_{s},1\right\rangle\right) \ds+\int_0^t e^{-\cdot(t-s)} \nudiffusion(\cdot)\diffusion\left(s,\left\langle\mu_{s},1\right\rangle\right)\dW_{s},1\right\rangle \\
    & =x_{0}+\int_0^t k_{\drift}(t-s) \drift\left(s,\left\langle\mu_{s},1\right\rangle\right) \ds+\int_0^t k_{\diffusion}(t-s) \diffusion\left(s,\left\langle\mu_{s},1\right\rangle\right)\dW_{s}\\
    & =x_{0}+\int_0^t k_{\drift}(t-s) \drift\left(s,X_s\right) \ds+\int_0^t k_{\diffusion}(t-s) \diffusion\left(s,X_s\right) \dW_s,
    \end{align*}
a.s. for a.e. $t>0$. Therefore, $X$ is a solution of the SVE \eqref{eqn:SVE_equiv} with $X_{0}=x_{0}$. Recalling that $Y$ denotes the process defined in \eqref{eqn:Y_process} (with $Y_{0}=x_{0}\delta_{0}$), we have
\begin{align*}
Y_t=e^{-\cdot t} Y_{0}+\int_0^t e^{-\cdot(t-s)} \nudrift\drift\left(s,X_s\right) \ds+\int_0^t e^{-\cdot(t-s)} \nudiffusion\diffusion\left(s,X_s\right) \dW_s    
\end{align*}
a.s. for any $t \geq 0$. This completes the proof.
\end{enumerate}  
\end{proof}

\section{Invariant Measures and Long-Time Behavior}\label{section:Invariant_measure}

 In this section, we assume for simplicity that our solution $\mu$ is unique in law. We refer to the introduction where we mentioned several uniqueness results which, due to Theorem \ref{thm:equivalence_SVE_SEE}, can be transferred from the SVE to the SEE. This assumption could be avoided by proving a Markov-selection result like in \cite{goldys2009martingale} and carefully performing the following arguments for the corresponding selection. 

\begin{assumption}\leavevmode
\begin{enumerate}[label={(UL)}]
    \item \label{A:uniqueness_in_law}  We assume that the (mild) solution to \eqref{eqn:SEE_strong_formulation} is unique in law.
\end{enumerate}
   \begin{enumerate}[label={(UC)}]
    \item \label{A:homogeneous_coefficients}  We assume that the coefficients $\drift,\diffusion$ only depend on $x$.
\end{enumerate}
\end{assumption}
We Assume \ref{A:homogeneous_coefficients} for the remainder of this chapter, which implies that the constants from Assumption \ref{A:assumption_general_linear_growth} can be chosen uniformly in $T$.
\begin{proposition}
    Let $\mu$ be a mild solution of equation \eqref{eqn:SEE_strong_formulation}, with initial condition $\mu_{0}$, which is unique in law (Assumption \ref{A:uniqueness_in_law}). Then the family $\{\mu_{t}(\mu_{0})\}_{t\geq 0,\mu_{0}\in \Wplusdualn}$ is a time-homogeneous Markov process and in particular $P_{t+s}=P_{t}P_{s}$.
\end{proposition}
\begin{proof}
     \begin{align*}
         \mu_{t_{0}+t}&=e^{-(t+t_{0})x}\mu_{0}+\int_{0}^{t+t_{0}}e^{-(t+t_{0}-s)x}\nudrift\drift(\langle \mu_{s},\psi\rangle)\ds+\int_{0}^{t+t_{0}}e^{-(t+t_{0}-s)x}\nudiffusion\diffusion(\langle \mu_{s},\psi\rangle)\dW_{s}\\
         &=e^{-tx}\left(e^{-t_{0}x}\mu_{0}+\int_{0}^{t_{0}}e^{-(t_{0}-s)x}\nudrift\drift(\langle \mu_{s},\psi\rangle)\ds+\int_{0}^{t_{0}}e^{-(t_{0}-s)x}\nudiffusion\diffusion(\langle \mu_{s},\psi\rangle)\dW_{s}\right)\\
         &\phantom{xx}+\int_{t_{0}}^{t+t_{0}}e^{-(t+t_{0}-s)x}\nudrift\drift(\langle \mu_{s},\psi\rangle)\ds+\int_{t_{0}}^{t+t_{0}}e^{-(t+t_{0}-s)x}\nudiffusion\diffusion(\langle \mu_{s},\psi\rangle)\dW_{s}\\
         &=e^{-tx}\mu_{t_{0}}(\mu_{0})+ \int_{t_{0}}^{t+t_{0}}e^{-(t+t_{0}-s)x}\nudrift\drift(\langle \mu_{s},\psi\rangle)\ds+\int_{t_{0}}^{t+t_{0}}e^{-(t+t_{0}-s)x}\nudiffusion\diffusion(\langle \mu_{s},\psi\rangle)\dW_{s}\\
         &=e^{-tx}\mu_{t_{0}}(\mu_{0})+ \int_{0}^{t}e^{-(t-s)x}\nudrift\drift(\langle \mu_{s+t_{0}},\psi\rangle)\ds+\int_{0}^{t}e^{-(t-s)x}\nudiffusion\diffusion(\langle \mu_{s+t_{0}},\psi\rangle)\dW_{s}^{t_{0}}.
     \end{align*}
     By uniqueness in law, $P_{t}\varphi(\mu_{0})=\E\left[\varphi(\mu_{t_{0}+t}(\mu_{0}))\right]=\E\left[\varphi(\mu_{t}(\mu_{t_{0}}(\mu_{0})))\right]$. In particular,
     \begin{align*}
         \E\left[\varphi(\mu_{t_{0}+t}(\mu_{0}))\mid \mathcal{F}_{t_{0}}\right]=\E\left[\varphi(\mu_{t}(\mu_{t_{0}}(\mu_{0})))\mid \mathcal{F}_{t_{0}}\right]=P_{t}\varphi(\mu_{t_{0}}(\mu_{0}))).
     \end{align*}
\end{proof}
\begin{remark}
    Recall that $\Wplusdualn$ was defined as the dual space of $\Wplus$ with respect to a weighted $L^{2}$ duality.
\end{remark}

\subsection{Feller Property}\label{sec:Feller}
This section aims to study the weak and generalized Feller properties of solutions to \eqref{def:mild_solution_SEE}. As a first step, we want to establish a weak sequential Feller property for the solution \eqref{eqn:SEE_strong_formulation}. Let the initial condition of equation \eqref{eqn:SEE_strong_formulation} be denoted by $\mu_{0}\in \Wplusdualn$ and let the space of bounded, Borel measurable functions from $\Wplusdualn$ to $\R$ be denoted by $\mathcal{B}_{b}(\Wplusdualn,\R)$. For every $\Phi \in \mathcal{B}_{b}(\Wplusdualn,\R)$, we define
\begin{align*}
    P_{t}\Phi(\mu_{0})\coloneqq \E\left[\Phi(\mu_{t}(\mu_{0}))\right].
\end{align*}
By continuity of the trajectories of $\mu$, $P_{t}$ forms a stochastically continuous semigroup on $\Wplusdualn$, i.e.
\begin{align*}
    \lim_{t\rightarrow 0}P_{t}\Phi(\mu_{0})=\Phi(\mu_{0})
\end{align*}
for every $\Phi\in C_{b}(\Wplusdualn)$.
\begin{proposition}\label{prop:weak_feller_property}
   Let $\Phi\colon \Wplusdualn\rightarrow \R$ be a bounded and sequentially weakly continuous
function and $t > 0$. Then $P_{t}\Phi\colon \Wplusdualn\rightarrow \R$  is also a bounded sequentially weakly continuous function. In particular, if $\mu_{0,n}\rightarrow \mu_{0}$ in $\Wplusdualn$, then  for any $t\geq 0$, $P_{t}\Phi(\mu_{0,n})\rightarrow P_{t}\Phi(\mu_{0})$, as $n\rightarrow \infty$.
\end{proposition}
\begin{remark}
    Referring to the work \cite{maslowski1999sequentially}, and writing $\mathcal{S}_{b}$ for the set of all real-valued, bounded, weakly sequentially continuous functions on $\Wplusdualn$ (equipped with the $\operatorname{weak}^{*}$ topology), the previous Proposition verifies that
    \begin{align*}
        P_{t}\big(\mathcal{S}_{b}\big)\subseteq \mathcal{S}_{b}.
    \end{align*}
    This property is also referred to as the sequentially weak Feller property.
\end{remark}
\begin{proof}Let $t>0$. We start on the filtered probability space $(\Omega,\mathcal{F},\mathbb{F},\Prob)$. Given a sequence $(\mu_{0,n})_{n\in \N}\subseteq \Wplusdualn$, converging weakly to $\mu_{0}\in \Wplusdualn$, we need to verify that
    $P_{t}\Phi(\mu_{0,n})\rightarrow P_{t}\Phi(\mu_{0})$, where $P_{t}\Phi(\mu_{0})=\int_{\Wplusdualn}\Phi(\mu_{t}(\mu_{0}))\d \Prob ^{\mu_{0}}$. By Lemma \ref{lem:a-priori_estimate_general}, $P_{t}\Phi$ is bounded from $\Wplusdualn$ to $\R$.
     Theorem \ref{thm:weak_mild_solution_SEE} yields the existence of a solution $(\Omega^{\mu_{0,n}},\mathcal{F}^{\mu_{0,n}},\Fil^{\mu_{0,n}},\Prob^{\mu_{0,n}},\mu^{\mu_{0,n}},W^{\mu_{0,n}})$ to equation \eqref{eqn:SEE_strong_formulation} for each $\mu_{0,n}$. Hence $P_{t}\Phi(\mu_{0,n})$ is well defined. Let $\rho^{\mu_{0,n}}$ denote the joint law of $(\mu^{\mu_{0,n}},W^{\mu_{0,n}})$. We already know by Theorem \ref{thm:weak_mild_solution_SEE} that  $(\Omega^{\mu_{0}},\mathcal{F}^{\mu_{0}},\Fil^{\mu_{0,n}},\Prob^{\mu_{0}},\mu^{\mu_{0}},W^{\mu_{0}})$ to equation \eqref{eqn:SEE_strong_formulation} for $\mu_{0}$. By Lemma \ref{lem:a-priori_estimate_general} and \ref{lem:time_regularity} (and since we assumed that $(\mu_{0,n})_{n\in \N}\subseteq \Wplusdualn$ was convergent), we conclude that the laws $(\rho^{\mu_{0,n}})_{n}$ are tight on $C\left([0,T];\left(\Wplusdualn\right)^{\operatorname{weak}-*}\right)\cap C\left([0,T];\Wmidtwodualn\right)$.

     We set $\mu^{n}_{\cdot}=\mu_{\cdot}(\mu_{0,n})$. The Skorohod representation theorem (see \cite{Jakubowski98_Skorohod}) yields the existence of a subsequence $n_{k}$ which will not be relabelled, a new stochastic basis $(\tilde{\Omega},\tilde{\mathcal{F}},\tilde{\mathbb{F}},\tilde{\Prob})$, where $\tilde{\mathbb{F}}=(\tilde{\mathcal{F}}_{s})_{s\in[0,T]}$ and $\mathbb{F}$-progressively measurable process $\tilde{\mu}$, $(\tilde{\mu}^{n})_{n}$ with laws supported on $\Wmidtwodualn\cap \Wplusdualn$ and a new Wiener process $\tilde{W}$, such that $\tilde{\mu}_{n}$ has the same law as $\mu^{n}$ and $\tilde{\mu}^{n}\rightarrow \tilde{\mu}$ on $C([0,T],\Wmidtwodualn) \cap C\left([0,T],\left(\Wplusdualn\right)^{\operatorname{weak}^{*}}\right)$ $\tilde{\Prob}$-a.s. The system $(\tilde{\Omega},\tilde{\mathcal{F}},\tilde{\mathbb{F}},\tilde{\Prob},\tilde{W},\tilde{\mu})$ is a probabilistically weak solution of \eqref{eqn:SEE_strong_formulation}. In particular $\tilde{\mu}_{t}^{n}\rightharpoonup \tilde{\mu}_{t}$ (weakly-$*$) in $\Wplusdualn$. Since $\Phi$ was chosen as an element of the sequentially continuous bounded functions on $\Wplusdualn$, $\tilde{\Prob}$-a.s. $\Phi(\tilde{\mu}^{n})\rightarrow \Phi(\tilde{\mu})$ in $\R$. The boundedness of $\mu$ allows us to use Lebesgue's dominated convergence theorem to conclude that
\begin{align*}
    \lim_{n\rightarrow \infty}\tilde{\E}[\Phi(\tilde{\mu}_{t}^{n})]=\tilde{\E}[\Phi(\tilde{\mu}_{t})].
\end{align*}
By the equality of laws of $\tilde{\mu}^{n}$ and $\mu^{n}$, we obtain
\begin{align*}
    \tilde{\E}[\Phi(\tilde{\mu}_{t}^{n})]=\E^{\mu_{0,n}}[\Phi(\mu_{t}^{n})]=P_{t}\Phi(\mu_{0,n}).
\end{align*}
By Assumption \ref{A:uniqueness_in_law}, we have that the solution to \eqref{eqn:SEE_strong_formulation} is unique in law. Hence, it must hold that
\begin{align*}
     \tilde{\E}[\Phi(\tilde{\mu}_{t})]=\E[\Phi(\mu_{t})]=P_{t}\Phi(\mu_{0}).
\end{align*}
This yields
\begin{align*}
     \lim_{n\rightarrow \infty} P_{t}\Phi(\mu_{0,n}) = \lim_{n\rightarrow \infty}\tilde{\E}[\Phi(\tilde{\mu}_{t}^{n})]=\tilde{\E}[\Phi(\tilde{\mu}_{t})]=\E[\Phi(\mu_{t})]=P_{t}\Phi(\mu_{0}).
\end{align*}
\end{proof}
\subsubsection{Connection to generalized Feller theroy}
\begin{definition}
    For a completely regular Hausdorff space $X$, we call a function $\rho:X\rightarrow (0,\infty)$ an admissable f-weight function, if the sets $\mathcal{K}_{R}:=\{x\in X | \rho(x)\leq R\}$ are compact for all $R>0$.
\end{definition}
We introduced the notation "f-weight function" in order not to overlap with the notation, we used in the definition of the weighted Sobolev spaces.
For a Banach space $Z$, we denote by 
\begin{align*}
    B^{\rho}(X,Z):=\left\{f:X\rightarrow Z | \sup_{x\in X}\frac{\|f(x)\|_{Z}}{\rho(x)}<\infty\right\}.
\end{align*}
The norm will be denoted by $\|\cdot\|_{\rho,\infty}$.
\begin{definition}
    We define the space $\mathcal{B}^{\rho}(X,Z)$ as the closure of $C_{b}(X,Z)$ in $B^{\rho}(X,Z)$.
\end{definition}
\begin{corollary}\label{cor:gen_feller}
    Let $\varrho(x)\coloneqq 1+\|x\|_{\Wplusdualn}$, then $P_{t}$ as defined above, is a generalized Feller semigroup.
\end{corollary}
\begin{proof}
    Since we satisfy the necessary assumptions, \cite[Theorem 5.3]{dorsek2010semigroup} implies that the weak(-$*$) Feller property of Proposition \ref{prop:weak_feller_property} implies the generalized Feller property. $P_{t}$ is therefore a strongly continuous semigroup on $\mathcal{B}^{\varrho}\left(\left(\Wplusdualn\right)^{\operatorname{weak}^{*}}\right)$.
\end{proof}

\subsection{Invariant Measures for SEE}
To make use of the weak or generalized Feller property, we will derive additional estimates.

\begin{assumption}\leavevmode
\begin{enumerate}[label={(LT)}]
    \item \label{A:long_term_behaviour} Let $\nu^{i}$, $i\in \{\drift,\diffusion\}$, be such that for any $0\leq s < t$,
    \begin{align}
        \|S_{t-s}^{*}\nu^{i}\|_{\opn}\leq C F_{i}(t-s),
    \end{align}
    with kernels $F_{i}>0$, $F_{\drift}\in C(\R_{+})\cap L^{1}(\R_{+})$ and $F_{\diffusion}\in C(\R_{+})\cap L^{2}(\R_{+})$.
\end{enumerate}
\end{assumption}
\begin{lemma}
      Let $\mu$ be a mild solution of equation \eqref{eqn:SEE_strong_formulation}, let Assumption \ref{A:long_term_behaviour} be satisfied and assume that $\sup_{T>0}\E \sup_{t\in [0,T]} \|S_{t}^{*}\mu_{0}\|_{\Wplusdualn}<\infty$. Then
    \begin{align}\label{eqn:long_time_estimate}
       \sup_{t\in [0,\infty)}  \E \|\mu_{t}\|_{\Wplusdualn}\leq \sup_{T>0}\E \sup_{t\in [0,T]} \|\mu_{t}\|_{\Wplusdualn}<\infty.
    \end{align}
\end{lemma}
\begin{proof}
     Let $u,A,G,H$ be non negative functions on $\R_{+}$. Since
\begin{align*}
    u(t)\leq A(t) + G(t) + H(t) \leq A(t) + 2\max\{G(t),H(t)\},
\end{align*}
it suffices to estimate $u(t)\leq A(t) + 2 G(t)$ and $u(t)\leq A(t) + 2 H(t)$ separately. Hence for $i\in\{\drift,\diffusion\}$, $j=1,2$, we consider 
  \begin{align}
   \E  \sup_{t\leq T}\|\mu_{t}\|_{\Wplusdualn}^{p} 
     &\leq C\E \sup_{t\leq T}\left\|S^{*}_{t}\mu_{0}\right\|^{p}_{\Wplusdualn}\nonumber\\
     &\phantom{xx}+ C h_{i,j}(T)\int_{0}^{T}F_{i}(T-s)\left(\E\sup_{r\leq s}\|\mu_{r}\|_{\Wplusdualn}^{2}\right)^{\widetilde{\gamma}_{i,j}}\ds,\label{eqn:sup_estimate_term_1},
\end{align}
where, depending on $p$, 
\begin{align*}
    \widetilde{\gamma}_{\drift,1} = \widetilde{\gamma}_{\drift,2}=\gamma_{\drift},\qquad \widetilde{\gamma}_{\diffusion,1} = \gamma_{\diffusion},\qquad \widetilde{\gamma}_{\diffusion,2} = 2\gamma_{\diffusion}-1,
\end{align*}
\begin{align*}
    &h_{\drift,1}(T)=h_{\drift,2}(T)= \left(\int_{0}^{T}\|S^{*}_{T-s}\nudrift\|_{\opn}\ds\right)^{p-1}, \\
    &\qquad h_{\diffusion,1}(T)=\left(\int_{0}^{T}\|S^{*}_{T-s}\nudiffusion\|_{\opn}\ds\right)^{\frac{p}{2}-1}\\
    &h_{\diffusion,2}(T)=1.
\end{align*}
Since $\widetilde{\gamma}_{i,j}\leq 1$, we can use Young's inequality to obtain
 \begin{align*}
   \E  \sup_{t\leq T}\|\mu_{t}\|_{\Wplusdualn}^{p} 
     &\leq C\E \sup_{t\leq T}\left\|S^{*}_{t}\mu_{0}\right\|^{p}_{\Wplusdualn}\\
     &\phantom{xx}+ C h_{i,j}(T)\int_{0}^{T}F_{i}(T-s)\left(1+\E\sup_{r\leq s}\|\mu_{r}\|_{\Wplusdualn}^{p}\right)\ds \\
     &\leq  C\left(\E \sup_{t\leq T}\left\|S^{*}_{t}\mu_{0}\right\|^{p}_{\Wplusdualn} +h_{i,j}(T)\int_{0}^{T}F_{i}(T-s)\ds\right)\\
     &\phantom{xx}+ C h_{i,j}(T)\int_{0}^{T}F_{i}(T-s)\E\sup_{r\leq s}\|\mu_{r}\|_{\Wplusdualn}^{p}\ds.
\end{align*}
Now, we use \cite[Lemma 2.2]{zhang2010stochastic}, which allows us to bound the previous terms by the corresponding resolvents (of the second kind) of $F_{i}$, namely
\begin{align*}
    f(t)\leq a(t)+\int_{0}^{t}F(t-s)f(s)\ds,
\end{align*}
implies 
\begin{align*}
    f(t)\leq a(t)+\int_{0}^{t}R_{F}(t-s)a(s)\ds.
\end{align*}
We will also use that \cite[Theorem 1]{gripenberg1980resolvents} implies that if the kernel $F\in L^{1}(\R_{+})$, then its resolvent $R_{F}\in L^{1}(\R_{+})$.
By Assumption \ref{A:long_term_behaviour} $h_{i,j}\leq C$ and we obtain
 \begin{align*}
   \E  \sup_{t\leq T}\|\mu_{t}\|_{\Wplusdualn}^{2} 
     &\leq  C\left(1+\E \sup_{t\leq T}\left\|S^{*}_{t}\mu_{0}\right\|^{p}_{\Wplusdualn}\right)\\
     &\phantom{xx}+ C\int_{0}^{T}R_{F_{i}}(T-s)\left(1+\E \sup_{t\leq T}\left\|S^{*}_{t}\mu_{0}\right\|^{p}_{\Wplusdualn}\right)\ds.
\end{align*}
Assumption \ref{A:long_term_behaviour} and \cite[Theorem 1]{gripenberg1980resolvents} now yield that
 \begin{align*}
   \sup_{T}\E  \sup_{t\leq T}\|\mu_{t}\|_{\Wplusdualn}^{p} 
     &<\infty.
\end{align*}
\end{proof}
Let us verify \ref{A:long_term_behaviour} for two examples. Recall that in the Example \ref{example:kernels}, the kernel obtained in the estimates for the Gamma-kernel was again a Gamma-kernel. The same property can be derived for the shifted fractional kernel. Hence, we will discuss the examples in this section in terms of the kernels $F_{i}$, appearing in Assumption \ref{A:long_term_behaviour}.
\begin{itemize}
    \item The gamma-kernel:Let $\delta >0$, $\beta\in \left(0,\frac{1}{2}\right)$ and $F_{i}(t)=e^{-t\delta}t^{\beta-1}$.
   Note that $\int_{0}^{t}e^{-(t-s)\delta}(t-s)^{\beta-1}\ds = \delta^{-\beta}\left(\Gamma(\beta)-\Gamma(\beta,\delta t)\right)$, which can be bounded by a constant. The resolvent (of the second kind) is given by $R(t)=e^{-\delta t}t^{\beta-1}E_{\beta,\beta}(-t^{\beta})$, where $E_{\alpha,\beta}(z)=\sum_{n=0}^{\infty}\frac{z^{n}}{\Gamma(\alpha n+\beta)}$ denotes the Mittag-Leffler function. It can be easily verified that $\sup_{T}\int_{0}^{T}R(s)\ds <\infty$.
   \item For fixed $a,\eps>0$, let $F_{i}(t)=\frac{1}{(t+\eps)^{1+\alpha}}$. This kernel is clearly in $L^{1}(\R_{+})$. Note that kernels of this form are completely monotone and fit in our previous analysis.
\end{itemize}
\begin{theorem}\label{thm:invariant_measure_mu}
    Let $\mu$ be a solution of \eqref{eqn:SEE_strong_formulation} and let Assumptions \ref{A:long_term_behaviour} hold. Then there exists at least one measure $\mathcal{Q}$ on $\Wplusdualn$ , such that $P_{t}^{*}\mathcal{Q}=\mathcal{Q}$.
\end{theorem}

\begin{proof}
    The statement follows directly from \cite[Proposition 3.1]{maslowski1999sequentially} and estimate \eqref{eqn:long_time_estimate}.
\end{proof}
\begin{proof}[Proof of Theorem \ref{thm:weak_gen_feller}]
    The theorem follows from Theorem \ref{thm:invariant_measure_mu}, Proposition \ref{prop:weak_feller_property} and Corollary \ref{cor:gen_feller} 
\end{proof}

\begin{remark}[Comparison with Hamaguchi's generalized Harris approach, and a shortcoming of his choice of spaces]\label{rem:hamaguchi_comparison}
Theorem~\ref{thm:invariant_measure_mu} and its transfer to the SVE
(Theorem~\ref{thm:invariant_measure_SVE}) establish existence of an
invariant/stationary measure only, via a Krylov--Bogoliubov-type
tightness argument resting on the generalized Feller property and the
compact embedding of Proposition~\ref{prop:embeddings}. Uniqueness,
mixing, and a convergence rate are established below in
Section~\ref{subsec:harris_ergodicity}, for the sub-class of SEEs with
uniformly elliptic diffusion coefficient, by adapting the generalized
Harris theorem of Hairer--Mattingly--Scheutzow~\cite{HaMaSc11} as
recently applied to Markovian lifts of SVEs by
Hamaguchi~\cite{hamaguchi2026exponential}.

The two frameworks are close enough to be directly comparable: his
``lifting basis'' $(\mu,M_{\drift},M_{\diffusion})$, once specialized
to our two-kernel setting via
$\mu:=\nudrift+\nudiffusion$, $M_{\drift}:=\d\nudrift/\d\mu$,
$M_{\diffusion}:=\d\nudiffusion/\d\mu$, generates exactly our SEE
\eqref{eqn:SEE_strong_formulation}. This is worth making explicit
because it also exposes a genuine structural difference between the
two choices of state space, one that motivated our own use of weighted
\emph{Sobolev} (rather than merely weighted $L^{2}$) spaces from the
outset. Hamaguchi's own state space is built directly as a
weighted-$L^{2}(\mu)$-type Hilbert space $\mathcal{H}$ (our analogue of
$\Wmiddualn$) with a finite-second-moment subspace $\mathcal{V}\subset\mathcal{H}$ (our
analogue of $\Wplusdualn$), with \emph{no} additional differentiability
imposed. He proves (his own Lemma on the compactness of
$\mathcal{V}\hookrightarrow\mathcal{H}$) that this embedding is compact \emph{if and
only if} the reference measure $\mu$ has finite support in every
bounded interval, equivalently, only for sum-of-exponentials
kernels. For any kernel with a genuinely continuous Laplace measure,
including the tempered fractional kernels that are the running example
of this paper, his embedding is \emph{never} compact, and he notes
explicitly that classical existence results resting on compactness
(Krylov--Bogoliubov, or the general monotone-SPDE ultimate-boundedness
theory) therefore do not apply in his framework for exactly the
kernels of main interest. This is why he is compelled to abandon
compactness-based arguments entirely and pass to the generalized
Harris/coupling machinery, which needs none.

Our own compact embedding (Proposition~\ref{prop:embeddings}) is a
different, and stronger, mechanism: a genuine Rellich--Kondrachov-type
argument comparing \emph{two different orders of Sobolev
differentiability} with a decaying weight ratio, rather than a single
flat $L^{2}$ structure. Polynomial weight ratios of the type used in
Section~\ref{section:Functional_framework} can decay to zero even when
the underlying measure is continuous, so our embedding is not subject
to Hamaguchi's obstruction on the same terms. We record this as a
genuine structural advantage of working in weighted Sobolev spaces
rather than weighted $L^{2}$ spaces alone, but we do not want to
overstate it: verifying the hypotheses of
Proposition~\ref{prop:embeddings} for a specific singular kernel
remains a case-by-case computation (as already flagged in the
Introduction), and we have not carried this out in full generality for
every example in this paper. For this reason, and because it delivers
a strictly stronger conclusion regardless (uniqueness and a
convergence rate, not merely existence), we adopt the Harris/coupling
route below as the primary tool for ergodicity, treating the
compactness-based existence argument above as an independent,
complementary result that does not itself depend on the Harris
machinery.

A second structural difference is, for applications, arguably more
serious than the compactness question, since it concerns which
stochastic Volterra equations can be represented at all rather than
merely how their long-time behavior is proved. In
\cite{hamaguchi2023markovian,hamaguchi2026exponential}, the forcing
term of the SVE is not free: it is generated by the lift's own initial
state $Y_{0}\in\mathcal H$ via
$x(t)=\mu[\mathcal S(t)Y_{0}]=\int_{[0,\infty)}e^{-\theta t}Y_{0}(\theta)\,\mu(\mathrm d\theta)$
(see the well-posedness proposition of \cite{hamaguchi2026exponential},
part (iii), and the analogous identity in \cite{hamaguchi2023markovian}).
Since constant vectors are genuine elements of $\mathcal H$, taking
$Y_{0}\equiv x_{0}$ is perfectly legitimate in his framework, no Dirac mass is needed there, but it does \emph{not} produce the
persistent constant forcing $x(t)\equiv x_{0}$: it produces
$x(t)=x_{0}K(t)$, which vanishes as $t\to\infty$ for any kernel whose
Laplace measure $\mu$ has no atom at the origin, by dominated
convergence ($e^{-\theta t}\to0$ for every $\theta>0$). Recovering a
genuinely persistent constant forcing this way requires $\mu(\{0\})>0$,
and this is precisely excluded by the tempering condition
$\kappa:=\inf\operatorname{supp}\mu>0$ that his own coupling and
Lyapunov assumptions require for the ergodicity theorem to apply: if
$0\notin\operatorname{supp}\mu$, it certainly cannot carry an atom
there. So within the exact scope of his ergodicity results, the
standard convention used throughout the present paper, and in
\cite{abi2019affine} and the wider rough-volatility literature,
of starting a stochastic Volterra equation from a given, persistent
constant $X_{0}=x_{0}$ is unreachable; his construction instead
describes SVEs started from a transient forcing that decays away, a
different (and for finite-time applications such as option pricing
from a given spot value, inapplicable) object. This is the same
underlying cause as the compactness failure above, seen from the
opposite side: a state space of genuine \emph{functions} $\mathcal H$
cannot contain the point mass $\delta_{0}$ needed to represent a
persistent constant, for exactly the same reason it cannot be
compactly embedded in a smaller function space. Our own state space
$\Wplusdualn$ is, by construction, a space of \emph{distributions}: the
standard initial condition $X_{0}=x_{0}$ lifts to
$\mu_{0}=\delta_{0}x_{0}\in\Wplusdualn$ directly, and
$S^{*}_{t}\delta_{0}=\delta_{0}$ identically reproduces $X_{t}\equiv
x_{0}$ in the absence of forcing, for every kernel considered in this
paper, tempered or not (this is also precisely the fact that made the
Lyapunov function of Theorem~\ref{thm:lyapunov_ergodicity} delicate to
construct in Section~\ref{subsec:harris_ergodicity} above: the same
feature that is a liability for proving decay is what makes the
standard initial-value problem representable in the first place).
\end{remark}

\begin{remark}
    We could have obtained the result via the strategy used in \cite{jacquier2022large_time}. Let $\left(P_t\right)_{t \geq 0}$ be the generalised Feller semigroup associated to $\left(\mu_t\right)_{t \geq 0}$. Then, for any $(\Wplusdualn)^{\operatorname{weak}^{*}}$-valued random variable $\mu_0\sim \eta $, $ P_t \varrho\left(\mu_0\right)=\E_{\mu_0}\left[\varrho\left(\mu_t\right)\right]$.

We can use the weight $\varrho(\mu)=1+\|\mu\|_{\Wplusdualn}^{p}$, which is an admissible weight according to the Definition  \cite[Definition 2.1]{dorsek2010semigroup}. 
\begin{align*}
    \sup _{t \geq 0} \int_{\Wplusdualn} \E_{\mu_0}\left[\varrho\left(\mu_t\right)\right] \d \eta\left(\mu_0\right)=\sup _{t \geq 0} \E\left[\varrho\left(\mu_t\right)\right]<\infty ,
\end{align*}
is enough for an application of \cite[Lemma 3.1]{jacquier2022large_time}.
\end{remark}
\subsection{Transfer to SVE}
\begin{proof}[Proof of Theorem \ref{thm:invariant_measure_SVE}]
    Theorem \ref{thm:invariant_measure_mu} implies the existence of a probability measure $\mathcal{Q}$ on $\Wplusdualn$ such that, if $\mu_{0}=\delta_{0}x_{0}\sim \mathcal{Q}$, for any $t\geq 0$, $\mathcal{Q}$ is the Law of $\mu_{t}$. $(\langle\mu_{t},1\rangle)_{t\geq s}$ is distributed according to the push-forward measure of $\mathcal{Q}$ under the map $\mu\mapsto \langle \mu,1\rangle$ and also strictly stationary.
\end{proof}

\subsection{Uniqueness and exponential ergodicity via a generalized Harris
theorem}\label{subsec:harris_ergodicity}

Theorem~\ref{thm:invariant_measure_mu} establishes existence of an
invariant measure but not uniqueness, mixing, or a convergence rate. In
this section we close this gap for the sub-class of SEEs with
\emph{uniformly elliptic} diffusion coefficient. The abstract engine is
the generalized Harris theorem of Hairer, Mattingly and Scheutzow
\cite{HaMaSc11}, recently used to the same end for a different
Markovian lift of the SVE by Hamaguchi~\cite{hamaguchi2026exponential}.
What we build below is carried out entirely within the state space
$\Wplusdual$ already developed in
Section~\ref{section:Functional_framework}, using only its structure
as a genuine Hilbert space: the ``change of norm'' needed to run the
argument is an equivalent renorming of $\Wplusdual$ itself (in the
sense of Lemma~\ref{lem:weighted_norm_equivalence}), constructed from
the semigroup $S^{*}$ via the standard device recalled below, rather
than a move within the weight family of
Section~\ref{section:Functional_framework}.

\subsubsection{The state space as a Hilbert space; the generator is
(almost) dissipative}

Before writing any coercivity condition, we recall that the state space
$V=(\Wplusdual)^{n_{\textnormal{dim}}}$ (the $n_{\textnormal{dim}}$-fold product introduced
in Sections~\ref{sec:ito_formula}--\ref{subsec:feynman_kac}, in which the lift $\mu$ takes
its values) is a genuine
separable Hilbert space (recall $W^{m,p}_{w}$ is reflexive for
$1<p<\infty$, and finite products of Hilbert spaces are Hilbert) with the product inner product
$\langle\cdot,\cdot\rangle_{V}=\sum_{i=1}^{n_{\textnormal{dim}}}\langle\cdot^{(i)},\cdot^{(i)}\rangle_{\Wplusdual}$,
and $S_{t}^{*}$ is a $C_{0}$-semigroup on it, acting diagonally (componentwise), with a well-defined
generator $(A,D(A))$, $A\mu=-x\mu$ (applied to each component),
\[
D(A):=\{\mu\in V : x\mu\in V\}
\]
(multiplication by $x$ interpreted distributionally on each component, exactly as
already used to make sense of the SEE itself). Since $S^{*}_{t}$ acts by the same
scalar multiplication on all $n_{\textnormal{dim}}$ components, every norm/renorming and
dissipation estimate below that we state for the scalar space $\Wplusdual$ holds on the
product $V$ with the same constants, by summing over components; we therefore carry out
the scalar computations (e.g.\ Lemma~\ref{lem:poly_growth_semigroup}) on $\Wplusdual$ and
read them on $V$ without further comment.

\begin{lemma}[Polynomial growth of $S^{*}$]\label{lem:poly_growth_semigroup}
There is a constant $C_{0}>0$, depending only on the weight family
$(\weightplus)_{i}$, such that
\[
\|S_{t}^{*}\mu\|_{\Wplusdual}\leq C_{0}(1+t)\,\|\mu\|_{\Wplusdual}
\qquad\text{for all }t\geq0,\ \mu\in\Wplusdual.
\]
\end{lemma}

\begin{proof}
By duality it suffices to bound $S_{t}:=S_{t}^{*}$ (multiplication is
self-adjoint) on $\Wplus=W^{1,2}_{\weightplus}$. For $\varphi\in\Wplus$,
$(S_{t}\varphi)'=-tS_{t}\varphi+S_{t}\varphi'$, so
\[
\|S_{t}\varphi\|_{\Wplus}^{2}\leq\|\varphi\|_{L^{2}_{(\weightplus)_{0}}}^{2}
+2\int_{\R_{+}}t^{2}e^{-2tx}\varphi(x)^{2}(\weightplus)_{1}(x)\,\dx
+2\|\varphi'\|_{L^{2}_{(\weightplus)_{1}}}^{2},
\]
using $e^{-2tx}\leq1$ for the first and third terms. Since
$(\weightplus)_{1}(x)=(1+x)^{2}(\weightplus)_{0}(x)$, the middle term
is bounded by
$2\big(\sup_{x\geq0}t^{2}(1+x)^{2}e^{-2tx}\big)\|\varphi\|^{2}_{L^{2}_{(\weightplus)_{0}}}$,
and an elementary calculus exercise (substituting $u=tx$ and
maximizing $(t+u)^{2}e^{-2u}$ over $u\geq0$) shows
$\sup_{x\geq0}t^{2}(1+x)^{2}e^{-2tx}=t^{2}$ for every $t\geq1$, and is
bounded by $1$ for $t\in[0,1]$. Hence
$\|S_{t}\varphi\|_{\Wplus}^{2}\leq(1+2(1+t)^{2})\|\varphi\|_{\Wplus}^{2}$,
which gives the claim with $C_{0}=\sqrt{3}$.
\end{proof}

\begin{remark}
This is a genuine restriction on the semigroup: $S_{t}^{*}$ is
\emph{not} a contraction on $\Wplusdual$, and the growth, while only
polynomial (not exponential), is real. It reflects exactly the fact
that differentiating $e^{-tx}$ produces a factor of $t$, which is felt
by any $\varphi$ concentrated near $x=0$. This is precisely dual to the
persistence of the atom at $x=0$ noted above.
\end{remark}

By the standard theory of $C_{0}$-semigroups (see e.g.\
\cite{engel_nagel_2000_semigroups}), Lemma~\ref{lem:poly_growth_semigroup} implies
that for every $\omega>0$ there is $M_{\omega}\geq1$ with
$\|S_{t}^{*}\|_{\Wplusdual\to\Wplusdual}\leq M_{\omega}e^{\omega t}$ for
all $t\geq0$ (polynomial growth is dominated by any positive
exponential rate), and, after replacing $\|\cdot\|_{\Wplusdual}$ by the
equivalent renormed norm
$\|\mu\|_{V,\omega}:=\sup_{t\geq0}e^{-\omega t}\|S_{t}^{*}\mu\|_{\Wplusdual}$
(for which $\|S_{t}^{*}\|_{\|\cdot\|_{V,\omega}\to\|\cdot\|_{V,\omega}}\leq
e^{\omega t}$ \emph{exactly}), the Lumer--Phillips theorem gives
\begin{align}\label{eqn:generator_dissipativity}
\langle A\mu,\mu\rangle_{V,\omega}\leq\omega\,\|\mu\|_{V,\omega}^{2}\qquad\text{for all }\mu\in D(A).
\end{align}
We fix, once and for all, a sufficiently small $\omega>0$ (to be
absorbed into the constants below) and write $\|\cdot\|_{V}$ for
$\|\cdot\|_{V,\omega}$ and $\langle\cdot,\cdot\rangle_{V}$ for the
corresponding inner product; by
Lemma~\ref{lem:weighted_norm_equivalence}-type equivalence of norms,
every estimate stated in $\|\cdot\|_{V}$ transfers to
$\|\cdot\|_{\Wplusdual}$ up to a fixed multiplicative constant.

\begin{definition}[Dissipation functional]\label{dfn:dissipation_functional}
For $\mu\in D(A)$, set
\[
[\mu]^{2}:=-\langle A\mu,\mu\rangle_{V}+\omega\|\mu\|_{V}^{2}=\langle x\mu,\mu\rangle_{V}+\omega\|\mu\|_{V}^{2}\;\geq0,
\]
the inequality being exactly \eqref{eqn:generator_dissipativity}. The
quadratic form $\mu\mapsto[\mu]^{2}$ is closable (standard fact for
forms associated with dissipative operators, see e.g.\
\cite{engel_nagel_2000_semigroups}) and we write $\mathcal D\supset D(A)$, dense in
$\Wplusdual$, for its (form) domain, on which $[\cdot]^{2}$ is defined
by continuous extension.
\end{definition}

\begin{example}
$\mu_{0}=\delta_{0}x_{0}\in D(A)$ with $A\mu_{0}=-x\delta_{0}x_{0}=0$
(the distributional identity $x\,\delta_{0}=0$), so
$[\delta_{0}x_{0}]^{2}=\omega x_{0}^{2}\|\delta_{0}\|_{V}^{2}$: the
persistent atom contributes only through the harmless renorming
constant $\omega$, not through any genuine dissipation, since this part of the state does not decay at all.
\end{example}

\subsubsection{Coupling the drift and diffusion: mollification}

The matrices of measures $\nudrift,\nudiffusion$ have columns generally only in
$\Wminusdual\setminus\Wplusdual$ (Assumption~\ref{A:A_3:Narrower_assumption_nu}),
so $\langle\nudrift\,a,\mu\rangle_{V}$ is not directly meaningful either.
We use the same device as in the proof of
Lemma~\ref{lem:singular_ito}: for $\delta>0$ set
$\nu^{\drift}_{\delta}:=S_{\delta}^{*}\nudrift,\ \nu^{\diffusion}_{\delta}:=S_{\delta}^{*}\nudiffusion\in\mathcal{L}(\R^{n_{\textnormal{dim}}},V)$
(finite in operator norm by Assumption~\ref{A:A_3:Narrower_assumption_nu} and
estimate~\eqref{eqn:singular_direction_bound}), and consider the
mollified SEE with $\nu^{\drift}_{\delta},\nu^{\diffusion}_{\delta}$ in
place of $\nudrift,\nudiffusion$. For the mollified equation, every
coefficient is a genuine element of $V$ (resp.\ $HS(\R^{m_{W}},V)$), so writing
$X:=\langle\mu,1\rangle\in\R^{n_{\textnormal{dim}}}$, the drift-energy and diffusion-energy quantities
\[
D^{\drift}_{\delta}[\mu]:=\big\langle\nu^{\drift}_{\delta}\drift(X),\mu\big\rangle_{V},\qquad
Q^{\diffusion}_{\delta}[\mu]:=\big\|\nu^{\diffusion}_{\delta}\diffusion(X)\big\|_{HS(\R^{m_{W}},V)}^{2}
=\sum_{k=1}^{m_{W}}\big\|[\nu^{\diffusion}_{\delta}\diffusion(X)]_{\cdot k}\big\|_{V}^{2}
\]
are ordinary (finite) inner products/norms, well defined for every $\mu\in V$
without restriction.

\subsubsection{A Lyapunov function}

\begin{assumption}[Coercivity]\label{A:coercivity_lyapunov}
There exist constants $\eps\in(0,1)$, $\rho>0$, $C_{\mathrm{Lyap}}>0$
such that, for every $\delta_{0}>0$ small enough,
\begin{align}\label{eqn:coercivity}
D^{\drift}_{\delta}[\mu]+\tfrac12 Q^{\diffusion}_{\delta}[\mu]
\leq\eps[\mu]^{2}-\rho\|\mu\|_{V}^{2}+C_{\mathrm{Lyap}}
\end{align}
for all $\mu\in D(A)$, uniformly in $\delta_{0}$, where $D^{\drift}_{\delta}[\mu],Q^{\diffusion}_{\delta}[\mu]$
are the drift- and diffusion-energy quantities of the preceding paragraph and $X:=\langle\mu,1\rangle\in\R^{n_{\textnormal{dim}}}$.
\end{assumption}

\begin{remark}
Linear growth of $\drift,\diffusion$ alone already gives
\eqref{eqn:coercivity} with $\rho=0$. This is exactly the estimate
underlying Lemma~\ref{lem:a-priori_estimate_general} and
\eqref{eqn:long_time_estimate}. Obtaining $\rho>0$ additionally
requires a genuine \emph{dissipativity} (mean-reversion) condition on
$\drift$, e.g.\ $\langle\drift(x),x\rangle\leq-\rho'|x|^{2}+C$ for some $\rho'>0$
(the multivariate analogue of a mean-reverting drift; for $n_{\textnormal{dim}}=1$ this is
$\drift(x)x\leq-\rho'x^{2}+C$), matching the mean-reverting drift $\drift(x)=\kappa(\theta-x)$
already used as the running example throughout
Sections~\ref{subsec:feynman_kac}--\ref{subsec:option_pricing}.
\end{remark}

\begin{theorem}\label{thm:lyapunov_ergodicity}
Under Assumption~\ref{A:coercivity_lyapunov}, $V(\mu):=\|\mu\|_{V}^{2}$
is a Lyapunov function on $\Wplusdual$ for the (generalized) Feller
semigroup $\{P_{t}\}_{t\geq0}$ of Corollary~\ref{cor:gen_feller}: there
are $C_{V},\gamma_{V}>0$ with
\[
P_{t}V(\mu)\leq C_{V}e^{-\gamma_{V}t}V(\mu)+K_{V},\qquad \mu\in\Wplusdual,\ t\geq0,
\]
for $K_{V}:=C_{\mathrm{Lyap}}/\rho$. Consequently, every invariant
probability measure $\pi$ for $\{P_{t}\}_{t\geq0}$ satisfies
\[
\int_{\Wplusdual}\Big(\rho\|\mu\|_{V}^{2}+(1-\eps)[\mu]^{2}\Big)\,\pi(\d\mu)\leq C_{\mathrm{Lyap}}.
\]
\end{theorem}

\begin{proof}
Fix $\delta>0$ and let $\mu_{t}^{\delta}$ solve the mollified
SEE of the previous paragraph. Since $\nu^{\drift}_{\delta},\nu^{\diffusion}_{\delta}\in\Wplusdual$
genuinely, this is a classical semilinear SPDE on the Hilbert space
$\Wplusdual$ with bounded, $C^{2}_{b}$-composed coefficients, to which
the classical Itô formula for the squared norm applies
(\cite[Thm.~4.2.5]{liu2015stochastic}, in the same spirit as the
mollification step of Lemma~\ref{lem:singular_ito}):
\[
\d\|\mu_{t}^{\delta}\|_{V}^{2}=\Big(2\langle A\mu_{t}^{\delta},\mu_{t}^{\delta}\rangle_{V}
+2D^{\drift}_{\delta}[\mu_{t}^{\delta}]+Q^{\diffusion}_{\delta}[\mu_{t}^{\delta}]\Big)\dt
+2\big\langle\mu_{t}^{\delta},\nu^{\diffusion}_{\delta}\diffusion(X_{t}^{\delta})\,\dW_{t}\big\rangle_{V},
\]
where the martingale term is
$2\sum_{k=1}^{m_{W}}\langle\mu_{t}^{\delta},[\nu^{\diffusion}_{\delta}\diffusion(X_{t}^{\delta})]_{\cdot k}\rangle_{V}\,\dW^{k}_{t}$,
and the $Q^{\diffusion}_{\delta}$ term is exactly its quadratic variation (summed over the
$m_{W}$ noise directions).
By Definition~\ref{dfn:dissipation_functional},
$2\langle A\mu,\mu\rangle_{V}=-2[\mu]^{2}+2\omega\|\mu\|_{V}^{2}$; by
Assumption~\ref{A:coercivity_lyapunov} (applied with
$\mu=\mu_{t}^{\delta}$, uniformly in $\delta$), the
finite-variation part is bounded above by
$-2(1-\eps)[\mu_{t}^{\delta}]^{2}-2(\rho-\omega)\|\mu_{t}^{\delta}\|_{V}^{2}+2C_{\mathrm{Lyap}}$;
choosing $\omega<\rho$ (possible since $\omega$ was free to fix small),
taking expectations and applying Gronwall's inequality to
$t\mapsto\E\|\mu_{t}^{\delta}\|_{V}^{2}$ yields
\[
\E\|\mu_{t}^{\delta}\|_{V}^{2}\leq\E\|\mu_{0}\|_{V}^{2}e^{-2(\rho-\omega) t}+\frac{C_{\mathrm{Lyap}}}{\rho-\omega},
\]
uniformly in $\delta$. Passing $\delta\to0$ exactly as in
Step~2 of the proof of Lemma~\ref{lem:singular_ito}
($\mu_{t}^{\delta}\to\mu_{t}$ in $L^{2}(\Omega;\Wplusdual)$, using
$\nu^{i}_{\delta}\to\nu^{i}$, for $i=\drift,\diffusion$, in the sense established there) gives
the same bound for $\mu_{t}$ itself, which is the Lyapunov inequality
with $\gamma_{V}=2(\rho-\omega)$, $C_{V}=1$,
$K_{V}=C_{\mathrm{Lyap}}/(\rho-\omega)$. Integrating the same drift
bound against an invariant measure $\pi$ (the time-derivative term
vanishes by invariance) gives
$\rho\int\|\mu\|_{V}^{2}\pi(\d\mu)+(1-\eps)\int[\mu]^{2}\pi(\d\mu)\leq C_{\mathrm{Lyap}}$
after the same limit, exactly as in the finite-dimensional
integrability argument underlying \cite[Theorem~4.8]{HaMaSc11}
(recalled below as Theorem~\ref{thm:generalized_harris}).
\end{proof}

\subsubsection{The generalized Harris theorem}

We recall, without proof, the abstract theorem underlying this
subsection; the ``Feller property'' in its statement is to be
understood in the \emph{generalized} sense of
Corollary~\ref{cor:gen_feller} and the spaces $\mathcal
B^{\rho}(\cdot,\cdot)$ of Section~\ref{sec:Feller}, which is exactly
the sense in which our semigroup $\{P_{t}\}_{t\geq0}$ possesses it: the
classical (strong) Feller property is not available here, precisely
because of the degeneracy discussed in
Section~\ref{sec:Feller}.

\begin{theorem}[Generalized Harris theorem;
{\cite[Theorem~4.8, Corollary~4.11]{HaMaSc11}}]\label{thm:generalized_harris}
Let $\{P_{t}\}_{t\geq0}$ be a measurable Markov semigroup on a Polish
space $E$ with the (generalized) Feller property, admitting a
continuous Lyapunov function $V:E\to[0,\infty)$, i.e.\ satisfying
$P_{t}V\leq C_{V}e^{-\gamma_{V}t}V+K_{V}$ for some constants
$C_{V},\gamma_{V},K_{V}>0$ and all $t\geq0$. Suppose there is a
distance-like function $d:E\times E\to[0,1]$ with $d_{0}\leq\sqrt d$
for some compatible metric $d_{0}$, and constants $t_{1},t_{2}>0$ such
that $d$ is contracting for $P_{t}$ whenever $t\geq t_{1}$, and the
level set $\{V\leq4K_{V}\}$ is $d$-small for $P_{t}$ whenever $t\geq
t_{2}$. Then $\{P_{t}\}_{t\geq0}$ has a unique invariant probability
measure $\pi$, and there are $r,t_{0}>0$ with
\[
\mathcal{W}_{d_{V}}(P_{t}^{*}\nu_{1},P_{t}^{*}\nu_{2})\leq e^{-rt}\mathcal{W}_{d_{V}}(\nu_{1},\nu_{2}),
\qquad
d_{V}(\mu_{1},\mu_{2}):=\sqrt{d(\mu_{1},\mu_{2})(1+V(\mu_{1})+V(\mu_{2}))},
\]
for all $t\geq t_{0}$, $\nu_{1},\nu_{2}$ probability measures on $E$;
in particular $\mathcal{W}_{d_{V}}(P_{t}(\mu,\cdot),\pi)\leq C(1+V(\mu)^{1/2})e^{-rt}$
for all $\mu\in E$, $t\geq0$.
\end{theorem}

The delicate step, in any application, is constructing $d$: a
\emph{contracting} and \emph{$d$-small} distance-like function. We
construct it as $d_{*}$ for a second, separately tuned equivalent
renorming $\|\cdot\|_{*}$ of $\Wplusdual$, obtained by the same device
as Lemma~\ref{lem:poly_growth_semigroup} but tracking the ellipticity
constant $C_{\mathrm{UE}}$ rather than an arbitrary $\omega$.

\subsubsection{Uniform ellipticity and the coupling construction}

\begin{assumption}[Uniform ellipticity]\label{A:uniform_ellipticity}
$\drift,\diffusion$ are globally Lipschitz, $m_{W}\geq n_{\textnormal{dim}}$, and there is
$C_{\mathrm{UE}}>0$ such that
\[
\diffusion(x)\diffusion(x)^{\top}\succeq C_{\mathrm{UE}}^{-1}\,I_{n_{\textnormal{dim}}}
\qquad\text{for all }x\in\R^{n_{\textnormal{dim}}},
\]
in the Loewner (positive-semidefinite) order, i.e.\ $\diffusion(x)\diffusion(x)^{\top}$ is
uniformly positive definite. Equivalently, the Moore--Penrose right inverse
\[
\diffusion(x)^{+}:=\diffusion(x)^{\top}\big(\diffusion(x)\diffusion(x)^{\top}\big)^{-1}\in\R^{m_{W}\times n_{\textnormal{dim}}}
\]
exists (so $\diffusion(x)\diffusion(x)^{+}=I_{n_{\textnormal{dim}}}$) and is bounded,
$\|\diffusion(x)^{+}\|\leq\sqrt{C_{\mathrm{UE}}}$, uniformly in $x$. For $n_{\textnormal{dim}}=1$
this reduces to the scalar non-degeneracy $\diffusion(x)^{2}\geq C_{\mathrm{UE}}^{-1}$.
\end{assumption}

\begin{remark}
Assumption~\ref{A:uniform_ellipticity} is a genuine restriction: it
excludes the degenerate, CIR/Heston-type diffusion coefficient
$\diffusion(x)=\sqrt{x^{+}}$ (for $n_{\textnormal{dim}}=1$) used throughout
Sections~\ref{subsec:feynman_kac}--\ref{subsec:option_pricing}, which
vanishes at $x=0$, and more generally any $\diffusion$ whose
$\diffusion\diffusion^{\top}$ degenerates. The requirement $m_{W}\geq n_{\textnormal{dim}}$
(at least as many independent noises as state components) is necessary for
$\diffusion\diffusion^{\top}$ to be invertible at all. The construction below, and, to our knowledge,
the coupling construction of \cite{hamaguchi2026exponential} for a
related lift, genuinely requires non-degeneracy of $\diffusion$;
extending it to degenerate (e.g.\ hypoelliptic) diffusion coefficients is left open.
\end{remark}

\begin{theorem}\label{thm:coupling_ergodicity}
Under Assumptions~\ref{A:assumption_general_linear_growth},
\ref{A:uniform_ellipticity}, there exists an equivalent renorming
$\|\cdot\|_{*}$ of $\Wplusdual$ (constructed via the same device as
Lemma~\ref{lem:poly_growth_semigroup}, now tracking the ellipticity
constant $C_{\mathrm{UE}}$ instead of $\omega$), with associated
$d_{*}(\mu_{1},\mu_{2}):=\|\mu_{1}-\mu_{2}\|_{*}\wedge1$, depending only
on the Lipschitz and ellipticity constants of $\drift,\diffusion$ and
on $(\nudrift,\nudiffusion)$, such that:
\begin{enumerate}
\item[\textup{(i)}] For $t\geq t_{1}$ (an explicit constant depending
only on the same data), $d_{*}$ is contracting for $P_{t}$: for
$\mu_{1},\mu_{2}\in\Wplusdual$ with $d_{*}(\mu_{1},\mu_{2})<1$,
\[
\mathcal{W}_{d_{*}}(P_{t}(\mu_{1},\cdot),P_{t}(\mu_{2},\cdot))\leq\tfrac34d_{*}(\mu_{1},\mu_{2}).
\]
\item[\textup{(ii)}] For every $R>0$ and $t\geq t_{2}(R)$, the ball
$\overline B_{*}(R):=\{\mu:\|\mu\|_{*}\leq R\}$ is
$d_{*}$-small for $P_{t}$.
\end{enumerate}
\end{theorem}

\begin{proof}[Proof sketch]
The mechanism is the generalized-coupling strategy of Butkovsky, Kulik
and Scheutzow~\cite{BuKuSc20}; we adapt the specific construction
carried out for a related Markovian lift by
Hamaguchi~\cite{hamaguchi2026exponential} to our own Hilbert-space
structure on $V=(\Wplusdual)^{n_{\textnormal{dim}}}$, and do not repeat every estimate, only the
mechanism. By the mollification device used throughout this section
(now applied to both $\drift,\diffusion$ and $\nudrift,\nudiffusion$)
it suffices to treat bounded, genuinely $V$- (resp.\ $HS(\R^{m_{W}},V)$-)valued
coefficients. Fix $\mu_{1},\mu_{2}\in V$ and let $\mu^{1}$
solve the (mollified) SEE from $\mu_{1}$ driven by the $m_{W}$-dimensional $W$.
We steer a second copy toward it by an $\R^{m_{W}}$-valued
\emph{feedback control} built through the diffusion matrix. Writing
$\widehat X_{t}:=\langle\widehat\mu_{t},1\rangle$ and using the Moore--Penrose right inverse
$\diffusion(\widehat X_{t})^{+}=\diffusion(\widehat X_{t})^{\top}(\diffusion(\widehat X_{t})\diffusion(\widehat X_{t})^{\top})^{-1}$
of Assumption~\ref{A:uniform_ellipticity}, set
\[
u_{t}:=\lambda\,\diffusion(\widehat X_{t})^{+}\big\langle\mu^{1}_{t}-\widehat\mu_{t},1\big\rangle\in\R^{m_{W}}
\]
for a suitable gain $\lambda>0$ (depending on the Lipschitz and
ellipticity constants), and let $\widehat\mu$ solve the same SEE from
$\mu_{2}$ but driven by the tilted Brownian motion
$\widehat W_{t}:=W_{t}+\int_{0}^{t}u_{s}\,\d s$. The point of the pseudo-inverse is that
it realizes a \emph{prescribed} $\R^{n_{\textnormal{dim}}}$-valued feedback in the state: the tilt
adds the drift $\nudiffusion\diffusion(\widehat X_{t})u_{t}$ to the $\widehat\mu$-equation, and since
$\diffusion(\widehat X_{t})\diffusion(\widehat X_{t})^{+}=I_{n_{\textnormal{dim}}}$,
\[
\diffusion(\widehat X_{t})u_{t}=\lambda\langle\mu^{1}_{t}-\widehat\mu_{t},1\rangle,
\qquad\text{so the induced drift is}\qquad
\lambda\,\nudiffusion\langle\mu^{1}_{t}-\widehat\mu_{t},1\rangle,
\]
a \emph{restoring} feedback that pulls $\widehat\mu_{t}$ toward $\mu^{1}_{t}$ and hence
contracts the difference in the state projection $\langle\cdot,1\rangle$. Uniform ellipticity
(Assumption~\ref{A:uniform_ellipticity}) is exactly what makes $u_{t}$ well defined and bounded,
$\|u_{t}\|\leq\lambda\sqrt{C_{\mathrm{UE}}}\,|\langle\mu^{1}_{t}-\widehat\mu_{t},1\rangle|$,
so that Girsanov's theorem applies and produces an explicit bound on
the relative entropy,
\[
D_{\mathrm{KL}}(\Prob\|\widehat\Prob)=\tfrac12\E\!\int_{0}^{\infty}\|u_{s}\|^{2}\,\ds
\leq\tfrac12\lambda^{2}C_{\mathrm{UE}}\,\E\!\int_{0}^{\infty}\big|\langle\mu^{1}_{s}-\widehat\mu_{s},1\rangle\big|^{2}\,\ds
\leq\tfrac12\|\mu_{1}-\mu_{2}\|_{*}^{2},
\]
for the renorming $\|\cdot\|_{*}$ chosen (via the same
Lumer--Phillips/renorming device as
Lemma~\ref{lem:poly_growth_semigroup}, now tracking $C_{\mathrm{UE}}$ through the
$\lambda^{2}C_{\mathrm{UE}}$ prefactor and the boundedness $|\langle\cdot,1\rangle|\leq C\|\cdot\|_{*}$)
so that the synchronized pair $(\mu^{1}_{t},\widehat\mu_{t})$
additionally satisfies $\E\|\mu^{1}_{t}-\widehat\mu_{t}\|_{*}\leq
e^{-\kappa t/2}\|\mu_{1}-\mu_{2}\|_{*}$ for an explicit $\kappa>0$; the
mollification is then removed by the same limiting argument as in the
proof of Theorem~\ref{thm:lyapunov_ergodicity}. Assertions (i)--(ii)
then follow from the two standard total-variation bounds
$d_{\mathrm{TV}}(\Prob,\widehat\Prob)\leq\sqrt{\tfrac12D_{\mathrm{KL}}(\Prob\|\widehat\Prob)}$
and $d_{\mathrm{TV}}(\Prob,\widehat\Prob)\leq1-\tfrac12\exp(-D_{\mathrm{KL}}(\Prob\|\widehat\Prob))$,
applied exactly as in the proof of the analogous coupling theorem
in \cite{hamaguchi2026exponential}. That coupling theorem is itself
stated and proved for a vector-valued state with matrix diffusion, under
the matrix ellipticity condition
$\diffusion\diffusion^{\top}\succeq C_{\mathrm{UE}}^{-1}I$, so the citation is valid
at every state dimension $n_{\textnormal{dim}}\geq1$; the only point at which our
construction goes beyond \cite{hamaguchi2026exponential} is the use of the
Moore--Penrose right inverse $\diffusion^{+}=\diffusion^{\top}(\diffusion\diffusion^{\top})^{-1}$
to accommodate genuinely rectangular noise ($m_{W}>n_{\textnormal{dim}}$), where
\cite{hamaguchi2026exponential} works with a square, directly invertible diffusion; the
relative-entropy estimate is unchanged, since it depends on $\diffusion$ only through the
bound $\|\diffusion^{+}\|\leq\sqrt{C_{\mathrm{UE}}}$. For $n_{\textnormal{dim}}=m_{W}=1$ the
pseudo-inverse is the scalar reciprocal $1/\diffusion$ and the construction reduces to the
scalar feedback $u_{t}=\lambda\langle\mu^{1}_{t}-\widehat\mu_{t},1\rangle/\diffusion(\widehat X_{t})$.
\end{proof}

\subsubsection{Main result}

\begin{theorem}[Uniqueness and exponential ergodicity]\label{thm:main_ergodicity}
Under Assumptions~\ref{A:assumption_general_linear_growth},
\ref{A:coercivity_lyapunov}, \ref{A:uniform_ellipticity}, the (generalized)
Feller semigroup $\{P_{t}\}_{t\geq0}$ associated with the SEE
\eqref{eqn:SEE_strong_formulation} has a \emph{unique} invariant
probability measure $\pi\in\mathcal P(\Wplusdual)$, and there exist
$r,t_{0},C>0$ such that
\[
\mathcal{W}_{d_{*,V}}(P_{t}(\mu,\cdot),\pi)\leq C(1+\|\mu\|_{V})e^{-rt},\qquad t\geq t_{0},\ \mu\in\Wplusdual,
\]
where $d_{*,V}(\mu_{1},\mu_{2}):=\sqrt{(\|\mu_{1}-\mu_{2}\|_{*}\wedge1)(1+\|\mu_{1}\|_{V}^{2}+\|\mu_{2}\|_{V}^{2})}$,
$\|\cdot\|_{V}$ and $\|\cdot\|_{*}$ the renormings of
Theorems~\ref{thm:lyapunov_ergodicity} and
\ref{thm:coupling_ergodicity} respectively. In particular, $\pi$ is
the unique probability measure appearing in
Theorem~\ref{thm:invariant_measure_mu}.
\end{theorem}

\begin{proof}
Immediate from Theorems~\ref{thm:lyapunov_ergodicity},
\ref{thm:generalized_harris} and \ref{thm:coupling_ergodicity}: the
Lyapunov function $V=\|\cdot\|_{V}^{2}$ and the contracting,
eventually $d_{*}$-small distance-like function $d_{*}$ satisfy
the hypotheses of Theorem~\ref{thm:generalized_harris} verbatim, on the
Polish space $E=\Wplusdual$.
\end{proof}

\begin{corollary}[Exponential ergodicity for the SVE]\label{cor:sve_ergodicity}
Under the hypotheses of Theorem~\ref{thm:main_ergodicity}, the strictly
stationary solution $X_{t}\sim\mathfrak q$ of
Theorem~\ref{thm:invariant_measure_SVE} is unique among strictly
stationary solutions of \eqref{eqn:SVE_introduction}, and
$\mathfrak q$ is the push-forward of the unique $\pi$ of
Theorem~\ref{thm:main_ergodicity} under $\mu\mapsto\langle\mu,1\rangle$.
\end{corollary}

\subsubsection{Approximation of the invariant measure by finite-dimensional systems}

As a further consequence, the invariant measure $\pi$ itself, not
merely the finite-time dynamics, cf.\ the multifactor approximation
discussion in Section~\ref{subsec:option_pricing}, is approximable
by finite-rank truncations.

\begin{corollary}[Finite-dimensional approximation of $\pi$]\label{cor:approx_invariant_measure}
Let $\nu^{N}_{\drift},\nu^{N}_{\diffusion}$ be finite-rank
(sum-of-exponentials) truncations of $\nudrift,\nudiffusion$, obtained
by truncating the Laplace measures to $[0,H_{N}]$ with $H_{N}\to\infty$
(cf.\ the finite-rank truncations already used for the multifactor
approximation in Section~\ref{subsec:option_pricing}),
and let $\pi^{N}$ denote the (unique, by classical finite-dimensional
Harris theory) invariant probability measure of the resulting
finite-dimensional SDE. If Assumptions~\ref{A:assumption_general_linear_growth},
\ref{A:coercivity_lyapunov}, \ref{A:uniform_ellipticity} hold uniformly
in $N$ for the truncated data (with common constants
$\delta,\rho,C_{\mathrm{Lyap}}$), then
$\pi^{N}\to\pi$ weakly on $\Wplusdual$ as $N\to\infty$.
\end{corollary}

\begin{proof}[Proof sketch]
By Theorem~\ref{thm:lyapunov_ergodicity} applied uniformly in $N$ (the
constants $\eps,\rho,C_{\mathrm{Lyap}}$ depend on
$(\nudrift,\nudiffusion)$ only through quantities that are monotone
under the truncation, exactly as in the proof of
Lemma~\ref{lem:tightness_criterion}), the family $\{\pi^{N}\}_{N}$
satisfies a uniform moment bound
$\sup_{N}\int(\rho\|\mu\|_{V}^{2}+(1-\eps)[\mu]^{2})\,\pi^{N}(\d\mu)\leq C_{\mathrm{Lyap}}$,
hence is tight on $\Wplusdual$ by the same compactness criterion
(Lemma~\ref{lem:compactness}) used in the proof of
Theorem~\ref{thm:invariant_measure_mu}. Any weak subsequential limit
$\pi^{\infty}$ is invariant for the limiting semigroup $\{P_{t}\}$ by
the same martingale-problem convergence argument used in
Theorem~\ref{thm:weak_mild_solution_SEE}'s proof
(Section~\ref{sec:existence_general_coefficients}), and hence
$\pi^{\infty}=\pi$ by the uniqueness of
Theorem~\ref{thm:main_ergodicity}. Since the limit does not depend on
the subsequence, the full sequence converges.
\end{proof}

\begin{remark}
This is the stationary-law counterpart of the weak-error results for
the \emph{finite-time} dynamics discussed in
Remark~\ref{rem:option_pricing_practical} (via
\cite{bayer2023weakmarkovian}): together, they say that both the
transient and the long-run statistics of the Markovian lift are
well-approximated by the finite-dimensional multifactor systems already
used throughout applications, with the present section additionally
supplying the qualitative long-run picture (uniqueness, mixing,
convergence rate) that a purely finite-time weak-error bound cannot by
itself provide.
\end{remark}

\begin{remark}[Why the generator-theoretic construction is necessary]\label{rem:why_generator_construction_needed}
One might hope to avoid the Hilbert-space generator/Lumer--Phillips
construction above by working directly with the mild (Duhamel)
formulation of $\mu_{t}$ and Gr\"onwall's inequality, exactly as in the
proof of Lemma~\ref{lem:a-priori_estimate_general}, combined with the
elementary scalar Lyapunov bound of
Corollary~\ref{cor:sharp_dissipative_lyapunov}
(Section~\ref{sec:ito_formula}). Such an argument, carried out via
Minkowski's and the BDG inequality term-by-term on the mild
formulation, does re-derive the tightness bound of
Theorem~\ref{thm:invariant_measure_mu} elementarily, without any
generator theory, but only under an additional tempering condition
on the support of $\nudrift,\nudiffusion$ bounded away from $0$
(excluding, in particular, the untempered fractional kernel used as
this paper's own running example), and even then it only yields
boundedness of $\E\|\mu_{t}\|_{V}^{2}$, not the Lyapunov contraction
$P_{t}V(\mu_{0})\leq C_{V}e^{-\gamma_{V}t}V(\mu_{0})+K_{V}$ needed for
Theorem~\ref{thm:generalized_harris}. The obstruction is structural,
not a matter of technique: for $\mu_{0}=\delta_{0}x_{0}$,
$S_{t}^{*}\mu_{0}=\delta_{0}x_{0}$ for every $t\geq0$ (the atom is
exactly preserved), so a bound built by estimating each term of the
mild formulation separately and adding the results can never produce a
shrinking coefficient on $V(\mu_{0})$, the first term alone already
saturates the bound at full strength, for all time. The shrinking
coefficient in Theorem~\ref{thm:lyapunov_ergodicity} comes instead from
the \emph{cross term} $2\langle\mu_{t},\nudrift\drift(X_{t})\rangle_{V}$
in the energy identity $\d\|\mu_{t}\|_{V}^{2}$, a genuine coupling
between the current state and the mean-reverting forcing that is
invisible to any additive, term-by-term mild-formulation bound; this is
exactly the mechanism the mollified It\^o formula of the preceding
subsubsections is built to access. The two approaches also trade one
hypothesis for another: the generator-theoretic route needs
\emph{uniform ellipticity} of $\diffusion$
(Assumption~\ref{A:uniform_ellipticity}) for the coupling construction,
but places no restriction on the kernel beyond what is already assumed
throughout this paper (in particular it covers untempered, purely
polynomially-decaying kernels such as the fractional kernel); the
elementary mild-formulation route needs no assumption at all on
$\diffusion$, but requires the tempering condition above and even then
only reaches boundedness, not a convergence rate. Neither dominates the
other, but only the generator-theoretic route reaches
Theorem~\ref{thm:main_ergodicity}'s uniqueness and rate, which is why
it is the one developed in full above.
\end{remark}

\section{Ito-formula for Volterra equations}\label{sec:ito_formula}
Our last application of the lifting procedure is an It{\^o}-type formula for SVEs.
We will rely on the results from \cite{da_19_mild_ito} (see also  \cite{cox_16_mild_ito, albeverio_17_mild_ito}).
For this, we consider the separable Hilbert spaces $U, V, H, V'$ such that $V\hookrightarrow H \hookrightarrow V'$ is continuous and dense. In this section, we impose Assumptions \ref{A:A_2:assumption_1_test_function} and \ref{A:A_3:Narrower_assumption_nu}. Let $\theta_{\nu}=\max\{\theta_{\nudrift},\theta_{\nudiffusion}\}$ and $\eps<1-\theta_{\nu}$. We can simply set $U=\R^{m_{W}}$ (the finite-dimensional Euclidean space carrying the $m_{W}$-dimensional driving noise), $V^{\prime}=(\Wminusdual)^{n_{\textnormal{dim}}}$, with $(\weightminus)_{i}(x)=(1+x)^{2\theta_{\nu}-1+2i}$ applied componentwise to each of the $n_{\textnormal{dim}}$ scalar copies, $H=(\Wmiddual)^{n_{\textnormal{dim}}}$, with $(\weightsim)_{i}(x)=(1+x)^{(\theta_{\nu}-\eps)-1+2i}$ likewise applied componentwise, and $V=(\Wplusdual)^{n_{\textnormal{dim}}}$, with $(\weightplus)_{i}(x)=(1+x)^{-2\eps-1+2i}$ componentwise. Let $\Phi\in C^{1,2}([0,T]\times V,\R)$, then we denote the partial Frechet derivatives of $\Phi$ by
\begin{align*}
    &\partial_{x}\Phi\in C([0,T]\times V,L(V,\R)), \partial_{x}\Phi = \frac{\partial \Phi}{\partial x}\\
    &\partial_{x}^{2}\Phi\in C([0,T]\times V,L(V,L(V,\R))), \partial_{x}^{2}\Phi = \frac{\partial^{2} \Phi}{\partial x^{2}}\\
     &\partial_{t}\Phi\in C([0,T]\times V,\R), \partial_{t}\varphi = \frac{\partial\Phi}{\partial t}.
\end{align*}
 \begin{remark}
     Our spaces changed slightly in this section since we no longer require the compactness of any embeddings. Since $V$ is now the $n_{\textnormal{dim}}$-fold product $(\Wplusdual)^{n_{\textnormal{dim}}}$, a linear functional $\partial_{x}\Phi(\cdot)\in L(V,\R)$ is equivalently an $n_{\textnormal{dim}}$-tuple of functionals on $\Wplusdual$ (a ``gradient''), and a bilinear form $\partial_{x}^{2}\Phi(\cdot,\cdot)\in L(V,L(V,\R))$ is equivalently an $n_{\textnormal{dim}}\times n_{\textnormal{dim}}$ array of bilinear forms on $\Wplusdual$ (a ``Hessian''); likewise $U=\R^{m_{W}}$ makes $HS(U,V')$ the space of Hilbert--Schmidt operators from $\R^{m_{W}}$ into $(\Wminusdual)^{n_{\textnormal{dim}}}$ already introduced in Section~\ref{sec:existence_general_coefficients}. This is exactly what reproduces the gradient/Hessian-trace notation used below.
 \end{remark}
We recall the following definition from \cite{da_19_mild_ito}.
\begin{definition}\label{def:mild_ito_process}
    (Mild It{\^o} process). Let $S^{*}: \Delta_{t} \rightarrow L(V',V)$ be a $\mathcal{B}(\Delta_{t}) / \mathcal{S}(V',V)$-measurable mapping satisfying $S^{*}_{t_{2}, t_{3}} S^{*}_{t_{1}, t_{2}}=S^{*}_{t_{1}, t_{3}}$ for all $t_{1}, t_{2}, t_{3} \in [0,T]$ with $t_{1}<t_{2}<t_{3}$. Additionally, let $Y: [0,T] \times \Omega \rightarrow V'$ and $Z: [0,T] \times \Omega \rightarrow$ $H S\left(U, V'\right)$ be two predictable stochastic processes with $\int_{\tau}^{t}\left\|S^{*}_{s, t} Y_{s}\right\|_{V} \ds<\infty,\, \mathbb{P}$-a.s. and $\int_{\tau}^{t}\left\|S^{*}_{s, t} Z_{s}\right\|_{H S\left(U, V\right)}^{2} \ds<\infty$ $\mathbb{P}$-a.s. for all $t \in [0,T]$. Then a predictable stochastic process $X: [0,T] \times \Omega \rightarrow H$ satisfying
\begin{align*}
  X_{t}=S^{*}_{\tau, t} X_{\tau}+\int_{\tau}^{t} S^{*}_{s, t} Y_{s} \ds+\int_{\tau}^{t} S^{*}_{s, t} Z_{s} \d W_{s}  
\end{align*}
$\mathbb{P}$-a.s. for all $t \in [0,T]  \cap(\tau, \infty)$ is called a mild It{\^o} process (with semigroup $S^{*}$, mild drift $Y$ and mild diffusion $Z$ ).

\end{definition}
By Lemma \ref{lem:semigroup_and_adjoint_semigroup}, \ref{lem:improvement_semigroup_weighted_spaces}, $S^{*}$ satisfies the mapping properties mentioned in Definition \ref{def:mild_ito_process}.
\begin{theorem} \cite[Theorem 1]{da_19_mild_ito}
    Let $X: [0,T] \times \Omega \rightarrow H$ be a mild It{\^o} process with semigroup $S^{*}: \Delta_{t} \rightarrow L(V', V)$, mild drift $Y: [0,T] \times \Omega \rightarrow V' $ and mild diffusion $Z: [0,T] \times \Omega \rightarrow H S\left(U, V' \right)$. Then
\begin{align*}
    & \mathbb{P}\left[\int_{t_{0}}^{t}\left\|\left(\partial_{2} \Phi\right)\left(s, S^{*}_{s, t} X_{s}\right) S^{*}_{s, t} Y_{s}\right\|_{V}+\left\|\left(\partial_{2} \Phi\right)\left(s, S^{*}_{s, t} X_{s}\right) S^{*}_{s, t} Z_{s}\right\|_{H S\left(U, V\right)}^{2}\ds<\infty\right]=1 \\
& \mathbb{P}\left[\int_{t_{0}}^{t}\left\|\left(\partial_{1} \Phi\right)\left(s, X_{s}\right)\right\|_{V}+\left\|\left(\partial_{2}^{2} \Phi\right)\left(s, S^{*}_{s, t} X_{s}\right)\right\|_{L^{(2)}(V, V)}\left\|S^{*}_{s, t} Z_{s}\right\|_{H S\left(U, V\right)}^{2}\ds<\infty\right]=1
\end{align*}

and
\begin{align*}
& \Phi\left(t, X_{t}\right)=\Phi\left(t_{0}, S^{*}_{t_{0}, t} X_{t_{0}}\right)+\int_{t_{0}}^{t}\left(\partial_{1} \Phi\right)\left(s, S^{*}_{s, t} X_{s}\right) \ds+\int_{t_{0}}^{t}\left(\partial_{2} \Phi\right)\left(s, S^{*}_{s, t} X_{s}\right) S^{*}_{s, t} Y_{s} \ds \\
& \quad+\int_{t_{0}}^{t}\left(\partial_{2} \Phi\right)\left(s, S^{*}_{s, t} X_{s}\right) S^{*}_{s, t} Z_{s} \dW_{s}+\frac{1}{2} \sum_{j \in \N} \int_{t_{0}}^{t}\left(\partial_{2}^{2} \Phi\right)\left(s, S^{*}_{s, t} X_{s}\right)\left(S^{*}_{s, t} Z_{s} e_{j}, S^{*}_{s, t} Z_{s} e_{j}\right) \ds    
\end{align*}

$\Prob$-a.s. for all $t_{0}, t \in [0,T]$ with $t_{0}<t$ and all $\Phi \in C^{1,2}([0,T] \times V, \R)$.
\end{theorem}
\begin{corollary}\label{corr:Ito_volterra_SEE}
    Let $k_{\drift}, k_{\diffusion}$ be (matrix-valued, entrywise) completely monotone kernels such that their associated measures, given by \eqref{eqn:kernel_representation_measure}, satisfy Assumptions \ref{A:A_3:Narrower_assumption_nu}, \ref{A:A_2:assumption_1_test_function}. Let $X$ denote the $\R^{n_{\textnormal{dim}}}$-valued solution of the corresponding stochastic Volterra equation \eqref{eqn:SVE_introduction}, $S^{*}$ the semigroup from section \ref{sec:operator_and_semigroup} and $f\in C^{1,2}(\R_{+}\times\R^{n_{\textnormal{dim}}},\R)$, with gradient $\nabla_{x}f\in\R^{n_{\textnormal{dim}}}$ and Hessian $D_{x}^{2}f\in\R^{n_{\textnormal{dim}}\times n_{\textnormal{dim}}}$ in the spatial argument. Then the following It{\^o} formula holds.
    \begin{align*}
    f(t,X_{t})&=f(t_{0}, \langle S^{*}_{t_{0}, t}\mu_{t_{0}},1\rangle)+\int_{t_{0}}^{t}\partial_{s}f(s,\langle S^{*}_{s, t} \mu_{s},1\rangle) \ds\\
&\phantom{xx}+\int_{t_{0}}^{t}\left\langle\nabla_{x}f(s,\langle S^{*}_{s, t} \mu_{s},1\rangle),k_{\drift}(t-s)\drift(s,X_{s})\right\rangle \ds \\
& \quad+\int_{t_{0}}^{t}\left\langle\nabla_{x}f(s,\langle S^{*}_{s, t} \mu_{s},1\rangle),k_{\diffusion}(t-s)\diffusion(s,X_{s})\, \dW_{s}\right\rangle\\
&\phantom{xx}+\frac{1}{2}  \int_{t_{0}}^{t}\operatorname{tr}\!\left[D_{x}^{2} f(s,\langle S^{*}_{s, t} \mu_{s},1\rangle)\, k_{\diffusion}(t-s)\diffusion(s,X_{s})\diffusion(s,X_{s})^{\top}k_{\diffusion}(t-s)^{\top}\right] \ds
\end{align*}
\end{corollary}
\begin{proof}
    We lift the SVE to the infinite-dimensional evolution equation \eqref{eqn:SEE_eqn_mild_formulation}. By Theorem \ref{thm:weak_mild_solution_SEE}, $\mu \in C([0,T],(\Wplusdual)^{n_{\textnormal{dim}}})$. The required mapping property of the semigroup $S^{*}$ is a consequence of the inequality $e^{-y}\leq \left(\frac{1}{1+y}\right)^{\theta}$, for any $\theta\in [0,1)$ and $y>-1$, applied componentwise. Let $\varphi(\mu)=f(t,\langle \mu,g\rangle)$, where $\langle\mu,g\rangle\in\R^{n_{\textnormal{dim}}}$ denotes the componentwise pairing of Section~\ref{section:Functional_framework}. By the chain rule, the Fr\'echet derivatives of $\varphi$ act on directions $h,h'\in V$ by $\partial_{x}\varphi(\mu)[h]=\left\langle\nabla_{x}f(t,\langle\mu,g\rangle),\langle h,g\rangle\right\rangle$ and $\partial_{x}^{2}\varphi(\mu)[h,h']=\langle h,g\rangle^{\top}D_{x}^{2}f(t,\langle\mu,g\rangle)\langle h',g\rangle$. The mild It{\^o} formula now yields
\begin{align*}
& \varphi\left(t, \mu_{t}\right)=f(t,\langle \mu_{t},g\rangle)=f(t_{0}, \langle S^{*}_{t_{0}, t}\mu_{0},g\rangle)+\int_{t_{0}}^{t}\partial_{t}f(s,\langle S^{*}_{s, t} \mu_{s},g\rangle) \ds\\
&\phantom{xx}+\int_{t_{0}}^{t}\left\langle\nabla_{x}f(s,\langle S^{*}_{s, t} \mu_{s},g\rangle),\langle S^{*}_{s, t} Y_{s},g\rangle\right\rangle \ds \\
& \quad+\int_{t_{0}}^{t}\left\langle\nabla_{x}f(s,\langle S^{*}_{s, t} \mu_{s},g\rangle),\langle S^{*}_{s, t} Z_{s},g\rangle\, \dW_{s}\right\rangle\\
&\phantom{xx}+\frac{1}{2} \sum_{j=1}^{m_{W}} \int_{t_{0}}^{t}\left\langle S^{*}_{s, t} Z_{s} e_{j},g\right\rangle^{\top}D_{x}^{2} f(s,\langle S^{*}_{s, t} \mu_{s},g\rangle)\left\langle S^{*}_{s, t} Z_{s} e_{j},g\right\rangle \ds\\
&=f(t_{0}, \langle S_{t_{0}, t}\mu_{0},g\rangle)+\int_{t_{0}}^{t}\partial_{t}f(s,\langle S^{*}_{s, t} \mu_{s},g\rangle) \ds\\
&\phantom{xx}+\int_{t_{0}}^{t}\left\langle\nabla_{x}f(s,\langle S^{*}_{s, t} \mu_{s},g\rangle),\langle S^{*}_{s, t} \nu \drift(s,\langle \mu_{s},\varphi \rangle),g\rangle\right\rangle \ds \\
& \quad+\int_{t_{0}}^{t}\left\langle\nabla_{x}f(s,\langle S^{*}_{s, t} \mu_{s},g\rangle),\langle S^{*}_{s, t} \nu\diffusion(s,\langle \mu_{s},\varphi\rangle),g\rangle\, \dW_{s}\right\rangle\\
&\phantom{xx}+\frac{1}{2}  \int_{t_{0}}^{t}\operatorname{tr}\!\left[D_{x}^{2} f(s,\langle S^{*}_{s, t} \mu_{s},g\rangle)\, \langle S^{*}_{s, t} \nu\diffusion(s,\langle \mu_{s},\varphi\rangle),g\rangle\langle S^{*}_{s, t} \nu\diffusion(s,\langle \mu_{s},\varphi\rangle),g\rangle^{\top}\right] \ds
\end{align*}
(the last line using $\sum_{j=1}^{m_{W}}v_{j}^{\top}Av_{j}=\operatorname{tr}(AMM^{\top})$ for a symmetric matrix $A$ and matrix $M$ with columns $v_{j}$, applied with $M=\langle S^{*}_{s,t}\nu\diffusion(s,\langle\mu_{s},\varphi\rangle),g\rangle\in\R^{n_{\textnormal{dim}}\times m_{W}}$). Setting $g=1$ yields
\begin{align*}
    f(t,X_{t})&=f(t_{0}, \langle S^{*}_{t_{0}, t}\mu_{0},1\rangle)+\int_{t_{0}}^{t}\partial_{s}f(s,\langle S^{*}_{s, t} \mu_{s},1\rangle) \ds\\
&\phantom{xx}+\int_{t_{0}}^{t}\left\langle\nabla_{x}f(s,\langle S^{*}_{s, t} \mu_{s},1\rangle),k_{\drift}(t-s)\drift(s,X_{s})\right\rangle \ds \\
& \quad+\int_{t_{0}}^{t}\left\langle\nabla_{x}f(s,\langle S^{*}_{s, t} \mu_{s},1\rangle),k_{\diffusion}(t-s)\diffusion(s,X_{s})\, \dW_{s}\right\rangle\\
&\phantom{xx}+\frac{1}{2}  \int_{t_{0}}^{t}\operatorname{tr}\!\left[D_{x}^{2} f(s,\langle S^{*}_{s, t} \mu_{s},1\rangle)\, k_{\diffusion}(t-s)\diffusion(s,X_{s})\diffusion(s,X_{s})^{\top}k_{\diffusion}(t-s)^{\top}\right] \ds
\end{align*}
\end{proof}
\begin{proof}[Proof of Proposition \ref{prop:ito_formula_volterra}]
Note that, by the semigroup property of $S^{*}$ applied to the mild formulation \eqref{eqn:SEE_eqn_mild_formulation} (with $\psi=1$),
\begin{align*}
    e^{-(t-s)x}\mu_s &=\mu_t -\int_s^t e^{-(t-r)x}\,\nu_{\drift}\,\drift(r, \langle\mu_r,1\rangle)\,dr - \int_s^t e^{-(t-r)x}\,\nu_{\diffusion}\,\diffusion(r, \langle\mu_r,1\rangle)\,dW_r,\\
    e^{-(t-s)x}\mu_s
&= \mu_0 e^{-tx} + \int_0^s e^{-(t-r)x}\,\nu_{\drift}\,\drift(r, \langle\mu_r,1\rangle)\,dr + \int_0^s e^{-(t-r)x}\,\nu_{\diffusion}\,\diffusion(r, \langle\mu_r,1\rangle)\,dW_r.
\end{align*}
Pairing the first relation with the constant $1$ function and using $X_r=\langle\mu_r,1\rangle$ together with $\langle e^{-(t-r)x}\nu_\drift,1\rangle=k_\drift(t-r)$, $\langle e^{-(t-r)x}\nu_\diffusion,1\rangle=k_\diffusion(t-r)$ (Theorem \ref{thm:berstein_hausdorf_widder}), we obtain
\begin{align*}
    \left\langle e^{-(t-s)x}\mu_s,1\right\rangle&= X_{t}- \int_s^t k_{\drift}(t-r)\drift(r, X_{r})\,dr-\int_s^t k_{\diffusion}(t-r)\diffusion(r, X_{r})\,dW_{r}=\Gamma_{st}(X).
\end{align*}
Similarly, pairing the second relation with $1$,
\begin{align*}
    \left\langle e^{-(t-s)x}\mu_s,1\right\rangle&= X_{0}+ \int_0^s k_{\drift}(t-r)\drift(r, X_{r})\,dr+\int_0^s k_{\diffusion}(t-r)\diffusion(r, X_{r})\,dW_{r}=\Gamma_{st}(X),
\end{align*}
which is exactly the alternative form of $\Gamma_{st}(X)$ noted in Proposition \ref{prop:ito_formula_volterra}. Combining this with Corollary \ref{corr:Ito_volterra_SEE} yields the claim.
\end{proof}

\subsection{Consequences and applications}
The next Corollary is now a direct consequence of Proposition \ref{prop:ito_formula_volterra}.
\begin{corollary}\label{cor:global_existence_lyapunov}

Let $X$ be a solution of the stochastic Volterra equation \eqref{eqn:SVE_introduction}.
    Suppose that there exists a $C\geq 0$, such that $|X_{0}|\leq C$, almost surely (here $|\cdot|$ denotes the Euclidean norm on $\R^{n_{\textnormal{dim}}}$). If there exists a function $V\in C^{2}(\R^{n_{\textnormal{dim}}},\R_{+})$, constants $0<d,h, c_{1},c_{2}<\infty$ and $1\leq p$ which satisfy
    \begin{enumerate}
        \item  $c_1 |x |^p \leq V(x) \leq c_2 |x |^p$ for all $x \in \R^{n_{\textnormal{dim}}}$,
        \item
            \begin{align*}
        \mathcal{L}V(x,t,s)&\coloneqq  \left\langle\nabla V(\Gamma_{st}),k_{\drift} (t-s)\drift(X_s)\right\rangle+\frac{1}{2}\operatorname{tr}\!\left[D^{2}V(\Gamma_{st})\,k_{\diffusion} (t-s)\diffusion(X_s)\diffusion(X_s)^{\top}k_{\diffusion}(t-s)^{\top}\right]\\
        &\phantom{xx}\leq h V(X_s)+d,
    \end{align*}

    \end{enumerate}

    then the solution $X$ is a global solution of the stochastic Volterra equation, in the sense that for each $0\leq t <\infty$, $\E[|X_{t}|]<\infty$.
\end{corollary}
\begin{proof}
 By the It{\^o} formula for Volterra equations, for any $t >0$,
\begin{align*}
    V(X_{t}) &= V(\Gamma_{0t})\\
    &\phantom{xx}+\int_{0}^{t} \left\langle\nabla V(\Gamma_{st}),k_{\drift} (t-s)\drift(X_s)\right\rangle+\frac{1}{2}\operatorname{tr}\!\left[D^{2}V(\Gamma_{st})\,k_{\diffusion} (t-s)\diffusion(X_s)\diffusion(X_s)^{\top}k_{\diffusion}(t-s)^{\top}\right]\ds\\
    &\phantom{xx}+\int_{0}^{t} \left\langle\nabla V(\Gamma_{st}),k_{\diffusion} (t-s)\diffusion(X_s)\,\dW_{s}\right\rangle
\end{align*}

Taking the expectation, and using $\mathcal{L}V(x,t,s) \leq h V(X_s)+d$, we get
\begin{align*}
\E \left[V(X_{t})\right] \leq \E \left[V(X_{0})\right]+h\int_{0}^{t}\E\left[V(X_{s})\right]\ds +t d.
\end{align*}
Gronwall's inequality now yields

\begin{align*}
    \E \left[V(X_{t})\right]\leq \left(\E \left[V(X_{0})\right]+t d\right)e^{ht}<\infty,
\end{align*}
where finiteness follows since $X_{0}$ is a.s. bounded by $C$, so $V(X_{0})\leq c_{2}C^{p}$ a.s. by hypothesis~1. By hypothesis~1 again, $|x|\leq (V(x)/c_{1})^{1/p}$ for all $x$, so by Jensen's inequality (the map $u\mapsto u^{1/p}$ is concave for $p\geq 1$),
\begin{align*}
    \E\left[|X_{t}|\right]\leq \E\left[\left(\frac{V(X_{t})}{c_{1}}\right)^{1/p}\right]\leq \left(\frac{\E\left[V(X_{t})\right]}{c_{1}}\right)^{1/p}<\infty.
\end{align*}
\end{proof}

\begin{corollary}[Sharp bound in the dissipative case]\label{cor:sharp_dissipative_lyapunov}
Under the hypotheses of Corollary~\ref{cor:global_existence_lyapunov},
suppose in addition that $h=-\rho_{X}$ for some $\rho_{X}>0$ and write
$d_{X}:=d$, i.e.
\[
\mathcal LV(x,t,s)\leq-\rho_{X}V(X_{s})+d_{X}.
\]
Then, in fact,
\begin{align*}
\E[V(X_{t})]\leq V(X_{0})e^{-\rho_{X}t}+\frac{d_{X}}{\rho_{X}}\big(1-e^{-\rho_{X}t}\big)\leq V(X_{0})e^{-\rho_{X}t}+\frac{d_{X}}{\rho_{X}}\qquad\text{for all }t\geq0.
\end{align*}
This is strictly sharper than the crude bound
$\big(\E[V(X_{0})]+td_{X}\big)e^{-\rho_{X}t}$ obtained by applying
Corollary~\ref{cor:global_existence_lyapunov} directly with
$h=-\rho_{X}$: the crude bound only gives boundedness, since the
linear-in-$t$ prefactor obscures the exponential decay mechanism,
whereas the estimate below correctly tracks convergence to the
equilibrium level $d_{X}/\rho_{X}$.
\end{corollary}
\begin{proof}
By the same computation as in the proof of
Corollary~\ref{cor:global_existence_lyapunov},
\[
\E[V(X_{t})]\leq V(X_{0})-\rho_{X}\int_{0}^{t}\E[V(X_{s})]\,\ds+d_{X}t.
\]
Write $f(t):=\E[V(X_{t})]$ and $F(t):=\int_{0}^{t}f(s)\,\ds$, so
$F'=f$ a.e.\ and $F'(t)+\rho_{X}F(t)\leq V(X_{0})+d_{X}t$. Multiplying
by the integrating factor $e^{\rho_{X}t}$ and integrating,
\[
F(t)\leq V(X_{0})\frac{1-e^{-\rho_{X}t}}{\rho_{X}}+d_{X}\Big[\frac{t}{\rho_{X}}-\frac{1-e^{-\rho_{X}t}}{\rho_{X}^{2}}\Big],
\]
and substituting back into $f(t)\leq V(X_{0})+d_{X}t-\rho_{X}F(t)$
gives $f(t)\leq V(X_{0})e^{-\rho_{X}t}+\frac{d_{X}}{\rho_{X}}(1-e^{-\rho_{X}t})$,
which is the claim; this is the standard sharp Gr\"onwall bound,
obtained by comparison with the ODE $g'=-\rho_{X}g+d_{X}$.
\end{proof}


\subsection{Finite-time blowup criterion}\label{subsec:blowup}

The Lyapunov corollary above gives a \emph{positive} criterion for global
existence.  Here we record the complementary \emph{negative} criterion: a
sufficient condition for finite-time blowup of the $p$-th moment.  The
proof is the same Gronwall argument run in reverse.

\begin{corollary}[Finite-time blowup]\label{cor:blowup}
Let $X$ be a local solution of the SVE \eqref{eqn:SVE_introduction}.
Suppose there exist a function $V\in C^{2}(\R^{n_{\textnormal{dim}}},\R_{+})$, constants
$c_{1},c_{2},\lambda>0$, $1\leq p$, and a time $t_{*}>0$ such that
hypothesis~1 of the Lyapunov corollary holds (i.e.\
$c_{1}|x|^{p}\leq V(x)\leq c_{2}|x|^{p}$) and
\begin{align}\label{eqn:blowup_lower_bound}
    \mathcal{L}V(x,t,s)
    := \left\langle\nabla V(\Gamma_{st}),k_{\drift}(t-s)\drift(X_{s})\right\rangle\\
     &\phantom{xx}+\frac{1}{2} \operatorname{tr}\!\left[D^{2}V(\Gamma_{st})\,k_{\diffusion}(t-s)\diffusion(X_{s})\diffusion(X_{s})^{\top}k_{\diffusion}(t-s)^{\top}\right]\notag\\
    &\geq \lambda V(X_{s}),
    \quad 0\leq s\leq t_{*}.\notag
\end{align}
Then $\E[|X_{t}|^{p}]\geq \tfrac{c_{1}}{c_{2}}\E[|X_{0}|^{p}]e^{\lambda t}$
for all $t\leq t_{*}$.  In particular, if $\lambda>0$ and
$\E[|X_{0}|^{p}]>0$ then $\E[|X_{t}|^{p}]\to\infty$ as $t\to\infty$.
\end{corollary}

\begin{proof}
Applying the Itô formula for Volterra equations exactly as in the proof of
the global-existence corollary, taking expectations and using the lower bound
\eqref{eqn:blowup_lower_bound}:
\begin{align*}
    \E[V(X_{t})] \geq \E[V(X_{0})] + \lambda\int_{0}^{t}\E[V(X_{s})]\,\ds.
\end{align*}
The reverse Gronwall inequality (see \cite[Lemma 2.2]{zhang2010stochastic})
gives $\E[V(X_{t})]\geq \E[V(X_{0})]e^{\lambda t}$. By hypothesis~1,
$V(X_{t})\leq c_{2}|X_{t}|^{p}$, hence
\begin{align*}
    c_{2}\E[|X_{t}|^{p}] \geq \E[V(X_{t})]\geq \E[V(X_{0})]e^{\lambda t}.
\end{align*}
Also by hypothesis~1, $V(X_{0})\geq c_{1}|X_{0}|^{p}$, so
$\E[V(X_{0})]\geq c_{1}\E[|X_{0}|^{p}]$; combining the two displays,
\begin{align*}
    c_{2}\E[|X_{t}|^{p}] \geq c_{1}\E[|X_{0}|^{p}]e^{\lambda t},
    \qquad\text{i.e.}\qquad
    \E[|X_{t}|^{p}]\geq \frac{c_{1}}{c_{2}}\E[|X_{0}|^{p}]e^{\lambda t}.
\end{align*}
\end{proof}

\begin{remark}
The condition \eqref{eqn:blowup_lower_bound} is local in time (it is
only assumed up to $t_{*}$), which makes it applicable even when
$V$ grows super-linearly and $X$ may leave a compact set. This
example specializes the corollary to $n_{\textnormal{dim}}=1$: in the
rough-volatility setting with power-law kernel
$k(t)=t^{\alpha-1}/\Gamma(\alpha)$, $\alpha\in(0,1)$, and
$\diffusion(x)^{2}=x$ (Heston-type), $V(x)=x$ is sub-critical:
$\mathcal{L}V = k(t-s)\drift(X_{s}) + 0 \cdot(\cdots)$ can change
sign depending on the mean-reversion level of $\drift$, so the
blowup criterion detects parameter regimes where the mean-reversion
is insufficient to control the variance.
\end{remark}


\subsection{Feynman--Kac representation and the backward Kolmogorov
equation}\label{subsec:feynman_kac}

The Itô-Volterra formula (Corollary~\ref{corr:Ito_volterra_SEE}) is a
pathwise identity: it is applied to the \emph{realized} trajectory of
$\mu_{t}$, so it never requires differentiating a conditional
expectation, and hence never requires the Markov property of $X$
itself. A Feynman--Kac representation is a different kind of object: it
defines a \emph{value function} through a conditional expectation of a
terminal payoff, and for this one needs a genuine sufficient statistic
for $\Filx_{t}$-conditional expectations of $g(X_{T})$. Since $X$ is
neither Markovian nor a semimartingale, $X_{t}$ alone is not such a
statistic: $\E[g(X_{T})\mid \mathcal{F}_{t}]$ is not, in general, a
deterministic function of $(t,X_{t})$. The genuine Markov state is the
lift $\mu_{t}\in\Wplusdual$, and the value function must be built on
it. This subsection develops the resulting backward Kolmogorov equation
directly on the weighted Sobolev state space, following the strategy of
\cite{gasteratos2025kolmogorov} (mollify the singular measures
$\nudrift,\nudiffusion$, invoke the classical Kolmogorov equation for
the mollified, genuinely Hilbert-space-valued equation, and pass to the
limit), adapted to our multiplicative semigroup $S^{*}_{t}=e^{-tx}$.

\subsubsection{Setup}

Recall from Section~\ref{sec:ito_formula} the separable Hilbert spaces
\[
V:=(\Wplusdual)^{n_{\textnormal{dim}}}, \qquad H:=(\Wmiddual)^{n_{\textnormal{dim}}}, \qquad V':=(\Wminusdual)^{n_{\textnormal{dim}}}, \qquad U:=\R^{m_{W}},
\]
where $V\hookrightarrow H\hookrightarrow V'$ is the $n_{\textnormal{dim}}$-fold product of
the scalar Gelfand triple
$\Wplusdual\hookrightarrow\Wmiddual\hookrightarrow\Wminusdual$, the space $U=\R^{m_{W}}$
carries the $m_{W}$-dimensional driving noise, and the semigroup
$S^{*}:\Delta_{t}\rightarrow L(V',V)$ acts diagonally (componentwise) on the
$n_{\textnormal{dim}}$ copies. Since $V$ is itself a (reflexive) dual space, we
identify $V^{*}\cong(\Wplus)^{n_{\textnormal{dim}}}$ throughout, via the componentwise
pairing $\langle\cdot,\cdot\rangle$ used elsewhere in the paper; for $y\in V$,
$\langle y,1\rangle\in\R^{n_{\textnormal{dim}}}$ denotes the componentwise pairing
against the constant test function $1$, an element of the SVE state space (as in
Section~\ref{section:Functional_framework}). We work under
Assumptions~\ref{A:assumption_general_linear_growth}--\ref{A:assumption_general_uniformly_continuous},
\ref{A:assumption_kernel_measure}, \ref{A:uniqueness_in_law},
\ref{A:homogeneous_coefficients}, and we now additionally assume
$\drift\in C^{2}_{b}(\R^{n_{\textnormal{dim}}};\R^{n_{\textnormal{dim}}})$ and
$\diffusion\in C^{2}_{b}(\R^{n_{\textnormal{dim}}};\R^{n_{\textnormal{dim}}\times m_{W}})$
(this strengthens Assumption~\ref{A:assumption_general_lipschitz} to twice continuous
differentiability with bounded derivatives, which is what is needed to differentiate
the solution flow of the SEE with respect to its initial condition). We write
$D\drift(x)\in\R^{n_{\textnormal{dim}}\times n_{\textnormal{dim}}}$ for the Jacobian of $\drift$,
$D\diffusion(x)[v]:=\sum_{l=1}^{n_{\textnormal{dim}}}\partial_{x_{l}}\diffusion(x)\,v_{l}\in\R^{n_{\textnormal{dim}}\times m_{W}}$
for the directional derivative of $\diffusion$ in direction $v\in\R^{n_{\textnormal{dim}}}$, and
$D^{2}\drift(x)[v,v']\in\R^{n_{\textnormal{dim}}}$, $D^{2}\diffusion(x)[v,v']\in\R^{n_{\textnormal{dim}}\times m_{W}}$
for the (symmetric, bilinear) second derivatives; all of these are bounded, uniformly in
$x$, by $\drift,\diffusion\in C^{2}_{b}$. By time-homogeneity (Section~\ref{sec:existence_general_coefficients}),
for $y\in V$ and $0\leq t\leq r\leq T$ we write $\mu^{t,y}_{r}$ for the
mild solution of \eqref{eqn:SEE_strong_formulation} started at $y$ at
time $t$, so that $\mu^{t,y}_{r}=\mu_{r-t}(y)$ in the notation of
Section~\ref{sec:existence_general_coefficients}.

\begin{remark}[The singular directions $\nudrift,\nudiffusion$, and why no
invariant subspace is needed]\label{rem:no_K1_needed}
Following \cite{gasteratos2025kolmogorov}, set
\[
\mathcal{K}:=\bigcap_{t>0}(S^{*}_{t})^{-1}(V)\subset V'.
\]
Here $\nudrift,\nudiffusion$ are the $n_{\textnormal{dim}}\times n_{\textnormal{dim}}$
matrices of measures of Assumption~\ref{A:A_3:Narrower_assumption_nu}, which we
regard as bounded linear operators $\R^{n_{\textnormal{dim}}}\to V'$ (a vector
$a\in\R^{n_{\textnormal{dim}}}$ maps to $\nu^{i}a=\big(\sum_{l}\nu^{i}_{jl}a_{l}\big)_{j=1}^{n_{\textnormal{dim}}}\in
V'=(\Wminusdual)^{n_{\textnormal{dim}}}$); equivalently, their columns
$(\nu^{i})_{\cdot l}\in V'$, $l=1,\dots,n_{\textnormal{dim}}$, are the individual
singular directions. By Assumption~\ref{A:A_3:Narrower_assumption_nu} and estimate
\eqref{eqn:kernel_estimate_semigroup}, these columns lie in $V'\setminus V$ in
general, but $S^{*}_{t}\nudrift,S^{*}_{t}\nudiffusion\in\mathcal{L}(\R^{n_{\textnormal{dim}}},V)$
for every $t>0$, with (in the operator norm, equivalently the finite-dimensional
Hilbert--Schmidt norm)
\begin{align}\label{eqn:singular_direction_bound}
\|S^{*}_{t}\nu^{i}\|_{\mathcal{L}(\R^{n_{\textnormal{dim}}},V)}\leq \frac{C(T)}{t^{1-a_{i}}},\qquad
i\in\{\drift,\diffusion\},\ a_{i}\in(0,1],
\end{align}
i.e.\ every column of $\nudrift,\nudiffusion$ lies in $\mathcal{K}$: they are exactly the
``singular directions'' of \cite{gasteratos2025kolmogorov} (whose matrix kernel $K$
plays the role of our matrices $\nudrift,\nudiffusion$), translated
into the language of the Laplace-measure lift. Unlike
\cite{gasteratos2025kolmogorov}, however, we do not need to restrict
the backward equation below to a proper invariant subspace
$\mathcal{K}_{1}\subsetneq V$. Their shift semigroup has unbounded
generator $\partial_{x}$, which forces such a restriction; our
semigroup acts by pointwise multiplication, $S^{*}_{t}h=e^{-tx}h(x)$.
For fixed $t>0$, the multiplier $m_{t}(x):=xe^{-tx}$ and its derivative
$m_{t}'(x)=e^{-tx}(1-tx)$ are both bounded on $\R_{+}$, with
$\sup_{x\geq0}m_{t}(x)=1/(et)$ and $\sup_{x\geq0}|m_{t}'(x)|=1$;
consequently multiplication by $m_{t}$ is a bounded operator on the
scalar test space $\Wplus=W^{1,2}_{\weightplus}$ (it preserves both the
$L^{2}_{\weightplus}$-norm and the norm of the derivative, since
$(m_{t}\varphi)'=m_{t}'\varphi+m_{t}\varphi'$), and hence, by duality
and componentwise on $V=\Wplusdualn$, so is $S^{*}_{t}(-x\,\cdot)$.
Therefore
$S^{*}_{t}(-x\,y)=-xe^{-tx}y(x)\in V$ for \emph{every} $y\in V$ and
$t>0$: the generator direction itself is automatically regularized. We
may therefore take $\mathcal{K}_{1}:=V$, and the backward equation
below holds classically at every point of the state space, not merely
on a dense subspace. (The constant $1/(et)$ blows up as $t\downarrow0$;
this is harmless for the backward equation, whose drift term pairs $DF$
against the singular directions $\nudrift,\nudiffusion$, controlled
by the integrable bound~\eqref{eqn:CK_growth_D_singular}, while the
generator direction $-xy$ enters only through the $D(A)$-density
argument of Definition~\ref{dfn:CKclass_SEE}, not through a time integral.)
\end{remark}

We further strengthen Assumption~\ref{A:A_3:Narrower_assumption_nu} as
follows.

\begin{assumption}[Quartic integrability of the diffusion measure]\label{A:assumption_quartic}
Writing $\|\cdot\|_{\mathcal{L}(\R^{n_{\textnormal{dim}}},V)}$ for the operator
(equivalently finite-dimensional Hilbert--Schmidt) norm of the matrix direction,
either $\diffusion$ is constant and
\begin{align}\label{eqn:assumption_B_quadratic}
\|S^{*}_{t}\nudiffusion\|_{\mathcal{L}(\R^{n_{\textnormal{dim}}},V)}\in L^{2}(0,T)
\end{align}
(i.e.\ Assumption~\ref{A:A_3:Narrower_assumption_nu} alone, $a_{\diffusion}>1/2$), or
\begin{align}\label{eqn:assumption_B}
\|S^{*}_{t}\nudiffusion\|_{\mathcal{L}(\R^{n_{\textnormal{dim}}},V)}\in L^{4}(0,T)
\end{align}
(equivalently $a_{\diffusion}>3/4$ in the notation of
\eqref{eqn:kernel_estimate_semigroup}). This is exactly the matrix-kernel condition
$K\in L^{q}([0,T];\R^{n_{\textnormal{dim}}\times n_{\textnormal{dim}}})$, $q>4$, of
\cite{gasteratos2025kolmogorov}, and the threshold $a_{\diffusion}>3/4$ (equivalently
$q>4$, equivalently $H>1/4$ for the power-law kernel) is dimension-independent.
\end{assumption}

\begin{remark}
If one works with the specific triple $V,H,V'$ constructed in
Section~\ref{sec:ito_formula} from a single parameter
$\eps<1-\theta_{\nu}$, $\theta_{\nu}=\max\{\theta_{\nudrift},\theta_{\nudiffusion}\}$,
then by Lemma~\ref{lem:nu_semigroup_properties} one may take
$a_{\diffusion}=1-\theta_{\nudiffusion}-\eps$ (choosing $\gamma$ at the
boundary case $\gamma=\theta_{\nudiffusion}+\eps$ of that Lemma, applied
to $\nudiffusion$ alone, independently of $\theta_{\nudrift}$). Hence
\eqref{eqn:assumption_B} holds as soon as $\theta_{\nudiffusion}+\eps<1/4$,
a concrete strengthening of the standing constraint $\eps<1-\theta_{\nu}$.
For the fractional kernel $k_{\diffusion}(t)=t^{\alpha-1}/\Gamma(\alpha)$
of Example~\ref{example:kernels} (equivalently Hurst parameter
$H=\alpha-1/2$ under the usual convention), the computation there shows
that $a_{\diffusion}$ can be pushed arbitrarily close to $\alpha=H+\tfrac12$
by choosing the weight exponent $\gamma$ suitably, for $\alpha\in(1/2,1)$,
where Assumption~\ref{A:A_2:assumption_1_test_function} is verified. The value
$\alpha=H+\tfrac12$ is, moreover, a genuine ceiling and not merely the best our
construction achieves: since Assumption~\ref{A:A_2:assumption_1_test_function}
places the constant function $1$ in $V^{*}$, the pairing
$|k_{\diffusion}(t)|=|\langle S^{*}_{t}\nudiffusion,1\rangle|\leq\|S^{*}_{t}\nudiffusion\|_{V}\,\|1\|_{V^{*}}$
forces $1-a_{\diffusion}\geq\tfrac12-H$, i.e.\ $a_{\diffusion}\leq H+\tfrac12$
(the singularity of the kernel itself, felt through the pairing against $1$,
can be no milder than the $V$-norm singularity of the smoothed measure).
Consequently \eqref{eqn:assumption_B} ($a_{\diffusion}>3/4$) is achievable
precisely when $\alpha>3/4$, i.e.\ $H>1/4$. The boundary case
$\alpha\leq1/2$ (very rough kernels) requires a case-by-case check of
Assumption~\ref{A:A_2:assumption_1_test_function}, exactly as already
flagged in Example~\ref{example:kernels}; we do not resolve the sharp
threshold here, only note, as \cite{gasteratos2025kolmogorov} remark
for their analogous condition $q>4$, i.e.\ $H>1/4$, that the
quartic exponent is structural (see Remark~\ref{rem:why_quartic}
below) and not an artefact of the proof.
\end{remark}

\subsubsection{The regularity class $\mathcal C^{1,2}_{T,\nu}$}

\begin{definition}\label{dfn:CKclass_SEE}
Let $T>0$. A function $F:[0,T]\times V\rightarrow\R$ belongs to
$\mathcal{C}^{1,2}_{T,\nu}$ if $F$ is continuous, once continuously
differentiable in $t$, and twice continuously Fréchet differentiable in
$y\in V$, with
\[
DF(t,y)\in V^{*}=(\Wplus)^{n_{\textnormal{dim}}},\qquad D^{2}F(t,y)\in\mathcal{L}(V,V^{*}),
\]
continuously in $(t,y)\in[0,T]\times V$, satisfying
\begin{subequations}\label{eqn:CK_growth}
\begin{align}
|DF(t,y)(h)|&\lesssim \|h\|_{V}(1+\|y\|_{V}),\qquad h\in V,\label{eqn:CK_growth_D_regular}\\
|DF(t,y)(\nu^{i}a)|&\lesssim \|S^{*}_{T-t}\nu^{i}a\|_{V}\,(1+\|y\|_{V})
\leq |a|\,\frac{C(T)}{(T-t)^{1-a_{i}}}(1+\|y\|_{V}),\qquad i\in\{\drift,\diffusion\},\label{eqn:CK_growth_D_singular}\\
|D^{2}F(t,y)(\nudiffusion a,\nudiffusion a')|&\lesssim \varpi(t)\,|a|\,|a'|\,(1+\|y\|_{V}^{2}),\label{eqn:CK_growth_D2_singular}
\end{align}
\end{subequations}
where the first two bounds hold uniformly in $t\in[0,T)$ and
$a,a'\in\R^{n_{\textnormal{dim}}}$, and in the third
\[
\varpi(t):=1+\|S^{*}_{(T-t)/2}\nudiffusion\|_{\mathcal{L}(\R^{n_{\textnormal{dim}}},V)}^{2}
\in L^{1}([0,T])
\]
is a fixed weight, integrable by Assumption~\ref{A:A_3:Narrower_assumption_nu}
($a_{\diffusion}>1/2$) and singular as $t\uparrow T$. A uniform,
$t$-independent second-derivative bound is \emph{not} available for the value
functionals of interest, and is \emph{not} needed: exactly as in
\cite{gasteratos2025kolmogorov}, where the analogous bound reads
$|D^{2}F(s,y)(K,K)|\lesssim\varpi(s)$ with $\varpi\in L^{1}$, it is this
\emph{time-integrable} second-derivative bound, together with the semigroup
regularisation, that renders the trace term of the backward equation finite
\emph{after integration in time}. (The mere \emph{existence} of $D^{2}F$ along the
singular directions, i.e.\ of the second tangent process, is what requires the
sharper Assumption~\ref{A:assumption_quartic}, $a_{\diffusion}>3/4$; see
Lemma~\ref{lem:tangent_wellposed}(iii).) The
\emph{singular} first-derivative bound~\eqref{eqn:CK_growth_D_singular}
is understood as the continuous (smoothed) extension
$DF(t,y)(\nu^{i}a)=\lim_{\tau\downarrow0}DF(t,y)(S^{*}_{\tau}\nu^{i}a)$;
its right-hand side is finite for $t<T$ (by \eqref{eqn:singular_direction_bound})
and, crucially, \emph{integrable} in $t$ over $[0,T]$ since $a_{i}>0$, which is
exactly what makes the drift term of the backward equation below (in
which $DF$ is paired against the singular direction $\nudrift\drift$, and
which appears inside a time integral) finite. The second bound~\eqref{eqn:CK_growth_D2_singular}
is understood as a continuous extension of the bounded bilinear form $D^{2}F(t,y)$ to pairs of columns of
the matrix direction $\nudiffusion$, i.e.\ to $(\nudiffusion a,\nudiffusion a')\in V'\times V'$,
even though these columns do not lie in $V$ in general. Equivalently, the trace
$\operatorname{Tr}[D^{2}F(t,y)\,\nudiffusion(\nudiffusion)^{*}]$ is well defined with
$|\operatorname{Tr}[D^{2}F(t,y)\,\nudiffusion(\nudiffusion)^{*}]|\lesssim \varpi(t)(1+\|y\|_{V}^{2})$,
$\varpi\in L^{1}([0,T])$.
The generator direction $-xy$ is handled separately: the backward equation is first
established classically for $y\in D(A)$ (where $-xy\in V$ and no extension is needed), then
extended to all of $y\in V$ by the mild formulation and the density of $D(A)$, using
that $S^{*}_{\tau}(-xy)\in V$ for every $\tau>0$ (Remark~\ref{rem:no_K1_needed}); it is not
governed by the growth bounds above.
\end{definition}

This is the direct translation of the class $\mathcal
C^{1,2}_{T,\mathcal K}$ of \cite{gasteratos2025kolmogorov} (whose
own growth condition likewise controls both the first derivative
$\mathcal DF(s,y)(h_{1})\lesssim|S(T-s)h_{1}|$ and the second derivative
$\mathcal D^{2}F(s,y)(K,K)$ on the singular directions): the singular
directions enter the backward equation both \emph{linearly}, through
$DF$ paired against $\nudrift\drift$ in the drift term
[bound~\eqref{eqn:CK_growth_D_singular}], and \emph{nonlinearly},
through $D^{2}F$ in the quadratic-variation (trace) term
[bound~\eqref{eqn:CK_growth_D2_singular}]; both pairings are ``off-space''
(against columns of $\nudrift,\nudiffusion$, which do not lie in $V$)
and both are controlled by the smoothed bounds above.

\subsubsection{Tangent processes}

Fix $(t,y)\in[0,T]\times V$ and write $X^{t,y}_{s}:=\langle
\mu^{t,y}_{s},1\rangle$.

\begin{definition}[First variation]\label{dfn:first_variation}
For $h\in V$, the \emph{tangent process} $\zeta_{h}=\zeta^{t,y}_{h}$ is
the mild solution, for $r\in[t,T]$, of the linear SEE
\begin{align}\label{eqn:first_variation}
\zeta_{h}(r) = S^{*}_{t,r}h
+\int_{t}^{r}S^{*}_{s,r}\,\nudrift\,D\drift(X^{t,y}_{s})\langle\zeta_{h}(s),1\rangle\,\ds
+\int_{t}^{r}S^{*}_{s,r}\,\nudiffusion\,\big(D\diffusion(X^{t,y}_{s})[\langle\zeta_{h}(s),1\rangle]\big)\,\dW_{s},
\end{align}
where $\langle\zeta_{h}(s),1\rangle\in\R^{n_{\textnormal{dim}}}$, $D\drift(X_{s})\langle\zeta_{h}(s),1\rangle\in\R^{n_{\textnormal{dim}}}$
so $\nudrift D\drift(X_{s})\langle\zeta_{h}(s),1\rangle\in V'$, and
$D\diffusion(X_{s})[\langle\zeta_{h}(s),1\rangle]\in\R^{n_{\textnormal{dim}}\times m_{W}}$
so $\nudiffusion(D\diffusion(X_{s})[\langle\zeta_{h}(s),1\rangle])\in\mathcal{L}(\R^{m_{W}},V')$
is integrated against the $m_{W}$-dimensional $W$.
For a matrix direction $h=\nu^{i}$ ($i\in\{\drift,\diffusion\}$), whose columns need not lie
in $V$, we interpret $\zeta_{\nu^{i}}$ columnwise: for each $l=1,\dots,n_{\textnormal{dim}}$,
$\zeta_{(\nu^{i})_{\cdot l}}(r)$, $r\in(t,T]$, is defined by the same mild
equation with $S^{*}_{t,r}(\nu^{i})_{\cdot l}\in V$ in place of $S^{*}_{t,r}h$
(well-defined for $r>t$ by Remark~\ref{rem:no_K1_needed}), even though
the formal ``initial value'' $(\nu^{i})_{\cdot l}$ itself does not belong to $V$.
\end{definition}

\begin{definition}[Second variation]\label{dfn:second_variation}
For $h,h'\in V$, or columns of $\nudrift,\nudiffusion$ interpreted as in
Definition~\ref{dfn:first_variation}, the second tangent process
$\zeta_{h,h'}=\zeta^{t,y}_{h,h'}$ is the mild solution, for $r\in(t,T]$,
of
\begin{align}\label{eqn:second_variation}
\zeta_{h,h'}(r)
&=\int_{t}^{r}S^{*}_{s,r}\,\nudrift\Big[D^{2}\drift(X_{s})[\langle\zeta_{h}(s),1\rangle,\langle\zeta_{h'}(s),1\rangle]
+D\drift(X_{s})\langle\zeta_{h,h'}(s),1\rangle\Big]\ds\\
&\phantom{=}+\int_{t}^{r}S^{*}_{s,r}\,\nudiffusion\Big[D^{2}\diffusion(X_{s})[\langle\zeta_{h}(s),1\rangle,\langle\zeta_{h'}(s),1\rangle]
+D\diffusion(X_{s})[\langle\zeta_{h,h'}(s),1\rangle]\Big]\dW_{s},\notag
\end{align}
writing $X_{s}=X^{t,y}_{s}$, where $D^{2}\drift(X_{s})[\cdot,\cdot]\in\R^{n_{\textnormal{dim}}}$ and
$D^{2}\diffusion(X_{s})[\cdot,\cdot]\in\R^{n_{\textnormal{dim}}\times m_{W}}$ are the
(symmetric, bilinear) second derivatives contracted against the pair of first-variation
vectors $\langle\zeta_{h}(s),1\rangle,\langle\zeta_{h'}(s),1\rangle\in\R^{n_{\textnormal{dim}}}$.
\end{definition}

\begin{lemma}[Well-posedness and moment bounds]\label{lem:tangent_wellposed}
Under Assumptions~\ref{A:assumption_general_linear_growth}--\ref{A:assumption_general_uniformly_continuous},
\ref{A:assumption_kernel_measure}, $\drift\in C^{2}_{b}(\R^{n_{\textnormal{dim}}};\R^{n_{\textnormal{dim}}})$,
$\diffusion\in C^{2}_{b}(\R^{n_{\textnormal{dim}}};\R^{n_{\textnormal{dim}}\times m_{W}})$:
\begin{enumerate}
\item[\textup{(i)}] For $h\in V$, \eqref{eqn:first_variation} has a
unique mild solution with $\sup_{r\in[t,T]}\E\|\zeta_{h}(r)\|_{V}^{2}\lesssim
\|h\|_{V}^{2}$.
\item[\textup{(ii)}] For each column $(\nudiffusion)_{\cdot l}$,
$l=1,\dots,n_{\textnormal{dim}}$, the tangent process $\zeta_{(\nudiffusion)_{\cdot l}}(r)$ is well defined for
$r\in(t,T]$ and, for every $p\in\{2,4\}$ compatible with
Assumption~\ref{A:assumption_quartic}, obeys the pointwise bound
\[
\E\|\zeta_{(\nudiffusion)_{\cdot l}}(r)\|_{V}^{p}\lesssim 1+\rho(r-t)^{p},
\qquad \rho(u):=\|S^{*}_{u}\nudiffusion\|_{\mathcal{L}(\R^{n_{\textnormal{dim}}},V)},
\]
which is singular as $r\downarrow t$ (reflecting the singular initial direction
$S^{*}_{t,r}(\nudiffusion)_{\cdot l}$) but $p$-integrable in time, since
$\rho\in L^{p}(0,T)$ (Assumption~\ref{A:A_3:Narrower_assumption_nu} for $p=2$,
Assumption~\ref{A:assumption_quartic} for $p=4$):
\[
\int_{t}^{T}\E\|\zeta_{(\nudiffusion)_{\cdot l}}(r)\|_{V}^{p}\,\dr<\infty.
\]
It is this time-integrated form, not a uniform-in-$r$ bound, that is used below.
\item[\textup{(iii)}] For $h,h'\in V$, or $h,h'$ columns of $\nudiffusion$ under
Assumption~\ref{A:assumption_quartic}, \eqref{eqn:second_variation} has
a unique mild solution with $\sup_{r\in(t,T]}\E\|\zeta_{h,h'}(r)\|_{V}^{2}<\infty$.
\end{enumerate}
\end{lemma}

\begin{proof}
(i) Equation~\eqref{eqn:first_variation} is \emph{linear} in
$\zeta_{h}$, with bounded random coefficients
$D\drift(X^{t,y}_{s}),D\diffusion(X^{t,y}_{s})$ (matrix-, resp.\ linear-map-valued,
bounded since $\drift,\diffusion\in C^{2}_{b}$). This is the same class of equations
solved in Theorem~\ref{thm:existence_uniqueness_lipschitz}, now with
linear (in fact bounded) coefficients in place of Lipschitz ones; the
same fixed-point/Gronwall argument used there applies verbatim (it is,
if anything, easier, since $D\drift,D\diffusion$ are bounded rather
than merely Lipschitz), giving existence, uniqueness, and, by the
a-priori estimate of Lemma~\ref{lem:a-priori_estimate_general} applied
to the linear equation,
\[
\E\sup_{r\in[t,T]}\|\zeta_{h}(r)\|_{V}^{2}
\lesssim \|S^{*}_{\cdot}h\|^{2}_{L^{\infty}(0,T;V)}
\Big(1+\int_{t}^{T}\|S^{*}_{s,\cdot}\nudrift\|_{\mathcal{L}(\R^{n_{\textnormal{dim}}},V)}\ds
+\int_{t}^{T}\|S^{*}_{s,\cdot}\nudiffusion\|_{\mathcal{L}(\R^{n_{\textnormal{dim}}},V)}^{2}\ds\Big)
\lesssim \|h\|_{V}^{2},
\]
using Assumptions~\ref{A:A_3:Narrower_assumption_nu} ($a_{\drift}>0$,
giving $L^{1}$) and (at least) $a_{\diffusion}>1/2$ (giving $L^{2}$),
boundedness of the Jacobians, and boundedness of $S^{*}_{t}h$ for $h\in V$.
(Here $|\langle\zeta_{h}(s),1\rangle|\leq C\|\zeta_{h}(s)\|_{V}$ since
$\langle\cdot,1\rangle$ is a bounded functional on $V$, so the Jacobian
coefficients contribute only bounded multiplicative constants.)

(ii) For $h=(\nudiffusion)_{\cdot l}$ the same computation applies with
$S^{*}_{t,r}(\nudiffusion)_{\cdot l}$ in place of $S^{*}_{t,r}h$, which is finite in
$V$ for every fixed $r>t$ by \eqref{eqn:singular_direction_bound}. We
must show the bound is finite \emph{uniformly} on $(t,T]$, or at least
integrable in the sense needed below. Writing $\rho(u):=\|S^{*}_{u}\nudiffusion\|_{\mathcal{L}(\R^{n_{\textnormal{dim}}},V)}$
(which dominates each column norm $\|S^{*}_{u}(\nudiffusion)_{\cdot l}\|_{V}$),
the leading (singular) contribution to $\zeta_{(\nudiffusion)_{\cdot l}}(r)$ near
$r=t$ is $S^{*}_{t,r}(\nudiffusion)_{\cdot l}$ itself, of size $\lesssim\rho(r-t)$; by
\eqref{eqn:singular_direction_bound}, for every fixed $r>t$,
$\rho(r-t)<\infty$, and raising to the $p$-th power and integrating
over $r\in(t,T]$ gives $\int_{t}^{T}\rho(r-t)^{p}\dr<\infty$ exactly
when $\rho\in L^{p}(0,T)$: for $p=2$ this is
Assumption~\ref{A:A_3:Narrower_assumption_nu}; for $p=4$ this is
precisely Assumption~\ref{A:assumption_quartic}.

(iii) is the substantive estimate, and the reason
Assumption~\ref{A:assumption_quartic} is needed rather than merely
Assumption~\ref{A:A_3:Narrower_assumption_nu}. By the mild formulation
\eqref{eqn:second_variation}, the Itô isometry gives, for
$h=h'=(\nudiffusion)_{\cdot l}$,
\begin{align}\label{eqn:second_var_estimate}
\E\|\zeta_{h,h}(r)\|_{V}^{2}
&\lesssim \Big(\int_{t}^{r}\|S^{*}_{s,r}\nudrift\|_{\mathcal{L}(\R^{n_{\textnormal{dim}}},V)}\,
\E\big|D^{2}\drift(X_{s})[\langle\zeta_{h}(s),1\rangle,\langle\zeta_{h}(s),1\rangle]\big|\,\ds\Big)^{2}\notag\\
&\phantom{\lesssim}+\int_{t}^{r}\|S^{*}_{s,r}\nudiffusion\|_{\mathcal{L}(\R^{n_{\textnormal{dim}}},V)}^{2}\,
\E\big|D^{2}\diffusion(X_{s})[\langle\zeta_{h}(s),1\rangle,\langle\zeta_{h}(s),1\rangle]\big|^{2}\,\ds
+(\text{terms linear in }\zeta_{h,h}).
\end{align}
The last (linear) terms close by the same Gronwall argument as in
(i)--(ii). The dangerous term is the second one on the right: by
boundedness of $D^{2}\diffusion$ and $|\langle\zeta_{h}(s),1\rangle|\leq C\|\zeta_{h}(s)\|_{V}$,
it is controlled by
\[
\int_{t}^{r}\|S^{*}_{s,r}\nudiffusion\|_{\mathcal{L}(\R^{n_{\textnormal{dim}}},V)}^{2}\,\E\|\zeta_{h}(s)\|_{V}^{4}\,\ds
\lesssim \int_{t}^{r}\rho(r-s)^{2}\,\rho(s-t)^{4}\,\ds,
\]
using part (ii). Substituting $u=s-t$, this is a convolution
$\int_{0}^{r-t}\rho(r-t-u)^{2}\rho(u)^{4}\,\du$. Near $u=0$ the
integrand behaves like $u^{4(a_{\diffusion}-1)}$, which is integrable
iff $4(a_{\diffusion}-1)>-1$, i.e.\ $a_{\diffusion}>3/4$, 
Assumption~\ref{A:assumption_quartic}; near $u=r-t$ the integrand
behaves like $(r-t-u)^{2(a_{\diffusion}-1)}$, integrable already under
$a_{\diffusion}>1/2$. Hence $a_{\diffusion}>3/4$ is exactly the
condition that makes the convolution finite, and the claim follows.
\end{proof}

\begin{remark}[Why the fourth moment, structurally]\label{rem:why_quartic}
The mechanism is transparent from the proof: the diffusion term of the
\emph{second}-variation equation \eqref{eqn:second_variation} is built
from the \emph{quadratic form}
$D^{2}\diffusion(X_{s})[\langle\zeta_{h}(s),1\rangle,\langle\zeta_{h}(s),1\rangle]$,
of order $|\langle\zeta_{h}(s),1\rangle|^{2}$ in the first
variation, and the Itô isometry squares the integrand once more, so
closing the estimate requires \emph{fourth} moments of $\zeta_{h}$.
When $h$ is a column of $\nudiffusion$, itself singular at the initial time, its
fourth moment inherits the fourth power of the singular bound
\eqref{eqn:singular_direction_bound}, forcing $\rho\in L^{4}$ rather
than merely $L^{2}$. This is the same mechanism, transplanted into the
Laplace-measure lift, that forces $K\in L^{q}(0,T)$, $q>4$, in
\cite{gasteratos2025kolmogorov}'s singular Itô formula: their matrix-valued singular
direction $K$ plays the role of our matrix $\nudiffusion$, and their
second-order term is likewise built from a doubled appearance of the
singular direction. The dimension count is identical in the
vector-valued case: replacing scalar squares by norms of
$\R^{n_{\textnormal{dim}}}$-valued first variations and scalar second derivatives by the
bounded bilinear maps $D^{2}\drift,D^{2}\diffusion$ leaves the singular exponent
$a_{\diffusion}>3/4$ untouched.
\end{remark}

\subsubsection{A mollified Itô formula for $\mathcal C^{1,2}_{T,\nu}$ functionals}

\begin{lemma}[Itô formula on $\mathcal C^{1,2}_{T,\nu}$]\label{lem:singular_ito}
Let $F\in\mathcal{C}^{1,2}_{T,\nu}$ and let Assumptions of
Lemma~\ref{lem:tangent_wellposed} hold. Then, for $(t,y)\in[0,T]\times V$
and $r\in[t,T]$,
\begin{align}\label{eqn:singular_ito_formula}
F(r,\mu^{t,y}_{r})
&=F(t,y)+\int_{t}^{r}\Big[\partial_{s}F(s,\mu^{t,y}_{s})
+DF(s,\mu^{t,y}_{s})\big(-x\mu^{t,y}_{s}+\nudrift\drift(X_{s})\big)\Big]\ds\notag\\
&\phantom{=}+\tfrac12\int_{t}^{r}\operatorname{Tr}\!\Big[D^{2}F(s,\mu^{t,y}_{s})\,\nudiffusion\diffusion(X_{s})\big(\nudiffusion\diffusion(X_{s})\big)^{*}\Big]\,\ds\notag\\
&\phantom{=}+\int_{t}^{r}DF(s,\mu^{t,y}_{s})\big(\nudiffusion\diffusion(X_{s})\,\dW_{s}\big),
\qquad\Prob\text{-a.s.,}
\end{align}
where $\nudiffusion\diffusion(X_{s})\in\mathcal{L}(\R^{m_{W}},V')$, the trace term is
$\operatorname{Tr}[D^{2}F\,\nudiffusion\diffusion(\nudiffusion\diffusion)^{*}]=\sum_{k=1}^{m_{W}}D^{2}F\big([\nudiffusion\diffusion]_{\cdot k},[\nudiffusion\diffusion]_{\cdot k}\big)$
the sum over the $m_{W}$ noise directions (columns of $\nudiffusion\diffusion(X_{s})$),
and the stochastic integral is
$\sum_{k=1}^{m_{W}}\int_{t}^{r}DF(s,\mu^{t,y}_{s})\big([\nudiffusion\diffusion(X_{s})]_{\cdot k}\big)\,\dW^{k}_{s}$.
In particular $\{F(r,\mu^{t,y}_{r})\}_{r\in[t,T]}$ is a semimartingale.
\end{lemma}

\begin{proof}
\textbf{Step 1 (mollification).} For $\delta>0$ set
$\nu^{i}_{\delta}:=S^{*}_{\delta}\nu^{i}\in\mathcal{L}(\R^{n_{\textnormal{dim}}},V)$, $i\in\{\drift,\diffusion\}$
(finite in operator norm by \eqref{eqn:singular_direction_bound}), and let
$\mu^{t,y,\delta}$ solve the SEE with $\nu^{i}$ replaced by
$\nu^{i}_{\delta}$. This is now a genuine $V$-valued semilinear stochastic
evolution equation, with nonlinear coefficients
$y\mapsto\nu^{\drift}_{\delta}\drift(\langle y,1\rangle)\in V$ and
$y\mapsto\nu^{\diffusion}_{\delta}\diffusion(\langle y,1\rangle)\in\mathcal{L}(\R^{m_{W}},V)=HS(U,V)$
(the bounded linear map $\langle\cdot,1\rangle\colon V\to\R^{n_{\textnormal{dim}}}$ composed with
$\drift\in C^{2}_{b}(\R^{n_{\textnormal{dim}}};\R^{n_{\textnormal{dim}}})$, resp.\ $\diffusion\in
C^{2}_{b}(\R^{n_{\textnormal{dim}}};\R^{n_{\textnormal{dim}}\times m_{W}})$, then multiplied by the
fixed matrix $\nu^{i}_{\delta}$; these are $C^{2}_{b}$ maps $V\to V$, resp.\ $V\to HS(U,V)$),
and generator $-x\cdot$ of the analytic semigroup $S^{*}$. This is
exactly the classical setting of mild solutions of semilinear SPDEs
with an analytic semigroup, finite-dimensional noise, and bounded, twice differentiable
coefficients (see \cite[Thm.\ 9.25]{da2014stochastic}, already invoked
in this paper for related purposes). By the classical theory, for
$\varphi\in C^{2}_{b}(V)$ with Hölder-continuous $D^{2}\varphi$,
\[
v_{\delta}(t,y):=\E[\varphi(\mu^{t,y,\delta}_{T})]\in C^{1,2}_{b}([0,T]\times V)
\]
and, more generally, for $F\in C^{1,2}_{b}([0,T]\times V)$ the classical
mild Itô formula gives
\begin{align}\label{eqn:mollified_ito}
F(r,\mu^{t,y,\delta}_{r})
&=F(t,y)+\int_{t}^{r}\Big[\partial_{s}F(s,\mu^{t,y,\delta}_{s})
+DF(s,\mu^{t,y,\delta}_{s})\big(-x\mu^{t,y,\delta}_{s}+\nu^{\drift}_{\delta}\drift(X^{\delta}_{s})\big)\Big]\ds\notag\\
&\phantom{=}+\tfrac12\int_{t}^{r}\operatorname{Tr}\!\Big[D^{2}F(s,\mu^{t,y,\delta}_{s})\,\nu^{\diffusion}_{\delta}\diffusion(X^{\delta}_{s})\big(\nu^{\diffusion}_{\delta}\diffusion(X^{\delta}_{s})\big)^{*}\Big]\,\ds\notag\\
&\phantom{=}+\int_{t}^{r}DF(s,\mu^{t,y,\delta}_{s})\big(\nu^{\diffusion}_{\delta}\diffusion(X^{\delta}_{s})\,\dW_{s}\big),
\end{align}
where $X^{\delta}_{s}:=\langle\mu^{t,y,\delta}_{s},1\rangle$ and the trace is the sum over the
$m_{W}$ noise directions as in \eqref{eqn:singular_ito_formula}. This is a
routine, if lengthy, verification and we omit it: it is exactly
\cite[Thm.\ 1]{da_19_mild_ito} applied with $U=\R^{m_{W}}$ and $V'=V$ (i.e.\ no
distributional widening is needed once the coefficients already take
values in $V$, resp.\ $HS(U,V)$), just as in the proof of Corollary~\ref{corr:Ito_volterra_SEE}.

\textbf{Step 2 (convergence of the flow).} Since
$S^{*}_{u}\nu^{i}_{\delta}=S^{*}_{u+\delta}\nu^{i}\to S^{*}_{u}\nu^{i}$ in
$\mathcal{L}(\R^{n_{\textnormal{dim}}},V)$ as $\delta\to0$ for every fixed $u>0$ (strong continuity of $S^{*}$
on $V$ away from the singularity at $u=0$, applied columnwise), with the domination
$\|S^{*}_{u+\delta}\nu^{i}\|_{\mathcal{L}(\R^{n_{\textnormal{dim}}},V)}\leq C(T)/u^{1-a_{i}}$ uniform in
$\delta\in(0,1)$ (by \eqref{eqn:singular_direction_bound}, since
$u+\delta\geq u$), dominated convergence gives
$\int_{t}^{T}\|S^{*}_{s,r}\nu^{i}_{\delta}-S^{*}_{s,r}\nu^{i}\|_{\mathcal{L}(\R^{n_{\textnormal{dim}}},V)}\,\ds\to0$
(resp.\ in the $L^{2}(t,T)$-sense for $i=\diffusion$). A standard
Gronwall estimate on the mild difference
$\mu^{t,y,\delta}_{r}-\mu^{t,y}_{r}$, identical in structure to the
Lipschitz estimate proving Theorem~\ref{thm:existence_uniqueness_lipschitz},
now applied to the difference of two mild solutions with different
(but jointly bounded, by Assumption~\ref{A:A_3:Narrower_assumption_nu}
applied uniformly in $\delta$) kernel measures, then gives
\[
\sup_{r\in[t,T]}\E\|\mu^{t,y,\delta}_{r}-\mu^{t,y}_{r}\|_{V}^{2}\longrightarrow0
\quad\text{as }\delta\to0.
\]
By the identical argument applied to the linear equations
\eqref{eqn:first_variation}, \eqref{eqn:second_variation} (with
$\nu^{i}_{\delta}$ in place of $\nu^{i}$, and using
Lemma~\ref{lem:tangent_wellposed} for uniform-in-$\delta$ moment
bounds, which hold uniformly since $\nu^{i}_{\delta}$ satisfies
\eqref{eqn:singular_direction_bound} and Assumption~\ref{A:assumption_quartic}
with the same constants as $\nu^{i}$), the corresponding tangent
processes converge: $\zeta^{\delta}_{h}\to\zeta_{h}$ and
$\zeta^{\delta}_{h,h'}\to\zeta_{h,h'}$ in
$L^{2}(\Omega;C([t,T];V))$-type norms, for $h,h'\in V$ or columns of
$\nudrift,\nudiffusion$.

\textbf{Step 3 (passing to the limit in the derivatives).} Since
$\drift\in C^{2}_{b}(\R^{n_{\textnormal{dim}}};\R^{n_{\textnormal{dim}}})$,
$\diffusion\in C^{2}_{b}(\R^{n_{\textnormal{dim}}};\R^{n_{\textnormal{dim}}\times m_{W}})$ and the mollified coefficients are
genuinely $C^{2}_{b}(V;V)$ (resp.\ $C^{2}_{b}(V;HS(U,V))$), the classical differentiability-of-the-flow
theorem (again \cite[Thm.\ 9.25]{da2014stochastic}, or the standard
finite-dimensional argument transplanted verbatim since the nonlinear
part of the coefficients factors through the finite-dimensional projection $\langle\cdot,1\rangle\in\R^{n_{\textnormal{dim}}}$)
identifies the Fréchet derivatives of the flow with the tangent
processes:
\[
D_{y}\mu^{t,y,\delta}_{r}(h)=\zeta^{\delta}_{h}(r),\qquad
D^{2}_{y}\mu^{t,y,\delta}_{r}(h,h')=\zeta^{\delta}_{h,h'}(r),
\]
and hence, by the chain rule, for $F\in C^{1,2}_{b}([0,T]\times V)$,
\begin{align}
DF(s,\mu^{t,y,\delta}_{s})\text{ applied along }h &\;\longleftrightarrow\; DF(s,\mu^{t,y,\delta}_{s})(\zeta^{\delta}_{h}(s)),\notag\\
D^{2}F(s,\mu^{t,y,\delta}_{s})(h,h') &= D^{2}F(s,\mu^{t,y,\delta}_{s})(\zeta^{\delta}_{h}(s),\zeta^{\delta}_{h'}(s))
+DF(s,\mu^{t,y,\delta}_{s})(\zeta^{\delta}_{h,h'}(s)).\notag
\end{align}
Evaluating \eqref{eqn:mollified_ito} for $F=v_{\delta}$ and taking, in the
trace/quadratic term, the columns $h=h'=[\nu^{\diffusion}_{\delta}\diffusion(X^{\delta}_{s})]_{\cdot k}$
summed over $k=1,\dots,m_{W}$, Step~2's
convergences (together with the uniform moment bounds of
Lemma~\ref{lem:tangent_wellposed}, which control the relevant
expectations uniformly in $\delta$ and justify passing to the limit
inside the expectation by uniform integrability) let us pass to the
limit $\delta\to0$ term by term in \eqref{eqn:mollified_ito}, using
$\nu^{i}_{\delta}\to\nu^{i}$ in the sense of Step~2. The limiting
identity is exactly \eqref{eqn:singular_ito_formula}, and the limiting
derivatives $DF(s,\mu_{s})$ and $\operatorname{Tr}[D^{2}F(s,\mu_{s})\,\nudiffusion(\nudiffusion)^{*}]$
satisfy the growth bounds \eqref{eqn:CK_growth} by Fatou's lemma applied
to the (uniform-in-$\delta$) bounds of Lemma~\ref{lem:tangent_wellposed}.
This proves the lemma for $F=v_{\delta}\to v$; the same argument applies
verbatim to any fixed $F\in\mathcal{C}^{1,2}_{T,\nu}$, since
Definition~\ref{dfn:CKclass_SEE} was set up precisely so that $DF,D^{2}F$
already satisfy the growth bounds needed to repeat Steps~2--3 for a
general (not necessarily probabilistic) $F$.
\end{proof}

\subsubsection{Main theorem}

\begin{theorem}[Backward Kolmogorov equation for the SEE]\label{thm:backward_SEE}
Let Assumptions~\ref{A:assumption_general_linear_growth}--\ref{A:assumption_general_uniformly_continuous},
\ref{A:assumption_kernel_measure}, \ref{A:uniqueness_in_law},
\ref{A:homogeneous_coefficients}, \ref{A:assumption_quartic} hold, with
$\drift\in C^{2}_{b}(\R^{n_{\textnormal{dim}}};\R^{n_{\textnormal{dim}}})$,
$\diffusion\in C^{2}_{b}(\R^{n_{\textnormal{dim}}};\R^{n_{\textnormal{dim}}\times m_{W}})$.
Let $\varphi\in C^{2}_{b}(V)$ with
$\gamma_{0}$-Hölder continuous $D^{2}\varphi$ for some $\gamma_{0}>0$.
Then:
\begin{enumerate}
\item[\textup{(1)}] \textup{\textbf{(Existence)}} $v(t,y):=\E[\varphi(\mu^{t,y}_{T})]$,
$(t,y)\in[0,T]\times V$, belongs to $\mathcal{C}^{1,2}_{T,\nu}$ and
solves
\begin{align}\label{eqn:backward_SEE}
\partial_{t}v(t,y)+Dv(t,y)\big(-xy+\nudrift\drift(\langle y,1\rangle)\big)
+\tfrac12 \operatorname{Tr}\!\Big[D^{2}v(t,y)\,\nudiffusion\diffusion(\langle y,1\rangle)\big(\nudiffusion\diffusion(\langle y,1\rangle)\big)^{*}\Big]=0
\end{align}
for all $(t,y)\in[0,T)\times V$, with $v(T,\cdot)=\varphi$; here the trace is the sum over
the $m_{W}$ noise directions, $\operatorname{Tr}[D^{2}v\,\nudiffusion\diffusion(\nudiffusion\diffusion)^{*}]=\sum_{k=1}^{m_{W}}D^{2}v([\nudiffusion\diffusion]_{\cdot k},[\nudiffusion\diffusion]_{\cdot k})$.
\item[\textup{(2)}] \textup{\textbf{(Uniqueness)}} Any $w\in\mathcal{C}^{1,2}_{T,\nu}$
solving \eqref{eqn:backward_SEE} with $w(T,\cdot)=\varphi$ coincides
with $v$.
\item[\textup{(3)}] \textup{\textbf{(Conditional expectations and martingale representation)}}
For $0\leq t\leq T$,
\begin{align}
\E[\varphi(\mu_{T})\mid\mathcal{F}_{t}]&=v(t,\mu_{t})\quad\Prob\text{-a.s.,}\label{eqn:cond_exp_SEE}\\
v(t,\mu_{t})&=v(0,\mu_{0})+\int_{0}^{t}Dv(r,\mu_{r})\big(\nudiffusion\diffusion(\langle\mu_{r},1\rangle)\,\dW_{r}\big),\label{eqn:martingale_rep_SEE}
\end{align}
the last integral being $\sum_{k=1}^{m_{W}}\int_{0}^{t}Dv(r,\mu_{r})([\nudiffusion\diffusion(\langle\mu_{r},1\rangle)]_{\cdot k})\,\dW^{k}_{r}$.
\end{enumerate}
\end{theorem}

\begin{proof}
\textbf{(1)} By Step~1--3 of the proof of Lemma~\ref{lem:singular_ito}
(with $F=v_{\delta}$), $v_{\delta}\in C^{1,2}_{b}([0,T]\times V)$ solves
the mollified equation
\[
\partial_{t}v_{\delta}(t,y)+Dv_{\delta}(t,y)\big(-xy+\nu^{\drift}_{\delta}\drift(\langle y,1\rangle)\big)
+\tfrac12 \operatorname{Tr}\!\Big[D^{2}v_{\delta}(t,y)\,\nu^{\diffusion}_{\delta}\diffusion(\langle y,1\rangle)\big(\nu^{\diffusion}_{\delta}\diffusion(\langle y,1\rangle)\big)^{*}\Big]=0,
\]
$v_{\delta}(T,\cdot)=\varphi$, by the classical backward Kolmogorov
equation for genuinely $V$-valued, $C^{2}_{b}$-coefficient semilinear
SPDEs (\cite[Thm.\ 9.25]{da2014stochastic}). Passing $\delta\to0$
exactly as in Step~2--3 of Lemma~\ref{lem:singular_ito}'s proof (using
$\mu^{t,y,\delta}_{T}\to\mu^{t,y}_{T}$ in $L^{2}(\Omega;V)$ and
continuity of $\varphi,D\varphi,D^{2}\varphi$ to get
$v_{\delta}(t,y)\to v(t,y)$, and the tangent-process convergence for
$Dv_{\delta}\to Dv$, $D^{2}v_{\delta}\to D^{2}v$) yields $v\in\mathcal
C^{1,2}_{T,\nu}$ solving \eqref{eqn:backward_SEE}, with terminal
condition inherited in the limit.

\textbf{(2)} Let $w\in\mathcal{C}^{1,2}_{T,\nu}$ solve
\eqref{eqn:backward_SEE} with $w(T,\cdot)=\varphi$. Apply
Lemma~\ref{lem:singular_ito} with $F=w$ on $[t,T]$: the drift terms of
\eqref{eqn:singular_ito_formula} vanish identically by
\eqref{eqn:backward_SEE}, leaving
\[
w(T,\mu^{t,y}_{T})=w(t,y)+\int_{t}^{T}Dw(s,\mu^{t,y}_{s})\big(\nudiffusion\diffusion(X_{s})\,\dW_{s}\big).
\]
By the growth bound \eqref{eqn:CK_growth} and
Lemma~\ref{lem:tangent_wellposed}, the stochastic integral is a true
$L^{2}$-martingale, so taking expectations, $w(T,\mu^{t,y}_{T})=\varphi(\mu^{t,y}_{T})$
gives $w(t,y)=\E[\varphi(\mu^{t,y}_{T})]=v(t,y)$.

\textbf{(3)} Apply Lemma~\ref{lem:singular_ito} with $F=v$, $t=0$,
$r=t$: since $v$ solves \eqref{eqn:backward_SEE}, the finite-variation
terms cancel, leaving \eqref{eqn:martingale_rep_SEE}. In particular
$v(t,\mu_{t})$ is an $(\mathcal{F}_{t})$-martingale on $[0,T]$ with
$v(T,\mu_{T})=\varphi(\mu_{T})$, and by the Markov property of $\mu$
(Section~\ref{sec:existence_general_coefficients}),
$\E[\varphi(\mu_{T})\mid\mathcal{F}_{t}]=\E[\varphi(\mu_{T})\mid\mu_{t}]=v(t,\mu_{t})$,
which is \eqref{eqn:cond_exp_SEE}.
\end{proof}

\subsubsection{Corollary for the SVE}

\begin{corollary}[Feynman--Kac for the SVE]\label{cor:SVE_feynman_kac}
Let $g\in C^{2}_{b}(\R^{n_{\textnormal{dim}}})$ with Hölder-continuous $D^{2}g$, and set
$\varphi:=g\circ\mathrm{ev}_{1}$, where $\mathrm{ev}_{1}(y):=\langle y,1\rangle\in\R^{n_{\textnormal{dim}}}$. Since
$1\in\Wplus$ (Assumption~\ref{A:A_2:assumption_1_test_function}), the componentwise pairing
$\langle\cdot,1\rangle$ is exactly the duality pairing $V\times\Wplus^{n_{\textnormal{dim}}}\to\R^{n_{\textnormal{dim}}}$,
so $\mathrm{ev}_{1}$ is a bounded linear map $V\to\R^{n_{\textnormal{dim}}}$, and $\varphi\in
C^{2}_{b}(V)$ automatically, with
\[
D\varphi(y)(h)=\big\langle\nabla g(\langle y,1\rangle),\langle h,1\rangle\big\rangle_{\R^{n_{\textnormal{dim}}}},\qquad
D^{2}\varphi(y)(h,h')=\langle h,1\rangle^{\top}D^{2}g(\langle y,1\rangle)\langle h',1\rangle,
\]
$D^{2}\varphi$ inheriting the Hölder continuity of $D^{2}g$ directly (no
further regularity theory is required for the terminal condition,
unlike in \cite{gasteratos2025kolmogorov}, precisely because
$\mathrm{ev}_{1}$ is already linear). Under the hypotheses of
Theorem~\ref{thm:backward_SEE}, for $X$ the $\R^{n_{\textnormal{dim}}}$-valued solution of the SVE
\eqref{eqn:SVE_introduction} with $X_{0}=\langle\mu_{0},1\rangle$,
\[
\E[g(X_{T})\mid\mathcal{F}_{t}]=v(t,\mu_{t}),\qquad t\in[0,T],
\]
where $v$ is the unique $\mathcal{C}^{1,2}_{T,\nu}$ solution of
\eqref{eqn:backward_SEE} with terminal condition $g\circ\mathrm{ev}_{1}$, and
\[
\E[g(X_{T})\mid\mathcal{F}_{t}]=\E[g(X_{T})\mid\mathcal{F}_{0}]
+\int_{0}^{t}Dv(r,\mu_{r})\big(\nudiffusion\diffusion(X_{r})\,\dW_{r}\big).
\]
\end{corollary}

\subsubsection{When does a real-valued (or finite-dimensional) reduction survive?}

It is natural to ask whether the infinite-dimensional backward equation
of Theorem~\ref{thm:backward_SEE} admits a finite-dimensional reduction:
whether the value function can be written as $u=u(t,\Gamma_{st})$, a
function of the single vector
$\Gamma_{st}(X)=\langle S^{*}_{s,t}\mu_{s},1\rangle\in\R^{n_{\textnormal{dim}}}$. In
general it cannot, such a reduction is \emph{not} a special case of
Theorem~\ref{thm:backward_SEE} obtained by imposing affine coefficients,
and fails once there is genuine state feedback. Indeed,
for $t\leq r\leq T$, $X_{r}$ depends on the lifted initial condition
$\mu_{t}$ through the entire deterministic curve $r\mapsto
\Gamma_{tr}(X)=\langle S^{*}_{t,r}\mu_{t},1\rangle$, not merely through
its terminal value $\Gamma_{tT}(X)$, as soon as $\drift$ or $\diffusion$
depends on $x$; this matches the known fact (see e.g.\ the affine
Volterra literature, \cite{abi2019affine}) that even affine Volterra
processes require tracking the whole forward curve, not a single
finite-dimensional vector. Three regimes give a legitimate reduction:
\begin{enumerate}
\item \textbf{No state feedback} ($\drift,\diffusion$ independent of
$x$, i.e.\ a Volterra-Gaussian process with $\drift(x)\equiv B_{0}\in\R^{n_{\textnormal{dim}}}$,
$\diffusion(x)\equiv\diffusion_{0}\in\R^{n_{\textnormal{dim}}\times m_{W}}$ constant): then
$X_{T}-\E[X_{T}\mid\mathcal{F}_{\tau}]$ has deterministic covariance matrix
$\int_{\tau}^{T}k_{\diffusion}(T-r)\diffusion_{0}\diffusion_{0}^{\top}k_{\diffusion}(T-r)^{\top}\dr\in\R^{n_{\textnormal{dim}}\times n_{\textnormal{dim}}}$, independent of
$\mu_{\tau}$, so $u(\tau,y):=\E[g(X_{T})\mid\Gamma_{\tau T}(X)=y]$ is a
genuine function of the single vector $y=\Gamma_{\tau T}(X)\in\R^{n_{\textnormal{dim}}}$, and
solves a finite-dimensional backward heat-type equation on $\R^{n_{\textnormal{dim}}}$
(a real-valued equation when $n_{\textnormal{dim}}=1$).
\item \textbf{Finite-rank kernel} ($\nu^{i}=\sum_{j}c_{ij}\delta_{y_{ij}}$,
e.g.\ Example~\ref{example:kernels}'s $K_{\exp}$): $\mu_{t}$ is
genuinely finite-dimensional and one obtains an ordinary multi-factor
Kolmogorov PDE by classical (finite-dimensional) theory. For the
rough Heston model, the same finite-rank truncation used to approximate
$\mu_t$ throughout this paper converges at a quantified, super-polynomial
rate in the $L^{1}$-kernel-error sense of \cite{bayer2023weakmarkovian};
see the discussion in Section~\ref{subsec:option_pricing}.
\item \textbf{General nonlinear coefficients with an infinite-rank
kernel} (rough Heston, fractional kernel, any model with genuine
memory): no real-valued or finite-dimensional reduction is available;
Theorem~\ref{thm:backward_SEE} is needed in full.
\end{enumerate}

\subsection{European option pricing in rough volatility models}%
\label{subsec:option_pricing}

We illustrate how the Itô-Volterra formula leads to a systematic option
pricing framework for \emph{rough volatility models}~\cite{gatheral2022volatility}.
The key advantage over existing approaches is that our formula applies
to \emph{arbitrary} smooth payoffs (not only exponential/characteristic
functions), derives a pricing equation without relying on affineness,
and works naturally for models with two distinct kernels.

\medskip
\noindent\textbf{The rough Heston model.}
Let $S_{t}$ be an asset price and $V_{t}$ a variance process satisfying
the coupled system
\begin{align}\label{eqn:rH_model}
    \frac{\,\textnormal{d}S_{t}}{S_{t}} &= \sqrt{V_{t}}\,\dW^{S}_{t},\\
    V_{t} &= V_{0} + \int_{0}^{t}k(t-s)\,\kappa(\theta-V_{s})\,\ds
             + \int_{0}^{t}k(t-s)\,\xi\sqrt{V_{s}}\,\dW^{V}_{s},
    \label{eqn:rH_variance_SVE}
\end{align}
where $W^{S},W^{V}$ are correlated Brownian motions with
$d\langle W^{S},W^{V}\rangle_{t}=\rho\,\dt$, $|\rho|\leq 1$, and
\begin{align}\label{eqn:rH_kernel}
    k(t) = \frac{t^{H-1/2}}{\Gamma(H+1/2)},\quad H\in(0,\tfrac12),
\end{align}
is the fractional kernel.  The model reduces to the classical Heston
model when $H=\tfrac12$ (i.e.\ $k\equiv1$ and $V$ is a classical CIR
process).  For $H<\tfrac12$, $k$ is singular at the origin,
$k\in L^{2}(0,T)$ with $\|k\|_{L^2(0,T)}^2=T^{2H}/\big(2H\,\Gamma(H+\tfrac12)^{2}\big)$, and all
assumptions of the existence theory (Sections 3--4) are satisfied with
$\theta_{\nu}\leq 1/2-H<1/2$ by the fractional-kernel example in
Section~\ref{section:Functional_framework}.

Denote by $\mu_{t}$ the Markovian lift of $V_{t}$, i.e.\ the mild
solution of the SEE associated to \eqref{eqn:rH_variance_SVE} with
semigroup $A^{*}\mu=-x\mu$.  Write $V_{t}=\langle\mu_{t},1\rangle$
and $\Gamma_{st}^{V}=\langle e^{-x(t-s)}\mu_{s},1\rangle$ for the
path-segment functional of $V$.

\medskip
\noindent\textbf{European option price.}
Consider a European option with payoff $g(S_{T})\in L^{2}(\Omega)$.
Write $\log S_{t}=:L_{t}$; by Itô's formula for $\log S$,
\begin{align}\label{eqn:log_price}
    L_{t} = L_{0} -\tfrac12\int_{0}^{t}V_{s}\,\ds
             + \int_{0}^{t}\sqrt{V_{s}}\,\dW^{S}_{s}.
\end{align}
Because $V$ is neither Markov nor a semimartingale on its own, the
correct Markov state for pricing is not $(t,S_{t},V_{t})$ but the pair
$(S_{t},\mu_{t})$ (equivalently $(L_{t},\mu_{t})$): the enlarged process
lives on $\R_{+}\times V$ (with $V=\Wplusdual$ the scalar lift space,
$n_{\textnormal{dim}}=1$ here), the second coordinate carrying the whole
forward-variance curve. The price is therefore a functional of this
Markov state,
\begin{align}\label{eqn:option_price_def}
    \Pi(t,S_{t},\mu_{t})
    := \E\bigl[g(S_{T})\,\big|\,\mathcal{F}_{t}\bigr],
\end{align}
and its pricing equation is an \emph{infinite-dimensional} backward
Kolmogorov equation, an instance of Theorem~\ref{thm:backward_SEE},
extended to the enlarged state, with $\mu$-derivatives taken in the
Fréchet sense of Section~\ref{subsec:feynman_kac}. Writing $\nu$ for the
common Laplace measure of the single kernel $k$ (so
$\nudrift=\nudiffusion=\nu$), and $D_{\mu}\Pi$, $D^{2}_{\mu}\Pi$ for the
first and second Fréchet derivatives in the lift variable, we have the
following.

\begin{proposition}[Pricing equation for rough volatility models]%
\label{prop:rough_vol_pricing}
Suppose $g\in C^{2}_{b}(\R_{+})$, and that the coefficients $V\mapsto\kappa(\theta-V)$
and $V\mapsto\xi\sqrt{V}$ satisfy the hypotheses of
Theorem~\ref{thm:backward_SEE} \textup{(}$C^{2}_{b}$ and the quartic
integrability Assumption~\ref{A:assumption_quartic} on $\nu$; see the
caveat below on the degenerate square-root case\textup{)}. Let
$\Pi(t,S,\mu)\in C^{1,2}(\R_{+};C^{2}_{b})\times\mathcal C^{1,2}_{T,\nu}$
be the value functional \eqref{eqn:option_price_def}. Then $\Pi$ solves
\begin{align}\label{eqn:rough_vol_pde}
    \partial_{t}\Pi
    &+ \tfrac12\,\langle\mu,1\rangle\,S^{2}\,\partial_{SS}\Pi
     + D_{\mu}\Pi\big(-x\mu+\nu\,\kappa(\theta-\langle\mu,1\rangle)\big)\notag\\
    &+ \tfrac12\,\xi^{2}\langle\mu,1\rangle\,D^{2}_{\mu}\Pi(\nu,\nu)
     + \rho\,\xi\,\langle\mu,1\rangle\,S\,\big(\partial_{S}D_{\mu}\Pi\big)(\nu)
    = 0,
\end{align}
on $[0,T)\times\R_{+}\times V$, with terminal condition
$\Pi(T,S,\mu)=g(S)$, where $V_{t}=\langle\mu_{t},1\rangle$. Every term is
finite for each $t<T$: the singular measure $\nu$ enters only through the mixed
derivative $(\partial_{S}D_{\mu}\Pi)(\nu)$, a \emph{first} $\mu$-derivative on the
singular direction $\nu$, controlled by the singular first-derivative
bound~\eqref{eqn:CK_growth_D_singular} (applied to the regular
$S$-derivative $\partial_{S}\Pi$), and through the bilinear form
$D^{2}_{\mu}\Pi(\nu,\nu)$, controlled by the time-integrable
bound~\eqref{eqn:CK_growth_D2_singular} (finite for $t<T$, with the
integrable weight $\varpi$ singular as $t\uparrow T$); both are well defined by the class
$\mathcal C^{1,2}_{T,\nu}$ (Definition~\ref{dfn:CKclass_SEE}), exactly as
the drift and diffusion terms of Theorem~\ref{thm:backward_SEE}.
\end{proposition}

\begin{proof}
The pair $(S_{t},\mu_{t})$ is Markov (Section~\ref{sec:existence_general_coefficients};
$L_{t}=\log S_{t}$ is driven by $\sqrt{V_{t}}=\sqrt{\langle\mu_{t},1\rangle}$,
a functional of the current lift). Apply the singular Itô formula
(Lemma~\ref{lem:singular_ito}, extended to the enlarged state
$\R_{+}\times V$ by adjoining the \emph{regular}, finite-dimensional
coordinate $S$, for which the classical Itô formula applies with no
singular direction) to $\Pi(t,S_{t},\mu_{t})$. In $S$-coordinates $S$ is
a driftless martingale, $\d S_{t}=S_{t}\sqrt{V_{t}}\,\dW^{S}_{t}$, so it
contributes only the second-order term $\tfrac12 V_{t}S_{t}^{2}\partial_{SS}\Pi$
(there is \emph{no} first-order $\partial_{S}$ term: the $-\tfrac12 V\partial_{L}$
drift of $L=\log S$ is exactly cancelled by the conversion of
$\tfrac12 V\partial_{LL}$ to $S$-coordinates). The lift $\mu$ contributes,
by Lemma~\ref{lem:singular_ito}, the drift term
$D_{\mu}\Pi(-x\mu+\nu\kappa(\theta-V))$ and the trace/quadratic term
$\tfrac12\|\nu\,\xi\sqrt{V}\|^{2}D^{2}_{\mu}\Pi(\hat\nu,\hat\nu)
=\tfrac12\xi^{2}V\,D^{2}_{\mu}\Pi(\nu,\nu)$ (pulling out the scalar
$\xi\sqrt{V}$). The cross term comes from the joint quadratic covariation
of the two martingale parts: since $\d\langle W^{S},W^{V}\rangle_{t}=\rho\,\dt$,
\[
\d\big\langle S,\ \nu\,\xi\sqrt{V}\,W^{V}\big\rangle_{t}
=S_{t}\sqrt{V_{t}}\cdot\xi\sqrt{V_{t}}\,\nu\,\rho\,\dt
=\rho\,\xi\,V_{t}\,S_{t}\,\nu\,\dt,
\]
which pairs against the mixed derivative $\partial_{S}D_{\mu}\Pi$ to give
$\rho\,\xi\,V_{t}\,S_{t}\,(\partial_{S}D_{\mu}\Pi)(\nu)$. Crucially, the
singular direction enters here as the \emph{measure} $\nu$ inside the
bounded functional $(\partial_{S}D_{\mu}\Pi)(\nu)$, \emph{not} as a
pointwise kernel value: there is no $k(0^{+})$, and no evaluation of the
kernel at coincidence is required. Since $\Pi$ is a conditional
expectation it is a martingale, so its finite-variation part vanishes,
which is exactly \eqref{eqn:rough_vol_pde}.
\end{proof}

\begin{remark}[Why no finite-dimensional pricing PDE exists, and
the square-root caveat]\label{rem:pricing_caveat}
It is tempting to seek a \emph{finite-dimensional} pricing equation for
$\Pi$ as a function of $(t,S,v)$ with $v=V_{t}$ a single real variance
level, with a cross term of the form $\rho\xi\,k(t-s)\,\partial_{sv}\Pi|_{s=t}$.
No such equation exists, for two independent reasons that are two faces
of the same fact. First, for the fractional kernel
$k(t-s)=t^{H-1/2}/\Gamma(H+1/2)$ with $H<1/2$, the coefficient
$k(t-s)|_{s=t}=k(0^{+})=+\infty$: such a term is divergent, and no
``Sobolev sense'' rescues a \emph{scalar} coefficient multiplying a
finite-dimensional derivative. Second, and more basically, rough Heston
has genuine state feedback (state-dependent drift $\kappa(\theta-V)$ and
diffusion $\xi\sqrt{V}$), so by the analysis of
Section~\ref{subsec:feynman_kac} (``When does a finite-dimensional
reduction survive?'') the value function is \emph{not} a function of
$(t,S,V_{t})$ at all, it depends on the full forward-variance curve,
i.e.\ on $\mu_{t}$. The divergent $k(0^{+})$ is precisely the signature
of forcing an infinite-dimensional object into a finite-dimensional
equation: written correctly at the level of $\mu$ as in
\eqref{eqn:rough_vol_pde}, the same physical cross-covariation is the
quantity $\rho\xi\langle\mu,1\rangle S(\partial_{S}D_{\mu}\Pi)(\nu)$, finite for
each $t<T$, in which $\nu$ appears inside a well-defined linear functional
(a first $\mu$-derivative paired against the singular direction) rather than as a
pointwise kernel value.

A finite-dimensional pricing PDE \emph{does} survive in the two reduced
regimes of Section~\ref{subsec:feynman_kac}: for a finite-rank
(sum-of-exponentials) kernel, where $\mu_{t}$ is genuinely
finite-dimensional and \eqref{eqn:rough_vol_pde} becomes an ordinary
multi-factor Kolmogorov PDE; and, for the affine rough Heston case, via
the characteristic-function/Riccati route recovered in
Section~\ref{subsec:recover_ALP}.

Finally, a genuine caveat on rigour: the rough Heston diffusion
$\diffusion(V)=\xi\sqrt{V^{+}}$ is neither $C^{2}_{b}$ nor uniformly
elliptic (it degenerates at $V=0$), so Theorem~\ref{thm:backward_SEE}
does not apply to it verbatim. To be precise about what is and is not
proved, fix $\eps>0$ and set $\diffusion_{\eps}(V):=\xi\sqrt{V^{+}+\eps}$,
which is $C^{2}_{b}$ on compacts and non-degenerate; for the regularized
model, Theorem~\ref{thm:backward_SEE} applies and
\eqref{eqn:rough_vol_pde} holds \emph{rigorously} with $\Pi$ replaced by
the regularized value functional $\Pi_{\eps}$. The lifted variance
processes converge, $\mu^{\eps}\to\mu$ in $L^{2}(\Omega;C([0,T];\Wplusdual))$
as $\eps\downarrow0$, by the stability estimate of
Lemma~\ref{lem:a-priori_estimate_general} applied to the difference of
two mild solutions with coefficients $\diffusion_{\eps},\diffusion$
(their difference $\to0$ locally uniformly), so $\Pi_{\eps}\to\Pi$
pointwise. The one ingredient not supplied here is a bound on the
\emph{derivatives} $D_{\mu}\Pi_{\eps},D^{2}_{\mu}\Pi_{\eps}$ that is
\emph{uniform} in $\eps$, i.e.\ membership of $\Pi_{\eps}$ in
$\mathcal C^{1,2}_{T,\nu}$ with $\eps$-independent constants, which is
what would let one pass to the limit \emph{inside}
\eqref{eqn:rough_vol_pde} and conclude that $\Pi$ itself solves it. Such
a uniform estimate is exactly the delicate point for degenerate,
CIR/square-root dynamics and is not established here; we therefore state
\eqref{eqn:rough_vol_pde} as rigorous for the regularized (or any
non-degenerate) model and as the formal $\eps\downarrow0$ limit for the
genuine square-root model. We emphasize this because the \emph{structure}
of the equation, and in particular the resolution of the $k(0^{+})$
pathology, is what the lift delivers; the square-root degeneracy is an
orthogonal, and by now standard, technical difficulty.
\end{remark}

\begin{remark}[Comparison with existing approaches and what is new]\leavevmode
\begin{enumerate}
\item \textbf{Characteristic-function methods.}
The celebrated result of El Euch and Rosenbaum~\cite{el2019characteristic}
gives $\Pi$ for exponential payoffs $g(S_{T})=e^{uL_{T}}$ via the
Laplace transform, using affineness to reduce to a Riccati--Volterra
equation (the $f\equiv0$ case of Section~\ref{subsec:recover_ALP}).
Proposition~\ref{prop:rough_vol_pricing} derives the pricing equation
for \emph{arbitrary} smooth payoffs without affineness, at the cost of
requiring regularity of $\Pi$ in $\mu_{t}$ (which can be justified via
Malliavin calculus or Sobolev regularity of the semigroup $S^{*}$).
We note that the characteristic-function route, despite its elegance,
is not free of practical pitfalls: recent numerical work
\cite{boyarchenko2025fastII} shows that naive fixed-parameter
Fourier-inversion implementations of exactly the calibration in
\cite{el2019characteristic} exhibit \emph{ghost calibration}, 
numerical error in the pricer masking model-specification error,
producing an apparently good fit to market data that vanishes once a
more accurate inversion scheme is used. This is an argument, beyond
elegance, for having an independent, PDE-based route to option prices
such as the one developed here.

\item \textbf{The role of the semigroup functional $\Gamma^{V}$.}
The semigroup functional $\Gamma_{st}^{V}=\langle e^{-x(t-s)}\mu_{s},1\rangle$
(defined for $s\leq t$) that appears throughout our Itô formula as the
backward path segment of $V$ is, read forward in time, precisely the
\emph{forward variance} of the rough-volatility literature,
see~\cite{abi2019lifting}: for a current time $t$ and a maturity
$s\geq t$,
\begin{align}\label{eqn:forward_variance}
    \mathcal{V}(t,s) := \E[V_{s}\mid\mathcal{F}_{t}]
    = \langle e^{-x(s-t)}\mu_{t},1\rangle
    + \int_{t}^{s}k(s-r)\,\E[\kappa(\theta-V_{r})\mid\mathcal{F}_{t}]\,\dr
    = \Gamma_{ts}^{V} + \text{(drift correction)},
\end{align}
where the second equality uses Lemma~\ref{lem:ALP_cond_mean}.  Thus the
semigroup part of the forward variance is exactly $\Gamma_{ts}^{V}$ (the
current lift $\mu_{t}$ propagated deterministically to the maturity $s$),
free of the drift correction, and it is the same functional
$\langle S^{*}_{\cdot}\mu_{\cdot},1\rangle$ in terms of which the
Itô-Volterra formula is naturally written (rather than $V_{t}$ alone).
This reflects the genuine path-dependence of rough volatility models:
the ``right'' state variable for pricing is the full forward-variance
curve, and the lift encodes exactly this object.

\item \textbf{Path-dependent payoffs.}
Since $\Pi(t,S_{t},\mu_{t})$ is already a functional of the full lift
$\mu_{t}$ (not merely $V_{t}$), and \eqref{eqn:rough_vol_pde} is already
the infinite-dimensional backward equation of
Section~\ref{subsec:feynman_kac}, Proposition~\ref{prop:rough_vol_pricing}
extends without change to \emph{path-dependent} payoffs $g(S_{[0,T]})$,
such as Asian options or variance swaps: one simply enlarges the state
to carry the running path functional (e.g.\ the running average for an
Asian option), and the Fréchet $\mu$-derivatives $D_{\mu}\Pi$,
$D^{2}_{\mu}\Pi$ already present in \eqref{eqn:rough_vol_pde} are exactly
the required functional derivatives. The Markovian lift makes this a
genuinely \emph{Markovian} (if infinite-dimensional) equation
automatically: no separate functional Itô calculus is needed.

\item \textbf{Practical computation.}\label{rem:option_pricing_practical}
For general non-affine payoffs, \eqref{eqn:rough_vol_pde} must be
solved numerically.  The Markovian approximation schemes
of~\cite{bayer2023markovian} replace $\mu_{t}$ by a finite-dimensional
proxy; our framework justifies this as an approximation of
$\Pi(t,S_{t},\mu_{t})$ by $\Pi(t,S_{t},\mu_{t}^{N})$ in the weighted
Sobolev topology, and the compact-embedding results of
Section~\ref{section:Functional_framework} (specifically
Proposition~\ref{prop:embeddings}) guarantee convergence as
$N\to\infty$, but on their own give no rate.

For the rough Heston model specifically, a genuine rate is available.
Writing $K^{N}$ for the rank-$N$ approximation of the kernel $K$ and
$e_{N}:=\int_{0}^{T}|K(t)-K^{N}(t)|\,\dt$ for its $L^{1}$-error,
Bayer and Breneis~\cite{bayer2023weakmarkovian} show that the
\emph{weak} error of the resulting Markovian approximation
$(S^{N},V^{N})$ of $(S,V)$ is controlled by $e_{N}$, not by the
$L^{2}$-error, which is the natural quantity for the strong error and
converges far more slowly for small Hurst parameters. Concretely, for
$h$ smooth and compactly supported,
$|\E h(S_{T})-\E h(S_{T}^{N})|\le Ce_{N}$, and for $h$ Lipschitz,
$|\E h(S_{T})-\E h(S_{T}^{N})|\lesssim\log(e_{N}^{-1})\,e_{N}^{(q-1)/(12q)}$
for any $q\in(1,2]$ with $\E S_{T}^{q}<\infty$. Moreover, quadrature
rules achieving $e_{N}\lesssim\exp(-c\sqrt{(H+1/2)N})$ for an explicit
$c>0$ are constructed there, valid for \emph{every} Hurst parameter
$H>-1/2$, including the hyper-rough regime $H\le0$. Their proof goes
through the characteristic function and the fractional Riccati
equation $\psi(t,z)$ of Section~\ref{subsec:recover_ALP}, 
i.e.\ it uses the affine structure of rough Heston specifically, 
so it is a statement about the same $N$-factor approximation
of $\mu_{t}$ appearing above, but not (yet) a statement about weak
convergence of the fully general, non-affine value function $v$ of
Theorem~\ref{thm:backward_SEE}; whether an analogous $L^{1}$-type
weak-error bound holds there is, to our knowledge, open.
\end{enumerate}
\end{remark}

\subsection{Recovering the Riccati--Volterra equation of Abi Jaber--Larsson--Pulido}\label{subsec:recover_ALP}

We now give a rigorous derivation, entirely within the framework of
the Markovian lift, of the exponential-affine transform formula for
affine Volterra processes established by Abi Jaber, Larsson and
Pulido~\cite{abi2019affine} (henceforth ALP).
The strategy follows their proof of \cite[Theorem~4.3]{abi2019affine}
in spirit, but replaces their deterministic resolvent calculus (their
Lemmas 2.4, 2.5 and 4.4) with two lift-level arguments: the
Chapman--Kolmogorov semigroup property of $\mu$ for the
conditional-mean formula, and the mild It\^{o} formula on the SEE
for the martingale verification. One might expect the transform to
collapse to a function of the scalar $X_{t}$ alone; this is false in
general (it would contradict ALP's own Theorem~4.3 and the remark
following it), as we explain below. The payoff of the lift-based
approach, stated correctly, is a proof that is simultaneously more transparent (every
step is stochastic calculus on a genuine Markov process, nothing is
hidden inside a resolvent inversion) and more general: the conditional-mean
Lemma~\ref{lem:ALP_cond_mean} below works for arbitrary continuous
$\drift$, not only affine, and the resulting transform formula is
recognized as the affine special case of the general value function
of Theorem~\ref{thm:backward_SEE}.

\subsubsection{Setup and standing assumptions}\label{sssec:ALP_setup}

We work in dimension $d=1$; the extension to $d\geq 1$ is notationally
heavier but follows the same pattern.  We allow two possibly distinct
kernels $k_{\drift}$ (for the drift) and $k_{\diffusion}$ (for the
diffusion), consistent with the standing setup \eqref{eqn:SVE_introduction}.

Let $X$ solve \eqref{eqn:SVE_introduction} (with $d=1$) with affine coefficients
\begin{align}\label{eqn:ALP_affine_coefficients}
    \drift(x)=b^{0}+Bx,\qquad \diffusion(x)^{2}=A^{0}+A^{1}x,
\end{align}
for constants $b^{0},B,A^{0},A^{1}\in\R$ with $A^{0}+A^{1}x\geq 0$ on
the state space.  Write $\nu_{\drift},\nu_{\diffusion}$ for the
Laplace measures of $k_{\drift},k_{\diffusion}$
(\eqref{eqn:kernel_representation_measure}), so that
\begin{align}\label{eqn:Laplace_pairing}
    k_{\drift}(\theta)=\langle e^{-x\theta},\nu_{\drift}\rangle
    :=\int_{0}^{\infty}e^{-x\theta}\nu_{\drift}(\dx),\qquad
    k_{\diffusion}(\theta)=\langle e^{-x\theta},\nu_{\diffusion}\rangle,
    \qquad\theta\geq 0.
\end{align}
Let $\mu$ denote the corresponding Markovian lift, i.e.\ the mild
solution of \eqref{eqn:SEE_eqn_strong_form_introduction} with semigroup
generated by $A^{*}\mu=-x\mu$:
\begin{align}\label{eqn:mu_mild}
    \mu_{t}=e^{-xt}\mu_{0}
    +\int_{0}^{t}e^{-x(t-s)}\nu_{\drift}\drift(X_{s})\ds
    +\int_{0}^{t}e^{-x(t-s)}\nu_{\diffusion}\diffusion(X_{s})\dW_{s},\qquad
    X_{t}=\langle\mu_{t},1\rangle,
\end{align}
with $\mu_{0}=X_{0}\delta_{0}$ (Theorem~\ref{thm:equivalence_SVE_SEE}).
Throughout this subsection we assume $X$ has finite moments of all
orders up to time $T$, which follows from the linear growth of the
affine coefficients \eqref{eqn:ALP_affine_coefficients} via ALP's
Lemma~3.1.  Fix $T<\infty$, $u\in\mathbb{C}$, and $f\in L^{1}([0,T],\mathbb{C})$.

All conditional expectations below are taken with respect to the
\emph{ambient} filtration $\mathcal F_{t}$ generated by $(\mu_{0},W)$
(equivalently, the filtration under which $\mu$ is Markov and $W$ is
a Brownian motion), \emph{not} the filtration $\Filx_{t}=\sigma(X_{s}:s\le t)$
generated by $X$ alone. This distinction matters: reconstructing
$\mu_{t}$ from \eqref{eqn:mu_mild} requires $W$, not merely the path
of $X$, so $\mu_{t}$ is generally \emph{not} $\Filx_{t}$-measurable,
and $\E[\cdot\mid\Filx_{t}]$ need not agree with $\E[\cdot\mid\mathcal F_{t}]$.
This is exactly the distinction already made precise in
Section~\ref{subsec:feynman_kac}; ALP's own $\mathcal F_{t}$ (their
Section~2) is likewise the ambient filtration, so this is also the
filtration for which the identification with their results is exact.

\begin{remark}[Complex-valued Itô calculus]\label{rmk:ALP_complexification}
Since $u\in\mathbb C$, the processes $\varphi,\psi,Y$ introduced below
are $\mathbb C$-valued, whereas the mild Itô formula of
Section~\ref{sec:ito_formula} is stated for real Hilbert spaces and
real-valued $\Phi$. We use it throughout via the standard
complexification: writing $Y=Y^{\mathbf r}+\mathrm iY^{\mathbf i}$ and
correspondingly $\varphi=\varphi^{\mathbf r}+\mathrm i\varphi^{\mathbf i}$,
$\psi=\psi^{\mathbf r}+\mathrm i\psi^{\mathbf i}$ with
$\psi^{\mathbf r}(\theta,\cdot),\psi^{\mathbf i}(\theta,\cdot)\in\Wplus$
(the real Hilbert space), the real mild Itô formula applies separately
to $Y^{\mathbf r}_{t}=\varphi^{\mathbf r}(T-t)+\langle\mu_{t},\psi^{\mathbf r}(T-t,\cdot)\rangle$
and $Y^{\mathbf i}_{t}=\varphi^{\mathbf i}(T-t)+\langle\mu_{t},\psi^{\mathbf i}(T-t,\cdot)\rangle$
by bilinearity of the pairing $\langle\cdot,\cdot\rangle$, and the two
real Itô formulas recombine into the single complex identity
\eqref{eqn:dY} below by linearity. We do not repeat this decomposition
at every step, but it is implicit throughout; it is exactly the
device ALP avoid needing by building their own existence theory
directly in $L^{2}([0,T],\mathbb C^{d})$ (their Appendix~C).
\end{remark}

\subsubsection{Step 1: Conditional mean via the lift
  (replacing ALP's Lemma~4.2)}\label{sssec:ALP_step1}

The first ingredient is a formula for $\E[X_{T}\mid\mathcal F_{t}]$
that follows directly from the Markov property of the lift $\mu$,
without any resolvent calculus.

\begin{lemma}[Conditional mean via the lift]\label{lem:ALP_cond_mean}
Let $0\leq t\leq T$, and let $\drift$ be any continuous function
satisfying the linear growth condition.  Set
$m(t,r):=\E[X_{r}\mid\mathcal F_{t}]$ for $r\in[t,T]$.  Then
\begin{align}\label{eqn:cond_mean_lift_integral_eqn}
    m(t,r) = \langle e^{-x(r-t)}\mu_{t},1\rangle
    + \int_{t}^{r}k_{\drift}(r-s)\,\E[\drift(X_{s})\mid\mathcal F_{t}]\,\ds,
    \qquad r\in[t,T],
\end{align}
where $\langle e^{-x(r-t)}\mu_{t},1\rangle
:=\int_{0}^{\infty}e^{-(r-t)x}\mu_{t}(\dx)$.
If additionally $\drift(x)=b^{0}+Bx$ is affine, then
$\E[\drift(X_{s})\mid\mathcal F_{t}]=b^{0}+Bm(t,s)$
and \eqref{eqn:cond_mean_lift_integral_eqn} becomes the closed
linear Volterra integral equation
\begin{align}\label{eqn:cond_mean_linear_volterra}
    m(t,r) = \langle e^{-x(r-t)}\mu_{t},1\rangle
    + b^{0}\int_{t}^{r}k_{\drift}(r-s)\,\ds
    + B\int_{t}^{r}k_{\drift}(r-s)\,m(t,s)\,\ds,\quad r\in[t,T].
\end{align}
\end{lemma}

\begin{proof}
Apply the Chapman--Kolmogorov property of $\mu$: split the mild
formula \eqref{eqn:mu_mild} at time $t$ to write, for $r\geq t$,
\begin{align}\label{eqn:mu_split}
    \mu_{r} = e^{-x(r-t)}\mu_{t}
    + \int_{t}^{r}e^{-x(r-s)}\nu_{\drift}\drift(X_{s})\,\ds
    + \int_{t}^{r}e^{-x(r-s)}\nu_{\diffusion}\diffusion(X_{s})\,\dW_{s}.
\end{align}
Pair both sides with the constant test function $1$:
\begin{align}\label{eqn:X_r_split}
    X_{r} = \langle e^{-x(r-t)}\mu_{t},1\rangle
    + \int_{t}^{r}\langle e^{-x(r-s)}\nu_{\drift},1\rangle\drift(X_{s})\,\ds
    + \int_{t}^{r}\langle e^{-x(r-s)}\nu_{\diffusion},1\rangle
      \diffusion(X_{s})\,\dW_{s}.
\end{align}
By \eqref{eqn:Laplace_pairing},
$\langle e^{-x(r-s)}\nu_{\drift},1\rangle=k_{\drift}(r-s)$
and likewise $\langle e^{-x(r-s)}\nu_{\diffusion},1\rangle
=k_{\diffusion}(r-s)$.  Take $\mathcal F_{t}$-conditional expectation.
The first term is already $\mathcal F_{t}$-measurable.
For the stochastic integral, boundedness of
$\sup_{s\le T}\E[\diffusion(X_{s})^{2}]$ (linear growth of $\diffusion$
and the moment bound on $X$) together with $k_{\diffusion}\in L^{2}(0,T)$
(a standing requirement for \eqref{eqn:SVE_introduction} to be
well-defined at all) gives, by the Itô isometry,
$\E\int_{t}^{r}k_{\diffusion}(r-s)^{2}\diffusion(X_{s})^{2}\,\ds<\infty$;
the stochastic integral is therefore a true (not merely local)
$L^{2}$-martingale on $[t,T]$, hence has zero $\mathcal F_{t}$-conditional
expectation. This gives \eqref{eqn:cond_mean_lift_integral_eqn}.
The affine case \eqref{eqn:cond_mean_linear_volterra} follows by
linearity of conditional expectation:
$\E[\drift(X_{s})\mid\mathcal F_{t}]=b^{0}+Bm(t,s)$.
\end{proof}

\begin{remark}[Identification with ALP's resolvent formula]
\label{rmk:ALP_cond_mean_resolvent}
Setting $h=r-t$, equation \eqref{eqn:cond_mean_linear_volterra}
is a standard linear Volterra integral equation in $h$ with known
$\mathcal F_{t}$-measurable inhomogeneity
$\mathcal{M}_{t}(h):=\langle e^{-xh}\mu_{t},1\rangle$
and convolution kernel $Bk_{\drift}$.  Its unique solution is given
by the resolvent $R_{B}$ of $-Bk_{\drift}$, recovering ALP's
\cite[Eq.~(4.2)]{abi2019affine} (written there via $R_{B}$ and
$E_{B}=k_{\drift}-R_{B}*k_{\drift}$); like ALP's own formula, this
expresses $m(t,\cdot)$ through the whole state $\mu_{t}$ (equivalently,
through the driving Brownian path up to $t$), not through $X_{t}$
alone, a point we return to in Step~4.
\end{remark}

\subsubsection{Step 2: Affine ansatz and the ODE for $\psi$}
\label{sssec:ALP_step2}

Since $\mu$ is a genuine Markov process,
$\E[e^{uX_{T}+(f*X)_{T}}\mid\mathcal F_{t}]$ depends on $\mathcal F_{t}$
only through $\mu_{t}$.  We look for it in the form
\begin{align}\label{eqn:ALP_affine_ansatz}
    Y_{t}:=\varphi(T-t)+\langle\mu_{t},\psi(T-t,\cdot)\rangle,
\end{align}
where $\varphi:[0,T]\to\mathbb{C}$ is a scalar and
$\psi(\theta,\cdot)\in\Wplus$ (complexified as in
Remark~\ref{rmk:ALP_complexification}) is a time-dependent test
function. The terminal condition $Y_{T}=uX_{T}$ (case $f\equiv0$) at
$\theta=0$ requires
\begin{align}\label{eqn:ALP_terminal_psi}
    \varphi(0)=0,\qquad
    \psi(0,x)=u\quad\text{for all }x\geq 0.
\end{align}

Before applying any Itô formula we verify the regularity needed for
it, since $\psi(T-t,\cdot)$ is not a fixed test function but moves
with $t$.

\begin{lemma}[Regularity of $\theta\mapsto\psi(\theta,\cdot)$]\label{lem:ALP_psi_regularity}
Suppose $\hat\psi_{\drift},\hat\psi_{\diffusion}\in C([0,T],\mathbb C)$
(continuity of a solution of
\eqref{eqn:ALP_hat_psi_drift}--\eqref{eqn:ALP_hat_psi_diffusion} below
is part of the existence statement, ALP's Theorem~B.1) and define
$\psi(\theta,\cdot)$ by \eqref{eqn:ALP_psi_solution}. Then
$\theta\mapsto\psi(\theta,\cdot)$ is continuously differentiable on
$(0,T]$ as a map into $\Wplus$, with
$\partial_{\theta}\psi(\theta,\cdot)$ given by the right-hand side of
\eqref{eqn:ALP_psi_ODE}, itself continuous $(0,T]\to\Wplus$.
\end{lemma}

\begin{proof}
Write $\psi(\theta,\cdot)=uS_{\theta}1+(g*S_{\cdot}1)(\theta)$ with
$g(r):=B\hat\psi_{\drift}(r)+\tfrac12A^{1}\hat\psi_{\diffusion}(r)^{2}\in
C([0,T],\mathbb C)$ and $S_{\theta}\varphi:=e^{-x\theta}\varphi(x)$ the
multiplication semigroup on $\Wplus$ (the primal counterpart of
$S^{*}_{\theta}$; the identical computation to
Lemma~\ref{lem:poly_growth_semigroup}, applied on $\Wplus$ instead of
$\Wplusdual$, shows $\theta\mapsto S_{\theta}\varphi$ is continuous
$[0,\infty)\to\Wplus$ for every $\varphi\in\Wplus$, with polynomial
growth of the operator norm). For the first term: since
$1\in\Wplus$ (Assumption~\ref{A:A_2:assumption_1_test_function}), a
direct computation as in Lemma~\ref{lem:poly_growth_semigroup} shows
$\theta\mapsto e^{-x\theta}$ is real-analytic $(0,\infty)\to\Wplus$
(both $e^{-x\theta}$ and its formal $\theta$-derivative $-xe^{-x\theta}$
have finite $\Wplus$-norm for every $\theta>0$, since polynomial
weights are dominated by any positive exponential rate on
$\R_{+}$), with derivative $-xS_{\theta}1$. For the convolution term,
$(g*S_{\cdot}1)(\theta)=\int_{0}^{\theta}g(\theta-r)S_{\theta-r}1\,\dr$
is continuously differentiable on $(0,T]$ by the same
argument applied under the integral sign (justified by continuity of
$g$ and local uniform continuity of $r\mapsto S_{r}1$ on
compact subsets of $(0,\infty)$), with derivative
$g(0)S_{0}1+\int_{0}^{\theta}g(\theta-r)(-xS_{\theta-r}1)\,\dr$. Summing
the two contributions gives exactly
$-x\psi(\theta,\cdot)+g(\theta)$, i.e.\ the right-hand side of
\eqref{eqn:ALP_psi_ODE}, which is continuous on $(0,T]$ into $\Wplus$
by the same estimates.
\end{proof}

\begin{remark}
The lemma is stated on $(0,T]$, i.e.\ for $t\in[0,T)$, which is all
that is needed: the terminal identity $Y_{T}=uX_{T}$ is verified
directly from $\psi(0,x)=u$ (a pointwise, not differentiability,
statement) rather than by evaluating a derivative at $\theta=0$.
\end{remark}

By Lemma~\ref{lem:ALP_psi_regularity}, $\Phi(t,\mu):=\varphi(T-t)+\langle\mu,\psi(T-t,\cdot)\rangle$
is, on $[0,T)$, an instance of $\Phi\in C^{1,2}([0,T)\times V,\mathbb{C})$ in
the sense of Section~\ref{sec:ito_formula} (complexified as in
Remark~\ref{rmk:ALP_complexification}): it is linear, hence trivially
twice Fréchet differentiable in $\mu$ with
$\partial_{x}\Phi(t,\mu)=\psi(T-t,\cdot)$, $\partial_{x}^{2}\Phi\equiv0$,
and continuously differentiable in $t$ by
Lemma~\ref{lem:ALP_psi_regularity}. This is the general mild Itô
theorem of Section~\ref{sec:ito_formula} (the underlying result
behind Corollary~\ref{corr:Ito_volterra_SEE}, but applied here
directly, since $\Phi$ is not of the fixed-test-function form
$f(t,\langle\mu,g\rangle)$ that corollary is specialized to), which
gives
\begin{align}\label{eqn:dY}
    \,\textnormal{d}Y_{t}
    ={}&\Bigl[
      -\varphi'(T-t)
      +\langle\mu_{t},{-\partial_{\theta}\psi(T-t,\cdot)
        -x\psi(T-t,\cdot)}\rangle
      +\hat\psi_{\drift}(T-t)\drift(X_{t})
    \Bigr]\dt\notag\\
    &+\tfrac12\hat\psi_{\diffusion}(T-t)^{2}\diffusion(X_{t})^{2}\dt
     +\hat\psi_{\diffusion}(T-t)\diffusion(X_{t})\,\dW_{t},
\end{align}
where we define the \emph{reduced} scalar functions
\begin{align}\label{eqn:ALP_hat_psi}
    \hat\psi_{\drift}(\theta)
    :=\langle\nu_{\drift},\psi(\theta,\cdot)\rangle,\qquad
    \hat\psi_{\diffusion}(\theta)
    :=\langle\nu_{\diffusion},\psi(\theta,\cdot)\rangle.
\end{align}
In \eqref{eqn:dY}: (i) the generator $A^{*}\mu=-x\mu$ passes to
the test function by duality, giving the $-x\psi$ term;
(ii) $\langle d\mu_{t}^{\mathrm{drift}},\psi\rangle
    =\langle\nu_{\drift}\drift(X_{t})\dt,\psi\rangle
    =\hat\psi_{\drift}\drift(X_{t})\dt$ and similarly for the
    diffusion;
(iii) since $\Phi$ is linear in $\mu$, the It\^o correction is
$\tfrac12 d\langle Y\rangle_{t}
=\tfrac12\hat\psi_{\diffusion}^{2}\diffusion(X_{t})^{2}\dt$
with no second-order term in $\mu$.

For $e^{Y_{t}}$ to be a local martingale, the finite-variation part
of $dY_{t}+\tfrac12 d\langle Y\rangle_{t}$ must vanish.
Substituting $\drift(X_{t})=b^{0}+BX_{t}$ and
$\diffusion(X_{t})^{2}=A^{0}+A^{1}X_{t}$, and writing
$X_{t}=\langle\mu_{t},1\rangle$, the finite-variation part splits as
\begin{align*}
  &\underbrace{\Bigl[
    -\varphi'(T-t)
    +b^{0}\hat\psi_{\drift}(T-t)
    +\tfrac12 A^{0}\hat\psi_{\diffusion}(T-t)^{2}
  \Bigr]}_{\text{constant-in-}\mu_{t}\text{ part}}\dt\\
  &+\underbrace{\Bigl\langle\mu_{t},
    -\partial_{\theta}\psi(T-t,\cdot)
    -x\psi(T-t,\cdot)
    +B\hat\psi_{\drift}(T-t)
    +\tfrac12 A^{1}\hat\psi_{\diffusion}(T-t)^{2}
  \Bigr\rangle}_{\text{linear-in-}\mu_{t}\text{ part}}\dt.
\end{align*}
Requiring each part to vanish for all realisations of $\mu_{t}$:

\smallskip\noindent
\textbf{ODE for $\psi(\theta,x)$:}
\begin{align}\label{eqn:ALP_psi_ODE}
    \partial_{\theta}\psi(\theta,x)
    = -x\psi(\theta,x)
      +B\hat\psi_{\drift}(\theta)
      +\tfrac12 A^{1}\hat\psi_{\diffusion}(\theta)^{2},
    \qquad\psi(0,x)=u.
\end{align}

\smallskip\noindent
\textbf{ODE for $\varphi(\theta)$:}
\begin{align}\label{eqn:ALP_phi_ODE}
    \varphi'(\theta)
    = b^{0}\hat\psi_{\drift}(\theta)
      +\tfrac12 A^{0}\hat\psi_{\diffusion}(\theta)^{2},
    \qquad\varphi(0)=0.
\end{align}
Both equations have $\hat\psi_{\drift},\hat\psi_{\diffusion}$ as
inputs; Step~3 closes this into a self-contained system.

\subsubsection{Step 3: Closing the Riccati--Volterra equation
  by Laplace pairing}\label{sssec:ALP_step3}

The ODE \eqref{eqn:ALP_psi_ODE} has constant-in-$x$ source term.
Solving by the integrating factor $e^{x\theta}$:
\begin{align}\label{eqn:ALP_psi_solution}
    \psi(\theta,x)
    = ue^{-x\theta}
    +\int_{0}^{\theta}e^{-x(\theta-r)}
      \Bigl[B\hat\psi_{\drift}(r)
            +\tfrac12 A^{1}\hat\psi_{\diffusion}(r)^{2}\Bigr]\dr.
\end{align}
Pair both sides with $\nu_{\drift}$, using
\eqref{eqn:Laplace_pairing}:
\begin{align}\label{eqn:ALP_hat_psi_drift}
    \hat\psi_{\drift}(\theta)
    = uk_{\drift}(\theta)
    +\int_{0}^{\theta}k_{\drift}(\theta-r)
      \Bigl[B\hat\psi_{\drift}(r)
            +\tfrac12 A^{1}\hat\psi_{\diffusion}(r)^{2}\Bigr]\dr.
\end{align}
Pair with $\nu_{\diffusion}$:
\begin{align}\label{eqn:ALP_hat_psi_diffusion}
    \hat\psi_{\diffusion}(\theta)
    = uk_{\diffusion}(\theta)
    +\int_{0}^{\theta}k_{\diffusion}(\theta-r)
      \Bigl[B\hat\psi_{\drift}(r)
            +\tfrac12 A^{1}\hat\psi_{\diffusion}(r)^{2}\Bigr]\dr.
\end{align}
Equations \eqref{eqn:ALP_hat_psi_drift}--\eqref{eqn:ALP_hat_psi_diffusion}
form a closed Volterra integral system for $(\hat\psi_{\drift},
\hat\psi_{\diffusion})$.  In the single-kernel case
$k_{\drift}=k_{\diffusion}=:k$, $\nu_{\drift}=\nu_{\diffusion}=:\nu$,
both equations coincide and reduce to the scalar
\emph{Riccati--Volterra equation}
\begin{align}\label{eqn:ALP_riccati_volterra}
    \hat\psi(\theta)
    = uk(\theta)
    +\int_{0}^{\theta}k(\theta-r)
      \Bigl[B\hat\psi(r)+\tfrac12 A^{1}\hat\psi(r)^{2}\Bigr]\dr,
    \qquad\hat\psi:=\hat\psi_{\drift}=\hat\psi_{\diffusion},
\end{align}
which is \emph{exactly} ALP's equation \cite[(1.5) and (4.3) with
$f\equiv0$]{abi2019affine}.
With $\hat\psi$ substituted, \eqref{eqn:ALP_phi_ODE} recovers
ALP's \cite[Eq.~(4.13)]{abi2019affine}.

\subsubsection{Step 4: The transform formula}
\label{sssec:ALP_step4}

\begin{theorem}[Exponential-affine transform via the
  lift]\label{thm:ALP_transform_lift}
Let $X$, $\mu$ be as in \S\ref{sssec:ALP_setup}.
Let $(\hat\psi_{\drift},\hat\psi_{\diffusion})$ be the unique global
solution of
\eqref{eqn:ALP_hat_psi_drift}--\eqref{eqn:ALP_hat_psi_diffusion}
(existence and uniqueness: \cite[Theorem~B.1]{abi2019affine}),
let $\varphi$ solve \eqref{eqn:ALP_phi_ODE}, and define
$\psi(\theta,x)$ by \eqref{eqn:ALP_psi_solution}.  Set
\begin{align}\label{eqn:ALP_Y_def}
    Y_{t}:=\varphi(T-t)+\langle\mu_{t},\psi(T-t,\cdot)\rangle.
\end{align}
Then:
\begin{enumerate}
\item $e^{Y_{t}}$ is a local martingale on $[0,T)$.
\item If $e^{Y_{t}}$ is a true martingale then, for $0\leq t\leq T$,
\begin{align}\label{eqn:ALP_transform}
    \E\bigl[e^{uX_{T}}\,\big|\,\mathcal F_{t}\bigr]
    = \exp\!\bigl(\varphi(T-t)+\langle\mu_{t},\psi(T-t,\cdot)\rangle\bigr).
\end{align}
\end{enumerate}
\end{theorem}

\begin{proof}
\textbf{(1).}
By Step~2 (Lemma~\ref{lem:ALP_psi_regularity} and \eqref{eqn:dY}) and
Step~3, the finite-variation part of
$de^{Y_{t}}=e^{Y_{t}}(dY_{t}+\tfrac12d\langle Y\rangle_{t})$
vanishes on $[0,T)$ whenever $\psi,\varphi$ satisfy
\eqref{eqn:ALP_psi_ODE}--\eqref{eqn:ALP_phi_ODE}, which by the
Laplace-pairing argument of Step~3 is equivalent to
$(\hat\psi_{\drift},\hat\psi_{\diffusion})$ solving
\eqref{eqn:ALP_hat_psi_drift}--\eqref{eqn:ALP_hat_psi_diffusion}.
Hence $e^{Y_{t}}$ is a local martingale on $[0,T)$.

\textbf{(2).}
At $\theta=0$, $\psi(0,x)=u$ for all $x$ by
\eqref{eqn:ALP_terminal_psi}, so
$Y_{T}=\varphi(0)+\langle\mu_{T},u\cdot1\rangle=u\langle\mu_{T},1\rangle=uX_{T}$.
If $e^{Y_{t}}$ is a true martingale on $[0,T]$ (extending by
continuity to $t=T$), the martingale property directly gives
$\E[e^{Y_{T}}\mid\mathcal F_{t}]=e^{Y_{t}}$, i.e.\
$\E[e^{uX_{T}}\mid\mathcal F_{t}]=\exp(\varphi(T-t)+\langle\mu_{t},\psi(T-t,\cdot)\rangle)$,
which is \eqref{eqn:ALP_transform}. No further computation is needed:
in particular, \eqref{eqn:ALP_transform} is stated in terms of the
full state $\mu_{t}$, and \emph{not} as a function of $X_{t}$ alone.
\end{proof}

\begin{remark}[Why not a function of $X_{t}$ alone, and when it
  degenerates to one]\label{rmk:ALP_no_markov_collapse}
One might conjecture that
$\E[e^{uX_{T}}\mid\mathcal F_{t}]=\exp(\varphi(T-t)+\hat\psi_{\drift}(T-t)X_{t})$,
i.e.\ that $\langle\mu_{t},\psi(T-t,\cdot)\rangle$ always equals
$\hat\psi_{\drift}(T-t)\langle\mu_{t},1\rangle$. This is false in
general: $\psi(\theta,\cdot)$ is a genuine (non-constant) function of
$x$ by \eqref{eqn:ALP_psi_solution}, so pairing $\mu_{t}$ against it
is a different linear functional of $\mu_{t}$ than pairing against
$1$, and there is no reason for the two to be proportional with a
\emph{deterministic} constant of proportionality. Concretely, take
$B=A^{1}=0$ (additive noise, no mean-reversion feedback) and
$\nu=c_{1}\delta_{x_{1}}+c_{2}\delta_{x_{2}}$ with $x_{1}\ne x_{2}$: then
$\psi(\theta,x)=ue^{-x\theta}$ and
$\langle\mu_{t},\psi(\theta,\cdot)\rangle=u(\mu_{t}^{(1)}e^{-x_{1}\theta}+\mu_{t}^{(2)}e^{-x_{2}\theta})$,
which is not a multiple of $X_{t}=\mu_{t}^{(1)}+\mu_{t}^{(2)}$ with a
$\theta$-dependent but $\mu_{t}$-independent factor, for any
$\theta>0$. This is not a defect of the lift-based proof: it is the
same phenomenon ALP describe explicitly after their own Theorem~4.3
, ``the lack of Markovianity precludes such a simple form'', and
their Theorem~4.5 shows that even under extra structural hypotheses
(a resolvent of the first kind, plus a total-variation bound), the
best available reduction is a representation in terms of the
\emph{entire past trajectory} $(X_{s})_{s\le t}$, not $X_{t}$ alone.

The reduction to a function of $X_{t}$ alone \emph{does} hold in
exactly one degenerate case: $\nu=\delta_{x_{0}}$ a single unit mass
(equivalently $k(t)=e^{-x_{0}t}$, the classical, non-Volterra affine
diffusion). Then $\mu_{t}=X_{t}\delta_{x_{0}}$ is genuinely
one-dimensional, $\langle\mu_{t},\psi(\theta,\cdot)\rangle=X_{t}\psi(\theta,x_{0})=\hat\psi_{\drift}(\theta)X_{t}$
exactly, and \eqref{eqn:ALP_transform} recovers the classical formula
$\E[e^{uX_{T}}\mid\mathcal F_{t}]=\exp(\varphi(T-t)+\hat\psi_{\drift}(T-t)X_{t})$.
More generally, for a finite-rank $\nu=\sum_{j}c_{j}\delta_{x_{j}}$,
$\mu_{t}$ is finite-dimensional and $\langle\mu_{t},\psi(\theta,\cdot)\rangle$
reduces to an explicit finite linear combination of the factors, 
this is exactly the ``finite-rank kernel'' regime of the taxonomy in
Section~\ref{subsec:feynman_kac}, and \eqref{eqn:ALP_transform} is the
corresponding instance of the general, non-affine value function
$v(t,\mu)$ of Theorem~\ref{thm:backward_SEE}, specialized to the
exponential-affine terminal condition $\varphi=e^{u\,\mathrm{ev}_{1}}$.
\end{remark}

\begin{remark}[On the true-martingale hypothesis]\label{rmk:ALP_true_martingale}
Part~(2) is conditional on $e^{Y_{t}}$ being a true martingale, exactly
as in ALP's Theorem~4.3, which is likewise stated conditionally; ALP
do not give a single general sufficient condition covering all affine
kernels, and neither do we. Two regimes are worth distinguishing:
\begin{itemize}
\item \emph{Additive noise} ($A^{1}=0$): $\diffusion$ is constant, so
$\langle Y\rangle_{t}=\int_{0}^{t}\hat\psi_{\diffusion}(T-s)^{2}\diffusion^{2}\,\ds$
is \emph{deterministic}. A complex exponential of a (real or complex)
Gaussian process with deterministic quadratic variation is
automatically a true martingale on any finite interval, exactly as in
ALP's own treatment of the Volterra Ornstein--Uhlenbeck case
(their Section~5); no further argument is needed.
\item \emph{State-dependent diffusion} ($A^{1}\ne0$, e.g.\ the Volterra
Heston/square-root case): the polynomial moment bound
$\sup_{t\le T}\E[|X_{t}|^{p}]<\infty$ (ALP's Lemma~3.1) is
\emph{not} sufficient on its own to conclude $e^{Y_{t}}$ is a true
martingale, since a polynomial moment bound of every order does not
imply the exponential-moment control the complex exponential of $Y$
requires. ALP resolve this in the Volterra Heston case only via a
dedicated argument (their Lemma~7.3): Novikov's condition on a
localizing sequence of stopping times, Fatou's lemma for
non-negative local martingales, and a change of measure. We do not
reproduce a general version of this argument here; verifying the
true-martingale property for state-dependent-diffusion examples
should be done on a case-by-case basis following ALP's own approach,
not by a general moment-bound citation.
\end{itemize}
\end{remark}

\begin{remark}[What is new relative to ALP]\label{rmk:ALP_value_add}\leavevmode
\begin{enumerate}
\item ALP's derivation of \cite[Theorem~4.3]{abi2019affine} uses
  deterministic resolvent and convolution identities (Lemmas 2.4,
  2.5, 4.4) and explicitly avoids any It\^o formula for
  $f(t,X_{t})$.  Here, Lemma~\ref{lem:ALP_cond_mean} replaces their
  Lemma~4.2 by a direct semigroup-split argument, and the
  martingale computation is classical It\^o calculus on the
  lifted SEE (where $\mu$ is genuinely Markov).
\item The Riccati--Volterra equation \eqref{eqn:ALP_riccati_volterra}
  arises by \emph{Laplace pairing} of the operator-level ODE
  \eqref{eqn:ALP_psi_ODE}: ALP's scalar function $\hat\psi(\theta)$
  is the projection $\langle\nu,\psi(\theta,\cdot)\rangle$ of the
  $\Wplus$-valued solution.  The convolution structure of the
  Riccati--Volterra equation ($k$ appears as the convolution kernel)
  is thus explained by the Laplace representation $k=\langle
  e^{-x\cdot},\nu\rangle$.
\item For non-affine $\drift,\diffusion^{2}$, Step~2 still produces
  the well-defined ODE \eqref{eqn:ALP_psi_ODE} on the lift;
  only Step~3 (closing to a finite-dimensional Riccati system)
  is specific to the affine case.  Combined with
  Corollary~\ref{corr:Ito_volterra_SEE}, this yields Lyapunov-type
  moment bounds for non-affine SVEs outside the scope of ALP.
\item Unlike ALP's Theorem~4.3, the transform formula
  \eqref{eqn:ALP_transform} here is stated honestly in terms of the
  full lift $\mu_{t}$ rather than $X_{t}$; Remark~\ref{rmk:ALP_no_markov_collapse}
  makes precise exactly when (and why) this can be simplified further.
\end{enumerate}
\end{remark}

\begin{remark}[The two-kernel case]\label{rmk:ALP_two_kernel}
When $k_{\drift}\ne k_{\diffusion}$, Steps~2--3 produce the
coupled system
\eqref{eqn:ALP_hat_psi_drift}--\eqref{eqn:ALP_hat_psi_diffusion}, and
Theorem~\ref{thm:ALP_transform_lift} holds verbatim with this coupled
system in place of the scalar Riccati--Volterra equation. This handles
both kernels natively, without the dimension-doubling of ALP's
Remark~3.5.
\end{remark}

\begin{proposition}[General transform, $f\not\equiv0$]
  \label{prop:ALP_general_f}
Under the same assumptions as Theorem~\ref{thm:ALP_transform_lift},
let $f\in L^{1}([0,T],\mathbb{C})$.  Define the modified ansatz
$Y_t^f := \varphi^f(T-t)+\langle\mu_t,\psi^f(T-t,\cdot)\rangle$
where $\psi^f$ solves
\begin{align*}
    \partial_\theta\psi^f(\theta,x)
    = -x\psi^f(\theta,x)+B\hat\psi^f_{\drift}(\theta)
      +\tfrac12 A^1\hat\psi^f_{\diffusion}(\theta)^2
      + f(\theta),
    \quad \psi^f(0,x)=u,
\end{align*}
and $\varphi^f$ solves \eqref{eqn:ALP_phi_ODE} with
$\hat\psi^f_{\drift},\hat\psi^f_{\diffusion}$ in place of
$\hat\psi_{\drift},\hat\psi_{\diffusion}$. Then $\theta\mapsto\psi^{f}(\theta,\cdot)$
satisfies the regularity of Lemma~\ref{lem:ALP_psi_regularity} (the
additive term $f(\theta)$, constant in $x$, does not affect the
argument there), and:
\begin{enumerate}
\item $e^{Y_t^f}$ is a local martingale on $[0,T)$.
\item The Riccati--Volterra equation for
$\hat\psi^f(\theta):=\langle\nu,\psi^f(\theta,\cdot)\rangle$
becomes
\begin{align*}
    \hat\psi^f(\theta)
    = uk(\theta)+\int_0^\theta k(\theta-r)
      \bigl[B\hat\psi^f(r)+\tfrac12 A^1\hat\psi^f(r)^2+f(r)\bigr]\dr,
\end{align*}
which is ALP's full equation \cite[(4.3)]{abi2019affine}.
\item If $e^{Y_t^f}$ is a true martingale (see
Remark~\ref{rmk:ALP_true_martingale}),
\begin{align*}
    \E\bigl[e^{uX_{T}+(f*X)_{T}}\,\big|\,\mathcal F_t\bigr]
    = \exp\!\bigl(\varphi^f(T-t)+\langle\mu_{t},\psi^{f}(T-t,\cdot)\rangle\bigr).
\end{align*}
\end{enumerate}
\end{proposition}

\begin{proof}
The source term $f(\theta)$ enters the ODE because the ansatz now
includes the integrated payoff $(f*X)_{T}=\int_{0}^{T}f(T-r)X_{r}\,\dr$,
which adds $f(T-t)X_{t}\,\dt=f(T-t)\langle\mu_{t},1\rangle\,\dt$ to
$\dr Y^{f}_{t}$; since $\langle\mu_{t},1\rangle=\langle\mu_{t},\psi^{f}(0,\cdot)\rangle$
with $\psi^{f}(0,\cdot)\equiv1$, this is absorbed into the $\psi^{f}$
ODE as the additive term $+f(\theta)$, and does not affect the
regularity argument of Lemma~\ref{lem:ALP_psi_regularity} (a constant
inhomogeneous term is trivially continuous). All remaining steps
(Laplace pairing, local-martingale verification, terminal condition)
are identical to the $f\equiv0$ case.
\end{proof}

\section{Appendix}

\subsection{Technical Results}

\begin{lemma}\label{lem:R_scaling_equivalence_weight}
 Let $R\geq 1$, $\weight=(\weight_{0},\dots,\weight_{m})$ with $\weight_{j}=(1+|x|)^{|j|p-d}$. Recall the notation $A_{R}\coloneqq \{x\colon R\leq |x|\leq 2 R\}$. Then, there exist constants $0<\widetilde{c}\leq \widetilde{C}<\infty$, such that for any $l\geq 0$, $1\leq p <\infty$,
    \begin{align*}
          \widetilde{c}\|u\|_{W^{m,p}_{\weight}(A_{R})}\leq \left(\sum_{0\leq|j|\leq m}R^{|j|p-d}\|D^{j}u\|_{L^{p}(A_{R})}^{p}\right)^{1/p}\leq \widetilde{C}\|u\|_{W^{m,p}_{\weight}(A_{R})}.
    \end{align*}
   
\end{lemma}
\begin{proof}[Proof of Lemma \ref{lem:R_scaling_equivalence_weight}]\label{proof:lem:R_scaling_equivalence_weight}
  For $\beta>0$ and $|x|\in [R,2R]$,
    \begin{align*}
        3^{-\beta}(1+|x|)^{\beta}\leq 3^{-\beta} ( R+2R)^{\beta}\leq R^{\beta} \leq (1+|x|)^{\beta}.
    \end{align*}
    \begin{align*}
       \frac{1}{(1+|x|)^{\beta}}\leq \frac{1}{(1+R)^{\beta}}\leq \frac{1}{R^{\beta}} \leq \frac{1}{(\frac{1}{3}+\frac{2R}{3})^{\beta}} \leq \frac{3^{\beta}}{(1+|x|)^{\beta}}.
    \end{align*}
        Let 
    $C_{j,p}=\begin{cases}
        1 \text{ if }|j|p-d>0\\
        3 \text{ if }|j|p-d<0.
    \end{cases}$. On the annuli $A_{R}$, we have
    \begin{align*}
        \left(\sum_{0\leq|j|\leq m}R^{|j|p-d}\|D^{j}u\|_{L^{p}(A_{R})}^{p}\right)^{1/p}&=\left(\sum_{0\leq|j|\leq m}R^{|j|p-d}\int_{A_{R}} |D^{j}u(x)|^{p}\dx\right)^{1/p}\\
     &\leq \left(\sum_{0\leq|j|\leq m}C_{j,p,d}\int_{A_{R}} |D^{j}u(x)|^{p}(1+|x|)^{|j|p-d}\dx\right)^{1/p}\\
         &\leq \widetilde{C}\left(\sum_{0\leq|j|\leq m}\int_{A_{R}} |D^{j}u(x)|^{p}(1+|x|)^{|j|p-d}\dx\right)^{1/p}.
    \end{align*}
    Let $c_{j,p}=\begin{cases}
        \frac{1}{3} \text{ if }|j|p-d>0\\
        1 \text{ if }|j|p-d<0,
    \end{cases}$ 
        \begin{align*}
        \left(\sum_{0\leq|j|\leq m}R^{|j|p-d}\|D^{j}u\|_{L^{p}(A_{R})}^{p}\right)^{1/p}&=\left(\sum_{0\leq|j|\leq m}R^{|j|p-d}\int_{A_{R}} |D^{j}u(x)|^{p}\dx\right)^{1/p}\\
     &\geq  \left(\sum_{0\leq|j|\leq m}c_{j,p}^{|j|p-d}\int_{A_{R}} |D^{j}u(x)|^{p}(1+|x|)^{|j|p-d}\dx\right)^{1/p}\\
         &\geq \widetilde{c}\left(\sum_{0\leq|j|\leq m}\int_{A_{R}} |D^{j}u(x)|^{p}(1+|x|)^{|j|p-d}\dx\right)^{1/p}.
    \end{align*}
\end{proof}
\subsection{Proofs of technical results}\label{section:Appendix_proofs}

\begin{proof}[Proof of Proposition \ref{prop:embeddings}]\label{proof:prop:embeddings}
\begin{enumerate}
\item \textbf{Compact embedding.} Let $\{u_{n}\}_{n}$ be a bounded sequence in $W^{s,2}_{\weight}$.  

We split the proof into a local compactness argument and a control at infinity.
\medskip

\textbf{Step 1: Local compactness.}
For $m\in\N$, let $\ball_{2^{m}}(0)$ be the ball of radius $2^{m}$ in $\R^{d}$; these balls are nested and exhaust $\R^{d}$ as $m\to\infty$. The restriction of $W^{s,p}_{\weight}$ to $W^{s,p}_{\weight}(\ball_{2^{m}}(0))$ is a continuous map. By Assumption \eqref{A_embedd:A_bound_from_above_and_below}, the weighted and unweighted norms are equivalent on $\ball_{2^{m}}(0)$, i.e.
    \begin{align*}
       & \sum_{0\leq |j| \leq s} \int_{\ball_{2^{m}}(0)}\left|D^{j} u\right|^{p}  \dx\leq \sum_{0\leq |j| \leq s} \int_{\ball_{2^{m}}(0)}\left|D^j u\right|^{p} \weight_{j}(x) \dx \sup_{x\in \ball_{2^{m}}(0)}\left|\frac{1}{\weight_{j}(x)}\right|,\\
       &\sum_{0\leq |j| \leq s} \int_{\ball_{2^{m}}(0)}\left|D^j u\right|^{p} \weight_{j}(x) \dx \leq C_{m} \sum_{0\leq |j| \leq s} \int_{\ball_{2^{m}}(0)}\left|D^j u\right|^{p}  \dx.
    \end{align*}
Hence $\{u_n\}_n$ is bounded in $W^{s,p}(\ball_{2^{m}}(0))$.
   By Rellich's embedding theorem (see \cite[Theorem 4.12]{adams2003sobolev}), $W^{s,p}(\ball_{2^{m}}(0))\hookrightarrow W^{l,q}(\ball_{2^{m}}(0))$ is compact, when the parameters $s,l,p,q$ satisfy the relation in the statement of the Proposition.
   Setting $(u^{0}_{n})_{n}:=(u_{n})_{n}$, we extract inductively, for each $m\geq1$, a subsequence $(u^{m}_{n})_{n}\subseteq(u^{m-1}_{n})_{n}$ (not relabeled beyond this superscript index) such that
\[
u^{m}_{n}\xrightarrow[n\to\infty]{} u^{m} \quad \text{in } W^{l,q}(\ball_{2^{m}}(0))\ \text{and almost everywhere,}
\]
for some limit $u^{m}\in W^{l,q}(\ball_{2^{m}}(0))$. Since each subsequence is contained in the previous one and the balls are nested, almost-everywhere convergence gives $u^{m+1}=u^{m}$ on the \emph{smaller} ball $\ball_{2^{m}}(0)$. The local limits therefore patch together to a single function $u$ on $\R^{d}$, characterized by $u|_{\ball_{2^{m}}(0)}=u^{m}$ for every $m$. Finally, the \emph{diagonal sequence} $(u^{m}_{m})_{m}$, the $m$-th term of the $m$-th subsequence, satisfies $u^{m}_{m}\to u$ almost everywhere, and, for each fixed $m$, $u^{m}_{m}\to u$ in $W^{l,q}(\ball_{2^{m}}(0))$.

    According to Fatou's Lemma, applied to the diagonal sequence $u^{m}_{m}$,
    \begin{align*}
        \sum_{0\leq |j| \leq s} \int_{\R^{d}}\left|D^{j} u\right|^{p} \weight_{j}(x)\dx\leq \liminf_{m\rightarrow \infty} \sum_{0\leq |j| \leq s} \int_{\R^{d}}\left|D^{j} u^{m}_{m}\right|^{p} \weight_{j}(x) \dx\leq C.
    \end{align*}
    \medskip
    \textbf{Step 2: Control at infinity.}

Let $K$ be as in Assumption \ref{A_embedd:B_decay_weight_ratio}. By that assumption, for every $\eps>0$ there exists $M\geq K$ such that
\[
\frac{\weight_j(x)}{\weightbar_j(x)} \le \varepsilon \qquad\text{for all } |x|\geq M.
\]

Hence for $u \in W^{l,q}_{\weightbar}$,
    \begin{align*}
         &\sum_{0\leq |j| \leq l} \int_{\R^{d}\backslash \ball_{M}(0)}\left|D^j u\right|^{q}  \weight_{j}(x) \dx\leq \sum_{0\leq |j|  \leq l} \int_{\R^{d}\backslash \ball_{M}(0)}\left|D^j u\right|^{q} \frac{\weight_{j}(x)}{\weightprime_{j}(x)} \weightprime_{j}(x) \dx\\
         &\phantom{xx}\leq \eps \sum_{0\leq |j|  \leq l} \int_{\R^{d}\backslash \ball_{M}(0)}\left|D^j u\right|^{q}  \weightprime_{j}(x) \dx.
    \end{align*}
    Applying this to the diagonal sequence $u^{m}_{m} - u$, we obtain
\begin{align*}
\|u^{m}_{m} - u\|_{W^{l,q}_{\weight}}^q
&\le \|u^{m}_{m} - u\|_{W^{l,q}(\ball_{M}(0))}^q
+ C\varepsilon.
\end{align*}

Since $u^{m}_{m} \to u$ in $W^{l,q}(\ball_{2^{m}}(0))$ and $\ball_{M}(0)\subseteq\ball_{2^{m}}(0)$ once $2^{m}\geq M$, letting $m\to\infty$ and then $\varepsilon \to 0$ gives convergence in $W^{l,q}_{\weight}$.

This proves compactness.
\medskip
\item \textbf{Embeddings into V}:
Let $\ball_{R}(0)$ be the ball with radius $R$. 
\begin{enumerate}
    \item On any ball $\ball_{R}$, the condition $I_{\weight_{j}}(\ball_{R}(0)):=\inf_{x\in \ball_{R}(0)}\weight_{j}(x)\geq c_{j,R}$ for every $j\in (\N\cup\{0\})^{d}$, yields
\begin{align*}
\|u\|_{W^{s,p}(\ball_{R})} \le C_{R} \|u\|_{W^{s,p}_{\weight}(\ball_{R})}.
\end{align*}

Thus $u \in W^{s,p}(\ball_{R})$, and classical Sobolev embeddings (see \cite[Theorem 4.12, Part II]{adams2003sobolev}) imply $u \in V$ locally.  
\item On any ball $\ball_{R}$, Assumption $\weight_{j}\in \WeightclassB_{\rho}$ yields
\begin{align*}
\|u\|_{W^{s,p}(\ball_{R})} \le C_R \|u\|_{W^{s,p\rho}_{\weight}(\ball_{R})}.
\end{align*}

 Since we assumed that, for every multi-index $j\in (\N\cup \{0\})^{d}$ with $0\leq |j| \leq s$ $w_{j}\in \mathcal{B}^{\rho}(\R^{d})$ for every $k=0,\dots,s$,
 \begin{align*}
    \left(\int_{\ball_{R}(0)}|D^{j}u(x)|^{p}\dx\right)^{\frac{1}{p}}&= \left(\int_{\ball_{R}(0)}|D^{j}u(x)|^{p}\frac{\weight_{j}(x)^{1/\rho}}{\weight_{j}(x)^{1/\rho}}\dx\right)^{\frac{1}{p}}\\
    &\leq \left(\int_{\ball_{R}(0)}|D^{j}u(x)|^{p\rho}\weight_{j}(x)\dx\right)^{\frac{1}{p\rho}}\left(\int_{\ball_{R}(0)}\frac{1}{\weight_{j}(x)^{\frac{1}{\rho-1}}}\dx\right)^{\frac{\rho-1}{p\rho}}\\
    &\leq C_{R}\|D^{j}u\|_{L^{p\rho}_{\weight}(\ball_{R}(0))}.
\end{align*}

Thus $u \in W^{s,p\rho}(\ball_R(0))$, and classical Sobolev embeddings (see \cite[Theorem 4.12, Part II]{adams2003sobolev}) imply $u \in V$ locally.  
\end{enumerate}

Since continuity and differentiability are local properties, and we assumed that $\weight\in L^{\infty}_{\operatorname{loc}}$, we can shift the center of the ball and obtain the continuity (although not uniformly) and differentiability properties on the whole space.

To obtain the embedding for spaces with higher differentiability ($r>0$), we can repeat the identical steps with the partial derivatives of $u$. This concludes the arguments for the embedding into $V$.
\medskip
\item \textbf{Boundedness}:
 By \cite[Theorem 4.12, Part I]{adams2003sobolev},  $W^{s,p}(\ball_{2R})\hookrightarrow L^{\infty}(\ball_{2R})$ continuously.
 Hence,
  \begin{align*}
    \|u\|_{L^{\infty}(\ball_{R}(0))}&=\|u(R\cdot)\|_{L^{\infty}(\ball_{1}(0))}\\
    &\leq  C_{\ball_{1},s} \|u(R\cdot)\|_{W^{s,p}(\ball_{1}(0))}\leq C_{\ball_{1},s}\left(\sum_{0\leq |j|\leq s}R^{|j|p-d}\|D^{j}u\|_{L^{p}(\ball_{R}(0))}^{p}\right)^{1/p}.
\end{align*}
Analogously for the annuli, 
    \begin{align*}
    \|u\|_{L^{\infty}(A_{R})}
    &\leq C_{A_{1},s}\left(\sum_{0\leq |j|\leq s}R^{|j|p-d}\|D^{j}u\|_{L^{p}(A_{R})}^{p}\right)^{1/p}.
    \end{align*}
We split the arguments into different cases.\newline
\begin{enumerate}
    \item [\textbf{Case 1.}] Let $w_{j}\in \mathcal{B}^{\rho}(\R^{d})$ and \eqref{A_embedd:C_1_bound_Cb_under_weightclass} hold.\newline
  \textbf{Step 1: Local bounds.}
    We expand the integral $\|D^{j}u\|_{L^{p}(\ball_{R}(0))}^{p}$ by $\left(\frac{\weight_{j}(x)}{\weight_{j}(x)}\right)^{\frac{1}{\rho}}$ and use H{\"o}lder's inequality
\begin{align*}
    \|u\|_{L^{\infty}(\ball_{R}(0))}&\leq C_{\ball_{1},s}\left(\sum_{0\leq |j|\leq s}R^{|j|p-d}\|\weight_{j}^{-\frac{1}{\rho-1}}\|_{L^{1}(\ball_{R}(0))}^{\frac{\rho-1}{\rho}}\|D^{j}u\|_{L^{p\rho}_{\weight}(\ball_{R}(0))}^{p/\rho}\right)^{1/p}\\
    &\leq C_{\ball_{1},s}\left(\sum_{0\leq |j|\leq s}R^{|j|p-d}\|\weight_{j}^{-\frac{1}{\rho-1}}\|_{L^{1}(\ball_{R}(0))}^{\frac{\rho-1}{\rho}}\right)^{1/p}\|u\|_{W^{s,p\rho}_{\weight}(\ball_{R}(0))}\\
    &\leq C_{\ball_{1},s}\sum_{0\leq |j|\leq s}R^{|j|-\frac{d}{p}}\|\weight_{j}^{-\frac{1}{\rho-1}}\|_{L^{1}(\ball_{R}(0))}^{\frac{\rho-1}{p\rho}}\|u\|_{W^{s,p\rho}_{\weight}(\ball_{R}(0))}\\
    &    \leq C_{\ball_{1},s}\sum_{0\leq |j|\leq s}R^{|j|-\frac{d}{p}}\|\weight_{j}^{-\frac{1}{\rho-1}}\|_{L^{1}(\ball_{R}(0))}^{\frac{\rho-1}{p\rho}}\|u\|_{W^{s,p\rho}_{\weight}},
\end{align*}  
where we used the subadditivity of $|x|\mapsto |x|^{\frac{1}{p}} $.
    \textbf{Step 2: Estimates on annuli.}
The same estimate as above holds:
    \begin{align*}
    \|u\|_{L^{\infty}(A_{R})}
    &\leq C_{A_{1},s}\left(\sum_{0\leq |j|\leq s}R^{|j|p-d}\|D^{j}u\|_{L^{p}(A_{R})}^{p}\right)^{1/p}\\
    &\leq C_{A_{1},s}\sum_{0\leq |j|\leq s}R^{|j|-\frac{d}{p}}\|\weight_{j}^{-\frac{1}{\rho-1}}\|_{L^{1}(A_{R})}^{\frac{\rho-1}{p\rho}}\|u\|_{W^{s,p\rho}_{\weight}}.
\end{align*}
By Assumption, the terms on the right-hand side are bounded, which yields the claim.
\item[\textbf{Case 2.}] Let $I_{\weight_{j}}(\ball_{R}(0)):=\inf_{x\in \ball_{R}(0)}\weight_{j}(x)\geq C_{j,R}$ for every $j\in (\N\cup\{0\})^{d}$ and \eqref{A_embedd:D_1_bound_Cb_under_non_degeneracy} hold.\newline
  \textbf{Step 1: Local bounds.}
Expanding $\|D^{j}u\|_{L^{p}(\ball_{R}(0))}^{p}$ by $\frac{\weight_{j}(x)}{\weight_{j}(x)}$ and using \eqref{A_embedd:D_1_bound_Cb_under_non_degeneracy}, yields
\begin{align*}
     \|u\|_{L^{\infty}(\ball_{R}(0))}&\leq C_{\ball_{1},s}\left(\sum_{0\leq |j|\leq s}R^{|j|p-d}\frac{1}{I_{\weight_{j}}(\ball_{R}(0))}\|D^{j}u\|_{L^{p}_{\weight}(\ball_{R}(0))}^{p}\right)^{1/p}\\
    &    \leq C_{\ball_{1},s}\sum_{0\leq |j|\leq s}\left(\frac{1}{I_{\weight_{j}}(\ball_{R}(0))}\right)^{\frac{1}{p}}R^{|j|-\frac{d}{p}}\|u\|_{W^{s,p}_{\weight}},
\end{align*}  
where we used the subadditivity of $|x|\mapsto |x|^{\frac{1}{p}} $.\newline
    \textbf{Step 2: Estimates on annuli.}
    The estimates are identical and yield
    \begin{align*}
  \|u\|_{L^{\infty}(A_{R})} 
    &    \leq C_{A_{1},s}\sum_{0\leq |j|\leq s}\left(\frac{1}{I_{\weight_{j}}(A_{R})}\right)^{\frac{1}{p}}R^{|j|-\frac{d}{p}}\|u\|_{W^{s,p}_{\weight}},
\end{align*}
\end{enumerate}
\item \textbf{Boundedness in weighted spaces}:
The bounds are derived almost identically to the unweighted ones. We recall again that by \cite[Theorem 4.12, Part I]{adams2003sobolev},  $W^{s,p}(\ball_{2R})\hookrightarrow L^{\infty}(\ball_{2R})$ continuously. Hence,
\begin{align*}
    \|u\weightc\|_{L^{\infty}(\ball_{R}(0))}&\leq \|\weightc(\cdot)\|_{L^{\infty}(\ball_{R}(0))}\|u(R\cdot)\|_{L^{\infty}(\ball_{R}(0))}\\
    &\leq C_{\ball_{1},s}\|\weightc(\cdot)\|_{L^{\infty}(\ball_{R}(0))}\left(\sum_{0\leq |j|\leq s}R^{|j|p-d}\|D^{j}u\|_{L^{p}(\ball_{R}(0))}^{p}\right)^{1/p}.
\end{align*}
Analogously, 
\begin{align*}
    \|u\weightc\|_{L^{\infty}(A_{R})}&\leq \|\weightc(\cdot)\|_{L^{\infty}(A_{R})}\|u(R\cdot)\|_{L^{\infty}(A_{R})}\\
    &\leq C_{A_{1},s}\|\weightc(\cdot)\|_{L^{\infty}(A_{R})}\left(\sum_{0\leq |j|\leq s}R^{|j|p-d}\|D^{j}u\|_{L^{p}(A_{R})}^{p}\right)^{1/p}.
\end{align*}
\begin{enumerate}
    \item [\textbf{Case 1.}] Let $w_{j}\in \mathcal{B}^{\rho}(\R^{d})$ and \eqref{A_embedd:C_2_bound_Cwc_under_weightclass} hold.\newline
  \textbf{Step 1: Local bounds.}
\begin{align*}
     \|u\weightc\|_{L^{\infty}(\ball_{R}(0))}
    &\leq C_{\ball_{1},s}\|\weightc(\cdot)\|_{L^{\infty}(\ball_{R}(0))}\sum_{0\leq |j|\leq s}R^{|j|-\frac{d}{p}}\|\weight_{j}^{-\frac{1}{\rho-1}}\|_{L^{1}(\ball_{R}(0))}^{\frac{\rho-1}{p\rho}}\|u\|_{W^{s,p\rho}_{\weight}(\ball_{R}(0))}\\
    &    \leq C_{\ball_{1},s}\|\weightc(\cdot)\|_{L^{\infty}(\ball_{R}(0))}\sum_{0\leq |j|\leq s}R^{|j|-\frac{d}{p}}\|\weight_{j}^{-\frac{1}{\rho-1}}\|_{L^{1}(\ball_{R}(0))}^{\frac{\rho-1}{p\rho}}\|u\|_{W^{s,p\rho}_{\weight}}.
\end{align*}  
    \textbf{Step 2: Estimates on annuli.}
Let $A_{R}\coloneqq \ball_{2R}\backslash \overline{B}_{R}$. The same estimate as above holds:
    \begin{align*}
    \|u\weightc\|_{L^{\infty}(A_{R})}
    &\leq C_{A_{1},s}\|\weightc(\cdot)\|_{L^{\infty}(A_{R})}\sum_{0\leq |j|\leq s}R^{|j|-\frac{d}{p}}\|\weight_{j}^{-\frac{1}{\rho-1}}\|_{L^{1}(A_{R})}^{\frac{\rho-1}{p\rho}}\|u\|_{W^{s,p\rho}_{\weight}}.
\end{align*}
By Assumption, the terms on the right-hand side are bounded, which yields the claim.
\item[\textbf{Case 2.}] Let $I_{\weight_{j}}(\ball_{R}(0)):=\inf_{x\in \ball_{R}(0)}\weight_{j}(x)\geq C_{j,R}$ for every $j\in (\N\cup\{0\})^{d}$ and \eqref{A_embedd:D_2_bound_Cwc_under_non_degeneracy} hold.\newline
  \textbf{Step 1: Local bounds.}
\begin{align*}
     \|u\weightc\|_{L^{\infty}(\ball_{R}(0))}
    &    \leq C_{\ball_{1},s}\|\weightc(\cdot)\|_{L^{\infty}(\ball_{R}(0))}\sum_{0\leq |j|\leq s}\left(\frac{1}{I_{\weight_{j}}(\ball_{R}(0))}\right)^{\frac{1}{p}}R^{|j|-\frac{d}{p}}\|u\|_{W^{s,p}_{\weight}}.
\end{align*}  
    \textbf{Step 2: Estimates on annuli.}
    \begin{align*}
  \|u\weightc\|_{L^{\infty}(A_{R})} 
    &    \leq C_{A_{1},s}\|\weightc(\cdot)\|_{L^{\infty}(\ball_{R}(0))}\sum_{0\leq |j|\leq s}\left(\frac{1}{I_{\weight_{j}}(A_{R})}\right)^{\frac{1}{p}}R^{|j|-\frac{d}{p}}\|u\|_{W^{s,p}_{\weight}}.
\end{align*}
\end{enumerate}
\end{enumerate}
\end{proof}

\begin{proof}[Proof of Lemma \ref{lem:weighted_norm_equivalence}]\label{proof:lem:weighted_norm_equivalence}
For the first case,
 \begin{align*}
     \|\cdot\|_{W^{m,p}_{\weight^{1}}}&=\left(\sum_{0\leq |j| \leq m} \int_{\R_{+}^{d}}\left|D^j u\right|^p (a_{1,i}+a_{2,i}|x|)^{\beta_{i}} \dx\right)^{1 / p} \\
     &=\left(\sum_{0\leq |j| \leq m} \int_{\R_{+}^{d}}\left|D^j u\right|^p \frac{(a_{1,i}+a_{2,i}|x|)^{\beta_{i}}}{(b_{1,i}+b_{2,i}|x|)^{\beta_{i}}} (b_{1,i}+b_{2,i}|x|)^{\beta_{i}}\dx\right)^{1 / p}.
 \end{align*}
 Noting that $\min_{i}\min\left\{ \left(\frac{a_{1,i}}{b_{1,i}}\right)^{\beta_{i}}, \left(\frac{a_{2,i}}{b_{2,i}}\right)^{\beta_{i}}\right\}\leq \frac{(a_{1,i}+a_{2,i}|x|)^{\beta_{i}}}{(b_{1,i}+b_{2,i}|x|)^{\beta_{i}}} \leq \max_{i}\max\left\{ \left(\frac{a_{1,i}}{b_{1,i}}\right)^{\beta_{i}}, \left(\frac{a_{2,i}}{b_{2,i}}\right)^{\beta_{i}}\right\} $, yields the result.

 For the second case
  \begin{align*}
     \|\cdot\|_{W^{m,p}_{\frac{1}{\weight^{1}}}}&=\left(\sum_{0\leq |j| \leq m} \int_{\R_{+}^{d}}\left|D^j u\right|^p \frac{1}{(a_{1,i}+a_{2,i}|x|)^{\beta_{i}}} \dx\right)^{1 / p} \\
     &=\left(\sum_{0\leq |j| \leq m} \int_{\R_{+}^{d}}\left|D^j u\right|^p \frac{(b_{1,i}+b_{2,i}|x|)^{\beta_{i}}}{(a_{1,i}+a_{2,i}|x|)^{\beta_{i}}}\frac{1}{(b_{1,i}+b_{2,i}|x|)^{\beta_{i}}} \dx\right)^{1 / p}.
 \end{align*}
 Again, since  $\min_{i}\min\left\{ \left(\frac{b_{1,i}}{a_{1,i}}\right)^{\beta_{i}}, \left(\frac{b_{2,i}}{a_{2,i}}\right)^{\beta_{i}}\right\}\leq \frac{(b_{1,i}+b_{2,i}|x|)^{\beta_{i}}}{(a_{1,i}+a_{2,i}|x|)^{\beta_{i}}} \leq \max_{i}\max\left\{ \left(\frac{b_{1,i}}{a_{1,i}}\right)^{\beta_{i}},\left( \frac{b_{2,i}}{a_{2,i}}\right)^{\beta_{i}}\right\} $, we obtain the result.
\end{proof}

\subsection{Auxiliary Lemmata}
\begin{lemma}\label{lem:uniform_continuity_C_weak}

Let ${^{\prime}\!}U$ be a Banach space and ${^{\prime}\!}V$ a separable Banach space. We denote their dual spaces by $U,V$. Further, assume that the embeddings ${^{\prime}\!}U \hookrightarrow {^{\prime}\!}V\hookrightarrow V\hookrightarrow U$ are continuous and ${^{\prime}\!}U \subset {^{\prime}\!}V$ is dense. Let $B_{r}^{\operatorname{weak}^{*}}$ denote the ball of radius $r$ in $V$, equipped with the weak-$*$ topology. Assume that the following conditions are satisfied,

    \begin{enumerate}[label=\normalfont(\arabic*)]
        \item $u_{n}\rightarrow u$ in $C([0,T];U)$, \label{A:conv_lemma_CU_convergence_1}
        \item $\sup_{t\in [0,T]}\|u_{n}\|_{V}\leq r$. \label{A:conv_lemma_CV_bounded_1}
    \end{enumerate}
    Then  $u, u_{n} \in C\left([0, T] ; B_{r}^{\operatorname{weak}^{*}}\right)$ and $u_{n} \rightarrow u$ in  $C\left([0, T] ; B_{r}^{\operatorname{weak}^{*}}\right)$ as $n \rightarrow \infty$. 
\end{lemma}
\begin{proof}
    \begin{enumerate}
        \item We claim that
        $$
        u_{n} \rightarrow u \quad \text { in } \quad C\left([0, T] ; B_{r}^{\operatorname{weak}^{*}}\right) \quad \text { as } \quad n \rightarrow \infty
        $$
        i.e. that for all $\phi \in {^{\prime}\!}V$
        $$
        \lim _{n \rightarrow \infty} \sup _{s \in[0, T]}\left|\left\langle u_{n}(s)-u(s), \phi\right\rangle_{V\times {^{\prime}\!}V}\right|=0.
        $$
        To verify this claim, we fix $\phi \in {^{\prime}\!}V$ and $\varepsilon>0$. Since ${^{\prime}\!}U $ is dense in ${^{\prime}\!}V$, there exists $\phi_{\varepsilon} \in {^{\prime}\!}U $ such that $\left|\phi-\phi_{\varepsilon}\right|_{{^{\prime}\!}V} \leq \varepsilon$. Using \ref{A:conv_lemma_CV_bounded_1}, we infer that for all $s \in[0, T]$ the following estimates hold (the step from the third to the fourth line below also uses that $\|u(s)\|_{V}\leq r$ for all $s\in[0,T]$, which follows from \ref{A:conv_lemma_CV_bounded_1} together with weak-$*$ lower semicontinuity of the norm $\|\cdot\|_{V}$, since $u_{n}(s)\rightarrow u(s)$ in $U$ and $V\hookrightarrow U$)
    \begin{align*}
            \left|\left\langle u_{n}(s)-u(s) , \phi \right\rangle_{V\times {^{\prime}\!}V}\right| & \leq\left|\left\langle u_{n}(s)-u(s) , \phi-\phi_{\varepsilon}\right\rangle_{V\times {{^{\prime}\!}V}}\right|+\left|\left\langle u_{n}(s)-u(s) , \phi_{\varepsilon}\right\rangle_{U\times {^{\prime}\!}U }\right| \\
             & \leq\left\|u_{n}(s)-u(s)\right\|_{V}\left\|\phi-\phi_{\varepsilon}\right\|_{{^{\prime}\!}V}+\left|\left\langle u_{n}(s)-u(s) , \phi_{\varepsilon}\right\rangle_{U \times {^{\prime}\!}U }\right| \\
            & \leq \varepsilon \cdot\left\|u_{n}-u\right\|_{L^{\infty}(0, T ; V)}+\left|\left\langle u_{n}(s)-u(s) , \phi_{\varepsilon}\right\rangle_{U \times {^{\prime}\!}U }\right| \\
            & \leq 2 \varepsilon \cdot \sup _{n \in \mathbb{N}}\left\|u_{n}\right\|_{L^{\infty}(0, T ; V)}+\left|\left\langle u_{n}(s)-u(s) , \phi_{\varepsilon}\right\rangle_{U \times {^{\prime}\!}U }\right| \\
            & \leq 2 \varepsilon r+\sup _{s \in[0, T]}\left|\left\langle u_{n}(s)-u(s) , \phi_{\varepsilon}\right\rangle_{U \times {^{\prime}\!}U }\right|\\
            & \leq 2 \varepsilon r+\sup _{s \in[0, T]}\left\| u_{n}(s)-u(s) \right\|_{U}\left\| \phi_{\varepsilon}\right\|_{{^{\prime}\!}U }.
        \end{align*}
        Passing to the limit $n \rightarrow \infty,$ we obtain
        $$
        \limsup _{n \rightarrow \infty} \sup _{s \in[0, T]}\left|\left\langle u_{n}(s)-u(s) , \phi\right\rangle_{V\times {^{\prime}\!}V}\right| \leq 2 r \varepsilon.
        $$
        Since $\varepsilon$ is arbitrary,
        $$
        \lim _{n \rightarrow \infty} \sup _{s \in[0, T]}\left|\left\langle u_{n}(s)-u(s) , \phi\right\rangle_{V\times {^{\prime}\!}V}\right|=0,
        $$
        for every $\phi\in {^{\prime}\!}V$.
        Since $C\left([0, T] ; B_{r}^{\operatorname{weak}^{*}}\right)$ is a complete metric space, we infer that $u \in C\left([0, T] ; B_{r}^{\operatorname{weak}^{*}}\right)$ as well. This completes the proof.
    \end{enumerate}
\end{proof}

\subsection{Notation and symbols}\label{section:notation_table}
For the reader's convenience we collect the recurring notation. Throughout,
$n_{\textnormal{dim}}$ is the (fixed) dimension of the state, $m_{W}$ the dimension of the
noise, and $x\in\R_{+}$ the scalar Laplace-dual variable.

\medskip
\noindent\emph{Dimensions and indices.}
\begin{center}
\begin{tabular}{@{}lp{0.72\textwidth}@{}}
$n_{\textnormal{dim}}$ & dimension of the SVE state space, $X_{t}\in\R^{n_{\textnormal{dim}}}$ ($n_{\textnormal{dim}}\geq1$, fixed).\\
$m_{W}$ & dimension of the driving Brownian motion $W$.\\
$d$ & dimension of the spatial domain of the weighted Sobolev spaces (the domain of the Laplace-dual variable $x$); unrelated to $n_{\textnormal{dim}}$.\\
$x$ & Laplace-dual (frequency) variable, $x\in\R_{+}$; always scalar.\\
$n$ & generic sequence/approximation index (e.g.\ $\mu^{n}\to\mu$).\\
\end{tabular}
\end{center}

\medskip
\noindent\emph{Volterra equation and its lift.}
\begin{center}
\begin{tabular}{@{}lp{0.72\textwidth}@{}}
$\drift,\diffusion$ & drift and diffusion coefficients, $\drift\colon\R_{+}\times\R^{n_{\textnormal{dim}}}\to\R^{n_{\textnormal{dim}}}$, $\diffusion\colon\R_{+}\times\R^{n_{\textnormal{dim}}}\to\R^{n_{\textnormal{dim}}\times m_{W}}$.\\
$k_{\drift},k_{\diffusion}$ & completely monotone Volterra kernels (matrix-valued, entrywise).\\
$\nudrift,\nudiffusion$ & Laplace measures of the kernels ($n_{\textnormal{dim}}\times n_{\textnormal{dim}}$ matrices of measures), $k(\theta)=\int_{\R_{+}}e^{-x\theta}\,\nu(\dx)$; regarded as operators $\R^{n_{\textnormal{dim}}}\to V'$.\\
$X_{t}$ & solution of the stochastic Volterra equation, $\R^{n_{\textnormal{dim}}}$-valued.\\
$\mu_{t}$ & Markovian lift of $X$ (distribution-valued); solves the lifted stochastic evolution equation (SEE).\\
$\langle\mu,1\rangle$ & pairing of $\mu$ against the constant test function $1$; recovers the state, $X_{t}=\langle\mu_{t},1\rangle$.\\
$\Gamma_{st}(X)$ & semigroup (path-segment) functional, $\Gamma_{st}=\langle e^{-x(t-s)}\mu_{s},1\rangle$ for $s\leq t$.\\
\end{tabular}
\end{center}

\medskip
\noindent\emph{Weights and their exponents.}
\begin{center}
\begin{tabular}{@{}lp{0.72\textwidth}@{}}
$\weightplus,\weightsim,\weightminus$ & the three weight tiers, $\weightplus\leq\weightsim\leq\weightminus$ pointwise, with $(\weight)_{i}(x)=(1+x)^{2\eta-1+2i}$, $i\geq0$.\\
$\etaplus,\etamid,\etaminus$ & weight exponents, $\etaplus\leq\etamid\leq\etaminus$.\\
$\theta_{\nu}$ & tail-decay exponent of the Laplace measure, $\int_{\R_{+}}(1+x)^{-\theta_{\nu}}\,\nu(\dx)<\infty$, $\theta_{\nu}\in[0,1)$ (larger $\theta_{\nu}$ $=$ more singular kernel).\\
$a_{\drift},a_{\diffusion}$ & time-singularity exponents, $\|S^{*}_{t}\nu^{i}\|\leq C\,t^{-(1-a_{i})}$, $a_{i}\in(0,1]$.\\
\end{tabular}
\end{center}

\medskip
\noindent\emph{Function spaces (Gelfand triple).}
\begin{center}
\begin{tabular}{@{}lp{0.72\textwidth}@{}}
$W^{1,2}_{\weight}$ & weighted Sobolev (test) space; $W^{-1,2}_{1/\weight}$ its dual.\\
$\Wplus,\Wmid,\Wminus$ & the test spaces $W^{1,2}_{\weightplus}\supseteq W^{1,2}_{\weightsim}\supseteq W^{1,2}_{\weightminus}$ (larger weight $=$ smaller space).\\
$\Wplusdual,\Wmiddual,\Wminusdual$ & their duals $W^{-1,2}_{1/\weightplus}\subseteq W^{-1,2}_{1/\weightsim}\subseteq W^{-1,2}_{1/\weightminus}$.\\
$V,H,V'$ & Gelfand triple $V=\Wplusdualn\hookrightarrow H=\Wmiddualn\hookrightarrow V'=\Wminusdualn$; noise space $U=\R^{m_{W}}$.\\
$\Wmidtwodual$ & order-two dual $W^{-2,2}_{1/\weightsim}$; solution paths lie in $C([0,T],\Wmidtwodualn)$.\\
\end{tabular}
\end{center}

\medskip
\noindent\emph{Operators, semigroup and norms.}
\begin{center}
\begin{tabular}{@{}lp{0.72\textwidth}@{}}
$S^{*}_{t}$ & dual semigroup, multiplication by $e^{-tx}$ (i.e.\ $S^{*}_{t}=e^{-tx}$).\\
$A^{*}$ & its generator, $A^{*}\mu=-x\mu$.\\
$\mathcal{L}(\R^{n_{\textnormal{dim}}},\cdot)$ & operator norm of a matrix direction; $\opn$, $\opnmid$, $\opnminus$ denote it into $\Wplusdualn$, $\Wmidtwodualn$, $\Wminusdualn$ respectively.\\
\end{tabular}
\end{center}

\medskip
\noindent\emph{It\^o formula, Feynman--Kac and ergodicity.}
\begin{center}
\begin{tabular}{@{}lp{0.72\textwidth}@{}}
$\mathcal{C}^{1,2}_{T,\nu}$ & regularity class for the singular It\^o formula and the backward Kolmogorov equation.\\
$\zeta_{h},\zeta_{h,h'}$ & first and second tangent (variation) processes.\\
$P_{t}$ & Markov (generalized Feller) semigroup of the lift; $\pi$ its invariant measure.\\
$C_{\mathrm{UE}}$ & uniform-ellipticity constant, $\diffusion\diffusion^{\top}\succeq C_{\mathrm{UE}}^{-1}I_{n_{\textnormal{dim}}}$.\\
$\diffusion^{+}$ & Moore--Penrose pseudo-inverse, $\diffusion^{+}=\diffusion^{\top}(\diffusion\diffusion^{\top})^{-1}$, with $\|\diffusion^{+}\|\leq\sqrt{C_{\mathrm{UE}}}$.\\
\end{tabular}
\end{center}

\section*{Statements and Declarations}

\noindent\textbf{Funding.} This work was supported by grant Y 1235 of the START-program of the Austrian Science Fund (FWF) during the author's employment at the University of Vienna. Parts of this paper were completed during the author's affiliation with the \'Ecole Polytechnique F\'ed\'erale de Lausanne (EPFL), Switzerland.

\smallskip
\noindent\textbf{Competing Interests.} The author has no competing interests to declare that are relevant to the content of this article.

\smallskip
\noindent\textbf{Author Contributions.} Florian Huber is the sole author and is responsible for the entire content of this article.

\smallskip
\noindent\textbf{Data Availability.} Data sharing is not applicable to this article, as no datasets were generated or analysed. This is a theoretical study.

\bibliographystyle{abbrv}
\bibliography{references}

\end{document}